%% file: natural-HFI.tex
\documentclass[reqno]{amsart}
\usepackage{bm,amssymb,mathtools,verbatim,amsfonts,tikz-cd,mathrsfs}
\usepackage{enumitem}
\usepackage{float}

\usepackage[hidelinks,colorlinks=true,linkcolor=blue, citecolor=black,linktocpage=true]{hyperref}
\usepackage{graphicx}
\usepackage[alphabetic]{amsrefs}
\usepackage{caption}
\usepackage{morefloats}
\usepackage[top=1.3in, bottom=1.1in, left=1.3in, right=1.3in,marginpar=1in]{geometry}

\counterwithin{figure}{section}

\newtheorem{thm}{Theorem}[section]
\newtheorem{prop}[thm]{Proposition}

\newtheorem{cor}[thm]{Corollary}
\newtheorem{lem}[thm]{Lemma}

\theoremstyle{definition}
\newtheorem{define}[thm]{Definition}

\theoremstyle{remark}
\newtheorem{rem}[thm]{Remark}
\newtheorem{example}[thm]{Example}

\newcommand{\ve}[1]{\boldsymbol{\mathbf{#1}}}
\newcommand{\R}{\mathbb{R}}

\newcommand{\Z}{\mathbb{Z}}

\newcommand{\C}{\mathbb{C}}

\renewcommand{\d}{\partial}
\renewcommand{\subset}{\subseteq}

\renewcommand{\tilde}{\widetilde}
\renewcommand{\bar}{\overline}

\newcommand{\iso}{\cong}

\DeclareMathOperator{\Diff}{{Diff}}

\DeclareMathOperator{\ev}{{ev}}

\DeclareMathOperator{\gr}{{gr}}

\DeclareMathOperator{\Hom}{{Hom}}
\DeclareMathOperator{\Id}{{Id}}
\DeclareMathOperator{\id}{{id}}

\DeclareMathOperator{\Int}{{int}}

\DeclareMathOperator{\im}{{im}}

\DeclareMathOperator{\Mor}{{Mor}}

\DeclareMathOperator{\Pin}{{Pin}}

\DeclareMathOperator{\Spin}{{Spin}}

\DeclareMathOperator{\Span}{{Span}}
\DeclareMathOperator{\Sym}{{Sym}}

\DeclareMathOperator{\tw}{tw}

\newcommand{\bA}{\mathbb{A}}
\newcommand{\bB}{\mathbb{B}}

\newcommand{\bE}{\mathbb{E}}
\newcommand{\bF}{\mathbb{F}}

\newcommand{\bK}{\mathbb{K}}
\newcommand{\bL}{\mathbb{L}}

\newcommand{\bO}{\mathbb{O}}

\newcommand{\bS}{\mathbb{S}}
\newcommand{\bT}{\mathbb{T}}

\newcommand{\bW}{\mathbb{W}}
\newcommand{\bX}{\mathbb{X}}
\newcommand{\bZ}{\mathbb{Z}}

\newcommand{\bXI}{\mathbb{XI}}

\newcommand{\cB}{\mathcal{B}}
\newcommand{\cC}{\mathcal{C}}
\newcommand{\cD}{\mathcal{D}}

\newcommand{\cF}{\mathcal{F}}
\newcommand{\cG}{\mathcal{G}}
\newcommand{\cH}{\mathcal{H}}
\newcommand{\cI}{\mathcal{I}}

\newcommand{\cL}{\mathcal{L}}
\newcommand{\cM}{\mathcal{M}}

\newcommand{\cQ}{\mathcal{Q}}

\newcommand{\cV}{\mathcal{V}}

\newcommand{\frD}{\mathfrak{D}}

\newcommand{\frF}{\mathfrak{F}}

\newcommand{\frS}{\mathfrak{S}}

\newcommand{\frX}{\mathfrak{X}}
\newcommand{\frY}{\mathfrak{Y}}

\newcommand{\frs}{\mathfrak{s}}
\newcommand{\frt}{\mathfrak{t}}

\newcommand{\frx}{\mathfrak{x}}
\newcommand{\fry}{\mathfrak{y}}

\newcommand{\cCFL}{\mathcal{C\!F\!L}}
\newcommand{\cCFK}{\mathcal{C\hspace{-.5mm}F\hspace{-.3mm}K}}
\newcommand{\cCFKI}{\mathcal{C\hspace{-.5mm}F\hspace{-.3mm}K\hspace{-.3mm} \cI}}
\newcommand{\cCFLI}{\mathcal{C\hspace{-.5mm}F\hspace{-.4mm}\cL\hspace{-.1mm} \cI}}

\newcommand{\CF}{\mathit{CF}}
\newcommand{\HF}{\mathit{HF}}

\newcommand{\CFK}{\mathit{CFK}}

\newcommand{\CFI}{\mathit{CFI}}
\newcommand{\BI}{\mathit{BI}}

\newcommand{\PD}{\mathit{PD}}

\newcommand{\xs}{\ve{x}}
\newcommand{\ys}{\ve{y}}
\newcommand{\zs}{\ve{z}}
\newcommand{\ws}{\ve{w}}

\newcommand{\qs}{\ve{q}}
\newcommand{\as}{\ve{\alpha}}
\newcommand{\bs}{\ve{\beta}}
\newcommand{\gs}{\ve{\gamma}}
\newcommand{\ds}{\ve{\delta}}

\newcommand{\Ds}{\ve{\Delta}}

\newcommand{\Dt}{\Delta}

\renewcommand{\a}{\alpha}
\renewcommand{\b}{\beta}
\newcommand{\g}{\gamma}
\newcommand{\dt}{\delta}
\newcommand{\veps}{\varepsilon}

\newcommand{\SO}{\mathit{SO}}

\usepackage{leftidx}

\DeclareMathOperator{\Cone}{{Cone}}

\newcommand{\Ss}[1]{\scriptstyle{#1}}

\numberwithin{equation}{section}

\newcommand{\ar}{\mathrm{a.r.}}

\newcommand{\cell}{\mathrm{cell}}

\newcommand{\scA}{\mathscr{A}}
\newcommand{\scB}{\mathscr{B}}
\newcommand{\scC}{\mathscr{C}}

\newcommand{\scE}{\mathscr{E}}
\newcommand{\scH}{\mathscr{H}}

\newcommand{\scU}{\mathscr{U}}
\newcommand{\scV}{\mathscr{V}}
\newcommand{\scS}{\mathscr{S}}

\title{Naturality and functoriality in involutive Heegaard Floer homology}

\author{Kristen Hendricks}
\thanks{KH was partially supported by NSF grant DMS-2019396 and a Sloan Research Fellowship.}
\address{Department of Mathematics, Rutgers University, New Brunswick, NJ, USA}
\email{kristen.hendricks@rutgers.edu}
\author{Jennifer Hom}
\thanks{JH was partially supported by NSF grants DMS-1552285 and DMS-2104144.}
\address{School of Mathematics, Georgia Institute of Technology, Atlanta, GA, USA}
\email{hom@math.gatech.edu}
\author{Matthew Stoffregen}
\address{Department of Mathematics, Michigan State University, East Lansing, MI, USA}
\email{stoffre1@msu.edu}
\thanks{MS was partially supported by NSF grant DMS-1702532.}
\author{Ian Zemke}
\address{Department of Mathematics\\Princeton University\\  Princeton, NJ, USA}
\email{izemke@math.princeton.edu}
\thanks{IZ was partially supported by NSF grants DMS-1703685 and 2204375.}

\begin{document}

\begin{abstract}We prove first-order naturality of involutive Heegaard Floer homology, and furthermore construct well-defined maps on involutive Heegaard Floer homology associated to cobordisms between three-manifolds. We also prove analogous naturality and functoriality results for involutive Floer theory for knots and links.  The proof relies on the doubling model for the involution, as well as several variations.\end{abstract}

\maketitle

\tableofcontents

\input{section-00-intro}

\input{section-0-algebra}
\input{section-1-2-doubling-prelims}

\input{section-3a-multi-stabilizations-v2}

\input{section-2-expanded-model}

\input{section-2a-change-models-v2}

\input{section-3-sigma-fixed}

\input{section-4-1handlesv4}

\input{section-5-2handles}

\input{section-6a-stabilizations}

\input{section-7-8-handleswaps}

\input{section-10-final-overview}

\input{section-duality}

\input{section-knots}

\input{section-surgery-formula}

\input{section-twist}

\bibliographystyle{custom}
\def\MR#1{}
\bibliography{biblio}

\end{document}

%% file: section-00-intro.tex
\section{Introduction}

Heegaard Floer homology, defined by Ozsv{\'a}th and Szab{\'o} in the early 2000s \cites{OSDisks, OSProperties}, is a powerful suite of invariants of 3-manifolds, knots and links inside of them, and 4-dimensional cobordisms between them. Given a basepointed 3-manifold $(Y,w)$ together with a choice of $\Spin^c$ structure $\frs$ on $Y$, the initial construction of the invariant goes by choosing a Heegaard diagram $\cH$ for $(Y,w)$ and associating to it a free finitely-generated chain complex $\CF^{-}(\cH,\frs)$ over the ring $\mathbb F[U]$, where $U$ is a variable of degree $-2$. If $\cH$ and $\cH'$ are two Heegaard diagrams representing $(Y,w)$, the Reidemeister-Singer theorem implies that there is some sequence of Heegaard moves that may be applied to $\cH$ to produce a diagram related to $\cH'$ by a basepoint-preserving diffeomorphism isotopic to the identity in $Y$. Ozsv{\'a}th and Szab{\'o} \cite[Section 9.2]{OSDisks} proved that these sequences of moves induce transition maps, that is, chain homotopy equivalences
\[
\Psi_{\cH \to \cH'}\colon \CF^{-}(\cH,\frs) \rightarrow \CF^{-}(\cH',\frs),
\] 
\noindent showing that the isomorphism class of the homology group $\HF^-(\cH) = \oplus_{\frs \in \Spin^c(Y)}\HF^-(\cH, \frs)$ is an invariant of $(Y,w)$. 

Juh{\'a}sz, Thurston, and the fourth author \cite{JTNaturality} proved that this data is \emph{first-order natural} in the sense that the maps $\Psi_{\cH \to \cH'}$ are independent up to chain homotopy of the choice of Heegaard moves and diffeomorphisms isotopic to the identity relating the two diagrams. (Here, ``first-order'' means that the chain homotopies relating the transition maps associated to two possible sequences of Heegaard moves are not known to be independent of the choices involved in their definition.) It follows that Heegaard Floer homology associates to $(Y,w)$ a specific module $\HF^{-}(Y,w)$ rather than an isomorphism class of $\mathbb F[U]$-modules.

In 2015 the first author and Manolescu \cite{HMInvolutive} put additional structure on the Heegaard Floer package in the form of a homotopy involution $\iota$, leading to involutive Heegaard Floer homology. The construction of the invariant goes as follows: given a basepointed 3-manifold $(Y,w)$ together with a conjugation-invariant $\Spin^c$ structure $\frs$ on $Y$, there is a chain isomorphism
\[ \eta \colon \CF^-(\cH, \frs) \to \CF^-(\overline{\cH},\frs)\]
\noindent where $\overline{\cH}$ denotes the conjugate Heegaard diagram. One may then compose with the transition map
\[ \Psi_{\overline{\cH} \to \cH} \colon \CF^-(\overline{\cH}, \frs) \to \CF^{-}(\cH, \frs)\]
\noindent obtaining the map $\iota$. This data is variously packaged either as the pair $(\CF^{-}(Y), \iota)$, called an \emph{$\iota$-complex}, or as the mapping cone
\[
\CFI^{-}(\cH, \frs) = \Cone(\CF^{-}(\cH,\frs) \xrightarrow{Q(1+\iota)} Q\cdot\CF^-(\cH, \frs)[-1]).
\]
\noindent Here, $Q$ is a formal variable of degree $-1$, and we view the cone as a complex over $\mathbb F[U,Q]/Q^2$. Juh{\'a}sz-Thurston-Zemke naturality implies that the equivariant chain homotopy equivalence class of the pair $(\CF^{-}(\cH,\frs), \iota)$, or equivalently the $\mathbb F[U,Q]/Q^2$ chain homotopy equivalence class of $\CFI^{-}(\cH, \frs)$, is an invariant of $Y$ \cite[Section 2.1]{HMInvolutive}. However, this does not suffice to show that involutive Heegaard Floer homology is itself natural.

In this paper we prove that involutive Heegaard Floer homology is a natural invariant of the basepointed 3-manifold $(Y,w)$ together with a choice of framing $\xi = (\xi_1, \xi_2, \xi_3)$ of the oriented normal bundle to the basepoint; equivalently, that transition maps between the involutive complexes induced by Heegaard moves and basepointed diffeomorphisms isotopic to the identity in $Y$ which preserve $\xi$ are unique up to homotopy. Our proof relies on equipping a Heegaard diagram $\cH$ for $(Y,w, \xi)$ with a set of \emph{doubling data} which is used to give a tractable model for the involution. We write $\frD$ for a Heegaard diagram $\cH$ equipped with such a collection of data, and we refer to such a $\frD$ as a \emph{doubling enhanced Heegaard diagram} for $(Y,w,\xi)$. We refer to the involutive Heegaard Floer homology of $(Y,w,\xi)$ as defined using the Heegaard diagram $\cH$ with this model for the involution as $\CFI^{-}(\frD)$. (For more information on the doubling model for the involution, see Section \ref{subsec:doubling-models}.)

 We also describe a set of moves which relate any two choices of doubling enhanced Heegaard diagrams. Given a sequence of such moves relating $\frD$ and $\frD'$, we define a transition map
\[ \Psi_{\frD \to \frD'}: \CFI^{-}(\frD) \rightarrow \CFI^{-}(\frD')\]
\noindent and show they are independent of the sequence chosen, as follows.

\begin{thm}\label{thm:final-naturality}
The transition map $\Psi_{\frD\to \frD'}$ is independent up to $\bF[U,Q]/Q^2$-equivariant chain homotopy from the sequence of moves between doubling enhanced Heegaard diagrams. Furthermore, $\Psi_{\frD\to \frD}\simeq \id$, and $\Psi_{\frD'\to \frD''}\circ \Psi_{\frD\to \frD'}\simeq \Psi_{\frD\to \frD''}$.
\end{thm}

\begin{rem}
Our involutive transition map $\Psi_{\frD\to \frD'}$ contains data equivalent to a pair $(\Psi_{\cH\to \cH'},H_{\frD\to \frD'})$ where $\Psi_{\cH\to \cH'}$ is the non-involutive transition map from $\CF^-(\cH)$ to $\CF^-(\cH')$, and $H_{\frD\to \frD'}$ is a distinguished chain-homotopy between $\Psi_{\cH\to \cH'}\iota_{\frD}$ and $\iota_{\frD'}\Psi_{\cH\to \cH'}$. In this manner, Theorem~\ref{thm:final-naturality} implies that the chain homotopy $H_{\frD\to \frD'}$ is also well-defined, up to a suitable notion of further chain-homotopy.
\end{rem}

Given a 4-dimensional cobordism $W$ with $\partial W = -Y_1 \coprod Y_2$ and a $\Spin^c$ structure $\frs$ on $W$, Ozsv{\'a}th and Szab{\'o} \cite{OSDisks} also constructed cobordism maps 
\[
\CF(W,\frs) : \CF^{-}(Y_1,  w_1,\frs|_{Y_1}) \rightarrow \CF^{-}(Y_2, w_2, \frs|_{Y_2}).\]
\noindent These maps depend on a choice of path $\gamma$ connecting the basepoints in $Y_1$ and $Y_2$ but are otherwise independent of choices made in their definition \cite[Sections 8 and 9]{OSDisks}, cf. also \cite[Corollary F]{ZemGraphTQFT}.

In \cite{HMInvolutive}, the first author and Manolescu gave a construction of cobordism maps on involutive Heegaard Floer homology, but did not show their construction was invariant of a choice of handle decomposition of the cobordism. In this paper we refine the construction to give maps
\[
\CFI(W, \xi, \frs): \CFI^-(Y_1, w_1, \frs|_{Y_1}, \xi_{Y_1}) \rightarrow \CFI^-(Y_2,w_2,\frs|_{Y_2}, \xi_{Y_2})
\]

\noindent associated to a cobordism $W$ from $Y_1$ to $Y_2$, a conjugation-invariant $\Spin^c$ structure on $W$, and a choice of framing $\xi$ of the normal bundle to a choice of path $\gamma$ which induces the framings $\xi_{Y_i}$ of the normal bundle to the basepoint $w_i$ in $Y_i$. We prove these maps are well-defined in the following sense.

\begin{thm}\label{thm:well-defined-cobordism-map}
The cobordism map $\CFI(W, \xi, \frs)$ is well-defined; that is, it depends only on the choice of $W$, $\frs$, $\g$, and $\xi$.
\end{thm}

Our theory has some subtle differences from the Oszv{\'a}th-Szab{\'o} cobordism maps; in particular, the role of duality is not presently obvious. See Remark~\ref{rem:duality} for further discussion. 

Heegaard Floer homology also has a counterpart for knots and links in 3-manifolds, introduced by Ozsv{\'a}th and Szab{\'o} \cite{OSKnots} and independently in the case of knots by J.~Rasmussen \cite{RasmussenKnots}. In its modern form, this theory associates to a Heegaard diagram $\cH$ for a link $L$ in a 3-manifold $Y$, with two basepoints $w_i$ and $z_i$ on each component, a chain complex $\cCFL(\cH)$ over the ring $\mathbb F[\scU_1, \scV_1, \dots, \scU_{\ell}, \scV_{\ell}]$. This again decomposes over $\Spin^c$ structures on $Y$, such that $\cCFL(\cH) = \oplus_{\frs \in \Spin^c(Y)}\cCFL(\cH,\frs)$. As previously there are maps associated to sequence of Heegaard moves between Heegaard diagrams $\cH$ and $\cH'$ for $(Y, L, \ws,\zs)$
\[ \Psi_{\cH \to \cH'}: \cCFL(\cH, \frs) \rightarrow \cCFL(\cH', \frs).\]
These maps are again well-defined up to chain homotopy \cite[Theorem 1.8]{JTNaturality}.   

The first author and Manolescu \cite{HMInvolutive} defined an endomorphism $\iota_L$ on $\cCFL(\cH)$, called the link involution, in analogy with the procedure for defining the 3-manifold involution $\iota$, and showed that the equivariant chain homotopy equivalence class of the pair $\cCFLI(\cH) = (\cCFL(\cH),\iota_L)$ is an invariant of $(Y,L)$. (We note that \cite{HMInvolutive} focused on the case of knots, rather than links, but the same construction works for links). The doubling model for the involution $\iota$ on 3-manifolds may be adapted to a model for the endomorphism $\iota_L$. We use this model to prove naturality and functoriality of the link variant of Floer homology in the following sense. Once again, we let $\frD$ denote an appropriate set of doubling data for a Heegaard diagram $\cH$ for $(Y, L, \ws, \zs)$, and let $\cCFLI(\frD)$ refer to the $\iota_L$-complex defined using this data. This data need \emph{not} include a framing of either basepoint; see the discussion in Section \ref{subsec:twist-knots} for more on this issue. Given a sequence of Heegaard moves relating two doubling-enhanced Heegaard link diagrams, we define a transition map
\[
\Psi_{\frD \rightarrow \frD'} \colon \cCFLI(\frD) \rightarrow \cCFLI(\frD')
\] 
\noindent which are enhanced $\iota_L$-homomorphisms. (Types of morphisms between $\iota_L$ complexes are reviewed in Section \ref{sec:knot-and-link-complexes}.) We then prove the following.
\begin{thm} \label{thm:link-naturality}
The transition maps $\Psi_{\frD \rightarrow \frD'}$ are well-defined up to homotopies of morphisms of $\iota_L$-complexes.
\end{thm}
The fourth author \cite{ZemCFLTQFT}, following work of Juhasz \cite{JCob}, defined maps associated to decorated link cobordisms $(W, \Sigma, \frs)$ between $(Y_1,L_1, \ws_1,\zs_2,\frs|_{Y_1})$ and $(Y_2, L_2, \ws_2,\zs_2, \frs|_{Y_2})$
\[\cCFL(W,\Sigma,\frs): \cCFL(Y_1,L_1, \ws_1,\zs_1,\frs|_{Y_1}) \to \cCFL(Y_2,L_2,\ws_2,\zs_2, \frs|_{Y_2})\]
\noindent which are again invariant of the choices involved in their definitions. Here the pair $(W,\Sigma)$ consists of a compact 4-manifold $W$ with embedded surface $\Sigma$ such that $W$ is a cobordism from $Y_1$ to $Y_2$ and $\partial \Sigma = -L_1 \coprod L_2$, together with some decorations on $\Sigma$. We will be interested in the case that $\Sigma$ is a collection of annuli, each with one boundary component in $Y_1$ and one in $Y_2$, and each decorated with two parallel longitudinal arcs. Maps associated to knot concordances or knot cobordisms have also appeared in \cite{JMConcordance, TNConcordance, AECobordism}.

Given a link cobordism $(W, \Sigma)$ from $(Y_1, L_1, \ws, \zs)$ to $(Y_2, L_2, \ws, \zs)$ consisting of a set of annuli each with two parallel longitudinal arcs as described above, and a $\Spin^c$ structure $\frs$ on $W$ with the property that $\bar{\frs} = \frs + PD([\Sigma])$, we construct cobordism maps
\[
\cCFLI(W, \Sigma, \frs) \colon \cCFLI(Y_1, L_1, \ws, \zs, \frs|_{Y_1}) \rightarrow \cCFLI(Y_2, L_2, \ws, \zs, \frs|_{Y_2}).
\]
\noindent which are again enhanced homomorphisms of $\iota_L$-complexes. We prove the following functoriality theorem.

\begin{thm} \label{thm:link-functoriality}
The cobordism maps $\cCFLI(W, \Sigma, \frs)$ are well-defined up to homotopies of $\iota_L$-complexes.
\end{thm}

\subsection{Computations and examples}

In our paper, we also perform several example computations. The first example we present is for the cobordism $W=S^2\times S^2$, with its unique self-conjugate $\Spin^c$ structure $\frs$. If we remove two 4-balls from $W$, we obtain a cobordism from $S^3$ to $S^3$, and hence a map
\[
\CFI(W,\xi, \frs)\colon \bF[U,Q]/Q^2\to \bF[U,Q]/Q^2.
\]
In Section~\ref{sec:example}, we show via a direct diagrammatic computation that $\CFI(W, \xi, \frs)$ is multiplication by $Q$. (The choice of framings does not affect the result in this case.) This agrees with the prediction from $\Pin(2)$-equivariant monopole Floer homology; see the proof of \cite{LinExact}*{Theorem~5}.

 As a second example, we compute certain 2-handle cobordisms using the knot surgery formula from \cite{HHSZExact}. We recall that to a knot $K\subset S^3$, Ozsv\'ath and Szab\'{o} \cite{OSIntegerSurgeries}*{Theorem~1.1} described a complex $\bX_n(K)$ which computes $\CF^-(S^3_n(K))$. Therein they also described a way of computing the non-involutive cobordism map $\CF(W_n(K),\frs)$ in terms of $\cCFK(K)$, where $W_n(K)$ denotes the natural 2-handle cobordism from $S^3$ to $S^3_n(K)$, for each $\frs\in \Spin^c(W_n(K))$. In \cite{HHSZExact}, we described an involutive refinement of Ozsv\'{a}th and Szab\'{o}'s mapping cone formula, which we denoted by $\bXI_n(K)$, and which computes $\CFI(S^3_n(K))$. In Section~\ref{sec:knot-surgery}, we describe a refinement of our work from \cite{HHSZExact} which computes certain cobordism maps. From the description in \cite{HHSZExact}, it turns out to be most convenient to consider the cobordism $W_n'(K)$  from $S^3_n(K)$ to $S^3$ obtained by reversing the orientation of $W_n(K)$. This manifold possesses a self-conjugate $\Spin^c$ structure if and only if $n$ is even. In Theorem~\ref{thm:knot-surgery-cobordism-map}, we describe an algebraic formula for the involutive cobordism map $\CFI(W_{2n}'(K),\xi, \frs)$ when $\frs$ is the unique self-conjugate $\Spin^c$ structure on $W_{2n}'(K)$. Again the choice of framings does not affect the computation.

\subsection{The doubling model of the involution and several variations} \label{subsec:doubling-models} The constructions of this paper rely on a model for the involution $\iota$, used in our previous paper \cite{HHSZExact}, based on the procedure of \emph{doubling} a Heegaard diagram. We now briefly recall this model. If $\Sigma$ is a Heegaard splitting of $Y$, containing a basepoint $w$, we construct another Heegaard splitting $D(\Sigma)$ of $Y$, with  $D(\Sigma)\iso \Sigma\# \bar{\Sigma}$, which is embedded as the boundary of a regular neighborhood of $\Sigma\setminus N(w)$. A schematic appears in Figure \ref{fig:double}. The attaching curves for doubled diagrams are described in Section~\ref{sec:doubling-def}. An important property of the doubling operation is that if $\cH$ is a diagram for $Y$, and $D(\cH)$ is a double, then the maps relating $\CF^-(\cH)$ and  $\CF^-(D(\cH))$ have a conceptually simple form, and similarly for the maps from $\CF^-(D(\cH))$ to $\CF^-(\bar{\cH})$; see Section \ref{sec:doubling-model-manifolds}. (Doubled Heegaard diagrams had previously been considered in \cite{ZemDuality}, \cite{JZContactHandles}, and \cite{JZStabilizationDistance}.)

 \begin{figure}[H]
	\centering
	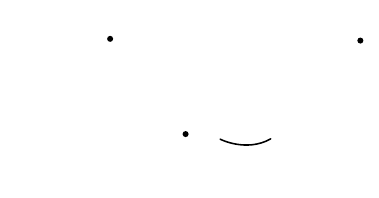
	\caption{Realizing the involution on $S^3$ by doubling.
	 }\label{fig:double}
\end{figure}

The proofs in the present paper require additional variations of this model for the involution, based on adding extra basepoints or extra tubes between the two copies of $\Sigma$ and $\bar \Sigma$.  In the 3-manifold case, we begin by describing transition maps between the involutive complexes associated to Heegaard diagrams for a 3-manifold which are related by \emph{elementary equivalences} in terms of maps on the complexes associated to the doubled diagrams, and define maps for general handleslide equivalences as compositions of the maps for elementary equivalences. However, it is somewhat difficult to see that this construction fails to depend on the sequence of elementary equivalences chosen. To show this, we consider an expanded version of this model, in which an additional basepoint and pair of curves are added in the doubling region. We introduce this model in detail in Section~\ref{sec:expanded-def}. Using this model we define transition maps for handleslide equivalences which are more obviously invariant of the choices involved, and then use this to prove invariance of the transition maps for the standard doubling model. In the case of knots and links, our situation is somewhat more technically complex, so for simplicity we in fact work solely with the analog of the expanded model for the involution. Throughout, the curve-counting arguments in our computations rely heavily on the cylindrical formulation of \cite{LipshitzCylindrical}.

\begin{rem} We emphasize here that there are two possible doubling models one could use to construct the map $\iota$ and our cobordism maps. The model described above and pictured in Figure~\ref{fig:double} is called the \emph{beta-doubling} model of the involution; the analogous construction with the roles of the alpha and beta curves reversed is the \emph{alpha-doubling} model. This distinction is examined more closely in Section~\ref{sec:models}. Our techniques also show naturality and functoriality for involutive Heegaard Floer homology constructed using the alpha-doubling model. We do not, however, attempt to relate these two models.
\end{rem}

\subsection{The role of framings} Theorem~\ref{thm:final-naturality} differs from Juh{\'a}sz-Thurston-Zemke naturality in the incorporation of a framing of the normal bundle to the basepoint $w$ in $Y$.  In particular, $\CFI(Y,w)$ is acted on not by $\Diff^+(Y,w)$, the orientation-preserving diffeomorphisms of $(Y,w)$, but by $\Diff^f(Y,w)$, the diffeomorphisms of $(Y,w)$ which preserving the framing $\xi$ at the basepoint.  The Dehn twist $\tw$ in a neighborhood of $w$ is an example of a class in $\pi_0(\Diff(Y,w,\xi))$ which is trivial in $\pi_0(\Diff(Y,w))$, but is not \emph{a priori} trivial in $\pi_0(\Diff(Y,w,\xi))$.  We calculate: 
\begin{thm}
\label{thm:twist-map} The element $\tw\in \pi_0(\Diff^f(Y,w))$ acts on $\CFI(Y,w)$ by $\Id+Q\Phi $. 
\end{thm}
Here $\Phi$ is a formal derivative of the differential with respect to $U$; for a review of this map, see Section~\ref{subsec:iotacomplexes}. The action of $\tw$ on $\CFI(Y,w)$ can also be interpreted as an action of loops in the space of framings of $(Y,w)$, as in Section ~\ref{sec:twist}. The space of such framings is a copy of $\SO(3)$, and therefore has fundamental group $\pi_1(\SO(3))\iso \Z/2$, whose nontrivial element acts as $\tw$ above.

In the doubling model described above, the choice of framing affects the choice of curves in the doubled diagram; concretely, it corresponds to a Dehn twist around a meridian of the connect sum tube in $\Sigma\#\overline{\Sigma}$, as in Figure~\ref{fig:dehn}.

\begin{figure}[H]
	\centering
	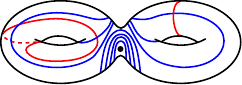
	\caption{The effect of changing the framing on the doubled diagram from Figure~\ref{fig:double}.
	 }\label{fig:dehn}
\end{figure}

The reader will note that we do not require a framing of the basepoint for the naturality and functoriality results for links. Heuristically, this follows from the fact that the connect sum region of (any variation of) the doubled diagram for a link does not contain a basepoint; for more detail, see Section~\ref{subsec:twist-knots}.

We expect that the dependence on the framings and the appearance of the map $\Phi$ in this computation are related to a conjectural $\Pin(2)$-equivariant structure on Heegaard Floer homology, in analogy with Manolescu's $\Pin(2)$-Seiberg Witten Floer homology and Lin's $\Pin(2)$-monopole Floer homology, and hope to investigate this connection in future work. We further note that the appearance of the framings in our theorems implies that any infinite order naturality statement for Heegaard Floer homology would necessarily also take framings of the basepoint into account.

\subsection{Organization} This paper is organized as follows. In Section \ref{sec:algebra} we discuss some algebraic preliminaries and their Floer-theoretic counterparts; in particular, we discuss hypercubes of chain complexes and hypercubes of attaching curves for a Heegaard surface. We additionally review the algebraic formalism of $\iota$-complexes and their relationship to hypercubes of chain complexes. In Section \ref{sec:involutive} we review some background on involutive Heegaard Floer homology, and in particular recall the operation of doubling a Heegaard diagram and the resulting model for the involution on Heegaard Floer homology from \cite{HHSZExact}. In Section \ref{sec:naturality} we recall the structure of Juh{\'a}sz, Thurston, and Zemke's proof of naturality of ordinary Heegaard Floer homology and describe how to adapt it to our case. In Section \ref{sec:elementary-equivalences} we define transition maps for Heegaard diagrams related by elementary equivalences, and define general transition maps for Heegaard diagrams related by handleslide equivalences as compositions of maps for elementary equivalences. We prove in Proposition \ref{prop:continuity} that these maps satisfy the continuity axiom, in analogy with \cite[Proposition 9.27]{JTNaturality}. In Section \ref{sec:polygons} we prove some necessary technical results concerning stabilizations and holomorphic polygons, generalizing the results from our previous paper \cite{HHSZExact}. In Section \ref{sec:expanded-def} we introduce the basepoint exanded doubling diagrams and the basepoint expanded model of the involution, and construct transition maps for handleslide equivalences using this model. In Section \ref{sec:relate-expanded} we define a chain homotopy equivalence between involutive chain complexes associated to the ordinary and basepoint expanded doubles of a Heegaard diagram, and prove that these maps commute with the transitions maps associated to handleslide equivalences. We further show that the transition maps defined in Section \ref{sec:expanded-def} for generalized handleslide equivalences are unique up to chain homotopy equivalence, which proves the result for the non-expanded model. We then turn our attention to constructing the maps between involutive complexes associated to cobordism; in Section \ref{sec:1-and-3-handles} we construct the maps associated to one- and three-handles and in Section \ref{sec:2-handles} we construct the maps associated to two-handles. In Section \ref{sec:stable} we define construct naturality maps for stabilizations as compositions of cobordism maps associated to cancelling handles and prove they commute with the transition maps associated to handleslide equivalences and with each other. In Section \ref{sec:handleswap} we prove invariance of our transition maps under handleswaps, in analogy with \cite[Section 9.3]{JTNaturality}. Finally, in Section~\ref{sec:final-overview} we complete the proofs of Theorem~\ref{thm:final-naturality}, following the structure of the proof in \cite[Section 9.4]{JTNaturality}, and of Theorem~\ref{thm:well-defined-cobordism-map}. As an example, in Section \ref{sec:example} we present the calculation of the cobordism map for a twice punctured copy of $S^2\times S^2$, decomposed as two handle $2$-handle cobordisms $S^3\to S^1\times S^2$  and $S^1\times S^2\to S^3$. In Section \ref{sec:knots} we describe how to adapt the naturality and functoriality results of the previous sections to the case of knots and links, and prove Theorems \ref{thm:link-naturality} and \ref{thm:link-functoriality}. We use this in Section \ref{sec:knot-surgery} to provide a computation of the involutive cobordism map associated to the cobordism $W'_n(K)$ from $S^3_n(K)$ to $S^3$ for $n$ even in terms of the involutive mapping cone formula. We conclude by analyzing the effect of changing the framing in Section ~\ref{sec:twist}, giving a proof of Theorem~\ref{thm:twist-map}.

\subsection*{Acknowledgments} The authors are grateful to Andr{\'a}s Juh{\'a}sz, Ciprian Manolescu, and Dylan Thurston for helpful conversations (over many years). The authors are further grateful to the anonymous referee for helpful comments. Parts of this work were carried out while the first and second author were in residence at the Mathematical Sciences Research Institute in Fall 2022, and therefore supported by NSF Grant DMS-1440140; the authors are grateful to the institute for its hospitality.

%% file: double.pdf_tex
\begingroup%
  \makeatletter%
  \providecommand\color[2][]{%
    \errmessage{(Inkscape) Color is used for the text in Inkscape, but the package 'color.sty' is not loaded}%
    \renewcommand\color[2][]{}%
  }%
  \providecommand\transparent[1]{%
    \errmessage{(Inkscape) Transparency is used (non-zero) for the text in Inkscape, but the package 'transparent.sty' is not loaded}%
    \renewcommand\transparent[1]{}%
  }%
  \providecommand\rotatebox[2]{#2}%
  \newcommand*\fsize{\dimexpr\f@size pt\relax}%
  \newcommand*\lineheight[1]{\fontsize{\fsize}{#1\fsize}\selectfont}%
  \ifx\svgwidth\undefined%
    \setlength{\unitlength}{181.49601871bp}%
    \ifx\svgscale\undefined%
      \relax%
    \else%
      \setlength{\unitlength}{\unitlength * \real{\svgscale}}%
    \fi%
  \else%
    \setlength{\unitlength}{\svgwidth}%
  \fi%
  \global\let\svgwidth\undefined%
  \global\let\svgscale\undefined%
  \makeatother%
  \begin{picture}(1,0.52198877)%
    \lineheight{1}%
    \setlength\tabcolsep{0pt}%
    \put(0,0){\includegraphics[width=\unitlength,page=1]{double.pdf}}%
    \put(0.03419793,0.27993578){\color[rgb]{0,0,0}\makebox(0,0)[t]{\lineheight{1.25}\smash{\begin{tabular}[t]{c}$\cH$\end{tabular}}}}%
    \put(0.93131096,0.25697733){\color[rgb]{0,0,0}\makebox(0,0)[t]{\lineheight{1.25}\smash{\begin{tabular}[t]{c}$\bar{\cH}$\end{tabular}}}}%
    \put(0.49047204,0.00443363){\color[rgb]{0,0,0}\makebox(0,0)[t]{\lineheight{1.25}\smash{\begin{tabular}[t]{c}$D(\cH)$\end{tabular}}}}%
    \put(0,0){\includegraphics[width=\unitlength,page=2]{double.pdf}}%
  \end{picture}%
\endgroup%

%% file: dehn.pdf_tex
\begingroup%
  \makeatletter%
  \providecommand\color[2][]{%
    \errmessage{(Inkscape) Color is used for the text in Inkscape, but the package 'color.sty' is not loaded}%
    \renewcommand\color[2][]{}%
  }%
  \providecommand\transparent[1]{%
    \errmessage{(Inkscape) Transparency is used (non-zero) for the text in Inkscape, but the package 'transparent.sty' is not loaded}%
    \renewcommand\transparent[1]{}%
  }%
  \providecommand\rotatebox[2]{#2}%
  \newcommand*\fsize{\dimexpr\f@size pt\relax}%
  \newcommand*\lineheight[1]{\fontsize{\fsize}{#1\fsize}\selectfont}%
  \ifx\svgwidth\undefined%
    \setlength{\unitlength}{115.92867116bp}%
    \ifx\svgscale\undefined%
      \relax%
    \else%
      \setlength{\unitlength}{\unitlength * \real{\svgscale}}%
    \fi%
  \else%
    \setlength{\unitlength}{\svgwidth}%
  \fi%
  \global\let\svgwidth\undefined%
  \global\let\svgscale\undefined%
  \makeatother%
  \begin{picture}(1,0.3484296)%
    \lineheight{1}%
    \setlength\tabcolsep{0pt}%
    \put(0,0){\includegraphics[width=\unitlength,page=1]{dehn.pdf}}%
  \end{picture}%
\endgroup%

%% file: section-0-algebra.tex
\section{Algebraic preliminaries} \label{sec:algebra}

\subsection{Hypercubes and hyperboxes}
\label{sec:hypercubes}
In this section, we recall the hypercube formalism of Manolescu and Ozsv\'{a}th \cite{MOIntegerSurgery}. We write $\bE_n$ for the set of points $\{0,1\}^n\subset \Z^n$. Similarly if $\ve{d}=(d_1,\dots, d_n)$ is a tuple of $n$ positive integers, we write $\bE(\ve{d})=\{0,\dots, d_1\}\times \cdots \times \{0,\dots, d_n\}$.  If $\veps,\veps'\in \R^n$, we say that $\veps\le\veps'$ if the inequality holds for each component of $\veps$ and $\veps'$.

\begin{define} 
An \emph{$n$-dimensional hypercube of chain complexes} $(C^\veps, D^{\veps,\veps'})_{\veps\in \bE_n}$ is a collection of groups $C^{\veps}$, ranging over $\veps\in \bE_n$, together with a collection of maps $D^{\veps,\veps'}\colon C^\veps\to C^{\veps'}$, whenever $\veps\le \veps'$.  Furthermore, we assume the following compatibility condition is satisfied whenever $\veps\le \veps''$:
\begin{equation}
\sum_{\veps': \veps\le \veps'\le\veps''} D^{\veps',\veps''}\circ D^{\veps,\veps'}=0. \label{eq:hypercube-structure-relations}
\end{equation}
\end{define}

We usually call $D^{\veps,\veps}$ the internal differential of $C^\veps$. The hypercube structure relation in equation~\eqref{eq:hypercube-structure-relations} is equivalent to $(\bigoplus_{\veps\in \bE_n} C^\veps,\sum_{\veps\le \veps'} D^{\veps,\veps'})$ being a chain complex.

\begin{define}
 A \emph{hyperbox of chain complexes of size~$\ve{d}$} consists of a collection of groups $(C^\veps)_{\veps\in \bE(\ve{d})}$, together with a choice of map $D^{\veps,\veps'}\colon C^\veps\to C^{\veps'}$ whenever $|\veps'-\veps|_{L^\infty}\le 1$. We assume equation~\eqref{eq:hypercube-structure-relations} holds whenever $|\veps''-\veps|_{L^\infty}\le 1$. 
\end{define}

An important operation involving hyperboxes is \emph{compression} \cite{MOIntegerSurgery}*{Section~5}. This operation takes a hyperbox of size $\ve{d}=(d_1,\dots, d_n)$ and returns an $n$-dimensional hypercube of chain complexes. We illustrate with an example. Consider a 1-dimensional hyperbox
\[
\begin{tikzcd}
C_0
	\ar[r, "f_{0,1}"]
&
C_1
	\ar[r, "f_{1,2}"]
&
\cdots
	\ar[r, "f_{n-1,n}"]
&
C_n.
\end{tikzcd}
\]
The compression of the above hyperbox is the hypercube
\[
\begin{tikzcd}[column sep=2.5cm]
C_0
\ar[r, "f_{n-1,n}\circ \cdots \circ f_{0,1}"]
&
C_n.
\end{tikzcd}
\]

\noindent More generally, the same description works for an $n$-dimensional hyperbox $\cC$ of size $(1,\dots, 1,d)$. That is, we may view $\cC$ as a 1-dimensional hyperbox of size $(d)$, which we call $\cC'$, where the complex at each point $\varepsilon_0$ of $\bE(d)$ is the $(n-1)$-dimensional hypercube $\cC'_{\varepsilon_0}=(\oplus_{\varepsilon\in \bE_{n-1}} C^{\varepsilon\varepsilon_0},\sum_{\varepsilon\leq \varepsilon'}D^{\varepsilon \varepsilon_0,\varepsilon'\varepsilon_0})$, viewed as a chain complex.  Schematically, $\cC'$ is represented as a diagram
\[
\begin{tikzcd}
\cC'_0
\ar[r, "f_{0,1}"]
&
\cC'_1
\ar[r, "f_{1,2}"]
&
\cdots
\ar[r, "f_{d-1,d}"]
&
\cC'_d.
\end{tikzcd}
\]
where the maps $f_{i,i+1}$ are constructed as a sum of the maps $D^{\varepsilon i,\varepsilon' i+1}$ for $\varepsilon\leq \varepsilon'\in \bE_{n-1}$.

The definition given above for compression of a $1$-dimensional hyperbox of chain complexes then applies, to construct a $1$-dimensional hypercube:
\[
\begin{tikzcd}[column sep=2.5cm]
\cC'_0
\ar[r, "f_{d-1,d}\circ \cdots \circ f_{0,1}"]
&
\cC'_n.
\end{tikzcd}
\]
For $\varepsilon\leq \varepsilon'\in \bE_n$ with $\varepsilon_n\neq \varepsilon_n'$, the map $D^{\varepsilon,\varepsilon'}$ of the compression of $\cC$ is the component of $f_{d-1,d}\circ\dots \circ f_{0,1}$ from $\cC^\varepsilon$ to $\cC^{\varepsilon'}$.

For a hyperbox $C$ of size $(d_1,\dots, d_n)$, the above description may be iterated to give the compression, as follows. This description depends on a choice of ordering of the axes of the cube. (The motivated reader may verify that re-ordering the axes results in a homotopic hypercube.) Firstly, we may view $C$ as a collection of $d_1\cdots d_{n-1}$ hyperboxes of size $(1,\dots, 1,d_n)$. We compress each of these hyperboxes using the function-composition description above. Stacking these hyperboxes results in a hyperbox of size $(d_1,\dots, d_{n-1},1)$. We may view this as a collection of $d_1\cdots d_{n-2}$ hypercubes of size $(1,\dots, 1,d_{n-1}, 1)$, which we compress using the function-composition description. Stacking the resulting hyperboxes gives a hyperbox of size $(d_1,\dots, d_{n-2},1,1)$. Repeating this procedure gives the compression of $C$. It is possible to relate this description to the description in terms of the algebra of songs in \cite{MOIntegerSurgery}, cf. \cite{Liu2Bridge}*{Section~4.1.2}.

\subsection{Hypercubes of attaching curves}

In this section, we describe the notion of a hypercube in the Fukaya category. Such objects are described by Manolescu and Ozsv\'{a}th \cite{MOIntegerSurgery}, and are referred to as \emph{hyperboxes of Heegaard diagrams}. The reader may also compare this construction to Lipshitz, Ozsv\'{a}th and Thurston's notion of a \emph{chain complex of attaching circles} \cite{LOTDoubleBranchedII}. The construction is also described in \cite{HHSZExact}*{Section~5.5}.

We begin with a preliminary definition about $\Spin^c$ structures. Firstly, we say $\cL_{\b}=(\bs^\veps)_{\veps\in \bE_n}$ is an \emph{empty hypercube of beta (or alpha)-attaching curves} if it is a collection of attaching curves on a surface $\Sigma$, indexed by points of $\bE_n$. Let $Y_{\veps_1, \veps_2}$ be the 3-manifold determined by the two sets of attaching curves $(\bs^{\veps_1}, \bs^{\veps_2})$ and $X_{\veps_1,\dots, \veps_m}$ be the smooth 4-manifold associated to the collection of sets of attaching curves $(\bs^{\veps_1},\dots, \bs^{\veps_m})$ in the usual way \cite[Section 8.1]{OSDisks}. 
\begin{define}
\begin{enumerate}
\item Suppose that $\cL_{\b}=(\bs^\veps)_{\veps\in \bE_n}$ is an empty hypercube of beta-curves. A \emph{hypercube of $\Spin^c$-structures} for $\cL_{\b}$ consists of a collection of $\Spin^c$ structures
\[
\frS_{\veps_1,\dots, \veps_m}\subset \Spin^c(X_{\veps_1,\dots, \veps_m})
\] 
for each sequence $\veps_1<\cdots< \veps_m$ in $\bE_n$, satisfying $m>1$ and the following compatibility relations. Firstly, if $1\le i<j\le m$, then $\frS_{\veps_1,\dots, \veps_m}$ is closed under the action of $\delta^1 H^1(Y_{\veps_i,\veps_j})$. Secondly if $\veps_1<\cdots< \veps_m$ is a sequence, and $\veps_{i_1}<\dots< \veps_{i_j}$ is a subsequence, where $1\le i_1<\dots< i_j\le m$, then $\frS_{\veps_{i_1},\dots, \veps_{i_j}}$ is the image of $\frS_{\veps_1,\dots, \veps_m}$ under the natural restriction map
\[
\Spin^c(X_{\veps_1,\dots, \veps_m})\to \Spin^c(X_{\veps_{i_1},\dots, \veps_{i_j}}).
\]
\item If $\cL_{\a}=(\as^\veps)_{\veps\in \bE_n}$ and $\cL_{\b}=(\bs^\veps)_{\veps\in \bE_m}$ are two empty hypercubes of alpha and beta attaching curves, then a hypercube of $\Spin^c$ structures for the pair $(\cL_{\a},\cL_{\b})$ consists of the following. For every pair of sequences $(\nu_1<\dots<\nu_k)$ in $\bE_n$ and $(\veps_1<\dots< \veps_\ell)$ in $\bE_m$ such that $k+\ell>1$, a set of $\Spin^c$ structures $\frS_{\nu_k,\dots, \nu_1,\veps_1,\dots, \veps_\ell}\subset\Spin^c(X_{\nu_k,\dots, \nu_1,\veps_1,\dots, \veps_{\ell}})$ that are closed under $\delta^1$-orbits and compatible with restriction, similar to above. The case that one of $(\nu_1<\dots<\nu_k)$ and $(\veps_1<\dots <\veps_\ell)$ is the empty list is allowed.
\end{enumerate}
\end{define}

Note that a hypercube of $\Spin^c$ structures for the pair $(\cL_{\a}, \cL_{\b})$ induces a hypercube of $\Spin^c$ structures for each of $\cL_{\a}$ and $\cL_{\b}$.

\begin{define} 
An $n$-dimensional hypercube of beta attaching curves $\cL_{\b}$ on a pointed surface $(\Sigma,w)$ consists of an empty hypercube of beta attaching curves $(\bs^\veps)_{\veps\in \bE_n}$, a hypercube of $\Spin^c$ structures on $\cL_{\b}$, together with a choice of Floer chain $\Theta_{\veps,\veps'}\in \CF^-(\Sigma, \bs^\veps,\bs^{\veps'}, w)$ whenever $\veps<\veps'$. Furthermore, the chains are required to satisfy the following compatibility condition, whenever $\veps<\veps'$:
\begin{equation}
\sum_{\veps=\veps_1<\cdots<\veps_n=\veps'} f_{\b^{\veps_1},\b^{\veps_2},\dots, \b^{\veps_n}} (\Theta_{\veps_1,\veps_2},\dots, \Theta_{\veps_{n-1},\veps_n})=0.
\label{eq:hypercube-Lagrangians}
\end{equation}

Here $f_{\b^{\veps_1},\b^{\veps_2},\dots, \b^{\veps_n}}$ is the polygon-counting map associated to $(\b^{\veps_1},\b^{\veps_2},\dots, \b^{\veps_n})$ \cite[Section 8.1]{OSDisks}. The sum is taken over $\Spin^c$ structures according to the hypercube of $\Spin^c$ structures.
\end{define}

We will also write the Floer chain $\Theta_{\veps,\veps'}$ as $\Theta_{\b^{\veps}, \b^{\veps'}}$ when it seems clearest to refer to the sets of curves involved. In the above, we assume that each Heegaard multi-diagram is weakly admissible, and the appropriate finiteness of counts so that the sum makes sense. Assuming weak admissibility, one may always obtain a sensible expression by working over the power series ring $\bF[[U]]$. Hypercubes of \emph{alpha} attaching curves are defined by a notational modification of the above definition.

We may \emph{pair} hypercubes of attaching curves, as follows. If $\cL_{\a}=(\as^\nu,\Theta_{\nu',\nu})_{\nu\in \bE_m}$ and $\cL_{\b}=(\bs^\veps,\Theta_{\veps,\veps'})_{\veps\in \bE_n}$ are hypercubes of attaching curves for some hypercubes of $\Spin^c$ structures, and we have a hypercube of $\Spin^c$ structures $\frS$ for the pair $(\cL_{\a},\cL_{\b})$ which extends the hypercubes for $\cL_{\a}$ and $\cL_{\b}$, then we may form an $(n+m)$-dimensional hypercube of chain complexes denoted $\CF^-(\cL_{\a},\cL_{\b}, \frS)$. The group at index $(\nu,\veps)$ is the ordinary Floer complex $\CF^-(\Sigma, \as^{\nu},\bs^\veps,\frS_{\nu,\veps})$. The morphism from $(\nu,\veps)$ to $(\nu',\veps')$ is the map
\[
D_{(\nu,\veps),(\nu',\veps')}(\xs)=\sum_{\substack{\nu=\nu_1<\cdots<\nu_k=\nu'\\ \veps=\veps_1<\cdots<\veps_j=\veps'}}f_{\nu_k,\dots, \nu_1,\veps_1,\dots, \veps_j,\frS_{\nu_k,\dots, \nu_1,\veps_1,\dots, \veps_j}}(\Theta_{\nu_k,\nu_{k-1}},\dots, \Theta_{\nu_2,\nu_1}, \ve{x},\Theta_{\veps_1,\veps_2},\dots, \Theta_{\veps_{j-1},\veps_{j}}).
\]

It is straightforward to use the compatibility condition in~\eqref{eq:hypercube-Lagrangians} to see that $\CF^-(\cL_{\a},\cL_{\b},\frS)$ is a hypercube of chain complexes. In specific examples, we will also use the abbreviated notation 
\[
f_{\a^{\nu_1}\to \dots\to  \a^{\nu_k}}^{\b^{\veps_1}\to \cdots \to \b^{\veps_j}}(\xs)=f_{\nu_k,\dots, \nu_1,\veps_1,\dots, \veps_j,\frS_{\nu_k,\dots, \nu_1,\veps_1,\dots, \veps_j}}\left(\Theta_{\nu_k,\nu_{k-1}},\dots, \Theta_{\nu_2,\nu_1}, \ve{x},\Theta_{\veps_1,\veps_2},\dots, \Theta_{\veps_{j-1},\veps_{j}}\right).
\]
\noindent If $\nu=\nu'$, then we will usually write $f_{\a^{\nu}}^{\b^{\veps}\to \b^{\veps'}}$, and similarly if $\veps=\veps'$. With respect to this notation the hypercube map becomes
\[
D_{(\nu,\veps),(\nu',\veps')}(\xs)=\sum_{\substack{\nu=\nu_1<\dots<\nu_k=\nu' \\ \veps=\veps_1<\dots< \veps_j=\veps'}}f_{\a^{\nu_1}\to \dots\to  \a^{\nu_k}}^{ \b^{\veps_1}\to \cdots \to \b^{\veps_j}}.\]
We will also frequently substitute the letter $h$ for $f$ in the specific case of a map counting pseudoholomorphic rectangles.

\subsection{$\iota$-complexes} \label{subsec:iotacomplexes}

We recall the following definition from \cite{HMInvolutive}.

	\begin{define}
		\label{def:iota-complex}
		An \emph{$\iota$-complex} is a chain complex $(C,\d)$ which is free and finitely generated over $\bF[U]$ and equipped with an endomorphism $\iota$. Furthermore, the following hold:
		\begin{enumerate}
			\item\label{iota-1} $C$ is equipped with a $\Z$-grading, such that $U$ has grading $-2$. We call this grading the \emph{Maslov} or \emph{homological} grading.
			\item\label{iota-2} There is a grading preserving isomorphism $U^{-1}H_*(C)\iso \bF[U,U^{-1}]$.
			\item \label{iota-3} $\iota$ is a grading preserving chain map and $\iota^2\simeq \id$.
		\end{enumerate}
	\end{define}

\begin{define}\label{def:enhanced-iota-homomorphisms}
	If $(C,\iota)$ and $(C',\iota')$ are $\iota$-complexes, we define the group of \emph{enhanced $\iota$-morphisms} to be
	\[
	\underline{\Mor}((C,\iota),(C',\iota')):=\Hom_{\bF[U]}(C,C')\oplus \Hom_{\bF[U]}(C,C')[1],
	\]
	where $\Hom_{\bF[U]}(C,C')$ denotes homogeneous (but not necessarily grading-preserving) morphisms and $[1]$ denotes a grading shift. We define
	\[
	\d_{\underline{\Mor}}(F,h)=(\d' F+F\d, F\iota+\iota'F+\d' h+h\d),
	\]
	which makes the category of $\iota$-complexes into a dg-category. 
	An enhanced $\iota$-\emph{homo}morphism is an enhanced $\iota$-morphism $(F,h)$ satisfying $\d_{\underline{\Mor}}(F,h)=0$. An \emph{enhanced $\iota$-homotopy equivalence} is an enhanced $\iota$-homomorphism $(F,h)$ such that there exists an enhanced $\iota$-homomorphism $(G,k)$ with the property that $(G,k)\circ (F,h)$ and $(F,h)\circ (G,k)$ both differ from $(\id,0)$, the identity enhanced $\iota$-morphism, by a boundary in $\underline{\Mor}$.  Composition is given by $(F,h)\circ (F',h')=(F\circ F',F'h+h'F)$.
\end{define}

If $(C,\d)$ is a free, finitely generated chain complex over $\bF[U]$, we now describe the chain map
\[
\Phi\colon C\to C.
\]
 Write $\d$ as a matrix with respect to a free $\bF[U]$-basis of $C$. We then define $\Phi$ to be the endomorphism obtained by differentiating each entry of this matrix with respect to $U$ and extend $\bF[U]$-linearly. This map appears naturally in considering the basepoint action on Heegaard Floer homology \cite{ZemHatGraphTQFT, ZemGraphTQFT, JTNaturality}. The maps $\Phi$ is independent of the choice of basis up to $\bF[U]$-equivariant chain homotopy. In the special case that the basis is the set of generators associated to a particular Heegaard diagram for $(Y,w)$, such that powers of the variable $U$ in the differential record intersections with the divisor $\{w\}\times \Sym^{g-1}(\Sigma)$, we may write this map $\Phi_w$.

\subsection{Hypercubes and $\iota$-complexes} \label{subsec:hypercube-equivalence}

In this section, we describe a relation between enhanced $\iota$-morphisms and hypercubes. An enhanced $\iota$-morphism $(F,h)\colon (C,\iota)\to (C',\iota')$ may be encoded into a diagram of the following form
\begin{equation}
\begin{tikzcd}[labels=description,row sep=1.5cm, column sep=2cm]
C\ar[d, "\iota"] \ar[r, "F"]\ar[dr,dashed, "h"] & C\ar[d, "\iota' "]\\
C\ar[r,"F"]& C'
\end{tikzcd}
\label{eq:hypercube=morphism-iota-complexes}
\end{equation}
The pair $(F,h)$ is an enhanced $\iota$-homomorphism if and only if the above diagram is a hypercube of chain complexes. Composition of $\iota$-morphisms is encoded by stacking and compressing hypercubes. 

\begin{rem}
The diagram in equation~\eqref{eq:hypercube=morphism-iota-complexes} is algebraically equivalent to the following diagram:
\begin{equation}
\begin{tikzcd}[labels=description,row sep=1.5cm, column sep=2cm]
C\ar[d, "Q(\iota+\id)"] \ar[r, "F"]\ar[dr,dashed, "Q\cdot h"] & C'\ar[d, "Q(\iota' +\id)"]\\
Q\cdot C\ar[r,"F"]& Q\cdot C'
\end{tikzcd}
\end{equation}
In particular, we may view enhanced iota morphisms between iota complexes as containing equivalent information to an $\bF[U,Q]/Q^2$-equivariant map between the associated complexes over $\bF[U,Q]/Q^2$. We will use the two perspectives interchangeably.
 \label{rem:q-to-iota}
\end{rem}

Relations between compositions of $\iota$-morphisms are also naturally encoded into hypercubes:
\begin{lem}\label{lem:iota-maps-from-hyperboxes}
 Suppose that  $(A,\iota_A)$, $(B,\iota_B)$, $(C,\iota_C)$ and $(D,\iota_D)$ are $\iota$-complexes, and $(F,f)$, $(G,g)$, $(H,h)$ $(I,i)$ and $(J,j)$ are enhanced $\iota$-morphisms which fit into the following diagram
 \[
 \begin{tikzcd}[column sep={1.5cm,between origins},row sep=.8cm,labels=description]
A
	\ar[dr, "G"]
	\ar[ddd, "\iota_A"]
	\ar[dddrr,dashed, "i"]
	\ar[rr, "I"]
	\ar[ddddrrr,dotted, "j"]
&[.7 cm]\,
&B
	\ar[rd, "H"]
	\ar[ddd,"\iota_B"]
	\ar[ddddr,dashed, "h"]
&[.7 cm]\,
\\
&[.7 cm] C
	\ar[rr, crossing over, "F"]
	\ar[dddrr,dashed, crossing over, "f"]
&\,
&[.7 cm] D
	\ar[from=ulll, dashed, crossing over, "J"]
	\ar[ddd, "\iota_D"]
\\
\\
A
	\ar[dr, "G"]
	\ar[rr,"I"]
	\ar[drrr,dashed, "J"]	
&[.7 cm]\,&
B
	\ar[dr, "H"]&[.9 cm]\,\\
& [.7 cm] C
	\ar[from=uuuul, dashed, crossing over, "g"]
	\ar[from=uuu,crossing over, "\iota_C"]
	\ar[rr, "F"]&\,
&[.9 cm] D
\end{tikzcd}
 \]
Then the hypercube relations for the above diagram are equivalent to each of $(F,f)$, $(G,g)$, $(H,h)$ and $(I,i)$ being enhanced $\iota$-homormophisms, as well as the relation
\[
(F,f)\circ (G,g)+(H,h)\circ (I,i)=\d_{\underline{\Mor}}(J,j).
\]
\end{lem}
\begin{proof}The length 1 relations imply that $F$, $G$, $H$ and $I$ are chain maps. The length 2 relations along the left, right, front and back faces are easily seen to be equivalent to $(F,f)$, $(G,g)$, $(H,h)$ and $(I,i)$ being enhanced $\iota$-homormorphisms. The length 2 relation along the top and bottom face is equivalent to 
\[
FG+H I=[\d, J].
\] 
The length 3 relation is equivalent to 
\[
F g+f G+H i+h I=\iota_D J+J\iota_A+[\d, j].
\]
On the other hand, by definition,
\[
(F,f)\circ (G,g)+(H,h)\circ (I,i)=(FG+HI,Fg+fG+Hi+hI)
\]
while
\[
\d_{\underline{\Mor}}(J,j)=([\d,J], \iota_D J+J\iota_A+[\d, j]).
\]
The main claim follows immediately.
\end{proof}

\subsection{$\iota_K$-complexes and $\iota_L$-complexes} \label{sec:knot-and-link-complexes} We now review the refinement of the analog of an $\iota$-complex for knots, called an $\iota_K$-complex, and its extension to links. First we recall some algebraic background. Let $(C_K,\d)$ be a free, finitely generated complex over the ring $\bF[\scU,\scV]$. There are two naturally associated maps
\[
\Phi,\Psi\colon C_K\to C_K,
\]
constructed as follows. Write $\d$ as a matrix with respect to a free $\bF[\scU,\scV]$-basis of $C_K$. We then define $\Phi$ to be the endomorphism obtained by differentiating each entry of this matrix with respect to $\scU$. We define $\Psi$ to be the endomorphism obtained by differentiating each entry with respect to $\scV$. These maps appear naturally in the study of knot Floer homology, see \cites{SarkarMaslov, ZemQuasi, ZemCFLTQFT}. The maps $\Phi$ and $\Psi$ are independent of the choice of basis up to $\bF[\scU,\scV]$-equivariant chain homotopy \cite{ZemConnectedSums}*{Corollary~2.9}.

An $\bF$-linear map $F\colon C_K\to C_K'$ is \emph{skew-$\bF[\scU,\scV]$-equivariant} if 
\[
F\circ \scV=\scU\circ F\quad \text{and} \quad F\circ \scU=\scV\circ F.
\]

We recall the following definition from \cite{HHSZExact}. (Compare also \cite{HMInvolutive}*{Definition~6.2} and \cite{ZemConnectedSums}*{Definition~2.2}.)

\begin{define}\label{def:iota-K-complex}
An \emph{$\iota_K$-complex} $(C_K,\d, \iota_K)$ is a finitely generated, free chain complex $(C_K,\d)$ over $\bF[\scU,\scV]$, equipped with a skew-equivariant endomorphism $\iota_K$ satisfying 
 \[
\iota_K^2\simeq \id+\Phi\Psi. 
 \]
\end{define}

\begin{rem}\label{rem:iotaKU}
If $(C_K,\d,\iota_K)$ is an $\iota_K$-complex, then $\iota_K$ commutes with $U=\scU\scV$, and hence we can view $C_K$ as an (infinitely generated) complex over $\bF[U]$ with an $\bF[U]$-equivariant endomorphism $\iota_K$. 
\end{rem}

We have the following notions of morphism of $\iota_K$-complexes:

\begin{define} Suppose that  $\mathscr{C}=(C_K,\d, \iota_K)$ and $\mathscr{C}'=(C'_K,\d',\iota_K')$ are $\iota_K$-complexes.
\begin{enumerate}
\item  An \emph{$\iota_K$-homomorphism} from $\mathscr{C}$ to $\mathscr{C}'$ consists of an $\bF[\scU,\scV]$-equivariant chain map $F\colon C_K\to C_K'$, which satisfies $\iota_K'F+F\iota_K\eqsim 0$ (where $\eqsim$ denotes skew-equivariantly chain homotopy equivalence).
\item The group of \emph{enhanced $\iota_K$-morphisms} is
\[
\underline{\Mor}(\scC,\scC'):=\Hom_{\bF[\scU,\scV]}(C_K,C_K')\oplus \bar{\Hom}_{\bF[\scU,\scV]}(C_K,C_K')[1,1],
\]
where $\bar{\Hom}_{\bF[\scU,\scV]}(C_K,C_K')$ denotes the group of $\bF[\scU,\scV]$-skew-equivariant maps. The differential on $\underline{\Mor}(\scC_K,\scC'_K)$ is given by 
\[
\d_{\underline{\Mor}}(F,g)=(F\d+\d'F, F\iota_K+\iota_K' F+\d' g+g\d).
\]
We say $(F,g)$ is an enhanced $\iota_K$-\emph{homo}morphism if $\d_{\underline{\Mor}}(F,g)=0$. Two enhanced $\iota_K$-morphisms are $\iota_K$-\emph{homotopic} if their sum is a boundary in $\underline{\Mor}(\scC_K,\scC_K')$.
\end{enumerate}
\end{define}

We will also be interested in the case of free, finitely-generated complexes $(C_L,\d)$ over the ring $\bF[\scU_1,\scV_1, \dots \scU_{\ell}, \scV_{\ell}]$, which arise when we study links $L$ with multiple components. In this case, there are maps $\Phi_i$ and $\Psi_i$ for $1 \leq i \leq {\ell}$ arising from differentiating with respect to each variable. An $\bF$-linear map $F \colon C_L \rightarrow C_L'$ is said to be skew-equivariant if it exchanges $\scU_i$ and $\scV_i$ for each $i$. We may then consider $\iota_L$-complexes $(C_L, \d, \iota_L)$, where the map $\iota_L$ satisfies the formula 
\[
\iota_L^2 \simeq \left(\id+\Phi_{\ell}\Psi_{\ell}\right)\circ \dots \circ \left(\id + \Phi_1\Psi_1\right).
\]
This is the diffeomorphism map for performing a Dehn twist on each link component. See \cite{ZemQuasi}*{Theorem~B} for further discussion. The definitions of $\iota_K$-morphisms and their variations now extend straightforwardly to $\iota_L$-complexes.

As in Section \ref{subsec:hypercube-equivalence}, an enhanced $\iota_K$-homomorphism is equivalent to a hypercubes of chain complexes, compositions of $\iota_K$-morphisms are encoded as hypercubes in the same manner as in Lemma~\ref{lem:iota-maps-from-hyperboxes}.

%% file: section-1-2-doubling-prelims.tex
\section{Involutive Heegaard Floer homology} \label{sec:involutive}

In this section, we review Hendricks and Manolescu's construction of involutive Heegaard Floer homology \cite{HMInvolutive}, and also define the doubling models which feature in our statements of naturality and functoriality. 

\subsection{The involutive Floer complexes}

We presently recall Hendricks and Manolescu's original construction \cite{HMInvolutive}. Suppose that $\cH=(\Sigma,\as,\bs,w)$ is a weakly admissible Heegaard diagram for $Y$. Suppose that $\frs$ is a self-conjugate $\Spin^c$ structure. Write $\bar \cH=(\bar \Sigma, \bar \bs, \bar \as, w)$ for the diagram obtained from $\cH$ by reversing the orientation of $\Sigma$ and reversing the roles of $\as$ and $\bs$. There is a canonical chain isomorphism
\[
\eta\colon \CF^-(\bar{\cH},\frs)\to \CF^-(\cH,\frs).
\]
Hendricks and Manolescu consider the map
\[
\iota \colon \CF^-(\cH,\frs)\to \CF^-(\cH,\frs)
\]
given by the formula
\[
\iota := \eta \circ \Psi_{\cH\to \bar {\cH}},
\]
where $\Psi_{\cH\to \bar \cH}$ is the map from naturality as in \cite{JTNaturality}. They define the involutive Heegaard Floer complex $\CFI(Y,\frs)$ to be the $\bF[U,Q]/(Q^2)$-chain complex whose underlying $\bF[U,Q]/(Q^2)$-module is 
\[
\CF^-(Y,\frs)\otimes_{\bF[U]} \bF[U,Q]/(Q^2)=\CF^-(Y,\frs)\oplus Q\CF^-(Y,\frs)
\] with differential $\partial_{\CFI(Y,\frs)}=\partial_{\CF^-(Y,\frs)}\otimes \id +(\id+\iota)\otimes Q$.  That is, the involutive Heegaard Floer complex is the mapping cone
\[
\Cone\left(Q(\id+\iota)\colon \CF^-(Y,\frs)\to Q \CF^-(Y,\frs)\right).
\]
with the evident action of $\bF[U,Q]/(Q^2)$.  Hendricks and Manolescu prove that $\CFI^-(Y,\frs)$ is well-defined up to chain homotopy equivalence.

Hendricks and Manolescu also define a refinement for knots, which we recall presently along with its extension to links. Beginning with the knot case, suppose that $K$ is a null-homologous, oriented knot in a 3-manifold $Y$, and $\frs\in \Spin^c(Y)$ is self-conjugate. We recall the following standard definition.

\begin{define} A Heegaard diagram for a pair $(Y,K)$ consists of a tuple $(\Sigma,\as,\bs,w,z)$, such that $(\Sigma,\as,\bs)$ is a Heegaard diagram for $Y$, such that the knot $K$ intersects $\Sigma$ in the points $w$ and $z$. Furthermore, if $U_{\a}$ and $U_{\b}$ are the closures of the two components of $Y\setminus \Sigma$, then $K$ intersects each of $U_{\a}$ and $U_{\b}$ in a boundary-parallel arc. The standard convention is that $K$ intersects $\Sigma$ positively at $z$, and negatively at $w$.
\end{define}

We will work over the version of knot Floer homology \cite{OSKnots} \cite{RasmussenKnots} which is freely generated over a 2-variable polynomial ring $\bF[\scU,\scV]$. We denote this chain complex by $\cCFK(K)$. See, e.g. \cite{ZemConnectedSums}*{Section~3} for background on this version of knot Floer homology. 

Let $\phi^+ \colon (Y,K,w,z)\to (Y,K,z,w)$ be the diffeomorphism of tuples supported in a neighborhood of $K$ corresponding to a positive half-twist along $K$, and similarly for $\phi^{-}$ and the negative half-twist. There is a canonical chain isomorphism
\[
\eta_K\colon\cCFK(\bar\Sigma,\bar \bs, \bar \as,z,w)\to \cCFK(\Sigma, \as, \bs,w,z)
\]
which satisfies $\eta_K(\scU\cdot x)=\scV \cdot \eta_K(x)$ and $\eta_K(\scV \cdot x)=\scU\cdot \eta_K(x)$. We define
\[
\iota_{K,+}:=\eta_K\circ \Psi_{\phi^+(\cH)\to \bar \cH} \circ \phi_*^+,
\]
and we define $\iota_{K,-}$ similarly. Here, $\phi_*^+$ denotes the tautological diffeomorphism map, and $\Psi_{\phi^+(\cH)\to \bar \cH}$ denotes the map from naturality obtained by picking a sequence of Heegaard moves relating $\phi^+(\cH)$ and $\bar \cH$. There is a map $\iota_K^-$ defined similarly. Hendricks and Manolescu define the involutive knot Floer complex to be the data consisting of $\cCFK(K)$ equipped with the endomorphism $\iota_{K,\pm}$, which gives an $\iota_K$-complex.

\begin{rem}
 In \cite{HMInvolutive}, Hendricks and Manolescu considered only the involution we call $\iota_{K,+}$, and wrote simply $\iota_K$. It is unknown whether there are any knots $K$ where $(\cCFK(K), \iota_{K,+})$ is inequivalent to $(\cCFK(K),\iota_{K,-})$. The distinction between $\iota_{K,+}$ and $\iota_{K,-}$ has previously been observed in \cite{BKST:genus}.
\end{rem}

The case of links is similar. For a link $L$ of $\ell=|L|$ components, we consider Heegaard diagrams $(\Sigma, \as, \bs, \ws, \zs)$ such that $\ws = \{w_1,\dots, w_n\}$ and $\zs = \{z_1, \dots, z_{\ell}\}$ and the $i$-th component $K_i$ of $L$  intersects $\Sigma$ positively at $z_i$ and negatively at $w_i$. The resulting chain complex $\cCFL(L)$ is freely-generated over the polynomial ring $\bF[\scU_1, \scV_1, \dots, \scU_{\ell}, \scV_{\ell}]$. In principle there are $2^{\ell}$ choices of orientation which may be used to define the involution; however for simplicity we assume that an orientation is fixed, and we consider only the two orientations which are either coherent, or opposite to our preferred orientation. It is now straightforward to adapt the construction of $\iota_{K,\pm}$ given above to models for $\iota_{L,\pm}$ for these versions of the link involution.

\subsection{Doubled Heegaard diagrams}
\label{sec:doubling-def}

Suppose that $\cH=(\Sigma,\as,\bs,w)$ is a pointed Heegaard diagram for $Y$. We view a regular neighborhood of $\Sigma$ as $\Sigma\times [0,1]$. Pick a small disk $D\subset \Sigma$, which contains $w$ along its boundary. The submanifold $U_0=(\Sigma\setminus D)\times [0,1]$ is a handlebody of genus $2g$, where $g$ denotes $g(\Sigma)$. The boundary of this submanifold is canonically identified with $\Sigma\# \bar \Sigma$, where the connected sum occurs at $w$. Write $U_1$ for the closure of the complement of $U_0$. Clearly $U_1$ is also a handlebody of genus $2g$. We write $D(\Sigma)=\Sigma\# \bar \Sigma$ for this Heegaard surface.

We may naturally equip $D(\Sigma)$ with attaching curves, as follows. For $U_1$, we may use the curves $\as \cup \bar \bs$, where $\as\subset \Sigma$ and $\bar\bs \subset \bar \Sigma$. For $U_0$, we pick compressing disks by picking a collection $\delta_1,\dots, \delta_{2g}$ of properly embedded arcs on $\Sigma\setminus N(w)$, called the doubling arcs. We add a basepoint to $\d N(w)$, for which we also write $w$. We assume that the arcs $\delta_1,\dots, \delta_{2g}$ are disjoint from $w$, and also form a basis of $H_1(\Sigma\setminus D, \d D\setminus\{w\})$.  We form attaching curves $\Ds\subset \Sigma\# \bar \Sigma$ by doubling the curves $\delta_1,\dots, \delta_{2g}$ onto $\Sigma\# \bar \Sigma.$ We write $D(\cH)=(\Sigma \# \bar \Sigma, \as \bar \bs, \Ds,w)$ for the resulting diagram.

If $\xi=(\xi_1,\xi_2,\xi_3)$ is a framing of $w\in Y$, then the framing is used in the construction as follows. We assume the Heegaard surface is tangent to the 2-plane field $(\xi_1,\xi_2)\subset T_w Y$. Additionally, in the tube connecting $\Sigma\# \bar \Sigma$, we must place the basepoint $w$. We place it in the direction corresponding to $\xi_1$.

\begin{rem}
 The diagram $D(\cH)$ may also be described by gluing two bordered Heegaard diagrams for genus $g$ handlebodies \cite{LOTBordered}.
\end{rem}  

\subsection{Doubling and the involution} \label{sec:doubling-model-manifolds}

We use doubled Heegaard diagrams to give a natural model of the involution. This model appears in \cite{ZemDuality} and \cite{HHSZExact}. In this section, we describe the construction.

Firstly, we define the following map as a  composition of 1-handle maps
\[
F_1^{\bar \b,\bar\b}\colon \CF^-(\Sigma,\as,\bs,w)\to \CF^-(\Sigma\# \bar \Sigma, \as \bar \bs, \bs\bar \bs, w).
\]
Next, we note that the diagram $(\Sigma\# \bar \Sigma, \bs\bar \bs,\Ds,w)$ is a doubled diagram for $(S^1\times S^2)^{\# g}$, and hence we may pick a cycle $\Theta_{\b\bar \b,\Dt}$ generating the top degree of homology. The cycle $\Theta_{\b\bar \b, \Dt}$ is not unique, but any two choices are related by a boundary. We define the holomorphic triangle map
\[
f_{\a\bar \b}^{\b \bar \b\to \Dt}\colon \CF^-(\Sigma\# \bar \Sigma, \as \bar \bs, \bs \bar \bs, w)\to \CF^-(\Sigma\# \bar \Sigma, \as \bar \bs, \Ds,w)
\] 
by the formula $f_{\a \bar \b}^{\b \bar \b\to \Dt}(x)=f_{\a \bar \b, \b \bar \b,\Dt}(x,\Theta_{\b \bar \b,\Dt}).$ 

Analogously, there is another holomorphic triangle map
\[
f_{\a \bar \b}^{\Dt\to \a \bar \a}\colon \CF^-(\Sigma \# \bar \Sigma, \as \bar \bs, \Ds, w)\to \CF^-(\Sigma\# \bar \Sigma, \as \bar \bs, \as \bar \as,w)
\]
as well as a 3-handle map
\[
F_3^{\a,\a}\colon \CF^-(\Sigma \# \bar \Sigma, \as \bar \bs, \as \bar \as, w)\to \CF^-(\bar \Sigma, \bar \bs, \bar \as,w).
\]

\begin{prop}[\cite{ZemDuality}*{Propositions~7.2, 7.8}]
\label{prop:doubling-involution}
 The composition $f_{\a \bar \b}^{\b \bar \b\to \Dt} \circ F_1^{\bar \b,  \bar\b}$ is chain homotopic to the map $\Psi_{\cH \to D(\cH)}$ from naturality. Dually, the composition $F_3^{\a,\a}\circ f_{\a \bar \b}^{\Dt\to \a \bar \a}$ is chain homotopic to the map $\Psi_{D(\cH)\to \bar \cH}$.
\end{prop}
The idea of the proof of the above result is that the composition is topologically equivalent to a composition of the maps for canceling 1-handles and 2-handles, or canceling 2-handles and 3-handles.

\begin{cor} \label{cor:doubling-model}
The involution satisfies
\[
\iota \simeq \eta\circ  F_3^{\a,\a} \circ f_{\a \bar \b}^{\Dt \to \a \bar \a}\circ f_{\a \bar \b}^{\b \bar \b \to \Dt} \circ F_1^{\bar \b, \bar \b},
\]
where $\eta$ is the canonical chain isomorphism
\[
\eta\colon \CF^-(\bar \Sigma, \bar \bs, \bar \as,w)\to \CF^-(\Sigma,\as,\bs,w).
\]
\end{cor}

\begin{define} Suppose $(Y,w,\xi)$ is a based manifold with a framing $\xi=(\xi_1,\xi_2,\xi_3)$ of $T_w Y$. 
A \emph{doubling enhanced Heegaard diagram} $\frD$ consists of the following data:
\begin{enumerate}
\item A Heegaard diagram $\cH=(\Sigma,\as,\bs,w)$. Furthermore, $\Sigma$ is positively tangent to the 2-plane spanned by $\xi_1$ and $\xi_2$.
\item A set of attaching curves $\Ds$ on $\Sigma\# \bar \Sigma$, constructed by doubling a basis of arcs $\delta_1,\dots, \delta_{2g}$. Furthermore, we view the endpoint of each $\delta_{i}$ arc as determining an oriented tangent space (oriented to point into $\delta_i$). We assume that none of these oriented tangent spaces coincide with the span of $\xi_1$. 
\item Choices of almost complex structures to compute the four Floer complexes $\CF(\Sigma,\as,\bs)$, $\CF(\Sigma \# \bar \Sigma, \as \bar \bs, \bs \bar \bs)$, $\CF(\Sigma\# \bar \Sigma, \as \bar \bs, \Ds)$ and $\CF(\Sigma\# \bar \Sigma, \as \bar \bs, \as \bar\as)$, as well as almost complex structures to compute the triangle maps $f_{\a \bar \b}^{\Dt \to \a \bar \a}$ and $f_{\a \bar \b}^{\b \bar \b \to \Dt}$.  
\end{enumerate}
We say that $\frD$ is a \emph{doubling enhancement} of $\cH$.  We say that $\frD$ is \emph{admissible} if each of the Heegaard diagrams and triples are weakly admissible.  
\end{define}

\subsection{Doubling and the knot and link involution}
\label{sec:doubling-knots}
The knot involution of \cite{HMInvolutive} also admits a convenient doubling model, which we introduced in \cite{HHSZExact} (see also \cite{JZStabilizationDistance}*{Section~9.4}). In this section we recall this model. In our proof of naturality, we will also need to use several stabilized models, which we introduce in Section~\ref{sec:knots}. 

The doubling model for knots has a similar description to Corollary~\ref{cor:doubling-model}, but more care will be needed to define the maps. There is furthermore one additional choice to be made, which will correspond to a choice of which direction along $K$ to perform a half-twist.

We begin by describing doubled knot diagrams. If $(\Sigma,\as,\bs,w,z)$ is a diagram for $(Y,K)$, we consider the following two diagrams:
\[
D^+(\cH)=(\Sigma\#_z \bar \Sigma, \as \bar \bs, \Ds, w ,\bar w)\quad \text{and} \quad D^-(\cH)=(\Sigma\#_w \bar \Sigma, \as \bar \bs, \Ds,\bar z, z).
\]
We start by analyzing the diagram $\cD^+(\cH).$ There is a map
\[
F_{1,+}^{\bar \b, \bar \b}\colon \cCFK(\Sigma,\as,\bs,w,z)\to \cCFK(\Sigma\#_z \bar \Sigma, \as \bar \bs, \bs \bar \bs, w, \bar w)
\]
obtained as the composition of a 1-handle map, as well as a diffeomorphism map for moving $z$ to $\bar w$. As in the setting of closed 3-manifolds, we have holomorphic triangle maps
$f_{\a \bar \b}^{\b \bar \b\to \Dt}$ and $f_{\a \bar \b}^{\Dt\to \a \bar \a}$. Finally, there is a map
\[
F_{3,+}^{\a,\a}\colon \cCFK(\Sigma\#_z \bar \Sigma, \as \bar \bs, \as \bar \as, w, \bar w)\to \cCFK(\bar\Sigma, \bar \bs, \bar \as, \bar z, \bar w),
\]
obtained by moving $w$ to the position of $\bar z$, and then applying the 3-handle maps to remove the $\as$ curves. Similar logic to Proposition~\ref{prop:doubling-involution} then implies that
\[
\iota_{K,+}\simeq \eta_K\circ  F_{3,+}^{\a,\a}\circ f_{\a \bar \b}^{\Dt \to \a \bar \a}\circ  f_{\a \bar \b}^{\b \bar \b\to \Dt}\circ  F_{1,+}^{\bar \b, \bar \b}.
\]

\noindent There are analogous maps $F_{1,-}^{\bar \b, \bar \b}$ and $F_{3,-}^{\a, \a}$ which instead involve the diagram $D^-(\cH)$, such that
\[
\iota_{K,-}\simeq \eta_K\circ  F_{3,-}^{\a,\a}\circ f_{\a \bar \b}^{\Dt \to \a \bar \a}\circ  f_{\a \bar \b}^{\b \bar \b\to \Dt}\circ  F_{1,-}^{\bar \b, \bar \b}.
\]

The same procedure with only notational changes may be used to give doubling models for the link involutions $\iota_{L^,\pm}$ for links with more than one component. See Section~\ref{sec:knots} for more details.

\subsection{Alpha versus beta doubling} \label{sec:models}

The reader may note that a somewhat arbitrary choice has been made in the definition of $\iota$.  Indeed, there is another diagram $\bar{D}(\cH)=(\Sigma\# \bar \Sigma, \Ds, \bs \bar \as,w)$, conjugate to the diagram $D(\cH)$ considered above.  Similar to Proposition~\ref{prop:doubling-involution} and using the same notation from that proposition, $f^{\b\bar\a}_{\a\bar\a\to \Delta}\circ F_1^{\bar\a,\bar\a}$ is also chain homotopic to $\Psi_{\cH\to \bar D(\cH)}$.  Dually, $F_3^{\b,\b}\circ f^{\b\bar\a}_{\Delta\to\b\bar\b}$ is homotopic to $\Psi_{\bar D(\cH)\to \bar{\cH}}$.  As a consequence, parallel to Corollary \ref{cor:doubling-model},
\begin{equation}
\iota \simeq \eta\circ  F_3^{\b,\b} \circ f_{\Dt\to \b \bar \b}^{\b\bar\a}\circ f_{\a \bar \a\to \Dt}^{\b \bar \a} \circ F_1^{\bar \a, \bar \a}.\label{eq:alpha-doubling}
\end{equation}
We will refer to the composite on the right side of equation~\eqref{eq:alpha-doubling} as the \emph{alpha-doubling model} for $\iota$, and we refer to the model  $\eta\circ  F_3^{\a,\a} \circ f_{\a \bar \b}^{\Dt \to \a \bar \a}\circ f_{\a \bar \b}^{\b \bar \b \to \Dt} \circ F_1^{\bar \b, \bar \b}$ from Corollary~\ref{cor:doubling-model}  as the \emph{beta-doubling} model. In spite of the fact that the alpha-doubling and beta-doubling models are homotopic by naturality, we do not claim there is a canonical homotopy between them. In this paper, we arbitrarily work over the beta-doubling model exclusively.

\section{On Naturality in Heegaard Floer theory} \label{sec:naturality}

In this section, we discuss naturality in Heegaard Floer theory and how it adapts to our setting. Our goal in this section is to state a version of the naturality theorem which covers our use of framed basepoints in the involutive Heegaard Floer complex.

We recall that the proof of \cite{JTNaturality} centers on proving naturality for invariants of sutured 3-manifolds. We recall that if $Y$ is equipped with a basepoint $w$, then we may obtain a sutured manifold by removing a regular neighborhood of $w$ and adding a single suture (closed loop) to the boundary. The reader should think of the suture as corresponding to the intersection of a Heegaard surface for $Y$ with the boundary.

We cannot apply the work of \cite{JTNaturality} without any change, because our construction of the involutive Floer complex requires a choice of framing at the basepoint. This is incompatible with the \emph{continuity} axiom of \cite{JTNaturality}, which requires any diffeomorphism of a sutured manifold  $\phi \in \Diff^+(M,\g)$ which is isotopic to the identity through sutured diffeomorphisms to act by the identity on the complex. In our present context, this is too strong of an axiom to impose, since such diffeomorphisms may correspond to diffeomorphisms of $(Y,w)$ which act non-trivially on the framing of the basepoint. Instead, we wish to weaken the continuity axiom to only require that diffeomorphisms which are isotopic to the identity \emph{relative to the boundary} to act by the identity on our invariant. In this section, we explain how the techniques of \cite{JTNaturality} adapt to our present situation.

\begin{rem}Note that in the original setting \cite{JTNaturality}, this was of little importance, since isotopies of sutured manifolds which are supported in a small neighborhood of the boundary do not change the position of the alpha or beta curves.
\end{rem}

We begin by recalling some notation from \cite{JTNaturality}. We first recall the graph $\cG_{(M,\g)}$ defined therein.  The vertices consist of embedded Heegaard diagrams for $(M,\g)$, with attaching curves taken up to isotopy on the Heegaard surface. The edges of $\cG_{(M,\g)}$ consist of the following:
\begin{enumerate}
\item Alpha equivalences on a fixed Heegaard surface.
\item Beta equivalences on a fixed Heegaard surface.
\item Index (1,2)-stabilizations.
\item Diffeomorphisms of $(M,\g)$ which are isotopic to the identity as sutured diffeomorphisms.
\end{enumerate}

We now define a subgraph $\cG_{(M,\g)}^{\d}\subset \cG_{(M,\g)}$, which we call the graph of \emph{boundary framed Heegaard moves}. This graph has the same vertices as $\cG_{(M,\g)}$. We add all of the same edges, except that diffeomorphisms are required to be the identity on the boundary of $M$, and also are required to be isotopic to the identity rel boundary. We consider the following loops of Heegaard diagrams (thought of as 2-cells). These are a restriction of the distinguished rectangles from \cite{JTNaturality}*{Definition~2.32}. These consist of the following loops:
\begin{enumerate}
\item Commutation of alpha and beta equivalences.
\item Commutations of two index $(1,2)$-stabilizations.
\item Commutations of an index $(1,2)$-stabilization and an alpha/beta equivalence.
\item Simple handleswap loops.
\item A sequence of diffeomorphism edges in $\cG_{(M,\g)}^{\d}$, starting and ending at a fixed Heegaard diagram $\cH$, such that the composition of the diffeomorphisms is a contractible loop in $\Diff(M,\d M)$.\label{itm:nat-5}
\end{enumerate}

\begin{prop}\label{prop:simply-connected}
 After attaching 2-cells to the graph $\cG_{(M,\g)}^{\d}$ corresponding to the loops (1)--(5), we obtain a simply connected space.
\end{prop}
\begin{proof}
The result follows from essentially the same argument as \cite{JTNaturality}, with some extra attention paid to the behavior of diffeomorphisms on $\d M$. Note that $\cG^\d_{(M,\g)}$ is a subcomplex of $\cG_{(M,\g)}$. One may go through the proof described in \cite{JTNaturality}*{Section~8}. We consider a loop of Heegaard moves in $\cG^{\d}_{(M,\g)}$ of $k$-steps:
\[
\begin{tikzcd}
\cH_0
	\ar[r, "e_1"]
&\cH_1
	\ar[r, "e_1"]
&\cH_2
	\ar[r, "e_2"]
&\cdots
	\ar[r, "e_k"]
&\cH_k=\cH_0.
\end{tikzcd}
\]
 we now build a map $S^1\to \cF\cV_{\ge 1}(M,\g)$, denoting the space of pairs $(f,v)$ of sutured functions and gradient like vector fields which lie in codimension 0 or 1. (See \cite{JTNaturality} for this notation). This is constructed the same as therein. Namely, for each diagram $\cH_i$, we pick a compatible pair $(f,v)\in \cG\cV_{0}(M,\g)$ inducing this pair. For each move of Heegaard diagrams, we pick a compatible 1-parameter family of pairs $(f_t,v_t)\in \cG\cV_{\le 1}(M,\g)$. Accordingly, one constructs a polyhedral decomposition of $D^2$, such that each edge contains at most one codimension 1 singularity, and each 2-cell contains at most one codimension 2 singularity. From here, one argues that after subdividing it is sufficient to consider only 2-cells whose edges compose to give one of the loops (1)--(5).
 
 The difference between our situation an \cite{JTNaturality} is that we must argue that it is sufficient to consider only sutured isotopies which fix the boundary pointwise. One must do this for both edges and for 2-cells. Consider first a 1-cell 
 \[
 \sigma_1\colon D^1\to \cF \cV_0(M,\g).
 \]
Given such a path, one may construct a family of embedded Heegaard surfaces. To see that these are related by an isotopy, one uses a version of the isotopy extension lemma, as stated in \cite{JTNaturality}*{Lemma~6.19}. Note that this lemma actually constructs an isotopy which is the identity on $\d M$, because it is obtained by integrating a time dependent family of vector fields on $(M,\g)$ which is supported on the interior of $(M,\g)$.

Given a map $\sigma_2\colon D^2\to \cF\cV_0(M,\g)$, the above construction, applied to $S^1=\d D^2$, defines a loop of diffeomorphisms $\phi_\theta\colon S^1\to \Diff(M,\d M)$. Since the corresponding path in $\cF \cV_0(M,\g)$ is contractible, we may extend this loop $\phi_\theta$ over $D^2$.  We observe that \cite{JTNaturality}*{Lemma~6.20} constructs a null-homotopy of the corresponding loop of diffeomorphisms. We observe that this null-homotopy is again constructed by integrating along a time dependent family of vector fields which is supported in $\Int(M)$, and hence the corresponding family of diffeomorphisms are the identity on $\d M$. Hence the boundary framed loops in (1)--(5) are sufficient to contract an arbitrary loop.
\end{proof}

\section{Maps for elementary handleslide equivalences}

\label{sec:elementary-equivalences}
In this section, we define the transition maps for changes of the doubling data which fix the Heegaard surface $\Sigma$. These correspond to changes of the $\as$, $\bs$ and $\Ds$ curves by handleslides and isotopies. Furthermore, in this section, we consider only a restricted set of alpha and beta equivalences, as follows.
\begin{define}\,
\begin{enumerate}
\item Suppose that $(\Sigma,\ws)$ is a Heegaard surface with $n\ge 1$ basepoints and $\gs$ and $\gs'$ are attaching curves on $\Sigma$. We say that $\gs$ and $\gs'$ are \emph{handleslide equivalent} if they may be related by a sequence of handleslides and isotopies.
\item  Suppose $\Sigma$ is a Heegaard surface and $\gs$ and $\gs'$ are handleslide equivalent attaching curves on $\Sigma$. We say that $\gs$ and $\gs'$ are related by an \emph{elementary equivalence} if $(\Sigma,\gs',\gs,\ws)$ is a weakly admissible diagram, and $\bT_{\g'}\cap \bT_{\g}$ consists of exactly $2^{|\g|-1}$ generators. 
\end{enumerate}
\end{define}

Suppose $(\Sigma,\as,\bs)$ is a Heegaard diagram, and $\as'$ and $\bs'$ are handleslide equivalent to $\as$ and $\bs$, respectively. Suppose further that $\Ds$ and $\Ds'$ are choices of doubling curves. Let $\frD$ denote the data consisting of $(\Sigma,\as,\bs)$, the curves $\Ds$, and appropriate choices of almost complex structures to compute the involution. Let $\frD'$ denote the analogous data with $\as'$, $\bs'$ and $\Ds'$. We say that $\frD'$ is obtained from $\frD$ by an \emph{elementary equivalence} if $\as'$ and $\bs'$ are obtained from $\as$ and $\bs$ by elementary equivalences, and also if the tuple
\[
(\Sigma\# \bar \Sigma,\as'\bar \as',\as\bar \as,\bs\bar \bs,\bs'\bar \bs',\Ds,\Ds',w)
\]
is weakly admissible.

In the case that $\frD'$ is obtained from $\frD$ by an elementary equivalence, we define the transition map
\[
\Psi_{\frD\to \frD'}\colon \CFI(\frD)\to \CFI(\frD') 
\]
as the compression of the hyperbox of chain complexes shown in Figure~\ref{def:transition-map-elementary-handleslide}. Therein, the rows with the 1-handle maps are constructed similarly to \cite{HHSZExact}*{Section~14}, where they are called \emph{hypercubes of stabilization}. The construction here requires that the $\as'$ be obtained from $\as$ by an elementary equivalence, and similar for $\bs$ and $\bs'$, so that the degeneration argument from \cite{HHSZExact}*{Section~10} applies.  The rows with 3-handles are similar. The remaining subcubes are obtained by pairing hypercubes of attaching curves.

Note that in Figure~\ref{def:transition-map-elementary-handleslide} we have omitted a final level involving the tautological map $\eta\colon \CF(\bar \bs,\bar \as)\to \CF(\as,\bs)$. Throughout the paper, we will usually omit this row.

\begin{figure}[h]
\[
\Psi_{\frD\to \frD'}:=\quad 
\begin{tikzcd}[labels=description, row sep=1cm]
\CF(\as,\bs)
	\ar[r]
	\ar[d, "F_{1}^{\bar{\b},\bar{\b}}"]
& \CF(\as',\bs)
	\ar[r]
	\ar[d, "F_{1}^{\bar{\b}',\bar{\b}}"]
& \CF(\as',\bs')
	\ar[d,"F_{1}^{\bar{\b}',\bar{\b}'}"]
\\
\CF(\as \bar{\bs}, \bs \bar{\bs})
	\ar[r]
	\ar[d]
	\ar[dr,dashed]
& \CF(\as'\bar{\bs}', \bs \bar{\bs})
	\ar[r]
	\ar[d]
	\ar[dr,dashed]
& \CF(\as' \bar{\bs}', \bs' \bar{\bs}')
	\ar[d]
\\
\CF(\as\bar \bs,\Ds)
	\ar[r]
	\ar[d]
	\ar[dr,dashed]
& \CF(\as' \bar \bs',\Ds)
	\ar[r]
	\ar[d]
	\ar[dr,dashed]
& \CF(\as'\bar \bs',\Ds')
	\ar[d]
\\
\CF(\as \bar{\bs}, \as \bar{\as})
	\ar[r]
	\ar[d, "F_3^{\a,\a }"]
&\CF(\as' \bar{\bs}', \as  \bar{\as})
	\ar[r]
	\ar[d, "F_3^{\a',\a }"]
& \CF(\as' \bar{\bs}', \as' \bar{\as}')
	\ar[d, "F_3^{\a',\a'}"]
\\
\CF(\bar{\bs}, \bar{\as})
	\ar[r]
	\ar[d,"\id"]
	\ar[drr,dashed]
& \CF(\bar{\bs}',\bar{\as})
	\ar[r]
& \CF(\bar{\bs}',\bar{\as}')
	\ar[d,"\id"]
\\
\CF(\bar \bs,\bar \as)
	\ar[r]
&
\CF(\bar \bs, \bar \as')
	\ar[r]
&
\CF(\bar \bs', \bar \as')
\end{tikzcd}
\]
\caption{The transition map for an elementary equivalence. An additional row involving the canonical map $\eta$ is omitted.}
\label{def:transition-map-elementary-handleslide}
\end{figure}

If $\frD$ and $\frD'$ are handleslide equivalent (but perhaps not related by an elementary equivalence), we pick an arbitrary sequence 
\[
\frD=\frD_1,\dots, \frD_n=\frD'
\]
such that $\frD_{i+1}$ and $\frD_i$ are related by an elementary equivalence for all $i$. We define
\begin{equation}
\Psi_{\frD\to \frD'}:=\Psi_{\frD_{n-1}\to \frD_n}\circ \cdots \circ \Psi_{\frD_1\to \frD_2}.
\label{eq:definition-transition-map-compose-elementary}
\end{equation}

A key step in our proof of naturality will be proving the following theorem:

\begin{thm}
\label{thm:naturality-fixed-Sigma}
 Suppose that $\frD$, $\frD'$ and $\frD''$  are doubling enhanced Heegaard diagrams which all have the same underlying Heegaard surface $\Sigma$ and are handleslide equivalent.
 \begin{enumerate}
 \item\label{thm:naturality-fixed-Sigma-1} The map $\Psi_{\frD\to \frD'}$ defined in~\eqref{eq:definition-transition-map-compose-elementary} is independent from the intermediate sequence of doubling enhanced Heegaard diagrams, up to $\bF[U,Q]/Q^2$-equivariant chain homotopy.
 \item\label{thm:naturality-fixed-Sigma-2} $\Psi_{\frD\to \frD}\simeq \id_{\CFI(\frD)}.$
 \item\label{thm:naturality-fixed-Sigma-3} $\Psi_{\frD'\to \frD''}\circ \Psi_{\frD\to \frD'}\simeq \Psi_{\frD\to \frD''}$.
 \end{enumerate}
\end{thm}

Part~\eqref{thm:naturality-fixed-Sigma-2} of Theorem~\ref{thm:naturality-fixed-Sigma} is proven in the subsequent Section~\ref{sec:continuity}. Parts~\eqref{thm:naturality-fixed-Sigma-1} and \eqref{thm:naturality-fixed-Sigma-3} are more technical and are proven in Sections \ref{sec:expanded-def} and ~\ref{sec:relate-expanded}; in particular, they follow from Proposition~\ref{prop:expansion}.

\begin{rem}
 We are only able to define the map $\Psi_{\frD\to \frD'}$ using Figure~\ref{def:transition-map-elementary-handleslide} when $\frD$ and $\frD'$ are related by an elementary equivalence because we are not able to construct the 1-handle hypercubes if $\frD$ and $\frD'$ are related by an arbitrary handleslide equivalence. To show the map $\Psi_{\frD\to \frD'}$ is independent of the choice of intermediate diagrams, we will define a more complicated transition map $\tilde{\Psi}_{\frD\to \frD'}$ in Section~\ref{sec:maps-expanded-model}.
\end{rem}

\subsection{Continuity}
\label{sec:continuity}

In this section, we prove the continuity axiom from \cite{JTNaturality}*{Definition~2.32} for our maps associated to a handleslide equivalence. In our present context, the continuity axiom amounts to the following:

\begin{prop}
\label{prop:continuity}
 Let $\frD$ denote an admissible doubling enhanced Heegaard diagram for $(Y,w,\xi)$. Let $g\colon (\Sigma,w)\to (\Sigma,w)$ be a pointed diffeomorphism which is smoothly isotopic to the identity through diffeomorphisms which are the identity on a neighborhood of $w$. Let $g(\frD)$ be the push-forward  of the diagram $\frD$ under $g$, and let 
 \[
\CFI(g)\colon \CFI(\frD)\to \CFI(g(\frD)) 
 \]
 denote the tautological map. Then
 \[
\Psi_{\frD\to g(\frD)}\simeq \CFI(g). 
 \]
\end{prop}
\begin{proof}
We adapt the argument of \cite{JTNaturality}*{Proposition~9.27} to the involutive case by also using the small translate theorems for triangles and quadrilaterals from our previous paper \cite{HHSZExact}. Write $(\Sigma,\as,\bs,w)$ for the underlying Heegaard diagram of $\frD$, and let $\Ds$ be the doubling curves. As a first step, pick a $C^\infty$ small diffeomorphism $g_0$ of $(\Sigma,w)$ and write $\as'=g_0(\as)$, and define $\bs'$ and $\Ds'$ similarly. Assume that $|\a'_i\cap \a_j|=2 \delta_{ij}$, where $\delta_{ij}$ denotes the Kronecker delta, and assume the analogous statement holds for $\bs'$, $\bs$, $\Ds'$ and $\Ds$. It follows from the small translate theorems \cite[Propositions 11.1 and 11.5]{HHSZExact} that if $g_0$ is chosen suitably small, then we may pick an almost complex structure on $\Sigma\times \Delta$ which interpolates between a fixed $J$ on the $(\Sigma,\as,\bs)$ cylindrical end and $(g_0)_*(J)$ on the $(\Sigma,\as',\bs)$ cylindrical end such that the map $f_{\a',\a,\b}(\Theta_{\a',\a},\xs)$ counts only small triangles, so that $f_{\a',\a,\b}(\Theta_{\a',\a},\xs)=\xs_{np}$ for all $\xs\in \bT_{\a}\cap \bT_{\b}$. Here $\xs_{np}$ denotes the nearest point of $\bT_{\a'}\cap \bT_{\b}$ to $\xs$.  In a similar manner, the small translate theorem for quadrilaterals implies that all of the length 2 maps counted by $\Psi_{\frD\to g_0(\frD)}$ vanish (as no holomorphic curves are counted). In particular, $\Psi_{\frD\to g_0(\frD)}$ coincides with the nearest point map.

Similarly, if $g_0'$ is $C^\infty$ close to $g_0$, then the same argument as the small translate theorems also imply that $\Psi_{\frD\to g_0'(\frD)}$ and $\Psi_{g_0'(\frD)\to \frD}$ coincide with the nearest point maps.

More generally, if $g$ is $C^\infty$ small relative to $g_0$, then we can decompose
\[
\Psi_{\frD\to g(\frD)}\simeq \Psi_{g_0(\frD)\to g(\frD)}\circ \Psi_{\frD\to g_0(\frD)}
\]
and apply the nearest point results to the latter two maps to identify $\Psi_{\frD\to g(\frD)}$ with $\CFI(g)$. This establishes the claim when $g$ is $C^\infty$-small. Since both sides of the desired equality are functorial under composition, and any diffeomorphism $g\colon (\Sigma,w)\to (\Sigma,w)$ which is smoothly isotopic to the identity may be decomposed into a composition of $C^\infty$-small diffeomorphisms, this proves the claim in general.
\end{proof}

%% file: section-3a-multi-stabilizations-v2.tex
\section{Holomorphic polygons and stabilizations} \label{sec:polygons}

In this section, we describe some technical results concerning stabilizations of Heegaard diagrams and holomorphic polygons. The results we cover are a generalization of those considered in \cite{HHSZExact}*{Section~9}.

\begin{define} Suppose that $\cD_0=(\Sigma_0,\gs_1,\dots, \gs_n,\ws)$ is a weakly admissible Heegaard tuple and $\frS$ is a set of $\Spin^c$ structures on $X_{\g_1,\dots, \g_n}$, which is closed under the action of $\delta H^1(Y_{\g_i,\g_j})$ for $i<j$.
\begin{enumerate}
\item We say that $(\cD_0,\frS)$ is a \emph{multi-stabilizing diagram} if the following are satisfied. All elements of $\frS$ restrict to a single element of $\Spin^c(\d X_{\g_1,\dots, \g_n})$ which is furthermore torsion. Additionally, we require that if $\frs_1,\frs_2\in \frS$, then $d(\frs_1)=d(\frs_2)$, where $d(\frs)=(c_1^2(\frs)-2\chi-3\sigma)/4$.
\item We say that $(\cD_0,\frS)$ is \emph{algebraically rigid} if $\d=0$ on $\widehat{\CF}(\Sigma_0, \gs_i, \gs_{i+1},\frs_{i,i+1},\ws)$, where $\frs_{i,i+1}$ is the restriction to $Y_{\g_i,\g_{i+1}}$ of the elements of $\frS$. 
\end{enumerate}
\end{define}

\begin{example}
 Let $\cD_0$ be a Heegaard triple where all $\gs_i$ are pairwise related by small Hamiltonian translates. Let $\frS=\{\frs_0\}$, where $\frs_0$ is the unique $\Spin^c$ structure which has torsion restriction to the boundary. Then $(\cD_0,\frS)$ is a multi-stabilizing diagram. For suitable choices of translates, $\cD_0$ is algebraically rigid.
\end{example}

\begin{rem} When the set $\frS$ is understood from context (such as the case when there is a unique $\Spin^c$ structure on $X_{\g_1,\dots, \g_n}$ which has torsion restriction to $\d X_{\g_1,\dots, \g_n}$), we will usually refer to the diagram $\cD_0$ as being a multi-stabilizing diagram.
\end{rem}

In this section, we state a helpful result about stabilizations and holomorphic curves. Before we state our result, we introduce some notation.

As a first step, we recall some notation and basic facts about Stasheff's associahedron, $K_n$. We view $K_n$ as a convex polytope which models the compactification of $n+1$ marked points on the boundary of a disk. It is well known that $K_n$ admits an embedding into a Euclidean space and has the homology of a point. There is a well known cell structure on $K_n$, giving a convenient model for its homology, which we denote $C_*^{\cell}(K_n)$. The cells of this decomposition are in bijective correspondence with planar trees with $n$ inputs, one output, and no internal vertices of valence less than 3. If $T$ is such a tree, then the degree of the corresponding cell is 
\[
n-1-\#(\text{internal vertices}).
\]
The differential of a tree $T$ is the sum of all ways of breaking an internal vertex into two vertices, both with valence at least 3. The degree 0 cells correspond to trees with only valence 3 internal vertices. There is a codimension $0$ cell of degree $n-2$, which has just one internal vertex.

Suppose that $\cD=(\Sigma,\gs_1,\dots, \gs_n,w)$ is a weakly admissible Heegaard tuple and $J=(J_x)_{x\in K_{n-1}}$ is a stratified family of almost complex structures on $\Sigma\times D_n$ for computing holomorphic $n$-gons. Suppose also that $\frS\subset \Spin^c(X_{\g_1,\dots, \g_n})$ is a set which is closed under the actions of $\delta^1 H^1(Y_{\g_i,\g_j})$ for all $i<j$, and assume for simplicity that all elements of $\frS$ restrict to a single element of $\Spin^c(\d X_{\g_1,\dots, \g_n})$. Given a tree $T\in C_*^{\cell}(K_{n-1})$, there is a map
\[
\widehat{f}_{\cD, \frS; J,T}\colon \widehat{\CF}(\gs_1,\gs_2, \frs_{1,2})\otimes \cdots \otimes \widehat{\CF}(\gs_{n-1},\gs_n,\frs_{n-1,n})\to \widehat{\CF}(\gs_1,\gs_n, \frs_{1,n}),
\]
obtained by composing the polygon maps corresponding to each internal vertex of $T$.
 If $T$ has only valence 3 internal vertices, then there is a unique point of $K_{n-1}$ corresponding to $T$, which is in the strata of maximal codimension. We write $J_T$ for the corresponding almost complex structure, and otherwise suppress the tree from the notation in this case. The map $\widehat{f}_{\cD, \frS; J,T}$ is obtained by successively composing holomorphic triangle maps according to the tree $T$.

More generally, we can view the holomorphic polygon maps as fitting together compatibly to give a chain map
\[
\widehat{f}_{\cD,\frS; J}\colon C_*^{\cell}(K_{n-1})\otimes \widehat{\CF}(\gs_1,\gs_2,\frs_{1,2})\otimes \cdots \otimes \widehat{\CF}(\gs_{n-1}, \gs_n,\frs_{n-1,n})\to \widehat{\CF}(\gs_{1},\gs_n, \frs_{1,n}).
\]

The same statement holds for $\CF^-$ as long as we are more careful about finiteness of curve counts, either by restricting $\Spin^c$ structures or working over $\bF[[U]]$. We call this version of the map $f_{\cD,\frS;J}$ and its specialization to a particular tree $f_{\cD, \frS; J, T}$ analogously.

\begin{lem}
 Suppose that $\cD_0=(\Sigma_0,\gs_1,\dots, \gs_n,\ws)$ is a diagram and $\frS$ is a set of $\Spin^c$ structures such that $(\cD_0,\frS)$ is algebraically rigid and multi-stabilizing. Let $T_1$ and $T_2$ be two degree 0 trees. Then,
 \[
\widehat{f}_{\cD_0,\frS;J_{T_1}}(\xs_1,\dots, \xs_{n-1})=\widehat{f}_{\cD_0,\frS;J_{T_2}}(\xs_1,\dots, \xs_{n-1}),
 \]
 for any $\xs_i\in \bT_{\g_i}\cap \bT_{\g_{i+1}}.$
\end{lem}
 \begin{proof}
The trees $T_1$ and $T_2$ are homologous as elements of $C_*^{\cell}(K_{n-1})$, since $H_*^{\cell}(K_{n-1})\iso \Z$. Hence, the maps
\[
\widehat{f}_{\cD_0,\frS;J_{T_1}}, \widehat{f}_{\cD_0,\frS;J_{T_2}}\colon \widehat{\CF}(\gs_1,\gs_2, \frs_{1,2})\otimes \cdots \otimes \widehat{\CF}(\gs_{n-1},\gs_n, \frs_{n-1,n})\to \widehat{\CF}(\gs_1,\gs_n,\frs_{1,n})
\]
are chain homotopic.
 The differentials on domain and codomain of the two maps vanish, so any two chain homotopic maps between them are equal.
\end{proof}

\begin{prop}\label{prop:multi-stabilization-counts}
Let $\cD=(\Sigma,\ds_1,\dots, \ds_{n},\ws)$ and $\cD_0=(\Sigma_0,\gs_1,\dots,  \gs_{n},\ws_0)$ be Heegaard $n$-diagrams, where $n>2$. Assume $\frs\in \Spin^c(X_{\dt_1,\dots, \dt_{n}})$ and $\frS\subset \Spin^c(X_{\g_1,\dots, \g_{n}})$ and assume that $(\cD_0,\frS)$ is an algebraically rigid multi-stabilizing Heegaard diagram. Form the connected sum $\cD\# \cD_0$ by adding a tube from $\cD_0$ at some $w_0\in \ws_0$ to a point on $\cD$. Let $T$ be a degree 0 tree representing the generator of $C_*^{\cell}(K_{n-1})$. Suppose $J=(J_x)_{x\in K_{n-1}}$ and $J^0=(J^0_x)_{x\in K_{n-1}}$ are stratified families of almost complex structures on $\Sigma\times D_n$ and $\Sigma_0\times D_n$, respectively, for counting holomorphic $n$-gons.  Let $\theta_1,\dots, \theta_{n-1}$ be homogeneously graded elements (necessarily cycles) of $\widehat{\CF}(\Sigma_0,\gs_1,\gs_2,\frs_{1,2}),\dots, \widehat{\CF}(\Sigma_0,\gs_{n-1}, \gs_{n},\frs_{n-1,n})$, respectively, and suppose that
\[
\ys=\widehat{f}_{\cD_0,\frS;J^0_T}(\theta_1,\dots, \theta_{n-1})
\]
is non-zero. Then for any tree $T'$ 
\[
f_{\cD\# \cD_0, \frs\# \frS; (J \wedge J^0), T'}(\xs_1\times \theta_1,\dots, \xs_{n-1}\times \theta_{n-1})=f_{\cD, \frs; J, T'}(\xs_1,\dots, \xs_{n-1})\otimes \ys+\sum_{\substack{
\zs\in \bT_{\g_1}\cap \bT_{\g_n}\\
\gr(\zs)> \gr(\ys)}} \qs_{\zs}\otimes \zs,
\]
where $\qs_{\zs}$ are elements of $\ve{\CF}^-(\Sigma,\ds_1,\ds_n)$.
\end{prop}

\begin{proof} The proof follows from similar reasoning to \cite{HHSZExact}*{Propositions~10.2 and 10.6}, as we now sketch. See also \cite{OSLinks}*{Section~6.1}. Since we are using almost complex structures which are singular at the connected sum point, the relevant moduli spaces are fibered products over $D_{n}^k\times K_{n-1}$, where $k$ is the multiplicity of a class at the connected sum point. Consider a class $\psi\wedge \psi_0$, where $\psi_0\in \pi_2(\theta_1,\dots, \theta_{n-1}, \ys_0)$, for some intersection point $\ys_0$, and such that $\frs_{\ws_0}(\psi_0)\in \frS$. The assumption that $\widehat{f}_{\g_1,\dots, \g_n,\frS;J_T^0}(\theta_1,\dots, \theta_{n-1})=\ys$ is non-zero implies that there exists a class of triangles  $\psi'\in \pi_2(\theta_1,\dots, \theta_{n-1},\ys)$ which has $\frs_{\ws_0}(\psi')\in \frS$, $n_{\ws_0}(\psi')=0$ and $\mu(\psi')=0$.
Since all elements of $\frS$ have the same degree,
 we conclude that 
\[
\gr(\ys_0)-2 n_{\ws_0}(\psi_0)+\mu(\psi_0)=\gr(\ys).
\]
Using Sarkar's formula for the Maslov index, we see that
\[
\mu(\psi\wedge \psi_0)=\mu(\psi)+\mu(\psi_0)-2n_{w_0}(\psi_0).
\]
(The only difference between the formulas for $\mu(\psi\wedge \psi_0)$ and $\mu(\psi)+\mu(\psi_0)$ is in the Euler measure, which is corrected on the right hand side by $2n_{w_0}(\psi_0)$).
Hence, combining the above two formulas, we obtain
\[
\mu(\psi\wedge \psi_0)=\mu(\psi)+\gr(\ys)-\gr(\ys_0)+n_{\ws_0\setminus \{w_0\}}(\psi_0).
\]
The main claim concerns only classes where $\gr(\ys)-\gr(\ys_0)\ge 0$. Since the map $f_{\cD\# \cD_0, \frs\# \frs_0; J\wedge J^0, T'}$ counts index $3-n$ curves, we assume
\[
\mu(\psi)+\gr(\ys)-\gr(\ys_0)+n_{\ws_0\setminus \{w_0\}}(\psi_0)=3-n.
\]
By transversality for curves representing $\psi$, we may assume that $\mu(\psi)\ge 3-n$. Hence, if $\gr(\ys)-\gr(\ys_0)\ge 0$, then we must have $\mu(\psi)=3-n$, $\gr(\ys)=\gr(\ys_0)=0$ and $n_{\ws_0\setminus \{w_0\}}(\psi_0)=0$. In particular, the unconstrained moduli space $\cM(\psi)$ is 0-dimensional. It suffices to understand the moduli spaces on the $\Sigma_0$ side.

Let $k\ge 0$ be a fixed integer, and assume that $\gr(\ys)=\gr(\ys_0)$. Write $\cM_k(\theta_1,\dots, \theta_{n-1},\ys_0)$ for the moduli space of curves representing a class in $\pi_2(\theta_1,\dots, \theta_{n-1},\ys_0)$ with multiplicity $k$ at $w_0$. We now claim that
\[
\ev\colon \cM_k(\theta_1,\dots, \theta_{n-1},\ys_0)\to D_{n}^k\times K_{n-1}
\]
is odd degree if $\ys_0$ is a summand of $\ys$ in $\widehat{\CF}(\gs_1,\gs_n)$, and is even degree if $\ys_0$ is not a summand of $\ys$. This clearly implies the claim.

 Let $(\ve{d}, x)\in D_{n}^k\times K_{n-1}$ be generic, and pick a path $\g\colon [0,\infty)\to D_{n}^k\times K_{n-1}$, such that $\g(0)=(\ve{d},x)$ and $\g(t)$ has the following behavior as $t\to \infty$. As $t\to \infty$,  we assume that $x(t)$ approaches the tree $T$, viewed as a point in $\d K_{n-1}$, while the $k$-tuple of points $\ve{d}(t)$ travel towards one of the boundary punctures of $D_{n}^k$. Furthermore, we assume that under the identification of this end as a half cylinder $[0,1]\times[0,\infty),$ they approach some fixed $\ve{d}'\subset ((0,1)\times \R)^k$, modulo the $\R$-action on $[0,\infty)$. We consider the 1-dimensional moduli space of curves which have $\ev(u)\in \im \g$. The generic degenerations of this moduli space at finite $t$ correspond exactly to index 1 disks breaking off of the cylindrical ends. These cancel modulo 2, since the differentials vanish on each $\widehat{\CF}(\gs_i,\gs_{i+1}, \frs_{i,i+1})$. The ends appearing as $t\to \infty$ correspond to points in the following set:
\[
\bigg(\coprod_{\substack{
\phi\in \pi_2(\theta_1,\theta_1)\\ \mu(\phi)=2k}  } \cM(\phi, \ve{d}')\bigg)\times \bigg(\coprod_{\substack{\psi_0\in \pi_2(\theta_1,\dots, \theta_{n-1},\ys_0)\\ \mu(\psi_0)=0 \\ n_{\ws_0}(\psi_0)=0}} \cM_{J_T}(\psi_0)\bigg).
\]
The left-hand factor has odd cardinality by \cite{ZemDuality}*{Equation~(31)}. The count of the right-hand side is exactly the $\ys_0$ coefficient of $\widehat{f}_{\cD_0,\frS; J_T^0}(\theta_1,\dots, \theta_{n-1})$, completing the proof.
\end{proof}

\subsection{Remarks and special cases}

Proposition~\ref{prop:multi-stabilization-counts} generalizes \cite{HHSZExact}*{Propositions~10.2 and 10.6}, which concern triangles and quadrilaterals. In this section we describe some special cases.

\begin{rem}  The results of \cite{HHSZExact}*{Propositions~10.2 and 10.6} are stated in terms of an \emph{index tuple}, which does not appear in our Proposition~\ref{prop:multi-stabilization-counts}. The stabilization formulas from \cite{HHSZExact} are stated only in the case that one of the index tuple coordinates was 0. The entry of the index tuple being zero implied a specific formula for $\ve{y}=\widehat{f}_{\cD_0,\frS;J^0_T}(\theta_1,\dots, \theta_n)$. See \cite{HHSZExact}*{Lemma~10.4}. In our present paper, we also need the stabilization formulas in cases when the entry of the index tuple is non-zero, but where we know the value of $\widehat{f}_{\cD_0,\frS;J^0_T}(\theta_1,\dots, \theta_n)$.
\end{rem}

\begin{example}  Suppose that $\cD$ and $\cD_0$ are as in the statement of Proposition~\ref{prop:multi-stabilization-counts}, and suppose further that $\cD_0$ has the property that each $\gs_i$ is obtained from $\gs_1$ by a small Hamiltonian translation of each of the curves. Then Proposition~\ref{prop:multi-stabilization-counts} implies that
\[
f_{\cD\# \cD_0, \frs \# \frs_0; J\wedge J^0, T'}(\xs_1\times \theta_{1,2}^+,\dots, \xs_{n-1}\times \theta_{n-1,n}^+)=f_{\cD, \frs;J, T'}(\xs_1,\dots, \xs_{n-1})\otimes \Theta^+_{1,n}.
\]
Similarly,
\[
\begin{split}
&f_{\cD\# \cD_0, \frs \# \frs_0; J\wedge J^0, T'}(\xs_1\times \theta_{1,2}^-,\xs_2\times \theta_{2,3}^+,\dots, \xs_{n-1}\times \theta_{n-1,n}^+)\\
=&f_{\cD, \frs;J, T'}(\xs_1,\dots, \xs_{n-1})\otimes \theta^-_{1,n}+\sum_{\substack{\Theta\in \bT_{\g_1}\cap \bT_{\g_n} \\ \gr(\Theta)>\gr(\theta^-_{1,n})}} \qs_{\zs}\otimes \Theta.
\end{split}
\]
\end{example}

%% file: section-2-expanded-model.tex
\section{An expanded model of the involution}

\label{sec:expanded-def}

In this section, we describe a \emph{basepoint-expanded} doubling model of the involution, which will allow us to understand the transition maps for non-elementary handleslide equivalences. An outline of this section is as follows. In Section~\ref{sec:expanded-model-def-details}, we define the notion of a \emph{basepoint expanded, doubling enhanced Heegaard diagram} $\tilde{\frD}$, which can be used to construct an $\bF[U,Q]/Q^2$-complex $\CFI(\tilde{\frD})$. In Section~\ref{sec:maps-expanded-model}, we construct transition maps
\[
\tilde{\Psi}_{\tilde{\frD}_1\to \tilde{\frD}_2}\colon \CFI(\tilde{\frD}_1)\to \CFI(\tilde{\frD}_2)
\]
in the case that $\tilde\frD_1$ and $\tilde\frD_2$ are two basepoint expanded, doubling enhanced Heegaard diagrams which have the same underlying Heegaard surface, and which satisfy a weak admissibility condition. Unlike in Section~\ref{sec:elementary-equivalences}, we do not require $\tilde{\frD}_1$ and $\tilde{\frD}_2$ to be related by an elementary handleslide. In the subsequent Section~\ref{sec:relate-expanded}, we relate the complexes $\CFI(\tilde{\frD})$ and the transition maps $\tilde{\Psi}_{\tilde{\frD}_1\to \tilde{\frD}_2}$ to the non-expanded models described earlier. The key motivation for considering the maps $\tilde{\Psi}_{\tilde\frD_1\to \tilde \frD_2}$ is that they may be defined for non-elementary handleslide equivalences in a manner which clearly is independent of additional choices.

\subsection{Doubling with extra basepoints}
\label{sec:expanded-model-def-details}

We now describe our expanded model of the involution, which uses an extra basepoint. The presence of the extra basepoint simplifies some of the gluing arguments.

\begin{define} 
A \emph{basepoint expanded, doubling enhanced Heegaard diagram} $\tilde{\frD}$ consists of a singly pointed Heegaard diagram $\cH=(\Sigma,\as,\bs,w)$, a small disk $D\subset \Sigma$, containing $w$ along its boundary, a choice of point $w'$, as well as a collection of $2g+2$ attaching curves $\Ds\subset \Sigma\# \bar \Sigma$, as follows. The curves $\Ds$ are constructed by doubling a basis of pairwise disjoint, properly embedded arcs $\delta_1,\dots, \delta_{2g+2}$ on $\Sigma\setminus D$, which avoid $w$ and $w'$ and which form a basis of $H_1(\Sigma\setminus D, \d D\setminus \{w,w'\})$. Additionally, $\tilde{\frD}$ contains the choice of almost complex structures used to compute the holomorphic triangle maps.
\end{define}

Using a basepoint expanded diagram $\tilde{\frD}$, as above, we obtain a model of the involution by modifying the formula in Corollary~\ref{cor:doubling-model}, as follows. We let $c$ and $c'$ be small perturbations of the circle $\d D$, as in Figure~\ref{fig:5}.

We define a 1-handle map 
\[
F_1^{c\bar \b, c\bar \b}\colon \CF(\Sigma,\as,\bs,w)\to \CF(\Sigma\# \bar \Sigma, \as c \bar \bs, \bs c \bar \bs, w,w')
\]
by tensoring with the top degree generator. This map adds the basepoint $w'$. Similarly, we define a map
\[
F_3^{\a c',\a c'}\colon \CF(\Sigma \# \bar \Sigma, \as c' \bar \bs, \as c'\bar \as,w,w')\to \CF(\bar \Sigma, \bar \bs, \bar \as, w)
\]
by using the same formula as the standard 3-handle map.

  We define the basepoint-expanded model for the involution via the formula
 \[
\iota:=\eta\circ F_3^{\a c',\a c'}\circ f_{\a c \bar {\b} \to \a c' \bar {\b}}^{\a c' \bar \a}\circ   A_{\lambda}\circ  f^{\Dt\to \a c'\bar{\a} }_{\a c\bar{\b}} \circ f_{\a c\bar{\b}}^{\b c\bar{\b} \to \Dt} \circ  F_1^{c\bar{\b},c\bar{\b}}.
 \]
In the above equation, $A_{\lambda}$ denotes the relative homology map for an arc $\lambda$ connecting $w$ and $w'$ in the connected sum region. This map is discussed further in the subsequent Section~\ref{sec:hypercubes-rel-homology}.

\begin{rem}
Similar to the non-expanded model, the expanded model also naturally depends on our framing $\xi=(\xi_1,\xi_2,\xi_3)$ of the basepoint. We assume that the original Heegaard surface $\Sigma$ is positively tangent to $(\xi_1,\xi_2)$, that $w$ is in the direction of $\xi_1$ and $w'$ is in the direction of $\xi_2$, in the tube region. Furthermore, we assume that $\lambda$ is chosen so that if we identify $\Span(\xi_1,\xi_2)$ with $\C$, where $\xi_1=1$ and $\xi_2=i$, and we identify the intersection with the tube region of $\Sigma\# \bar \Sigma$ with $S^1\subset \C$, then $\lambda$ corresponds to $e^{it}$ for $t\in [0,\pi/2]$.
\end{rem}

 \begin{figure}[H]
	\centering
	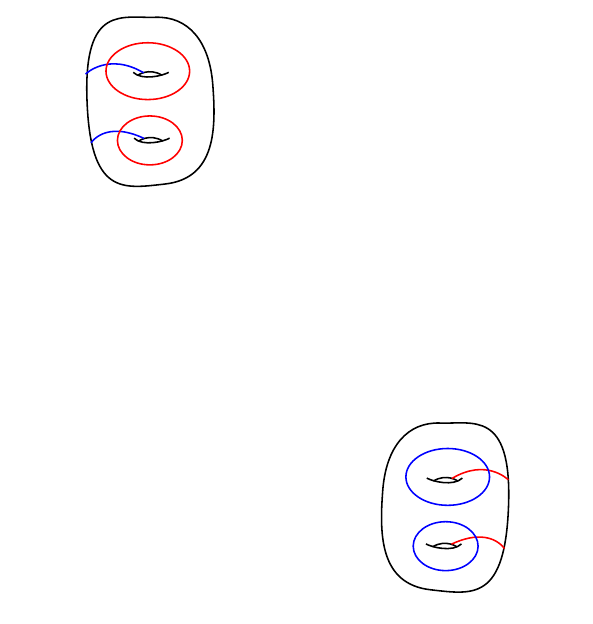
	\caption{The basepoint expanded doubling model of the involution. (Frames read in alphabetical order).}\label{fig:5}
\end{figure}

\subsection{Homology actions on hypercubes}
\label{sec:hypercubes-rel-homology}

In this section, we describe how to construct a homology action on a hypercube of Floer chain complexes. We can construct both a version for closed curves on $\Sigma$, as well as a relative one for arcs on $\Sigma$ with boundary in $\ws$. We focus on the relative case.

We recall the construction from \cite{ZemGraphTQFT}*{Section~5}. If $\cH=(\Sigma,\as,\bs,\ws)$ is a multi-pointed Heegaard diagram, and $\lambda$ is a path on $\Sigma$ connecting $w_1,w_2\in \ws$, there is an endomorphism
\begin{equation}
A_\lambda\colon \CF(\cH,\frs)\to \CF(\cH,\frs),\label{eq:A-lambda-original}
\end{equation}
which satisfies
\begin{equation}
[\d, A_\lambda]=U_{w_1}+U_{w_2}.\label{eq:[d,A-lambda]}
\end{equation}
In our present paper, we work with just one $U$-variable, so $A_{\lambda}$ becomes a chain map. The map $A_\lambda$ is defined as follows. If $\phi\in \pi_2(\xs,\ys)$ is a class of disks, we may define a quantity $a(\lambda,\phi)\in \Z$ by summing the changes of the multiplicities of $\phi$ across the alpha curves as one travels along $\lambda$. (Compare \cite{Ni-homological}). Sometimes it is helpful to write $a(\lambda,\phi)=(\d_{\a} D(\phi))\cdot \lambda$, where $\d_{\a} D(\phi)$ is the alpha boundary of the domain of $\phi$.  The map $A_\lambda$ is defined via the formula
\[
A_{\lambda}(\xs):=\sum_{\substack{\phi\in \pi_2(\xs,\ys)\\ \mu(\phi)=1}} a(\lambda,\phi) \# (\cM(\phi)/\R)U^{n_{w_1}(\phi)+\cdots+n_{w_n}(\phi)} \cdot \ys.
\]

The map $A_{\lambda}$ is sometimes a helpful tool when adding and removing basepoints. We now describe a version of $A_\lambda$ for hypercubes. Suppose we are given two hypercubes of attaching curves $\cL_\a$ and $\cL_{\b}$ on a multi-pointed Heegaard diagram, each consisting of handleslide equivalent attaching curves, such that the length 1 morphisms are all the top degree generators of $\CF^-(\#^n (S^1\times S^2))$. We now suppose that $\lambda$ is an arc which connects two basepoints, $w_1$ and $w_2$. Consider the pairing $\CF(\cL_\a,\cL_\b)$, where we identify the variables of the basepoints $w_1$ and $w_2$ to a single $U$.

 We now construct a morphism of hypercubes
\[
\scA_{\lambda}\colon \CF(\cL_{\a}, \cL_{\b})\to \CF(\cL_{\a}, \cL_{\b}).
\]
Recall that a morphism of hypercubes can itself be thought of as a hypercube of dimension 1 larger than $\dim \CF(\cL_{\a}, \cL_{\b})$. Our morphism of hypercubes $\scA_{\lambda}$ will have the property that the length one morphisms coincide with the chain maps $A_\lambda$ from~\eqref{eq:A-lambda-original}. We will show that $\scA_{\lambda}$ is well-defined up to chain homotopy of hypercube morphisms, and is natural with respect to restriction to sub-cubes of $\CF(\cL_{\a},\cL_{\b})$.

We note that we may formally view $\lambda$ as an input for a holomorphic polygon counting map by setting
\[
\begin{split}
&f_{\g_1,\dots,\g_j,\g_j,\dots, \g_m}(\xs_{1,2},\dots,\xs_{j-1,j}, \lambda, \xs_{j,j+1},\dots, \xs_{m-1,m})
\\:=&\sum_{\substack{\psi\in \pi_2(\xs_{0,1},\dots, \xs_{m-1,m},\zs)\\ \mu(\psi)=3-m}} (\d_{\g_j}D(\psi) \cdot \lambda)\# \cM(\psi)U^{n_{w_1}(\psi)+\cdots+n_{w_n}(\psi)}\cdot \zs,
\end{split}
\]
where $\cM(\psi)$ denotes the moduli space of holomorphic $m$-gons (in the usual sense).

\noindent As particular examples, 
\[
f_{\a,\a,\b}(\lambda,\xs)=A_{\lambda}(\xs)\quad \text{and} \quad f_{\a,\b,\b}(\xs,\lambda)=B_{\lambda}(\xs),
\] 
where $A_{\lambda}$ counts holomorphic disks weighted by $\#(\d_{\a}(\phi)\cap \lambda)$, while $B_{\lambda}$ counts disks weighted by $\#(\d_{\b}(\phi)\cap \lambda)$.

\begin{lem}\label{lem:associativity-extra-lambda}
The holomorphic polygon counting maps, defined with an extra input $\lambda$, satisfy the standard associativity rule as long as one uses the convention that $\d \lambda=U_{w_1}+U_{w_2}$, together with conventions that the polygon maps are $U_{w_i}$-equivariant and strictly unital (i.e. they vanish if they have $1$ as an input when there are more than 2 total inputs). In particular, if we work over a single $U$ variable, they satisfy the associativity relations with the convention that $\d \lambda=0$.
\end{lem}
\begin{proof}The proof differs based on whether $\ell=2$ or $\ell>2$. If $\ell=2$, then the result follows from Equation~\eqref{eq:[d,A-lambda]}. For the case that $\ell>2$, one counts the ends of index $(3-\ell)$ moduli spaces weighted by $\d_{\g_i} D(\psi)\cdot \lambda$ and quickly obtains the result.
\end{proof}

The construction of $\scA_{\lambda}$ is as follows. We will formally construct a cube-shaped diagram $\cL_{\a}^\lambda$ of dimension $\dim \cL_{\a}+1$. To begin, we take two copies of $\cL_{\a}$, for which we write $\cL_{\a}\times \{0\}$ and $\cL_{\a}\times \{1\}$, and formally adjoin additional length 1 arrows from $\cL_{\a}\times \{0\}$ to $\cL_{\a}\times \{1\}$, which we label by the character $\lambda$.

 We begin by constructing the length 2 chains of $\cL_{\a}^\lambda$, via the following argument. We will show that the 2-dimensional faces can be constructed to have the following form:
\begin{equation}
\begin{tikzcd}[column sep=2cm, row sep=2cm, labels=description]
\as
\ar[r, "\Theta_{\a',\a}"]\ar[d, "\lambda"]
\ar[dr,dashed, "U\eta_{\a',\a}"]& \as'\ar[d, "\lambda"]
\\
\as\ar[r,"\Theta_{\a',\a}"]& \as'
\end{tikzcd}
\label{eq:length-2-A-lambda}
\end{equation}
(We have not yet defined $\eta_{\a',\a}$.) The desired relation is that
\begin{equation}
A_\lambda(\Theta_{\a',\a})+A'_{\lambda}(\Theta_{\a',\a})=\d U \eta_{\a',\a}.\label{eq:length-2-relation-A-lambda}
\end{equation}
Here $A_{\lambda}$ counts changes of a disk class across $\as$, while $A'_{\lambda}$ counts changes across $\as'$.   Note that by \cite{HMZConnectedSum}*{Section~3.5}, we have that
\begin{equation}
A_{\lambda}(\Theta_{\a',\a})+A'_{\lambda}(\Theta_{\a',\a})=U(\Phi_{w_1}+\Phi_{w_2})(\Theta_{\a',\a}), \label{eq:telescope-A-lambda}
\end{equation}
where $\Phi_{w_i}$ denotes the map 
\[
\Phi_{w_i}(\xs)=U^{-1}\sum_{\substack{\phi\in \pi_2(\xs, \ys) \\ \mu(\phi)=1}} n_{w_i}(\phi) \# (\cM(\phi)/\R) U^{(n_{w_1}+\cdots+n_{w_n})(\phi)},
\]
extended $\bF[U]$-equivariantly. 

 We note that $\Phi_{w_i}(\Theta_{\a',\a})$ is a cycle of grading 1 higher than $\Theta_{\a',\a}$, so it is null-homologous, since $\Theta_{\a',\a}$ is the top degree generator. Let $\eta_{\a',\a}$ be any  chain of homogeneous grading one higher that $\Theta_{\a',\a}$ so that
\[
\d \eta_{\a',\a}=(\Phi_{w_1}+\Phi_{w_2})(\Theta_{\a',\a}).
\]
 We perform this construction to build all of the length 2 chains in $\cL_{\a}^{\lambda}$. (If we did not quotient by $U_{w_1}+U_{w_2}$, then we would instead obtain a length 2 morphism of the form $U_1\eta_{\a',\a}^1+U_2\eta_{\a',\a}^2$). 

We now construct the higher length chains in the cube $\cL^\lambda_{\a}$. Their construction is similar to the construction of the length 2 chains. We assume, by induction, that each length $m>1$ chain is of the form $U \eta_{\a',\a}$. (This is only necessary when $m\in \{2,3\}$, but is satisfiable for all $m>1$).

 Let $\veps<\veps'$ be points $\bE_{d}$, where $d=\dim \cL_\a$, and suppose that all chains of length less than $|\veps'-\veps|_{L^1}$ have already been defined. Set
\begin{equation}
\begin{split}
C_{\veps',\veps}:=&\sum_{\veps=\veps_1<\cdots <\veps_j=\veps'} \sum_{i=0}^j f_{\a^{\veps_j},\dots, \a^{\veps_1}}(\Theta_{\a^{\veps_j},\a^{\veps_{j-1}}},\dots,\Theta_{\a^{\veps_{i+1}},\a^{{\veps}_{i}}},\lambda,\Theta_{\a^{\veps_i},\a^{\veps_{i-1}}}\dots, \Theta_{\a^{\veps_2},\a^{\veps_1}})\\
+&\sum_{\veps=\veps_1<\cdots <\veps_j=\veps'} \sum_{i=0}^j f_{\a^{\veps_j},\dots, \a^{\veps_1}}(\Theta_{\a^{\veps_j},\a^{\veps_{j-1}}},\dots,\Theta_{\a^{\veps_{i+1}},\a^{\veps_i}},U\eta_{\a^{\veps_i},\a^{\veps_{i-1}}},\Theta_{\a^{{\veps}_{i-1}},\a^{{\veps}_{i-2}}}\dots, \Theta_{\a^{\veps_2},\a^{\veps_1}})
\end{split}
\label{eq:A-lambda-1st-equation}
\end{equation}

However, by an entirely analogous argument to ~\eqref{eq:telescope-A-lambda}, we have
\begin{equation}
\begin{split}
C_{\veps',\veps}=&\sum_{\veps=\veps_1<\cdots <\veps_j=\veps'}  f_{\a^{\veps_j},\dots, \a^{\veps_1}}^{n_{w_1}+n_{w_2}}(\Theta_{\a^{\veps_j},\a^{{\veps}_{j-1}}},\dots,\Theta_{\a^{\veps_2},\a^{\veps_1}})\\
+&\sum_{\veps=\veps_1<\cdots <\veps_j=\veps'} \sum_{i=0}^j f_{\a^{\veps_j},\dots, \a^{\veps_1}}(\Theta_{\a^{\veps_j},\a^{{\veps}_{j-1}}},\dots,\Theta_{\a^{\veps_{i+1}},\a^{\veps_i}},U\eta_{\a^{\veps_i},\a^{\veps_{i-1}}},\Theta_{\a^{{\veps}_{i-1}},\a^{\veps_{i-2}}}\dots, \Theta_{\a^{\veps_2},\a^{\veps_1}})
\end{split}
\label{eq:A-lambda-2nd-equation}
\end{equation}
where $f_{\a^{\veps_n},\dots, \a^{\veps_1}}^{n_{w_1}+n_{w_2}}$ counts holomorphic polygons with a factor of $n_{w_1}(\psi)+n_{w_2}(\psi)$. In particular, $C_{\veps',\veps}=U C'_{\veps',\veps}$, for some chain $C'_{\veps',\veps}$. Since the hypercube relations are satisfied for all proper faces, the element $C_{\veps',\veps}$ is a cycle. Since the complexes are free $\bF[U]$-modules, this implies that $C'_{\veps',\veps}$ is also a cycle. However $C_{\veps',\veps}$ has grading $m-2$ higher than the top degree generator, where $m=|\veps'-\veps|_{L^1}$. Hence $C'_{\veps',\veps}$ has grading $m$ higher than the top degree generator. As long as $m\ge 1$, we conclude that $C'_{\veps',\veps}$ is a boundary. We set $\eta_{\a^{\veps'}, \a^{\veps}}$ to be any chain of homogeneous grading $m+1$ such that $\d \eta_{\a^{\veps'},\a^{\veps}}=C'_{\veps',\veps}$.

Finally, we define the hypercube corresponding to the morphism $\scA_\lambda$ to be the pairing of $\CF(\cL^\lambda_{\a}, \cL_\b)$, where the hypercube maps are defined by using the normal formulas for the maps in a pairing, while allowing the chains $\lambda$ to be inputs.

\begin{lem}\,
\begin{enumerate}
\item The diagram $\cL_{\a}^{\lambda}$ satisfies the hypercube relations (interpreted in the way described above).
\item The map $\scA_{\lambda}$ satisfies $\d_{\Mor}(\scA_{\lambda})=0$ (i.e. $\scA_{\lambda}$ is a cycle).
\item The map $\scA_{\lambda}$ is independent, up to chain-homotopy of hypercube morphisms, from the choices of chains used in its construction.
\end{enumerate}
\end{lem}
\begin{proof}The first claim follows immediately from the construction.
The second claim follows quickly from Gromov compactness interpreted via Lemma~\ref{lem:associativity-extra-lambda}.

We consider the final claim. This is proven as follows. Suppose we have two models of $\cL_{\a}^{\lambda}$, which are both constructed using the above procedure. We consider a partially constructed hypercube of dimension $d+2$, given by the following diagram:
\[
\cL_{\a}^{\lambda_1\to \lambda_2}=
\begin{tikzcd}[column sep=2cm, row sep=2cm, labels=description]
\cL_{\a}\ar[d, "\id"] \ar[dr,dashed, "*"]\ar[r,"\lambda_1"] &\cL_{\a} \ar[d,"\id"]\\
\cL_{\a}\ar[r, "\lambda_2"]& \cL_{\a}.
\end{tikzcd}
\]
Here $\lambda_1$ denotes the chains in one model of $\cL_{\a}^{\lambda}$, while $\lambda_2$ denotes the chains constructed in the other model, and the $*$ arrows are yet to be defined. The hypercube relations are satisfied on the four $d$-dimensional faces of the above diagram. We wish to define chains which increment both of the displayed coordinates, such that the hypercube relations are satisfied. The construction of such arrows (corresponding to the $*$ arrow above) follows nearly verbatim from the construction of  the cube $\cL^\lambda_\a$ itself. Pairing $\cL_{\a}^{\lambda_1\to \lambda_2}$ with $\cL_{\b}$ gives the homotopy between the two models of $\scA_{\lambda}$, completing the proof.
\end{proof}

\begin{rem}
\begin{enumerate}
\item  If one works over the ring $\bF[U_{w_1},\dots, U_{w_n}]$, one obtains instead that $\d_{\Mor}(\scA_{\lambda})=(U_{w_1}+U_{w_2})\cdot \id$, where $\d \lambda=U_{w_1}+U_{w_2}$. 
\item The simplest case of the map $\scA_{\lambda}$ is when $\cL_{\a}$ is 0-dimensional. Then $\cL_\a^\lambda$ is just the formal hypercube
\[
\begin{tikzcd} \as \ar[r, "\lambda"] & \as \end{tikzcd}
\]
If $\cL_{\b}$ is a beta hypercube, then the map
\[
\scA_{\lambda}\colon \CF(\as, \cL_{\b})\to \CF(\as, \cL_{\b})
\]
requires no additional choices.
\item Naturally, one may also build a hypercube $\cL_{\b}^{\lambda}$, similar to above, and define
\[
\scB_{\lambda}\colon \CF(\cL_\a, \cL_{\b})\to \CF(\cL_{\a},\cL_{\b})
\]
as the hypercube $\CF(\cL_{\a}, \cL_{\b}^{\lambda})$.
\end{enumerate}
\end{rem}

\subsection{Maps for general equivalences on the expanded model}
\label{sec:maps-expanded-model}

We now define the naturality maps for a general alpha or beta equivalence in terms of the expanded model of the involution. Suppose that $\tilde\frD$ and $\tilde\frD'$ are two basepoint expanded doubling enhanced Heegaard diagrams, which share the same underlying Heegaard surface $\Sigma$. Write $\cH=(\Sigma,\as,\bs,w)$ and $\cH'=(\Sigma,\as',\bs',w)$ for the underlying non-involutive Heegaard diagrams of $\tilde\frD$ and $\tilde\frD'$, respectively. Write $\Ds$ and $\Ds'$ for the choices of doubling curves. Furthermore, we assume that the diagram
\[
 (\Sigma\# \bar{\Sigma},\as',\bar \as',\as,\bar \as,\bs,\bar \bs,\bs', \bar \bs',c,c',\Ds,\Ds',w,w')
 \]
  is weakly admissible. (This may be achieved by winding $\as$, $\as'$, $\bs$ and $\bs'$ sufficiently).
  \begin{rem}
  \label{rem:admissibility-unconventional}
   Note that the above is not a valid Heegaard diagram since none of the attaching curve sets have $2g(\Sigma)+1$ curves,  however it still makes sense to require weak-admissibility. Indeed the definition of weak-admissibility is that if $P$ is a non-zero integral 2-chain with boundary equal to a linear combination of curves on the diagram and $n_{w}(P)=n_{w'}(P)=0$, then $P$  has both positive and negative multiplicities.
\end{rem}

\begin{figure}[ht]
\[
\begin{tikzcd}[labels=description, row sep=1cm]
\CF(\as,\bs,w)
	\ar[r]
	\ar[d, "F_{1}^{c\bar{\b},c\bar{\b}}"]
& \CF(\as',\bs,w)
	\ar[r]
	\ar[d, "F_{1}^{c\bar{\b}',c\bar{\b}}"]
& \CF(\as',\bs',w)
	\ar[d,"F_{1}^{c\bar{\b}',c\bar{\b}'}"]
\\
\CF(\as c \bar{\bs}, \bs c \bar{\bs},w,w')
	\ar[r]
	\ar[d]
	\ar[dr,dashed]
& \CF(\as'c \bar{\bs}', \bs c \bar{\bs},w,w')
	\ar[r]
	\ar[d]
	\ar[dr,dashed]
& \CF(\as' c \bar{\bs}', \bs' c \bar{\bs}',w,w')
	\ar[d]
\\
\CF(\as c \bar \bs, \Ds,w,w')
	\ar[r]
	\ar[d, "A_{\lambda}"]
	\ar[dr,dashed]
& \CF(\as'c\bar \bs',\Ds,w,w')
	\ar[r]
	\ar[d, "A_{\lambda}"]
	\ar[dr,dashed]
& \CF(\as' c \bar \bs', \Ds',w,w')
	\ar[d, "A_{\lambda}"]
\\
\CF(\as c \bar \bs, \Ds,w,w')
	\ar[r]
	\ar[d]
	\ar[dr,dashed]
& \CF(\as'c\bar \bs',\Ds,w,w')
	\ar[r]
	\ar[d]
	\ar[dr,dashed]
& \CF(\as' c \bar \bs', \Ds',w,w')
	\ar[d]
\\
\CF(\as c\bar{\bs}, \as c' \bar{\as},w,w')
	\ar[r]
	\ar[d]
	\ar[dr,dashed]
&\CF(\as'c \bar{\bs}', \as c' \bar{\as},w,w')
	\ar[r]
	\ar[d]
	\ar[dr,dashed]
&\CF(\as' c \bar{\bs}', \as' c \bar{\as}',w,w')
	\ar[d]
\\
\CF(\as c' \bar{\bs}, \as c' \bar{\as},w,w')
	\ar[r]
	\ar[d, "F_3^{\a c',\a c'}"]
&\CF(\as' c' \bar{\bs}', \as c' \bar{\as},w,w')
	\ar[r]
	\ar[d, "F_3^{\a'c',\a c'}"]
& \CF(\as' c' \bar{\bs}', \as' c' \bar{\as}',w,w')
	\ar[d, "F_3^{\a'c',\a'c'}"]
\\
\CF(\bar{\bs}, \bar{\as},w)
	\ar[r]
	\ar[d,"\id"]
	\ar[drr,dashed, "h_{\bar \b\to \bar \b'}^{\bar \a \to \bar \a'}"]
& \CF(\bar{\bs}',\bar{\as},w)
	\ar[r]
& \CF(\bar{\bs}',\bar{\as}',w)
	\ar[d,"\id"]
\\
\CF(\bar \bs, \bar \as,w)
	\ar[r]
&
\CF(\bar \bs, \bar \as',w)
	\ar[r]
&
\CF(\bar \bs', \bar \as',w)
\end{tikzcd}
\]
\caption{The hyperbox whose compression is the basepoint expanded transition map $\tilde\Psi_{\tilde\frD\to \tilde\frD'}$.}
\label{fig:transition-map-1}
\end{figure}

Our transition map $\tilde \Psi_{\tilde \frD\to \tilde \frD'}$ is defined by the hyperbox in Figure~\ref{fig:transition-map-1}. We now describe aspects of the construction in more detail. The second and third levels are obtained by pairing hyperboxes of attaching curves in the obvious manner. The fourth level is obtained from the argument in Section~\ref{sec:hypercubes-rel-homology}, where we take $\lambda$ to be a path in the neck region which connects the two basepoints $w$ and $w'$.

 We now consider the levels involving the 1-handle and 3-handle maps. 
The map $F_{1}^{c\bar{\b}',c \bar{\b}}$ is different than the other 1-handle maps we have seen so far, since we do not require the diagram $(\Sigma, \bar{\bs}, \bar{\bs}')$ to be a standard diagram for $(S^1\times S^2)^{\# g}$. Instead, we only require the diagram to be admissible. We define the top level as the compression of the following hyperbox
\begin{equation}
\begin{tikzcd}[labels=description, row sep=1cm, column sep=.4cm]
\CF(\as,\bs,w) 
	\ar[r]
	\ar[d, "{\otimes \Theta_{\bar \b,\bar \b}}"]
& \CF(\as',\bs,w)
	\ar[r]
	\ar[d, "\otimes \Theta_{\bar{\b}', \bar \b}"]
& \CF(\as',\bs',w)
	\ar[d,"\otimes \Theta_{\bar \b', \bar \b'}"]
\\
\CF(\as, \bs ,w)\otimes \CF(\bar \bs, \bar \bs, w')
	\ar[r]
	\ar[d, "F_1^{c,c}"]
& \CF(\as', \bs,w) \otimes \CF(\bar{\bs}', \bar \bs,w')
	\ar[r]
	\ar[d,"F_1^{c,c}"]
& \CF(\as', \bs',w)\otimes \CF(\bar{\bs}', \bar{\bs}',w')
	\ar[d,"F_1^{c,c}"]
\\
\CF(\as c \bar{\bs}, \bs c \bar{\bs},w,w')
	\ar[r]
& \CF(\as'c \bar{\bs}', \bs c \bar{\bs},w,w')
	\ar[r]
& \CF(\as' c \bar{\bs}', \bs' c \bar{\bs}',w,w')
\end{tikzcd}
\label{eq:1-handle-expanded-hypercube}
\end{equation}

The horizontal arrows in Equation~\eqref{eq:1-handle-expanded-hypercube} are holomorphic triangle maps. In the middle, they are the tensor product of the holomorphic triangle maps on the two relevant Heegaard triples. (Or equivalently, the map which counts holomorphic triangles in the symplectic manifold $(\Sigma\sqcup \bar \Sigma)\times D_3$, where $D_3$ denotes a disk with three boundary punctures).

\begin{lem} The diagram in Equation~\eqref{eq:1-handle-expanded-hypercube} satisfies the hyperbox relations.
\end{lem}
\begin{proof} That the maps $\otimes \Theta_{\bar \b, \bar \b}$ (and so forth) are chain maps follows from the fact that $\Theta_{\bar \b, \bar \b}$ is a cycle. The maps $F_1^{c,c}$ are chain maps by \cite{ZemGraphTQFT}*{Proposition~8.5}. The fact that the top left and right squares commute is a consequence of the fact that
\[
f_{\bar \b', \bar \b, \bar \b}(\Theta_{\bar \b', \bar \b}, \Theta_{\bar \b, \bar \b})=\Theta_{\bar \b', \bar \b}\quad \text{and} \quad f_{\bar \b', \bar \b, \bar \b'}(\Theta_{\bar \b', \bar \b}, \Theta_{ \bar \b, \bar \b'})=\Theta_{\bar \b', \bar \b'}.
\]
The left relation follows from a small triangle argument \cite{HHSZExact}*{Proposition~11.1}, while the right follows from grading considerations.  Commutativity of the bottom two squares follows from the stabilization results for triangles \cite{ZemGraphTQFT}*{Theorem~8.8}. 
\end{proof}

\begin{rem}
 We remark that our main motivation for using the expanded model of involution is to simplify the construction of the 1-handle and 3-handle hypercubes in the transition map. Indeed, the 1-handle hypercube at the top level of Equation~\eqref{def:transition-map-elementary-handleslide} is challenging to construct for general $\bs$ and $\bs'$. The challenge is that in Heegaard Floer theory, degenerating the holomorphic $\ell$-gon maps along a connected sum neck in the Heegaard diagram yields moduli spaces which are fibered products over an evaluation map to $\Sym^n(D_\ell)\times K_{\ell-1}$, where $K_{\ell-1}$ denotes an associahedron. Here, $n$ denotes the multiplicity at the connected sum point. The factor of $\Sym^{n}(D_{\ell})$ records asymptotics of holomorphic curves at the connected sum point. The factor of $K_{\ell-1}$  records a choice of almost complex structure on $D_\ell$. The construction of the 1-handle hypercube amounts to a problem of deforming the diagonal in $\Sym^n(D_\ell)\times\Sym^n(D_\ell)\times K_{\ell-1}\times K_{\ell-1}$. However, in the expanded model of the involution we can use stabilization results like Proposition~\ref{prop:multi-stabilization-counts} to reduce the problem to considering diagonals in $K_{\ell-1}\times K_{\ell-1}$, which is a simpler problem and is sufficient for our purposes.
\end{rem}

%

%

%% file: fig5.pdf_tex
\begingroup%
  \makeatletter%
  \providecommand\color[2][]{%
    \errmessage{(Inkscape) Color is used for the text in Inkscape, but the package 'color.sty' is not loaded}%
    \renewcommand\color[2][]{}%
  }%
  \providecommand\transparent[1]{%
    \errmessage{(Inkscape) Transparency is used (non-zero) for the text in Inkscape, but the package 'transparent.sty' is not loaded}%
    \renewcommand\transparent[1]{}%
  }%
  \providecommand\rotatebox[2]{#2}%
  \newcommand*\fsize{\dimexpr\f@size pt\relax}%
  \newcommand*\lineheight[1]{\fontsize{\fsize}{#1\fsize}\selectfont}%
  \ifx\svgwidth\undefined%
    \setlength{\unitlength}{284.59940027bp}%
    \ifx\svgscale\undefined%
      \relax%
    \else%
      \setlength{\unitlength}{\unitlength * \real{\svgscale}}%
    \fi%
  \else%
    \setlength{\unitlength}{\svgwidth}%
  \fi%
  \global\let\svgwidth\undefined%
  \global\let\svgscale\undefined%
  \makeatother%
  \begin{picture}(1,1.0519527)%
    \lineheight{1}%
    \setlength\tabcolsep{0pt}%
    \put(0,0){\includegraphics[width=\unitlength,page=1]{fig5.pdf}}%
    \put(0.32251438,0.86412344){\makebox(0,0)[lt]{\lineheight{1.25}\smash{\begin{tabular}[t]{l}$w$\end{tabular}}}}%
    \put(0,0){\includegraphics[width=\unitlength,page=2]{fig5.pdf}}%
    \put(0.73973479,0.90390296){\makebox(0,0)[lt]{\lineheight{1.25}\smash{\begin{tabular}[t]{l}$w$\end{tabular}}}}%
    \put(0.7407333,0.84796411){\makebox(0,0)[lt]{\lineheight{1.25}\smash{\begin{tabular}[t]{l}$w'$\end{tabular}}}}%
    \put(0,0){\includegraphics[width=\unitlength,page=3]{fig5.pdf}}%
    \put(0.23605483,0.54840213){\makebox(0,0)[lt]{\lineheight{1.25}\smash{\begin{tabular}[t]{l}$w$\end{tabular}}}}%
    \put(0.23577561,0.5128264){\makebox(0,0)[lt]{\lineheight{1.25}\smash{\begin{tabular}[t]{l}$w'$\end{tabular}}}}%
    \put(0,0){\includegraphics[width=\unitlength,page=4]{fig5.pdf}}%
    \put(0.73577634,0.54608125){\makebox(0,0)[lt]{\lineheight{1.25}\smash{\begin{tabular}[t]{l}$w$\end{tabular}}}}%
    \put(0.73432751,0.49710702){\makebox(0,0)[lt]{\lineheight{1.25}\smash{\begin{tabular}[t]{l}$w'$\end{tabular}}}}%
    \put(0,0){\includegraphics[width=\unitlength,page=5]{fig5.pdf}}%
    \put(0.23921616,0.20562054){\makebox(0,0)[lt]{\lineheight{1.25}\smash{\begin{tabular}[t]{l}$w$\end{tabular}}}}%
    \put(0.65256257,0.17343228){\makebox(0,0)[lt]{\lineheight{1.25}\smash{\begin{tabular}[t]{l}$w$\end{tabular}}}}%
    \put(0.22788506,0.1538229){\makebox(0,0)[lt]{\lineheight{1.25}\smash{\begin{tabular}[t]{l}$w'$\end{tabular}}}}%
    \put(0,0){\includegraphics[width=\unitlength,page=6]{fig5.pdf}}%
    \put(0.0195063,0.70108149){\makebox(0,0)[lt]{\lineheight{1.25}\smash{\begin{tabular}[t]{l}$\mathrm{A}$\end{tabular}}}}%
    \put(0.51326605,0.70793631){\makebox(0,0)[lt]{\lineheight{1.25}\smash{\begin{tabular}[t]{l}$\mathrm{B}$\end{tabular}}}}%
    \put(0.01632694,0.36045541){\makebox(0,0)[lt]{\lineheight{1.25}\smash{\begin{tabular}[t]{l}$\mathrm{C}$\end{tabular}}}}%
    \put(0.51322903,0.36327896){\makebox(0,0)[lt]{\lineheight{1.25}\smash{\begin{tabular}[t]{l}$\mathrm{D}$\end{tabular}}}}%
    \put(0.01893061,0.0199392){\makebox(0,0)[lt]{\lineheight{1.25}\smash{\begin{tabular}[t]{l}$\mathrm{E}$\end{tabular}}}}%
    \put(0.51621924,0.02276269){\makebox(0,0)[lt]{\lineheight{1.25}\smash{\begin{tabular}[t]{l}$\mathrm{F}$\end{tabular}}}}%
    \put(0.25155822,0.70369029){\makebox(0,0)[t]{\lineheight{1.25}\smash{\begin{tabular}[t]{c}$(\Sigma,\as,\bs,w)$\end{tabular}}}}%
    \put(0.75397295,0.70369027){\makebox(0,0)[t]{\lineheight{1.25}\smash{\begin{tabular}[t]{c}$(\Sigma\# \bar \Sigma,\as c \bar \bs,\bs c \bar \bs,w,w')$\end{tabular}}}}%
    \put(0.2515716,0.36190748){\makebox(0,0)[t]{\lineheight{1.25}\smash{\begin{tabular}[t]{c}$(\Sigma\# \bar \Sigma,\as c \bar \bs,\Ds,w,w')$\end{tabular}}}}%
    \put(0.75396958,0.36190748){\makebox(0,0)[t]{\lineheight{1.25}\smash{\begin{tabular}[t]{c}$(\Sigma\# \bar \Sigma,\as c \bar \bs, \as c'\bar \as,w,w')$\end{tabular}}}}%
    \put(0.25157023,0.01983564){\makebox(0,0)[t]{\lineheight{1.25}\smash{\begin{tabular}[t]{c}$(\Sigma\# \bar \Sigma,\as c'\bar \bs,\as c' \bar \bs,w,w')$\end{tabular}}}}%
    \put(0.75396958,0.01983567){\makebox(0,0)[t]{\lineheight{1.25}\smash{\begin{tabular}[t]{c}$(\bar \Sigma,\bar \bs, \bar \as,w)$\end{tabular}}}}%
  \end{picture}%
\endgroup%

%% file: section-2a-change-models-v2.tex
\section{Relating the expanded and ordinary models of the involution}
\label{sec:relate-expanded}

In this section, we relate the expanded model of the involution from Section~\ref{sec:expanded-def} with the ordinary model from Section~\ref{sec:doubling-def}. If $\frD$ is a doubling enhanced Heegaard diagram, and $\tilde{\frD}$ is a basepoint expanded doubling diagram, such that the underlying (non-involutive) Heegaard diagrams for $\frD$ and $\tilde{\frD}$ coincide, we will construct a chain homotopy equivalence
\[
F_{\frD\to \tilde{\frD}}\colon \CFI(\frD)\to \CFI(\tilde{\frD}).
\]

\begin{prop} \label{prop:expansion}\,
\begin{enumerate}
\item\label{prop:expansion-part-2}
 If $\frD$ and $\frD'$ are ordinary doubling diagrams which differ by an elementary equivalence, then 
\[
\tilde\Psi_{\tilde \frD\to \tilde \frD'}\circ F_{\frD\to \tilde{\frD}}+F_{\frD'\to \tilde \frD'}\circ \Psi_{\frD\to \frD'}\simeq 0.
\]
\item \label{prop:expansion-part-3} If $\tilde \frD, $ $\tilde \frD'$ and $\tilde \frD''$ are three basepoint expanded doubling enhancements of three handleslide equivalent diagrams, then
\[
\tilde{\Psi}_{\tilde \frD'\to \tilde \frD''}\circ \tilde{\Psi}_{\tilde \frD\to \tilde \frD'}\simeq \tilde{\Psi}_{\tilde \frD\to \tilde \frD''}.
\]
\end{enumerate}
\end{prop}

The definition of the map $F_{\frD\to \tilde{\frD}}$ is in Section~\ref{subsec:expanding-the-involution}. The proof of Part~\eqref{prop:expansion-part-2} of Proposition~\ref{prop:expansion} is in Section~\ref{sec:expansion-relation-transition-maps}. Finally we prove Part~\eqref{prop:expansion-part-3} in Section~\ref{sec:composition-law-expanded-proof}.

\subsection{Preliminaries}

In this section, we review a few maps that will appear in the construction of $F_{\frD\to \tilde{\frD}}$, and also construct an important hypercube.

Firstly, we recall the \emph{free-stabilization} maps from \cite{ZemGraphTQFT}*{Section~6}. If $\cH=(\Sigma,\as,\bs,\ws)$ is a Heegaard diagram, and $w'$ is a point in $\Sigma\setminus(\as\cup \bs \cup \ws)$, we may form a new diagram $\cH^+=(\Sigma,\as\cup \a_0,\bs\cup \b_0, \ws\cup \{w'\})$ by adding attaching curves $\a_0$ and $\b_0$ which are contained in a small ball centered at $w'$. We assume that $\a_0$ and $\b_0$ intersect in two points, and furthermore are disjoint from $\as\cup \bs\cup \ws$.

 In the above setting, there are chain maps 
\[
S_{w'}^+\colon \CF^-(\cH,\frs)\to \CF^-(\cH^+, \frs)\quad \text{and} \quad S_{w'}^-\colon \CF^-(\cH^+,\frs)\to \CF^-(\cH,\frs)
\]
called the free-stabilization maps. They are given by the formulas
\[
S_{w'}^+(\xs)=\xs\times \theta^+,\qquad S_{w'}^-(\xs\times \theta^+)=0,\quad \text{and} \quad S_{w'}^-(\xs\times \theta^-)=\xs,
\]
extended equivariantly over $\bF[U]$.

If $\lambda$ is an arc which connects $w'$ to another basepoint in $\ws$, then
\begin{equation}
S_{w'}^- A_{\lambda} S_{w'}^+=\id,
\label{eq:free-stabilization}
\end{equation}
for appropriately chosen almost complex structure.  This is proven in \cite{ZemGraphTQFT}*{Lemma~7.10} by a model computation.

We now generalize Equation~\eqref{eq:free-stabilization} to the setting of hypercubes. Suppose that $\scH=\CF(\cL_{\a},\cL_{\b})$ is a hypercube formed by pairing hypercubes of attaching curves $\cL_{\a}$ and $\cL_{\b}$. 
 We form new hypercubes, $\cL_{\a}^+$ and $\cL_{\b}^+$, by adding small translates of $\a_0$ and $\b_0$ to each attaching curve of $\cL_{\a}$ and $\cL_{\b}$. We form the morphisms of $\cL_{\a}^+$ by tensoring each morphism of $\cL_{\a}$ with $\theta_{\a_0,\a_0}^+$. We construct $\cL_{\b}^+$ similarly. The hypercube relations for $\cL_{\a}^+$ and $\cL_{\b}^+$ follow from Proposition~\ref{prop:multi-stabilization-counts}. Finally, we define a new hypercube of chain complexes, $\scH^+=\CF(\cL_{\a}^+,\cL_{\b}^+)$.

We assume, additionally, that $\cL_{\a}$ is a hypercube of handleslide equivalent alpha attaching curves, and that the length 1 morphisms are all cycles representing the top degree generator of homology. We construct the morphism
\[
\scA_{\lambda}\colon \CF(\cL_{\a}^+,\cL_{\b}^+)\to \CF(\cL_{\a}^+,\cL_{\b}^+)
\]
as in Section~\ref{sec:hypercubes-rel-homology}.

 We consider the following diagram:
 \begin{equation}
 \begin{tikzcd}[labels=description, row sep=1.2cm, column sep=1.2cm]
 \scH 
 	\ar[r,"\id"]
 	\ar[ddd,"\id"]
 & \scH
  	\ar[d,"\scS_{w'}^+"]
 \\
 &
 \scH^+
 	\ar[d, "\scA_{\lambda}"]
 \\
 &
 \scH^+
 	\ar[d, "\scS_{w'}^-"]
 \\
 \scH
 	\ar[r,"\id"]
 &
 \scH
 \end{tikzcd}
 \label{eq:add-basepoint-hypercube}
 \end{equation}
 
 The maps $\scS_{w'}^+$ and $\scS_{w'}^-$ denote hypercube versions of the free-stabilization maps; that is, they are hypercube morphisms with only length 1 maps, which are the ordinary free-stabilization maps on the Floer complexes. In the language of \cite{HHSZExact}, these are \emph{hypercubes of stabilization}.

 If we compress the right-hand side of~\eqref{eq:add-basepoint-hypercube} we obtain a cube-shaped diagram. In the following, we verify that the resulting diagram is a hypercube of chain complexes: 
 
 \begin{prop}\label{prop:simple-expansion-hypercube}
 Let $\scH$ and $\scH^+$ be as above. Then $\scA_{\lambda}$ may be chosen so that the diagram in ~\eqref{eq:add-basepoint-hypercube} becomes a hypercube of chain complexes once we compress the right hand-side. In the hypercubes along the right-hand side, we use almost complex structures which are nodal near the free-stabilization region.
 \end{prop}

 We begin with a lemma:
\begin{lem}
\label{lem:A_lambda-cube-stabilized}
Let $\scH$ and $\scH^+$ be as above. The hypercube $(\cL_\a^+)^{\lambda}$ may be chosen so that each chain of length at least 2 in the $\lambda$ direction is of the form $U\eta_{\a',\a}\otimes \theta^+_{\a_0,\a_0}$.
\end{lem} 
\begin{proof} The proof is to follow the steps of the construction in Section~\ref{sec:hypercubes-rel-homology} and verify that each of the chains in $\cL_{\a}^\lambda$ of length 2 or more may be taken to be of the stated form. Suppose $\veps<\veps'$ and $|\veps'-\veps|_{L^1}=1$. The chain of $\cL_{\a}^+$ from $\as_{\veps}$ to $\as_{\veps'}$ is of the form $\Theta_{\a_{\veps'},\a_{\veps}}\otimes \theta^+_{\a_0,\a_0}$ by definition. As in Equation~\eqref{eq:length-2-relation-A-lambda}, the desired length 2 hypercube relation is 
\begin{equation}
(A_\lambda+A_{\lambda}')(\Theta_{\a',\a}\otimes \theta_{\a_0,\a_0}^+)=\d (U\eta_{\a',\a}\otimes \theta_{\a_0,\a_0}^+). \label{eq:A_lambda-length-2-stabilized}
\end{equation}
From Equation~\eqref{eq:telescope-A-lambda}, we know
\[
(A_\lambda+A_{\lambda}')(\Theta_{\a',\a}\otimes \theta_{\a_0,\a_0}^+)=U(\Phi_{w}+\Phi_{w'})(\Theta_{\a',\a}\otimes \theta_{\a_0,\a_0}^+).
\]
Using the differential computation from \cite{ZemGraphTQFT}*{Proposition~6.5}, we see that $\Phi_{w'}(\Theta_{\a',\a}\otimes \theta_{\a_0,\a_0}^+)=0$ and 
\[
\Phi_{w}(\Theta_{\a',\a}\otimes \theta^+_{\a_0,\a_0})=\Phi_{w}^-(\Theta_{\a',\a})\otimes \theta_{\a_0,\a_0}^+
\]
where $\Phi_{w}^-$ denotes the basepoint action on the unstabilized complexes. In particular $\Phi_{w}^-(\Theta_{\a',\a})=\d \eta_{\a',\a}$, for some $\eta_{\a',\a}$, since $\Phi_w^-$ increases grading by 1, and $\Theta_{\a',\a}$ is in the maximal grading of homology. Using the fact that $S_{w'}^+$ is a chain map \cite{ZemGraphTQFT}*{Proposition~6.5}, we also have
\[
\d(\eta_{\a',\a}\otimes \theta_{\a_0,\a_0}^+)=\d(\eta_{\a',\a})\otimes \theta_{\a_0,\a_0}^+.
\]
In particular, the chain $\eta_{\a',\a}$, constructed as above, satisfies Equation~\eqref{eq:A_lambda-length-2-stabilized}.

The higher length chains of $(\cL_{\a}^+)^\lambda$ are analyzed similarly. In our present context, by using the stablization result for holomorphic polygons in Proposition~\ref{prop:multi-stabilization-counts}, we see that the right-hand side of Equation~\eqref{eq:A-lambda-1st-equation} can be written as $C_{\veps',\veps}\otimes \theta_{\a_0,\a_0}^+$, where $C_{\veps',\veps}$ is a chain on the unstabilized complex. By the exact same reasoning as before, $C_{\veps',\veps}$ is a multiple of $U$, and is a cycle, so we may pick a primitive of the form $U\eta_{\a',\a}\otimes \theta_{\a_0,\a_0}^+$, completing the proof.
\end{proof}

 \begin{proof}[Proof of Proposition~\ref{prop:simple-expansion-hypercube}] The hypercube relations for the right-hand side follow from the fact that each of the maps $\scS_{w'}^+$, $\scA_{\lambda}$ and $\scS_{w'}^-$ is a homomorphism of hypercubes (i.e. a chain map).

The hypercube relations for the entire cube, after compressing the right-hand side of the diagram in Equation~\eqref{eq:add-basepoint-hypercube} are equivalent to the relation
\[
\scS_{w'}^- \scA_{\lambda} \scS_{w'}^+=\id,
\] 
as hypercube morphisms. First, we use Lemma~\ref{lem:A_lambda-cube-stabilized} to construct the hypercube $(\cL_{\a}^+)^{\lambda}$ so that each chain of length at least 2 which has non-trivial $\lambda$ direction is of the form $U \eta_{\a',\a}\otimes \theta_{\a_0,\a_0}^+$. In the composition $\scA_{\lambda} \scS_{w'}^+$, each arrow with length greater than 1 will be obtained by a holomorphic polygon count which only has inputs of the form $\xs\otimes \theta^+_{\a_0,\a_0}$. By Proposition~\ref{prop:multi-stabilization-counts}, the output of such a polygon count will also have a tensor factor of $\theta^+_{\a_0,\a_0}$. Such elements lie in the kernel of $S_{w'}^-$. Hence the composition $\scS_{w'}^- \scA_{\lambda} \scS_{w'}^+$ has no arrows of length greater than 1. The arrows of length 1 correspond to the composition $S_{w'}^- A_{\lambda} S_{w'}^+$, which is the identity, by \cite{ZemGraphTQFT}*{Lemma~7.10}.
 \end{proof}

\subsection{The map \texorpdfstring{$F_{\frD\to \tilde{\frD}}$}{F:D->D-tilde}}\label{subsec:expanding-the-involution}

In this section, we define the map $F_{\frD\to \tilde{\frD}}$. The map will correspond to a hypercube relating the expanded model of the involution from Section~\ref{sec:expanded-def} and the non-expanded model from Section~\ref{sec:doubling-def}. 
 
The construction proceeds in two steps. We first relate the singly-pointed involution to a model which adds a basepoint in a very simple way. We call this the \emph{trivially expanded} involution. As a second step, we relate the trivially expanded model to the expanded model defined in Section~\ref{sec:expanded-def}.
 
Let $(\Sigma,\as,\bs,w)$ denote a Heegaard diagram, and let $w'$ be a new basepoint, chosen near $w$. Let $\a_0$ and $\b_0$ be two new alpha and beta circles, centered at $w'$. We will let $s$ denote either of the curves $\a_0$ or $\b_0$ (or small translates thereof).

\begin{figure}[ht]
\[
 \begin{tikzcd}[labels=description, row sep=.8cm, column sep=1.5cm]
\CF(\as,\bs)
 	\ar[r,"\id"]
 	\ar[d,"F_1^{\bar \b, \bar \b}"]
 	\ar[dr,phantom, "\fbox{A}"]
&[-1cm]
\CF(\as,\bs)
 	\ar[d,"F_1^{\bar \b, \bar \b}"]
	\ar[r, "\id"]
 	\ar[dr,phantom, "\fbox{B}"]
&
\CF(\as,\bs)
	\ar[r, "\id"]
	\ar[d,"F_{1}^{s\bar \b , s\bar \b}"]
	 \ar[dr,phantom, "\fbox{C}"]
&[-1cm]
\CF(\as,\bs)
	\ar[d,"F_{1}^{s\bar \b, s\bar \b}"]
\\[1cm]
\CF(\as \bar \bs, \bs \bar \bs)
	\ar[r,"\id"]
	\ar[d, "f_{\a \bar \b}^{\b\bar \b\to \Dt}"]
 	\ar[dr,phantom, "\fbox{D}"]
&
\CF(\as\bar \bs, \bs \bar \bs)
	\ar[d, "f_{\a \bar \b}^{\b\bar \b\to \Dt}"]
	\ar[r, "S_{w'}^+"]
 	\ar[dr,phantom, "\fbox{E}"]
&
\CF(\as s \bar \bs, \bs s \bar \bs)
	\ar[d, "f_{\a s \bar \b}^{\b s\bar \b\to s\Dt}"]
	\ar[r,"\id"]
 	\ar[dr,phantom, "\fbox{F}"]
&
\CF(\as s \bar \bs , \bs s \bar \bs)
	\ar[d, "f_{\a s \bar \b}^{\b s\bar \b\to s\Dt}"]
\\[1cm]
\CF(\as \bar \bs, \Ds)
	\ar[r,"\id"]
	\ar[ddd, "\id"]
	\ar[dddr,phantom, "\fbox{G}"]
&
\CF(\as \bar \bs, \Ds)
	\ar[d, "S_{w'}^+"]
	\ar[r, "S_{w'}^+"]
 	\ar[dr,phantom, "\fbox{H}"]
&
\CF(\as s \bar \bs, s \Ds)
	\ar[r, "\id"]
	\ar[d," A_{\lambda}"]
 	\ar[dr,phantom, "\fbox{I}"]
&
\CF(\as s \bar \bs, s\Ds )
	\ar[d,"A_\lambda"]
\\
&
\CF(\as s \bar \bs, s \Ds)
	\ar[d, "A_\lambda"]
	\ar[r, "A_\lambda"]
 	\ar[dr,phantom, "\fbox{J}"]
&
\CF(\as s \bar \bs, s \Ds)
	\ar[r, "\id"]
	\ar[d,"f_{\a s \bar \b}^{ s\Dt\to \a s \bar \a}"]
 	\ar[dr,phantom, "\fbox{K}"]
&
\CF(\as s \bar \bs, s \Ds)
	\ar[d,"f_{\a s \bar \b}^{ s\Dt\to \a s \bar \a}"]
\\
&
\CF(\as s \bar \bs, s \Ds)
	\ar[d, "S_{w'}^-"]
	\ar[r, "f_{\a s \bar \b}^{ s\Dt\to \a s \bar \a}"]
 	\ar[dr,phantom, "\fbox{L}"]
&
\CF(\as s \bar \bs, \as s \bar \as)
	\ar[r,"\id"]
	\ar[d,"S_{w'}^-"]
 	\ar[dr,phantom, "\fbox{M}"]
&
\CF(\as s \bar \bs, \as s \bar \as)
	\ar[d,"F_3^{\a s,\a s}"]
\\
\CF(\as \bar \bs, \Ds)
	\ar[r,"\id"]
	\ar[d, "f_{\a \bar \b}^{\Dt\to \a \bar \a}"]
 	\ar[dr,phantom, "\fbox{N}"]
&
\CF(\as \bar \bs, \Ds)
	\ar[d,"f_{\a \bar \b}^{\Dt\to \a \bar \a}"]
	\ar[r," f_{\Dt\to \b\bar \b}^{\b\bar{\a}}"]
 	\ar[dr,phantom, "\fbox{O}"]
&
\CF(\as \bar \bs, \as \bar \as)
	\ar[r,"F_3^{\a,\a}"]
	\ar[d,"F_3^{\a,\a}"]
 	\ar[dr,phantom, "\fbox{P}"]
&
\CF(\bar \bs, \bar \as)
	\ar[d,"\id "]
\\[1cm]
\CF(\as\bar \bs,\as \bar \as)
	\ar[r,"\id"]
	\ar[d, "F_3^{\a,\a}"]
 	\ar[dr,phantom, "\fbox{Q}"]
&
\CF(\as\bar \bs, \as\bar \as)
	\ar[d, "F_3^{\a,\a}"]
	\ar[r,"F_3^{\a,\a}"]
 	\ar[dr,phantom, "\fbox{R}"]
&
\CF(\bar \bs, \bar \as)
	\ar[r, "\id"]
	\ar[d,"\id"]
 	\ar[dr,phantom, "\fbox{S}"]
&
\CF(\bar \bs, \bar \as)
	\ar[d,"\id"]
\\[1cm]
\CF(\bar \bs, \bar \as)
	\ar[r,"\id"]
&
\CF(\bar \bs, \bar \as)
	\ar[r,"\id"]
&
\CF(\bar \bs, \bar \as)
	\ar[r,"\id"]
&
\CF(\bar \bs, \bar \as)
\end{tikzcd}
\]
\caption{The first hyperbox used to construct the map $F_{\frD\to \tilde{\frD}}$. None of the faces have length 2 maps. The boxed letter will be referred to in the proof of part \eqref{prop:expansion-part-2} of Proposition~\ref{prop:expansion}.}
\label{eq:transition-standard-to-trivial}
\end{figure}

In Figure~\ref{eq:transition-standard-to-trivial} each square commutes, and no length 2 maps are necessary. As our definition, we take the compression of the right-hand side of Figure~\ref{eq:transition-standard-to-trivial} as the trivially expanded model of the involution.

We now describe how to relate the trivially expanded model in Figure~\ref{eq:transition-standard-to-trivial} to the expanded model described in Section~\ref{sec:expanded-def}. To do this, we use the hyperbox shown in Figure~\ref{eq:trivial-model-to-neck}.

\begin{figure}[ht]
\[
\begin{tikzcd}[labels=description, row sep=1.2cm, column sep=2.2cm]
\CF(\as,\bs)
	\ar[d,"F_{1}^{s\bar \b, s\bar \b}"]
	\ar[r,"\id"]
 	\ar[dr,phantom, "\fbox{A'}"]
&
\CF(\as,\bs)
	\ar[d,"F_{1}^{s\bar \b,c\bar \b}"]
	\ar[r,"\id"]
 	\ar[dr,phantom, "\fbox{B'}"]
&
\CF(\as,\bs)
	\ar[d,"F_{1}^{c\bar \b,c\bar \b}"]
\\
\CF(\as s \bar \bs, \bs s \bar \bs)
	\ar[d, "f_{\a s \bar \b}^{\b s \bar \b\to s \Dt}"]
	\ar[r,"f_{\a s \bar \b}^{\b s \bar \b\to \b c \bar \b}"]
	\ar[dr,dashed]
	\ar[dr,phantom, shift left=6,pos=.6, "\fbox{C'}"]
&
\CF(\as s \bar \bs, \bs c \bar \bs)
	\ar[d,"f_{\a s \bar \b}^{\b c \bar \b\to s \Dt}"]
	\ar[r,"f_{\a s \bar \b\to \a c \bar \b}^{\b c \bar \b}"]
	\ar[dr,dashed]
	\ar[dr,phantom, shift left=6,pos=.6, "\fbox{D'}"]
&
\CF(\as c \bar \bs, \bs c \bar \bs)
	\ar[d,"f_{\a c \bar \b}^{\b c\bar \b \to s\Dt}"]
\\
\CF(\as s \bar \bs, s \Ds)
	\ar[d,"A_\lambda"]
	\ar[r,"\id"]
	\ar[dr,phantom, "\fbox{E'}"]
&
\CF(\as s \bar \bs, s \Ds)
	\ar[d,"A_\lambda"]
	\ar[r,"f_{\a s \bar \b\to \a c \bar \b}^{s \Dt}"]
	\ar[dr,dashed]
	\ar[dr,phantom,  shift left=6,pos=.6,"\fbox{F'}"]
&
\CF( \as c \bar \bs, s\Ds)
	\ar[d,"A_\lambda"]
\\
\CF(\as s \bar \bs, s \Ds)
	\ar[d,"f_{\a s \bar \b}^{s \Dt \to \a s \bar \a}"]
	\ar[r,"\id"]
	\ar[dr,dashed]
	\ar[dr,phantom,  shift left=6,pos=.6,"\fbox{G'}"]
&
\CF(\as s \bar \bs, s \Ds)
	\ar[d,"f_{\a s \bar \b}^{s \Dt\to \a c'\bar \a}"]
	\ar[r,"f_{\a s \bar \b\to \a c \bar \b}^{s \Dt}"]
	\ar[dr,dashed]
	\ar[dr,phantom,  shift left=6,pos=.6,"\fbox{H'}"]
&
\CF(\as c \bar \bs, s \Ds)
	\ar[d,"f_{\a c \bar \b}^{s \Dt\to \a c'\bar \a}"]
\\
\CF(\as s \bar \bs, \as s \bar \as)
	\ar[d,"\id"]
	\ar[r,"f_{\a s \bar \b}^{\a s \bar \a\to \a c'\bar \a}"]
	\ar[dr,phantom, "\fbox{I'}"]
&
\CF(\as s \bar \bs, \as c'\bar \as)
	\ar[r,"f_{\a s \bar \b\to \a c \bar \b}^{\a c'\bar \a}"]
	\ar[d, "\id"]
	\ar[dr,dashed]
	\ar[dr,phantom,  shift left=6,pos=.6,"\fbox{J'}"]
&
\CF(\as c \bar \bs, \as c' \bar \as)
	\ar[d, "f_{\a c \bar \b\to \a c'\bar \b}^{\a c'\bar \a}"]
\\
\CF(\as s \bar \bs, \as s \bar \as)
	\ar[d,"F_3^{\a s,\a s}"]
	\ar[r, "f_{\a s \bar \b}^{\a s \bar \a\to \a c'\bar \a}"]
	\ar[dr,phantom,"\fbox{K'}"]
&
\CF(\as s \bar \bs, \as c'\bar \as)
	\ar[d,"F_3^{\a s, \a c'}"]
	\ar[r,"f_{\a s \bar \b\to \a c'\bar \b}^{\a c'\bar \a}"]
	\ar[dr,phantom,"\fbox{L'}"]
&
\CF(\as c'\bar \bs, \as c'\bar \as)
	\ar[d,"F_3^{\a c',\a c'}"]
\\
\CF(\bar \bs, \bar \as)
	\ar[r,"\id"]
&
\CF(\bar \bs, \bar \as)
	\ar[r,"\id"]
&
\CF(\bar \bs, \bar \as)
\end{tikzcd}
\]
\caption{The second hyperbox used to construct the map $F_{\frD\to \tilde{\frD}}$. The boxed letters label the faces (in particular, they do not label maps in the faces).}
\label{eq:trivial-model-to-neck}
\end{figure}

We now explain the maps appearing in equation Figure~\ref{eq:trivial-model-to-neck}. Except for the top and bottom-most cubes, all of the faces either commute on the nose, or are obtained by pairing hypercubes of attaching curves. The hypercubes are all of handleslide equivalent attaching curves, or handleslide equivalent attaching curves with an extra direction corresponding to $\lambda$. Hence the construction follows by the technique of filling empty hypercubes of handleslide equivalent attaching curves \cite{MOIntegerSurgery}*{Lemma~8.6}, or Section~\ref{sec:hypercubes-rel-homology}.

We now consider the top-most and bottom-most hypercubes. We claim that the hypercube relations may be verified by appealing to Proposition~\ref{prop:multi-stabilization-counts}. To this end, we state the following lemma:

\begin{lem}\,
	\label{lem:s->c-multi-stabilize-verify}
\begin{enumerate}
\item  With respect to the unique $\Spin^c$ structures which have torsion restriction to the boundaries of the corresponding 4-manifolds, the Heegaard triples 
\[
(\bar \Sigma, s \bar \bs, s \bar \bs, c \bar \bs,w,w')\quad \text{and}\quad  (\bar \Sigma, c \bar \bs, s \bar \bs, c \bar \bs,w,w')
\]
 are weakly admissible multi-stabilizing Heegaard triples. Furthermore
 \[
\widehat{f}_{s\bar \b, s \bar \b, c \bar \b}(\Theta^+_{s\bar \b, s \bar \b},\Theta^+_{s\bar \b, c \bar \b})=\Theta^+_{s\bar \b, c \bar \b}\quad \text{and} \quad \widehat{f}_{c \bar \b, s \bar \b, c \bar \b}(\Theta^+_{c \bar \b,s \bar \b},\Theta^+_{s \bar \b, c \bar \b})=\Theta^+_{c \bar \b, c \bar \b}. 
 \]
 Here, $\widehat{f}_{s\bar \b, s \bar \b, c \bar \b}$ denotes the map on $\widehat{\CF}$, and similarly for the other holomorphic triangle map.
 \item  Similarly, with respect to the unique $\Spin^c$ structures which have torsion restriction to the boundary,
   \[
 (\Sigma, \as s, \as s, \as c',w,w')\quad \text{and} \quad (\Sigma, \as c', \as s, \as c',w,w')
 \]
  are weakly admissible multi-stabilizing diagrams, and
  \[
\widehat{f}_{\a s, \a s, \a c'} (\Theta^-_{\a s, \a s}, \Theta^+_{\a s, \a c'})=\Theta^-_{\a s, \a c'}\quad \text{and}\quad \widehat{f}_{\a c', \a s, \a c'}(\Theta^+_{\a c', \a s}, \Theta^-_{\a s, \a c'})=\Theta^-_{\a c',\a c'}.  
  \]
\end{enumerate}
\end{lem}
\begin{proof} The claims about admissibility are proven by reducing to the genus zero triples $(S^2, s,s,c,w,w')$, and so forth, by removing pairwise isotopic triples of curves from $\bar \bs$ or $\as$. See Figure~\ref{fig:7} for the model case. It is straightforward to verify admissibility for the genus 0 diagrams. Clearly the other conditions of a multi-stabilizing diagram are satisfied. 

We now prove the claim about the holomorphic triangle maps. Both triangle maps may be interpreted as naturality maps for a sequence of handleslides. Since each of the Floer complexes appearing in the statement have a unique top degree intersection point, and a unique bottom degree intersection point, and also since the claims only involve the maps on $\widehat{\CF}$, the claims follow from grading considerations.
\end{proof}  
  
 \begin{figure}[H]
	\centering
	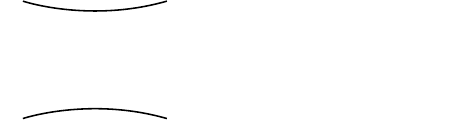
	\caption{The curves $s$, $c$ and $c'$ in the connected sum neck region. The $w$ basepoint is indicated $\bullet$ and the $w'$ is indicated $\bullet'$.}\label{fig:7}
\end{figure}

\subsection{Proof of Part \eqref{prop:expansion-part-2} of Proposition~\ref{prop:expansion}}
\label{sec:expansion-relation-transition-maps}

We now show that the map $F_{\frD\to \tilde \frD}$, defined in the previous section, commutes up to chain homotopy for the transition maps for elementary equivalences of the alpha and beta curves. This is part~\eqref{prop:expansion-part-2} of Proposition~\ref{prop:expansion}.

Our proof is to build a 3-dimensional hypercube which realizes the homotopy commutation
\[
F_{\frD'\to \tilde{\frD}'}\circ \Psi_{\frD\to \frD'}\simeq \tilde{\Psi}_{\tilde \frD\to \tilde \frD'}\circ F_{\frD\to \tilde \frD},
\]
whenever $\frD$ and $\frD'$ differ by an elementary equivalence. We build this hypercube as follows. For each of the faces labeled (A)--(S) in Figure~\ref{eq:transition-standard-to-trivial}, and each of the faces labeled (A')--(L') in Figure~\ref{eq:trivial-model-to-neck}, we build a hyperbox of size $(1,1,2)$, which extends the corresponding sub-cube of the map $F_{\frD\to \tilde{\frD}}$  in the direction of the maps $\Psi_{\frD\to \frD'}$ and $\tilde{\Psi}_{\tilde \frD\to \tilde \frD'}$. We can view each of these hyperboxes as a sequence of two hypercubes. The third coordinate of the first hypercube is an alpha equivalence, and the third coordinate of the second hypercube is a beta equivalence. (This pattern is parallel to the construction of the map $\Psi_{\frD\to \frD'}$; see Figure~\ref{def:transition-map-elementary-handleslide}). Write $\as',$ $\bs'$ and $\Ds'$ for the attaching curves of $\frD'$.

The construction of the size $(1,1,2)$ hyperbox over the facet labeled B in Figure~\ref{eq:transition-standard-to-trivial} is shown below.
\begin{equation}
\begin{tikzcd}[
	labels=description,
	row sep={2cm,between origins},
	column sep={3cm,between origins},
	execute at end picture={
	\foreach \Nombre in  {A,B,C,D}
	{\coordinate (\Nombre) at (\Nombre.center);}
	\fill[opacity=0.1] 
	(A) -- (B) -- (C) -- (D) -- cycle;
	}	
	]
 |[alias=A]|\CF(\as,\bs)
	\ar[rr,"f_{\a\to a'}^{\b}"]
	\ar[dd,"F_1^{\bar \b, \bar \b}"]
	\ar[dr,"\id"]
	\ar[dddr,phantom, "\fbox{B}"]
&[-1cm]
&
\CF(\as',\bs)
	\ar[dd,"F_1^{\bar \b', \bar \b}"]
	\ar[rr, "f_{\a'}^{\b\to\b'}"]
	\ar[dr,"\id"]
&[-1cm]
&
\CF(\as',\bs')
	\ar[dd,"F_{1}^{\bar \b',\bar \b'}",]
	\ar[dr,"\id"]
&[-1cm]
\,
\\[-.6cm]
&
 |[alias=B]|\CF(\as,\bs)
	\ar[rr,"f_{\a\to \a'}^{\b}",crossing over]
&&
\CF(\as',\bs)
	\ar[rr, "f_{\a'}^{\b\to \b'}",crossing over]
&&
\CF(\as',\bs')
	\ar[dd,"F_1^{s \bar \b', s \bar \b'}"]
\\[.7cm]
 |[alias=D]|
 \CF(\as\bar \bs, \bs \bar \bs)
	\ar[rr, "f_{\a \bar \b\to \a'\bar \b'}^{\b \bar \b} ", pos=.7]
	\ar[dr,"S_{w'}^+"]
&&
\CF(\as' \bar \bs', \bs  \bar \bs)
	\ar[rr, "f_{\a' \bar \b'}^{\b \bar \b \to \b'\bar \b'}", pos=.7]
	\ar[dr,"S_{w'}^+"]
&&
\CF(\as' \bar \bs', \bs'  \bar \bs')
	\ar[dr,"S_{w'}^+"]
\\[-.6cm]
&
 |[alias=C]|\CF(\as s \bar \bs, \bs s  \bar \bs)
	\ar[rr, "f_{\a s \bar \b \to \a' s \bar \b'}^{\b s \bar \b}"]
	\ar[from=uu,crossing over,"F_1^{s\bar \b',s \bar \b}"]
&&
\CF(\as's \bar \bs', \bs s \bar \bs)
	\ar[rr, "f_{\a' s \bar \b'}^{\b s \bar \b\to \b' s \bar \b'} "]
	\ar[from=uu, crossing over,"F_1^{s\bar \b,s \bar \b}"]
&&
\CF(\as' s \bar \bs', \bs' s \bar \bs')
\end{tikzcd}
\label{eq:hyperbox-over-B}
\end{equation}
The hypercube relations for the diagram in~\eqref{eq:hyperbox-over-B} follow by using Proposition~\ref{prop:multi-stabilization-counts} to destabilize the holomorphic triangle counts. 

The hyperbox over the face labeled (G) in Figure~\ref{eq:transition-standard-to-trivial} follows from the construction from Proposition~\ref{prop:simple-expansion-hypercube}. The size $(1,1,2)$ hyperboxes over the remaining faces (A)--(S) are similar to the construction in equation~\eqref{eq:hyperbox-over-B}. We leave the details to the reader.

Next, we move on to extending the hyperboxes from Figure~\ref{eq:trivial-model-to-neck}. The extensions of the faces labeled (C')--(J') are obtained by pairing hypercubes of handleslide equivalent attaching curves (with an extra dimension corresponding to $\lambda$ in the case of E' and F'). In particular, the extensions are constructed using techniques which are now standard, so we leave them to the reader.

 We investigate the hypercubes labeled A', B', K' and L'. All four are constructed similarly, so we focus our attention on the face labeled A'. The extension is shown below:
 \begin{equation}
 	\begin{tikzcd}[
 		labels=description,
 		row sep={2cm,between origins},
 		column sep={3cm,between origins},fill opacity=.8,
 		execute at end picture={
 			\foreach \Nombre in  {A,B,C,D}
 			{\coordinate (\Nombre) at (\Nombre.center);}
 			\fill[opacity=0.1] 
 			(A) -- (B) -- (C) -- (D) -- cycle;
 		}	
 		]
 		|[alias=A]|\CF(\as,\bs)
 		\ar[rr,"f_{\a\to a'}^{\b}"]
 		\ar[dd,"F_1^{s\bar \b, s\bar \b}"]
 		\ar[dr,"\id"]
 		\ar[dddr,phantom, "\fbox{A'}"]
 		&[-1cm]
 		&
 		\CF(\as',\bs)
 		\ar[dd,"F_1^{s\bar \b', s\bar \b}"]
 		\ar[rr, "f_{\a'}^{\b\to\b'}"]
 		\ar[dr,"\id"]
 		&[-1cm]
 		&
 		\CF(\as',\bs')
 		\ar[dd,"F_{1}^{s\bar \b',\bar s\b'}",]
 		\ar[dr,"\id"]
 		&[-1cm]
 		\,
 		\\[-.6cm]
 		&
 		|[alias=B]|\CF(\as,\bs)
 		\ar[rr,"f_{\a\to \a'}^{\b}",crossing over]
 		&&
 		\CF(\as',\bs)
 		\ar[rr, "f_{\a'}^{\b\to \b'}",crossing over]
 		&&
 		\CF(\as',\bs')
 		\ar[dd,"F_1^{s \bar \b', c \bar \b'}"]
 		\\[.7cm]
 		|[alias=D]|
 		\CF(\as s\bar \bs, \bs s \bar \bs)
 		\ar[rr, "f_{\a s \bar \b\to \a' s\bar \b'}^{\b s \bar \b} ", pos=.7]
 		\ar[dr,"f_{\a s \bar \b}^{\b s \bar \b \to \b c \bar \b}"]
 		&&
 		\CF(\as' s \bar \bs', \bs s \bar \bs)
 		\ar[rr, "f_{\a' s \bar \b'}^{\b s \bar \b \to \b' s\bar \b'}", pos=.7]
 		\ar[dr,"f_{\a' s \bar \b'}^{\b s \bar \b \to \b c \bar \b}"]
 		&&
 		\CF(\as's \bar \bs', \bs's  \bar \bs')
 		\ar[dr,"f_{\a' s \bar \b'}^{\b' s \bar \b'\to \b' c \bar \b'}"]
 		\\[-.6cm]
 		&
 		|[alias=C]|\CF(\as s \bar \bs, \bs c  \bar \bs)
 		\ar[rr, "f_{\a s \bar \b \to \a' s \bar \b'}^{\b c \bar \b}"]
 		\ar[from=uu,crossing over,"F_1^{s\bar \b,c \bar \b}"]
 		&&
 		\CF(\as's \bar \bs', \bs c \bar \bs)
 		\ar[rr, "f_{\a' s \bar \b'}^{\b c \bar \b\to \b' c \bar \b'} "]
 		\ar[from=uu, crossing over,"F_1^{s\bar \b',c \bar \b}"]
 		&&
 		\CF(\as' s \bar \bs', \bs' c \bar \bs')
 	\end{tikzcd}
 	\label{eq:hyperbox-over-A'}
 \end{equation}
The hypercube relations follow from the stabilization results of Proposition~\ref{prop:multi-stabilization-counts}, together with Lemma~\ref{lem:s->c-multi-stabilize-verify}.

%% file: fig7.pdf_tex
\begingroup%
  \makeatletter%
  \providecommand\color[2][]{%
    \errmessage{(Inkscape) Color is used for the text in Inkscape, but the package 'color.sty' is not loaded}%
    \renewcommand\color[2][]{}%
  }%
  \providecommand\transparent[1]{%
    \errmessage{(Inkscape) Transparency is used (non-zero) for the text in Inkscape, but the package 'transparent.sty' is not loaded}%
    \renewcommand\transparent[1]{}%
  }%
  \providecommand\rotatebox[2]{#2}%
  \newcommand*\fsize{\dimexpr\f@size pt\relax}%
  \newcommand*\lineheight[1]{\fontsize{\fsize}{#1\fsize}\selectfont}%
  \ifx\svgwidth\undefined%
    \setlength{\unitlength}{225.08590073bp}%
    \ifx\svgscale\undefined%
      \relax%
    \else%
      \setlength{\unitlength}{\unitlength * \real{\svgscale}}%
    \fi%
  \else%
    \setlength{\unitlength}{\svgwidth}%
  \fi%
  \global\let\svgwidth\undefined%
  \global\let\svgscale\undefined%
  \makeatother%
  \begin{picture}(1,0.25438367)%
    \lineheight{1}%
    \setlength\tabcolsep{0pt}%
    \put(0,0){\includegraphics[width=\unitlength,page=1]{fig7.pdf}}%
    \put(0.17001734,0.16192888){\makebox(0,0)[lt]{\lineheight{1.25}\smash{\begin{tabular}[t]{l}$\bullet'$\end{tabular}}}}%
    \put(0.17340057,0.06787472){\makebox(0,0)[lt]{\lineheight{1.25}\smash{\begin{tabular}[t]{l}$\bullet$\end{tabular}}}}%
    \put(0,0){\includegraphics[width=\unitlength,page=2]{fig7.pdf}}%
    \put(0.73534241,0.16192888){\makebox(0,0)[lt]{\lineheight{1.25}\smash{\begin{tabular}[t]{l}$\bullet'$\end{tabular}}}}%
    \put(0.72539747,0.06787472){\makebox(0,0)[lt]{\lineheight{1.25}\smash{\begin{tabular}[t]{l}$\bullet$\end{tabular}}}}%
    \put(0,0){\includegraphics[width=\unitlength,page=3]{fig7.pdf}}%
    \put(0.2539981,0.05902675){\makebox(0,0)[lt]{\lineheight{1.25}\smash{\begin{tabular}[t]{l}$c$\end{tabular}}}}%
    \put(0.70305281,0.05712318){\makebox(0,0)[rt]{\lineheight{1.25}\smash{\begin{tabular}[t]{r}$c'$\end{tabular}}}}%
    \put(0.2912972,0.17878603){\makebox(0,0)[lt]{\lineheight{1.25}\smash{\begin{tabular}[t]{l}$s$\end{tabular}}}}%
    \put(0.66002168,0.17878603){\makebox(0,0)[rt]{\lineheight{1.25}\smash{\begin{tabular}[t]{r}$s$\end{tabular}}}}%
    \put(0.07017546,0.08966299){\makebox(0,0)[rt]{\lineheight{1.25}\smash{\begin{tabular}[t]{r}$\Sigma$\end{tabular}}}}%
    \put(0.36942069,0.08966299){\makebox(0,0)[lt]{\lineheight{1.25}\smash{\begin{tabular}[t]{l}$\bar\Sigma$\end{tabular}}}}%
    \put(0.59974538,0.08966299){\makebox(0,0)[rt]{\lineheight{1.25}\smash{\begin{tabular}[t]{r}$\Sigma$\end{tabular}}}}%
    \put(0.89899128,0.08966299){\makebox(0,0)[lt]{\lineheight{1.25}\smash{\begin{tabular}[t]{l}$\bar\Sigma$\end{tabular}}}}%
  \end{picture}%
\endgroup%

%% file: section-3-sigma-fixed.tex
\subsection{The composition law for expanded transition maps}
\label{sec:composition-law-expanded-proof}

In this section we prove part \eqref{prop:expansion-part-3} of Proposition \ref{prop:expansion}. As in the hypotheses of the proposition, suppose that we have three handleslide equivalent Heegaard diagrams $(\Sigma,\as,\bs,w)$, $(\Sigma,\as',\bs',w)$ and $(\Sigma,\as'',\bs'',w)$, all with underlying surface $\Sigma$. We also pick three choices of basepoint expanded doubling curves $\Ds$, $\Ds'$ and $\Ds''$, and let $c$ and $c'$ denote the curves in the connected sum region. We further assume that the tuple 
 \[
(\Sigma \# \bar \Sigma, \as'',\bar \as'', \as',\bar \as', \as, \bar \as, \bs, \bar \bs, \bs',\bar \bs', \bs'',\bar \bs'',\Ds,\Ds',\Ds'',c,c',w)  
 \]
 is weakly admissible (cf. Remark~\ref{rem:admissibility-unconventional}). Write $\tilde \frD$, $\tilde\frD'$ and $\tilde\frD''$ for the associated basepoint expanded, doubling enhanced Heegaard diagrams. We will show that
 \[
\tilde{\Psi}_{\tilde\frD\to \tilde\frD''}\simeq \tilde{\Psi}_{\tilde\frD'\to \tilde\frD''}\circ \tilde{\Psi}_{\tilde\frD\to \tilde\frD'}.
 \]

Our proof goes by constructing a hyperbox which realizes the stated relation.
 For each vertical level in the construction of $\tilde{\Psi}_{\tilde \frD\to \tilde\frD'}$, the hyperbox we construct will have a level which takes the form of a hyperbox of size $(2,2,1)$.

We describe the 1-handle hyperbox first. See Figure~\ref{fig:naturality-Sigma-fixed-1-handles} for the overall shape of the 1-handle hyperbox.
As a first step, we construct the hypercube shown in Equation~\eqref{eq:top-degree-generator-hypercube}.
\begin{equation}
\begin{tikzcd}[
	column sep={2.3cm,between origins},
	row sep=1.4cm,
	labels=description,
	fill opacity=.7,
	text opacity=1,
	]
\bF[U]	\ar[dr, "\id", pos=.6]
	\ar[ddd, "\Theta_{\bar \b', \bar \b}"]
	\ar[rr, "\id"]
	\ar[ddddrrr,dotted, "\omega_{\bar \b'', \bar \b'}"]
&[.7 cm]
&\bF[U]
	\ar[rd, "\id"]
	\ar[ddd, "\Theta_{\bar \b', \bar \b'}"]
&[.7 cm]
\,
\\
&[.7 cm]
 \bF[U]
 	\ar[rr,line width=2mm,dash,color=white,opacity=.7]
	\ar[rr,  "\id"]
&\,
&[.7 cm]
 \bF[U]
	\ar[ddd, "\Theta_{\bar \b'', \bar \b'}"]
\\
\\
\CF(\bar \bs', \bar \bs)
	\ar[dr,"f_{\bar \b'\to \bar \b''}^{\bar \b}"]
	\ar[rr, "f_{\bar \b'}^{\bar \b\to\bar \b'}"]
	\ar[drrr,dashed, "h_{\bar \b'\to \bar \b''}^{\bar \b \to \bar \b'}"]
&[.7 cm]\,&
\CF(\bar \bs', \bar \bs')
	\ar[dr, "f_{\bar \b'\to \bar \b''}^{\bar \b'}"]
&[.7 cm]
\\
& [.7 cm]
\CF(\bar \bs'', \bar \bs)
 	\ar[from=uuu,line width=2mm,dash,color=white,opacity=.7]
	\ar[from=uuu, "\Theta_{\bar \b'', \bar \b}"]
	\ar[rr, "f_{\bar \b''}^{\bar \b\to \bar \b'}"]
 	\ar[from=uuuul,line width=2mm,dash,color=white,opacity=.7]
	\ar[from=uuuul, dashed, "\lambda_{\bar \b'', \bar \b}"]
&
&[.7 cm] 
\CF(\bar \bs'', \bar \bs')
 	\ar[from=uuull,line width=2mm,dash,color=white,opacity=.7]
	\ar[from=uuull,dashed,"\lambda_{\bar \b'', \bar \b'}"]
\end{tikzcd}
\label{eq:top-degree-generator-hypercube}
\end{equation}
The hypercube in Equation~\eqref{eq:top-degree-generator-hypercube} is constructed similarly to the procedure for \emph{filling} a hypercube of handleslide equivalent attaching curves, described by Manolescu and Ozsv\'{a}th \cite{MOIntegerSurgery}*{Lemma~8.6}. In Equation~\eqref{eq:top-degree-generator-hypercube}, an arrow labeled by an element $\xs\in \bT_{\a}\cap \bT_{\b}$ means the map from $\bF[U]$ to $\CF(\as,\bs)$ given by
\[
1\mapsto \xs,
\]
extended equivariantly over $U$. The existence of $\lambda_{\bar \b'',\bar \b'}$, $\lambda_{\bar \b'', \bar \b'}$ and $\omega_{\bar \b'', \bar \b'}$ so that the diagram in~\eqref{eq:top-degree-generator-hypercube} satisfies the hypercube relations follows from the fact that $\bar \bs$, $\bar\bs'$ and $\bar\bs''$ are handleslide equivalent.

It is helpful to concretely understand the hypercube relations in~\eqref{eq:top-degree-generator-hypercube}. The length 1 relations amount to each chain marked with $\Theta$ being a cycle. The length 2 relations are as follows:
\begin{equation}
\begin{split}
f_{\bar \b'}^{ \bar \b\to \bar \b'}(\Theta_{\bar \b', \bar \b})&=\Theta_{\bar \b', \bar \b'}\\
f_{\bar \b'\to \bar \b''}^{\bar \b'}(\Theta_{\bar \b', \bar \b'})&=\Theta_{\bar \b'', \bar \b'}\\
\Theta_{\bar \b'', \bar \b}+f_{\bar \b'\to \bar \b''}^{\bar \b}(\Theta_{\bar \b', \bar \b})&=\d \lambda_{\bar \b'', \bar \b}\\
f_{\bar \b''}^{\bar \b\to \bar \b'}(\Theta_{\bar \b'', \bar \b})+ \Theta_{\bar \b'', \bar \b'}&= \d \lambda_{\bar \b'', \bar \b'}\\
f_{\bar \b'\to \bar \b''}^{\bar \b'}\circ f_{\bar \b'}^{\bar \b\to \bar \b'}+f_{\bar \b''}^{\bar \b\to \bar \b'}\circ f_{\bar \b'\to \bar \b''}^{\bar \b}&=\left[\d, h_{\bar \b'\to \bar \b''}^{\bar \b\to \bar \b'}\right].
\end{split}
\label{eq:hypercube-relations-diagonal-product}
\end{equation}

Note that the first relation is forced by grading considerations, and the second is implied by a nearest point map computation. The third and fourth relations are clearly satisfiable, and the final relation is the associativity relations for quadrilaterals.

 Finally, the length 3 relation is
\[
f_{\bar \b''}^{\bar \b \to \bar \b'}(\lambda_{\bar \b'', \bar \b})+\lambda_{\bar \b'', \bar \b'}+h_{\bar \b'\to \bar \b''}^{\bar \b \to \bar \b'}(\Theta_{\bar \b', \bar \b})=\d \omega_{\bar \b'', \bar \b'}.
\]

Note that the hypercube in~\eqref{eq:top-degree-generator-hypercube} is a new type of hypercube. It is not naturally realized by pairing alpha and beta hypercubes together. We can think of the left side as determining the following hypercube of \emph{alpha} attaching curves
\begin{equation}
\begin{tikzcd}[labels=description, column sep=1.5cm, row sep=1.5cm]
\bar{\bs}
	\ar[r,"1"]
	\ar[dr,dashed,"\lambda_{\bar\b'',\bar\b}"]
	\ar[d,"\Theta_{\bar \b', \bar \b}"]
&\bar{\bs}
	\ar[d, "\Theta_{\bar\b'',\bar\b}"]
\\
\bar{\bs}'
	\ar[r, "\Theta_{\bar\b'',\bar\b'}"]
&\bar{\bs}''
\end{tikzcd}
\label{eq:subface-left-pure-1-handle}
\end{equation}
On the other hand, the front side is more naturally viewed as corresponding to the following hypercube of \emph{beta} attaching curves
\begin{equation}
\begin{tikzcd}[labels=description, column sep=1.5cm, row sep=1.5cm]
\bar{\bs}''
	\ar[r,"1"]
	\ar[dr,dashed,"\lambda_{\bar\b'',\bar\b'}"]
	\ar[d,"\Theta_{\bar \b'', \bar \b}"]
&\bar{\bs}''
	\ar[d, "\Theta_{\bar\b'',\bar\b'}"]
\\
\bar{\bs}
	\ar[r, "\Theta_{\bar\b,\bar\b'}"]
&\bar{\bs}'
\end{tikzcd}
\label{eq:subface-front-pure-1-handle}
\end{equation}
The bottom face of~\eqref{eq:top-degree-generator-hypercube} is the pairing of hypercubes of attaching curves.

Next, we pick hypercubes of alpha attaching curves and beta attaching curves (respectively):
\begin{equation}
\begin{tikzcd}[labels=description, column sep=1.5cm, row sep=1.5cm]
\as
	\ar[r,"\Theta_{\a',\a}"]
	\ar[dr,dashed,"\lambda_{\a'',\a}"]
	\ar[d,"1"]
&\as'
	\ar[d, "\Theta_{\a'',\a'}"]
\\
\as
	\ar[r, "\Theta_{\a'',\a}"]
&\as''
\end{tikzcd}
\quad \text{and} \quad
\begin{tikzcd}[labels=description, column sep=1.5cm, row sep=1.5cm]
\bs
	\ar[r,"\Theta_{\b,\b'}"]
	\ar[dr,dashed,"\lambda_{\b,\b''}"]
	\ar[d,"1"]
&\bs'
	\ar[d, "\Theta_{\b',\b''}"]
\\
\bs
	\ar[r, "\Theta_{\b,\b''}"]
&\bs''
\end{tikzcd}
\label{eq:alpha-hypercubes}
\end{equation} 
As usual, we assume that the length 1 morphisms represent top degree cycles. Note that we do not assume that the chains in~\eqref{eq:alpha-hypercubes} have any relation with the chains in~\eqref{eq:subface-left-pure-1-handle} or~\eqref{eq:subface-front-pure-1-handle}.

 \begin{figure}[hp]
\[
\begin{tikzcd}[
	column sep={2.1cm,between origins},
	row sep=1.5cm,
	labels=description,
	fill opacity=.7,
	text opacity=1,
	execute at end picture={
	\foreach \Nombre in  {A,B,C,D}
  {\coordinate (\Nombre) at (\Nombre.center);}
\fill[opacity=0.1] 
  (A) -- (B) -- (C) -- (D) -- cycle;
}]
(\as,\bs)
	\ar[dr, "\id"]
	\ar[ddd, "F_1^{c \bar{\b}, c \bar{\b}}"]
	\ar[rr, "f_{\a\to \a'}^{\b}"]
	\ar[ddddrrr,dotted]
&[.7 cm]
&(\as',\bs)
	\ar[rd, "f_{\a'\to \a''}^{\b}"]
	\ar[ddd, "F_1^{c\bar{\b}', c \bar{\b}}"]
	\ar[rr, "f_{\a'}^{\b\to \b'}"]
	\ar[ddddrrr,dotted]
	\ar[ddddr,dashed]
&[.7 cm]
&
(\as',\bs')
	\ar[dr, "f_{\a'\to \a''}^{\b'}"]
	\ar[ddd, "F_{1}^{c \bar{\b}', c \bar{\b}'}"]
&[.7 cm]
\\
&[.7 cm]
 |[alias=A]|(\as,\bs)
 	\ar[rr,line width=2mm,dash,color=white,opacity=.7]
	\ar[rr, "f_{\a\to \a''}^{\b}"]
&\,
&[.7 cm]
 (\as'',\bs)
	\ar[ddd, "F_1^{c\bar{\b}'', c\bar{\b}}"]
	\ar[rr, "f_{\a''}^{\b\to \b'}"]
 	\ar[from=ulll,line width=2mm,dash,color=white,opacity=.7]
	\ar[from=ulll, dashed]
&&
|[alias=B]|(\as'',\bs')
	\ar[ddd, "F_{1}^{c\bar{\b}'', c\bar{\b}'}"]
 	\ar[from=ulll,line width=2mm,dash,color=white,opacity=.7]	
	\ar[from=ulll, dashed]
\\
\\
(\as c \bar{\bs}, \bs c \bar{\bs})
	\ar[dr,"\id"]
	\ar[rr, "f_{\a c \bar{\b}\to \a' c \bar{\b}'}^{\b c \bar{\b}}"]
	\ar[drrr,dashed]
&[.7 cm]\,&
(\as' c \bar{\bs}', \bs c \bar{\bs})
	\ar[dr, "f_{\a' c \bar{\b}'\to \a'' c \bar{\b}''}^{\b c \bar{\b}}"]
	\ar[rr, "f_{\a'c \bar{\b}'}^{\b c \bar{\b}\to \b' c \bar{\b}'}"]
	\ar[drrr,dashed]
&[.7 cm]\,
&
(\as'c\bar{\bs}',\bs' c \bar{\bs}')
	\ar[dr, "f_{\a'c \bar{\b}'\to \a'' c \bar{\b}''}^{\b' c \bar{\b}'}"]
\\
& [.7 cm]
|[alias=D]|(\as c \bar{\bs}, \bs c \bar{\bs})
 	\ar[from=uuu,line width=2mm,dash,color=white,opacity=.7]
	\ar[from=uuu, "F_1^{c\bar{\b}, c\bar{\b}}"]
	\ar[rr, "f_{\a c \bar{\b}\to \a'' c \bar{\b}''}^{\b c \bar{\b}}"]
&
&[.7 cm] 
(\as'' c \bar{\bs}'', \bs c \bar{\bs})
	\ar[rr, "f_{\a''c \bar{\b}''}^{\b c \bar{\b}\to \b' c \bar{\b}'}" ]
	&&
|[alias=C]|(\as'' c \bar{\bs}'', \bs' c \bar{\bs}')
 	\ar[from=uuull,line width=2mm,dash,color=white,opacity=.7]	
	\ar[from=uuull,dashed]
\end{tikzcd}
\]
\vspace{1cm}
\[
\begin{tikzcd}[
	column sep={2.1cm,between origins},
	row sep=1.5cm,
	labels=description,
	fill opacity=.7,
	text opacity=1,
	execute at end picture={
	\foreach \Nombre in  {A,B,C,D}
  	{\coordinate (\Nombre) at (\Nombre.center);}
	\fill[opacity=0.1] 
  (A) -- (B) -- (C) -- (D) -- cycle;
}]
|[alias=A]|(\as,\bs)
	\ar[dr, "\id"]
	\ar[ddd, "F_1^{c \bar{\b}, c \bar{\b}}"]
	\ar[rr, "f_{\a\to \a''}^{\b}"]
&[.7 cm]
&(\as'',\bs)
	\ar[rd, "\id"]
	\ar[ddd, "F_1^{c\bar{\b}'', c \bar{\b}}"]
	\ar[rr, "f_{\a''}^{\b\to \b'}"]
	\ar[ddddrrr,dotted]
	\ar[dddrr,dashed]
&[.7 cm]
&
|[alias=B]|(\as'',\bs')
	\ar[dr, "f_{\a''}^{\b'\to \b''}"]
	\ar[ddd, "F_{1}^{c \bar{\b}'', c \bar{\b}'}"]
&[.7 cm]
\\
&[.7 cm]
 (\as,\bs)
 	\ar[rr,line width=2mm,dash,color=white,opacity=.7]
	\ar[rr, "f_{\a\to \a''}^{\b}"]
&\,
&[.7 cm]
 (\as'',\bs)
	\ar[ddd, "F_1^{c\bar{\b}'', c\bar{\b}}"]
	\ar[rr, "f_{\a''}^{\b\to \b''}"]
&&
(\as'',\bs'')
	\ar[ddd, "F_{1}^{c\bar{\b}'', c\bar{\b}''}"]
 	\ar[from=ulll,line width=2mm,dash,color=white,opacity=.7]	
	\ar[from=ulll, dashed]
\\
\\
|[alias=D]|(\as c \bar{\bs}, \bs c \bar{\bs})
	\ar[dr,"\id"]
	\ar[rr, "f_{\a c \bar{\b}\to \a'' c \bar{\b}''}^{\b c \bar{\b}}"]
	&[.7 cm]\,&
(\as'' c \bar{\bs}'', \bs c \bar{\bs})
	\ar[dr, "\id"]
	\ar[rr, "f_{\a''c \bar{\b}''}^{\b c \bar{\b}\to \b' c \bar{\b}'}"]
	\ar[drrr,dashed]
&[.7 cm]\,
&
|[alias=C]|(\as''c\bar{\bs}'',\bs' c \bar{\bs}')
	\ar[dr, "f_{\a''c \bar{\b}''}^{\b' c \bar{\b}'\to \b'' c \bar{\b}''}"]
\\
& [.7 cm]
(\as c \bar{\bs}, \bs c \bar{\bs})
 	\ar[from=uuu,line width=2mm,dash,color=white,opacity=.7]
	\ar[from=uuu, "F_1^{c \bar{\b}, c \bar{\b}}"]
	\ar[rr, "f_{\a c \bar{\b}\to \a'' c \bar{\b}''}^{\b c \bar{\b}}"]
&
&[.7 cm] 
(\as'' c \bar{\bs} '', \bs c \bar{\bs})
	\ar[rr, "f_{\a''c \bar{\b}''}^{\b c \bar{\b}\to \b'' c \bar{\b}''}" ]
	&&
(\as'' c \bar{\bs}'', \bs'' c \bar{\bs}'')
\end{tikzcd}
\]
\caption{Four hypercubes. The length 2 maps are sums of quadrilateral and triangle maps (obtained by pairing hypercubes). The two boxes are stacked along the gray faces.}
\label{fig:naturality-Sigma-fixed-1-handles}
\end{figure}

We now consider the diagram in Figure~\ref{fig:naturality-Sigma-fixed-1-handles}. We will presently define the maps appearing therein, and then subsequently prove the hypercube relations. All the maps are built from the hypercubes in~\eqref{eq:top-degree-generator-hypercube} and~\eqref{eq:alpha-hypercubes}. We note that the hypercube relations are slightly subtle, since they do not arise from simply pairing hypercubes of attaching curves together.

We focus first on the top left cube of Figure~\ref{fig:naturality-Sigma-fixed-1-handles}. The maps appearing there are as follows:
\begin{enumerate}
\item The length 2 map along the top is the one associated to pairing the 0-dimensional beta hypercube $\bs$ with the alpha hypercube
\[
\begin{tikzcd}[labels=description, column sep=1.5cm, row sep=1.5cm]
\as
	\ar[r,"\Theta_{\a',\a}"]
	\ar[dr,dashed,"\lambda_{\a'',\a}"]
	\ar[d,"1"]
&\as'
	\ar[d, "\Theta_{\a'',\a'}"]
\\
\as
	\ar[r, "\Theta_{\a'',\a}"]
&\as''
\end{tikzcd}
\]
\item The length 2 map along the bottom is the one obtained by pairing the 0-dimensional beta hypercube $\bs c \bar \bs$ with the hypercube
\[
\begin{tikzcd}[labels=description, column sep=3cm, row sep=1.5cm]
\as c \bar \bs
	\ar[r,"\Theta_{\a',\a}|\theta^+|\Theta_{\bar \b', \bar\b}"]
	\ar[d,"1"]
	\ar[dr,dashed, "\lambda"]
&\as' c \bar \bs'
	\ar[d, "\Theta_{\a'',\a'}|\theta^+|\Theta_{\bar \b '', \bar \b '}"]
\\
\as c \bar \bs
	\ar[r, "\Theta_{\a'',\a}|\theta^+|\Theta_{\bar \b'', \bar \b}"]
&\as'' c \bar \bs''
\end{tikzcd}
\]
where
\[
\lambda=\lambda_{\a'',\a}|\theta^+|f_{\bar{\b}'', \bar{\b}', \bar{\b}}(\Theta_{\bar{\b}'', \bar{\b}'}, \Theta_{\bar{\b}', \bar{\b}}) +\Theta_{\a'',\a}|\theta^+|\lambda_{\bar{\b}'', \bar{\b}}.
\]
\item The length 2 map along the right side is 
\[
\xs\mapsto f_{\a'\to \a''}^{\b}(\xs)| \theta^+| \lambda_{\bar{\b}'',\bar{\b}}.
\]
\item The length 3 map is 
\[
\xs \mapsto h_{\a \to \a'\to \a''}^{\b}(\xs)| \theta^+| \lambda_{\bar{\b}'', \bar \b}+f_{\a'', \a,\b}(\lambda_{\a'',\a},\xs)|\theta^+|\lambda_{\bar{\b}'',\bar{\b}}.
\]
\end{enumerate}

\begin{lem}\label{lem:top-left-1-handle-subcube}
The maps described above for the top left subcube of the diagram in Figure~\ref{fig:naturality-Sigma-fixed-1-handles} satisfy the hypercube relations.
\end{lem}
\begin{proof} The length 1 relations are immediate, as are the length 2 relations for all but the right face of the cube. For the right face of the cube, using the stabilization result from Proposition~\ref{prop:multi-stabilization-counts}, applied to holomorphic triangles, we see that the hypercube relations amount to showing that if $\xs\in \bT_{\a''}\cap \bT_{\b}$, then
\[
f_{\a'\to \a''}^{\b}(\xs)| \theta^+|\left(\Theta_{\bar \b'', \bar \b}+f_{\bar \b'', \bar \b', \bar \b}(\Theta_{\bar \b'', \bar \b'}, \Theta_{\bar \b', \bar \b})\right)=\d \left(f_{\a' \to \a''}^{\b}(\xs)| \theta^+ | \lambda_{\bar{\b}'', \bar \b}\right)+f_{\a'\to \a''}^{\b}(\d \xs)| \theta^+|\lambda_{\bar{\b}'', \bar{\b}},
\]
which follows from the form of the differential on the Heegaard diagram obtained by attaching a 1-handle. The key point is that the right-hand side of the above equation simplifies to just $f_{\a'\to \a''}^{\b}(\xs)| \theta^+|\d \lambda_{\bar{\b}'', \bar \b}$, which coincides with the left-hand side by Equation~\eqref{eq:hypercube-relations-diagonal-product}.

Next, we consider the length 3 relation of the cube. The desired relation is that if $\xs\in \bT_{\a}\cap \bT_{\b},$ then 
\begin{equation}
\begin{split}
&h_{\a\to \a'\to \a''}^{\b}(\xs)|\theta^+|\Theta_{\bar \b'', \bar \b} \\
&+f_{\a'',\a,\b}(\lambda_{\a'',\a},\xs)|\theta^+|\Theta_{\bar \b'', \bar \b}\\
&+h_{\a'' c \bar \b'', \a' c \bar \b', \a c\bar \b, \b c \bar \b}(\Theta_{\a'',\a'}|\theta^+|\Theta_{\bar \b'' \bar \b'}, \Theta_{\a',\a}|\theta^+|\Theta_{\bar \b' \bar \b},\xs|\theta^+|\Theta_{\bar \b, \bar \b})\\
&+ (f_{\a'\to \a''}^{\b}\circ f_{\a \to \a'}^{\b})(\xs)|\theta^+|\lambda_{\bar \b'', \bar \b}\\
&+f_{\a'' c \bar \b'', \a c \bar \b, \b c \bar \b}(\lambda_{\a'',\a}|\theta^+|f_{\bar \b'', \bar \b', \bar \b}(\Theta_{\bar \b'', \bar \b'}, \Theta_{\bar \b', \bar \b}), \xs |\theta^+|\Theta^+_{\bar \b, \bar \b})\\
&+f_{\a'' c \bar \b'', \a c \bar \b, \b c \bar \b}(\Theta_{\a'', \a}|\theta^+|\lambda_{\bar \b'', \bar \b},\xs |\theta^+|\Theta_{\bar \b, \bar \b})\\
=&\left[ h_{\a\to \a'\to \a''}^{\b}(-)|\theta^+|\lambda_{\bar{\b}'', \bar{\b}}+f_{\a'',\a,\b}(\lambda_{\a'',\a},-)|\theta^+|\lambda_{\bar \b'', \bar \b},\d\right](\xs).
\end{split}
\label{eq:hypercube-relations-diagonal-1-handle}
\end{equation}
We note that
\begin{equation}
\begin{split}
&\left[ h_{\a\to \a'\to \a''}^{\b}(-)|\theta^+|\lambda_{\bar{\b}'', \bar{\b}},\d\right](\xs)
\\
=&f_{\a'',\a,\b}(f_{\a'',\a',\a}(\Theta_{\a'',\a'},\Theta_{\a',\a}),\xs)|\theta^+|\lambda_{\bar{\b}'',\bar \b}\\
&+\left(f_{\a'\to \a''}^{\b}\circ f_{\a\to \a'}^{\b}\right)(\xs)|\theta^+|\lambda_{\bar \b'', \bar \b}\\
&+h_{\a\to \a'\to \a''}^{\b}(\xs)|\theta^+| \Theta_{\bar \b'', \bar \b}\\
&+h_{\a\to \a'\to \a''}^{\b}(\xs)|\theta^+| f_{\bar \b'', \bar \b', \bar \b}(\Theta_{\bar \b'', \bar \b'}, \Theta_{\bar \b', \bar \b}).
\end{split}
\label{eq:1-handle-hypercube-simpl-1}
\end{equation}
Also 
\begin{equation}
\begin{split}
&\left[f_{\a'',\a,\b}(\lambda_{\a'',\a},-)|\theta^+|\lambda_{\bar \b'', \bar \b},\d\right](\xs)\\
=&f_{\a'', \a, \b}(\Theta_{\a'',\a}+f_{\a'',\a',\a}(\Theta_{\a'',\a'}, \Theta_{\a',\a}),\xs)|\theta^+|\lambda_{\bar \b'', \bar \b}
\\
+&
f_{\a'',\a,\b}(\lambda_{\a'',\a},\xs)|\theta^+|(\Theta_{\bar \b'', \bar \b}+f_{\bar \b'', \bar \b', \bar \b}(\Theta_{\bar \b'', \bar \b'}, \Theta_{\bar \b', \bar \b})).
\end{split}
\label{eq:1-handle-hypercube-simpl-2}
\end{equation}
Next, we note that the stabilization result from Proposition~\ref{prop:multi-stabilization-counts}
implies that
\begin{equation}
\begin{split}f_{\a'' c \bar \b'', \a c \bar \b, \b c \bar \b}(\lambda_{\a'',\a}|\theta^+|f_{\bar \b'', \bar \b', \bar \b}(\Theta_{\bar{\b}'', \bar{\b}'}, \Theta_{\bar \b', \bar \b}), \xs |\theta^+|\Theta_{\bar \b, \bar \b})
&=f_{\a'', \a, \b}(\lambda_{\a'', \a}, \xs)|\theta^+|f_{\bar \b'', \bar \b', \bar \b}(\Theta_{\bar{\b}'', \bar{\b}'}, \Theta_{\bar \b', \bar \b}),
\\
f_{\a'' c \bar \b'', \a c \bar \b, \b c \bar \b}(\Theta_{\a'', \a}|\theta^+|\lambda_{\bar \b'', \bar \b},\xs |\theta^+|\Theta_{\bar \b, \bar \b})
&=f_{\a'', \a, \b}(\Theta_{\a'',\a}, \xs)|\theta^+|\lambda_{\bar \b'', \bar \b}.
\end{split}
\label{eq:1-handle-hypercube-simpl-3}
\end{equation}

Using~\eqref{eq:1-handle-hypercube-simpl-1}, ~\eqref{eq:1-handle-hypercube-simpl-2} and~\eqref{eq:1-handle-hypercube-simpl-3} to simplify~\eqref{eq:hypercube-relations-diagonal-1-handle}, we see that the length 3 hypercube relation becomes equivalent to
\begin{equation}
\begin{split}
&h_{\a'' c \bar \b'', \a' c \bar \b', \a c\bar \b, \b c \bar \b}(\Theta_{\a'',\a'}|\theta^+|\Theta_{\bar \b'' \bar \b'}, \Theta_{\a',\a}|\theta^+|\Theta_{\bar \b',\bar \b},\xs|\theta^+|\Theta_{\bar \b, \bar \b})\\
=&h_{\a\to \a'\to \a''}^{\b}(\xs)|\theta^+| f_{\bar \b'', \bar \b', \bar \b}(\Theta_{\bar \b'', \bar \b'}, \Theta_{\bar \b', \bar \b}).
\end{split}
\label{eq:hypercube-reformulated-diagonal-1handle}
\end{equation}

We now prove ~\eqref{eq:hypercube-reformulated-diagonal-1handle}.   Suppose that $\psi$ and $\psi_0$ are classes of rectangles on $(\Sigma,\as'',\as',\as,\bs)$ and $(\bar{\Sigma}, \bar{\bs}'', \bar{\bs}',\bar{\bs}, \bar{\bs})$ such that $\mu(\psi)+\mu(\psi_0)=-1$. Let $C(\psi,\psi_0)$ denote the set of all homology classes which coincide with $\psi$ and $\psi_0$ outside the 1-handle region, and have $\theta^+$ as their input and output in the 1-handle region. (Such rectangles necessarily have index $-1$).  Applying Proposition~\ref{prop:multi-stabilization-counts} to the 1-handle region allows us to remove the 1-handle region connecting $\Sigma$ and $\bar{\Sigma}$. We obtain
\begin{equation}
\sum_{\phi\in C(\psi,\psi_0)} \# \cM(\phi)\equiv \#( \cM(\psi)\times_{\ev_{\Box}} \cM(\psi_0)) \pmod 2 \label{eq:rectangles-and-1-handles}.
\end{equation}
 Here, we are also using the nodal almost complex structures around the 1-handle region on the left-hand side. Furthermore any class of rectangles which is not in some $C(\psi,\psi_0)$ for a pair $(\psi,\psi_0)$ satisfying $\mu(\psi)+\mu(\psi_0)=-1$ has no holomorphic representatives. Also, $\times_{\ev_{\Box}}$ denotes the fibered product over the evaluation map to the parameter space $\R\iso \cM(\Box)$.

We write $h^0$ for the map from $\CF^-(\as,\bs)$ to $\CF^-(\as''c\bar{\bs}'',\bs c \bar{\bs})$ which is the composition 
\[
h^0=h_{\a c \bar{\b} \to \a' c \bar{\b}'\to \a'' c \bar{\b}''}^{\b c \bar{\b}}\circ F_1^{c \bar{\b}, c \bar{\b} }.
\]
By Equation~\eqref{eq:rectangles-and-1-handles}, we may identify $h^0$ as the map which counts holomorphic rectangles on the disjoint union of two diagrams, viewed as a fibered product over $\ev_{\Box}$. We may consider a map $h^t$ for any generic $t\in \R$, which counts curve pairs $(u,u_0)$ on the disjoint union consisting of holomorphic quadrilaterals such that $\ev_{\Box}(u)=\ev_{\Box}(u_0)+t$.

We claim that $h^t=h^{t'}$ for generic $t$ and $t'$ in $\R$. Firstly, we note that since $\bar{\bs}$ appears twice (once as a small Hamiltonian isotopy), the nearest point argument of \cite{HHSZExact}*{Proposition~11.5} (cf. \cite{LOTDoubleBranchedII}*{Section~3.4})  implies that for suitably small choices of translations, the only rectangles on $\bar{\Sigma}$ with representatives have Maslov index at least 0. In particular all holomorphic rectangles counted by $h^t$ for generic $t$ have $\mu(\psi)=-1$ and $\mu(\psi_0)=0$. 

We consider the ends of the moduli spaces counted by $h^t$, ranging over $t\in [0,\infty)$. No degenerations are possible on the $\Sigma$-side, since all rectangles there have index $-1$. On the $\bar{\Sigma}$-side, rectangles have index $0$, so the possible degenerations are a holomorphic disk splitting off at some finite $t$, as well as the ends which appear as $t\to \infty$. The ends appearing at $t\to \infty$ correspond exactly to the map
\[
h_{\a\to \a'\to \a''}^{\b}(\xs)|\theta^+|f_{\bar{\b}'', \bar{\b}',\bar{\b}}(\Theta_{\bar{\b}'', \bar{\b}'}, f_{\bar{\b}', \bar{\b}, \bar{\b}}(\Theta_{\bar{\b}', \bar{\b}}, \Theta_{\bar{\b},\bar{\b}})).
\]
The ends at finite $t$ are slightly more subtle. First, one end which can appear is an end where a disk breaks off into the $\bar{\bs}$-$\bar{\bs}$ end. These cancel modulo two since $\Theta_{\bar{\b}, \bar{\b}}$ is a cycle. Holomorphic disks breaking off into the other ends are prohibited, since they would leave an index $-1$ holomorphic rectangle, with $\Theta_{\bar{\b},\bar{\b}}$ as input. (The most important end where this is prohibited is the outgoing $\bar{\bs}''$-$\bar{\bs}$ end, where a more general diagram might have degenerations.) In particular, we obtain~\eqref{eq:hypercube-reformulated-diagonal-1handle}, completing the proof.
 \end{proof}

 The hypercube in the bottom right of Figure~\ref{fig:naturality-Sigma-fixed-1-handles} is constructed by simply interchanging the roles of the alpha and beta curves in the above procedure.
 
 We now construct the maps in the top right hypercube of Figure~\ref{fig:naturality-Sigma-fixed-1-handles}. As a first step, we pick chains to give hypercubes of attaching curves, as follows:

\begin{enumerate}
 \item The top face is obtained by pairing the two hypercubes
 \[
\begin{tikzcd}
  \as'\ar[r, "\Theta_{\a'',\a'}"]
  &
  \as''
\end{tikzcd}
\quad \text{and} \quad 
 \begin{tikzcd}
  \bs\ar[r, "\Theta_{\b,\b'}"]
  &
  \bs'
\end{tikzcd}
 \]
 \item The bottom face is obtained by pairing the two hypercubes
 \[
 \begin{tikzcd}[column sep =2.5cm]
  \as'c \bar{\bs}'\ar[r, "\Theta_{\a'',\a'}|\theta^+|\Theta_{\bar \b'', \bar \b'}"]
  &
  \as''c \bar{\bs}''
\end{tikzcd}
\quad \text{and} \quad 
 \begin{tikzcd}[column sep =2.5cm]
  \bs c \bar \bs\ar[r, "\Theta_{\b,\b'}|c|\Theta_{\bar \b, \bar \b'}"]
  &
  \bs' c \bar \bs'
\end{tikzcd}
 \]
 \item The length 2 map along the left face is 
\[
\xs\mapsto f_{\a'\to \a''}^{\b}(\xs)| \theta^+| \lambda_{\bar{\b}'',\bar{\b}}.
\]
\item The length 2 map along the front face is
\[
\xs\mapsto f_{\a''}^{ \b\to \b'}(\xs)| \theta^+| \lambda_{\bar \b'', \bar \b'}.
\]
\item The length 3 map is the sum of the following two maps
\[
\begin{split}
 \xs&\mapsto \left(f_{\a''}^{\b\to \b'}\circ f_{\a'\to \a''}^{\b}\right)(\xs)|\theta^+|\omega_{\bar \b'', \bar \b'}\\
\xs&\mapsto h^{t<t_0}(\xs)
\end{split}
\]
\end{enumerate}
The map $h^{t<t_0}$ is a new map, as follows. It counts pairs $(u,u_0)$, where $u$ and $u_0$ are holomorphic rectangles on $(\Sigma,\as'',\as',\bs,\bs')$ and $(\bar{\Sigma}, \bar \bs'',\bar \bs',\bar\bs,\bar \bs')$, respectively, both of index $-1$, such that
\[
\ev_{\Box}(u)<\ev_{\Box}(u_0).
\]
If $\ys$ and $\zs$ are the outputs of $u$ and $u_0$, respectively, then the $h^{t<t_0}(\xs)$ has a summand of $\ys\times \theta^+\times \zs$ (weighted by the appropriate $U$-power).

\begin{lem}
 With the maps described above, the top right subcube of Figure~\ref{fig:naturality-Sigma-fixed-1-handles} is a hypercube of chain complexes.
\end{lem}
\begin{proof} The proof is similar to the proof of Lemma~\ref{lem:top-left-1-handle-subcube}. The length 1 and 2 hypercube relations are similar to the previous case, so we focus on the length 3 relation. If $\xs\in \bT_{\a'}\cap \bT_{\b}$, the desired relation is
\begin{equation}
\begin{split}
&h_{\a'\to \a''}^{\b\to \b'}(\xs)|\theta^+|\Theta_{\bar \b'', \bar\b'}\\
&+f_{\a'' c \bar \b'', \b c \bar \b,\b' c \bar \b'}\left(f_{\a'\to \a''}^{\b}(\xs)|\theta^+|\lambda_{\bar \b'', \bar \b}, \Theta_{\b, \b'}|\theta^+|\Theta_{\bar \b, \bar \b'}\right)\\
&+\left(f_{\a''}^{\b\to \b'}\circ f_{\a'\to \a''}^{\b}\right) (\xs)|\theta^+|\lambda_{\bar \b'', \bar \b'}\\
&+h_{\a' c \bar \b'\to \a'' c \bar \b''}^{\b c \bar \b\to \b' c \bar \b'}(\xs|\theta^+|\Theta_{\bar \b', \bar \b})\\
=&\left[\left(f_{\a''}^{\b\to \b'}\circ f_{\a'\to \a''}^{\b}\right)(-)|\theta^+|\omega_{\bar \b'', \bar \b'}+h^{t<t_0}, \d \right](\xs)
\end{split}
\label{eq:length-3-relation-well-defined-ab}
\end{equation}
By excising the special curves in the connected sum region using the results about holomorphic triangles and 1-handles from \cite{ZemGraphTQFT}*{Theorem~8.8} (compare Proposition~\ref{prop:multi-stabilization-counts}, above), we obtain
\begin{equation}
\begin{split}
&f_{\a'' c \bar \b'', \b c \bar \b,\b' c \bar \b'}\left(f_{\a'\to \a''}^{\b}(\xs)|\theta^+|\lambda_{\bar \b'', \bar \b}, \Theta_{\b, \b'}|\theta^+|\Theta_{\bar \b, \bar \b'}\right)\\
=&\left(f^{\b\to \b'}_{\a''}\circ f_{\a' \to \a''}^{\b}\right)(\xs)|\theta^+|f_{\bar \b''}^{\bar \b \to \bar \b'}(\lambda_{\bar \b'', \bar \b}).
\end{split}
\label{eq:hyperbox-holomorphic-triangles-ab-well-def}
\end{equation}

We now analyze the quadrilateral map $h_{\a' c \bar \b'\to \a'' c \bar \b''}^{\b c \bar \b\to \b' c \bar \b'}(\xs|\theta^+|\Theta_{\bar \b', \bar \b})$. As with the previous case, we can destabilize the 1-handle region and identify the expression with the count of index $-1$ holomorphic quadrilaterals on $(\Sigma\sqcup \bar \Sigma)\times \Box$. We can view this moduli space as the fibered product over the evaluation map to $\cM(\Box)\iso \R$. If $(\psi,\psi_0)$ is a homology class on $(\Sigma \sqcup \bar \Sigma)\times \Box$ which admits a holomorphic representative, then there are generically two possibilities:
\begin{enumerate}
\item\label{case:mu-1} $\mu(\psi)=0$ and $\mu(\psi_0)=-1$;
\item\label{case:mu-2} $\mu(\psi)=-1$ and $\mu(\psi_0)=0$.
\end{enumerate}
As before we will write $h^0$ for the map $h_{\a' c \bar \b'\to \a'' c \bar \b''}^{\b c \bar \b\to \b' c \bar \b'}(\xs|\theta^+|\Theta_{\bar \b', \bar \b})$. For $t\in [0,\infty)$, we write $h^t$ for a deformation, which counts curves such that 
\[
\ev_{\Box}(u_0)-\ev_{\Box}(u)=t.
\]
We count the ends of the moduli spaces defining the maps $h^t$ (ranging over $t\in [0,\infty)$). We obtain 
\begin{equation}
\begin{split}
[\d, h^{t<t_0}](\xs)=&h_{\a' c \bar \b'\to \a'' c \bar \b''}^{\b c \bar \b\to \b' c \bar\b'}(\xs|\theta^+|\Theta_{\bar \b', \bar \b})\\
&+(f^{\b\to \b'}_{\a''}\circ f_{\a'\to \a''}^{\b})(\xs)|\theta^+|h_{\bar \b'\to \bar \b''}^{\bar \b\to \bar \b'}(\Theta_{\bar \b', \bar \b})\\
&+h_{\a'\to \a''}^{\b\to \b'}(\xs)|\theta^+|\Theta_{ \bar \b'', \bar \b'}.
\end{split}
\label{eq:deform-diagonals-quadrilaterals}
\end{equation}
Note that  instead of $\Theta_{\bar \b'', \bar \b'}$ in the final factor, what more  naturally appears is 
\[
f_{\bar \b'', \bar \b', \bar \b'}(\Theta_{\bar \b'', \bar \b'}, f_{\bar \b', \bar \b, \bar \b'}(\Theta_{\bar \b', \bar \b}, \Theta_{\bar \b, \bar \b'})).
\]
However this coincides with $\Theta_{\bar \b'', \bar \b'}$, by a nearest point computation, cf. \cite{HHSZExact}*{Proposition~11.1}.

Finally,
\begin{equation}
\begin{split}
&\left[\left(f_{\a''}^{\b\to \b'}\circ f_{\a'\to \a''}^{\b}(-)\right)|\theta^+|\omega_{\bar \b'', \bar \b'}, \d \right](\xs)\\
=&\left(f_{\a''}^{\b\to \b'}\circ f_{\a'\to \a''}^{\b}(\xs)\right)|\theta^+|\left(f_{\bar \b''}^{\bar \b \to \bar \b'}(\lambda_{\bar \b'', \bar \b})+\lambda_{\bar \b'', \bar \b'}+h_{\bar \b'\to \bar \b''}^{\bar \b \to \bar \b'}(\Theta_{\bar \b', \bar \b})\right)
\end{split}
\label{eq:def-omega-b''b'}
\end{equation}
By combining ~\eqref{eq:hyperbox-holomorphic-triangles-ab-well-def}, ~\eqref{eq:deform-diagonals-quadrilaterals} and ~\eqref{eq:def-omega-b''b'}, we quickly obtain ~\eqref{eq:length-3-relation-well-defined-ab}, completing the proof.
\end{proof}

%% file: section-4-1handlesv4.tex
\section{1- and 3-handles} \label{sec:1-and-3-handles}

In this section, we define our cobordism maps for 1-handle and 3-handles. Since for the maps in this section there will be an obvious choice of $\Spin^c$ structure and framed path, we omit these from the notation throughout. 

\subsection{The construction}
We now construct the map for a 4-dimensional 1-handle. Suppose that $\bS_0=\{p,p'\}$ is a (framed) 0-sphere in $Y$. We pick a Heegaard surface $(\Sigma,\as,\bs)$ so that $\bS_0\subset \Sigma\setminus (\as\cup \bs)$. We form a Heegaard surface for the surgered manifold $Y(\bS_0)$ by cutting out small disks centered at $p$ and $p'$, and gluing in an annular diagram $(A,\g_0,\g_0)$, where $\g_0$ and $\g_0$ are meridians of the annulus which intersect in two points. (Here and onwards, we abuse notation and view one copy of $\g_0$ as a Hamiltonian translate of the other). Write $(\Sigma', \as \g_0, \bs \g_0,w)$ for the diagram obtained by this procedure.

We now define the 4-dimensional 1-handle map. To do this, we first pick paths $\lambda$ and $\lambda'$ from $p$ and $p'$ to $w$. We pick doubling arcs $\ds$ for $\Sigma$ which are disjoint from $\lambda$ and $\lambda'$. Furthermore, we choose two more doubling arcs $\ds_0$ on $\Sigma'$, which are contained in the union of a neighborhood of $\lambda$ and $\lambda'$, as well as the annular 1-handle region. We assume $\ds\cup \ds_0$ is a valid collection of doubling arcs for $\Sigma'$.  Write $\Ds$ and $\Ds_0$  for the curves obtained by doubling $\ds$ and $\ds_0$, respectively. 

Additionally, the involutive 1-handle map requires a choice of closed curve $\tau\subset \Sigma'\# \bar \Sigma'$, such that the path $\tau$ is obtained by doubling a properly embedded arc on $\Sigma'\setminus N(w)$ which intersects $\g_0$ transversely in a single point.

We define the 1-handle map to be the diagram obtained by compressing the following hyperbox:
\begin{equation}
\begin{tikzcd}[
labels=description,
row sep=1cm,
column sep=2.4cm,
fill opacity=.7,
text opacity=1
]
\CF(\as,\bs)
	\ar[d, "F_{1}^{\bar{\b},\bar{\b}}"]
	\ar[rrr, "F_1^{\g_0,\g_0}"]
&[-1.5cm]
&[-1.5cm]
&[-.2cm]
\CF(\as\g_0,\bs\g_0)
	\ar[d,"F_{1}^{\bar{\b}\bar \g_0,\bar{\b} \bar \g_0}"]
\\
\CF(\as \bar{\bs}, \bs \bar{\bs})
	\ar[rr, "F_1^{\g_0\bar \g_0, \g_0\bar \g_0}"]
	\ar[d,"f^{\b \bar{\b}\to \Dt}_{\a \bar{\b}}"]
&
& 
\CF(\as \bar \bs \g_0 \bar \g_0, \bs  \bar \bs \g_0  \bar \g_0)
	\ar[r,equal]
	\ar[d, "f^{\b \bar \b\g_0 \bar \g_0\to \Dt \g_0 \bar \g_0}_{\a \bar \b \g_0 \bar \g_0}"]
	\ar[dr,dashed]
& \CF(\as\bar \bs \g_0 \bar \g_0, \bs \bar \bs \g_0 \bar \g_0)
	\ar[d,"f^{\b \bar \b \g_0 \bar \g_0\to \Dt \Dt_0}_{\a \bar \b \g_0 \bar \g_0}"]
\\
\CF(\as \bar \bs,\Ds)
	\ar[rr, "F_1^{\g_0\bar \g_0, \g_0\bar \g_0}"]
	\ar[d, "f^{\Dt\to \a  \bar{\a}}_{\a  \bar{\b}}"]
&
&
 \CF(\as \bar \bs  \g_0 \bar \g_0, \Ds \g_0 \bar \g_0)
	\ar[r, "f_{\a \bar \b \g_0 \bar \g_0}^{\Dt \g_0 \bar \g_0 \to \Dt \Dt_0}"]
	\ar[d, "f_{\a \bar \b \g_0 \bar \g_0}^{\Dt \g_0 \bar \g_0\to \a \bar \a \g_0 \bar \g_0}"]
	\ar[dr,dashed]
& \CF(\as\bar{\bs} \g_0 \bar \g_0,\Ds \Ds_0)
	\ar[d, "f_{\a \bar \b \g_0 \bar \g_0}^{\Dt \Dt_0\to \a \bar \a \g_0 \bar \g_0}"]
\\
\CF(\as  \bar{\bs}, \as  \bar{\as})
	\ar[rr, "F_1^{\g_0\bar \g_0,\g_0\bar \g_0}"]
	\ar[d, "F_{3}^{\a,\a}"]
&&
\CF(\as \bar \bs \g_0 \bar \g_0, \as \bar \as \g_0 \bar \g_0)
	\ar[r, "B_\tau"]
&
\CF(\as \bar \bs \g_0 \bar \g_0, \as \bar \as \g_0 \bar \g_0)
	\ar[d, "F_3^{\a\g_0,\a\g_0}"]
	\\
\CF(\bar{\bs}, \bar{\as})
	\ar[rrr, "F_1^{\bar \g_0, \bar \g_0}"]
&&&
\CF(\bar{\bs} \bar{\g}_0,\bar{\as} \bar{\g}_0)
\end{tikzcd}
\label{eq:1-handles}
\end{equation}

It will be helpful to observe that the bottom-most level of~\eqref{eq:1-handles} has the following slightly expanded form
\begin{equation}
\begin{tikzcd}[
labels=description,
row sep=1cm,
column sep=1.5cm,
fill opacity=.7,
text opacity=1
]
\CF(\as  \bar{\bs}, \as  \bar{\as})
	\ar[r, "F_1^{\bar \g_0, \bar \g_0}"]
	\ar[d, equal]
&[-.5cm]
\CF(\as\bar \bs\bar \g_0, \as \bar \as\bar \g_0)
	\ar[r,"F_1^{\g_0,\g_0}"]
	\ar[d,equal]
&[-.5cm]
\CF(\as \bar \bs \g_0 \bar \g_0, \as \bar \as \g_0 \bar \g_0)
	\ar[r, "B_\tau"]
&[-1cm]
\CF(\as \bar \bs \g_0 \bar \g_0, \as \bar \as \g_0 \bar \g_0)
	\ar[d, "F_3^{\g_0,\g_0}"]
\\
\CF(\as\bar{\bs}, \as\bar{\as})
	\ar[r, "F_1^{\bar \g_0, \bar \g_0}"]
	\ar[d, "F_3^{\a,\a}"]
&
\CF(\as\bar \bs \bar \g_0,\as \bar \as \bar \g_0)
	\ar[rr,equal]
&&
\CF(\as\bar{\bs} \bar{\g}_0,\as\bar{\as} \bar{\g}_0)
	\ar[d,"F_3^{\a,\a}"]
\\
\CF(\bar{\bs}, \bar{\as})
	\ar[rrr, "F_1^{\bar \g_0, \bar \g_0}"]
&&&
\CF(\bar{\bs} \bar{\g}_0,\bar{\as} \bar{\g}_0)
\end{tikzcd}
\label{eq:1-handles-expanded-bottom-portion}
\end{equation}

We now describe the construction of~\eqref{eq:1-handles}.  The top-most face which has a diagonal map is obtained by pairing a hypercube of beta attaching curves with the hypercube of alpha attaching curves consisting of  $\as\bar \bs \g_0\bar \g_0$.

We now consider the lowest face of~\eqref{eq:1-handles} which has a diagonal map. This map is obtained by pairing the 0-dimensional alpha hypercube $\as \bar \bs \g_0\bar \g_0$ with the diagram below. The map $B_\tau$ is the $H_1$-action for $\tau$, defined by counting changes across the beta curves, $\as\bar \as \g_0\bar \g_0$.

\begin{equation}
\begin{tikzcd}[labels=description, row sep =2cm, column sep =2.5cm]
\Ds \g_0\bar \g_0
	\ar[r,"\Theta_{\Dt \g_0\bar \g_0, \Dt \Dt_0}"]
	\ar[d,"\Theta_{\Dt \g_0\bar \g_0, \a \bar \a \g_0\bar \g_0}"]
	\ar[dr, dashed, "\eta"]
&
\Ds \Ds_0
	\ar[d, "\Theta_{\Dt \Dt_0, \a \bar \a \g_0 \bar \g_0}"]
\\
\as \bar \as \g_0 \bar \g_0
	\ar[r ,"\tau"]
&
\as \bar \as \g_0 \bar \g_0
\end{tikzcd}
\label{eq:1-handle-hypercube-involving-H1}
\end{equation}
In equation~\eqref{eq:1-handle-hypercube-involving-H1}, the hypercube relations are to be interpreted using the formalism of Section~\ref{sec:hypercubes-rel-homology}. In particular, the hypercube relations amount to the equation
\begin{equation}
\d \eta=B_\tau(\Theta_{\Dt\g_0\bar \g_0, \a \bar \a \g_0\bar \g_0})+f_{\Dt \g_0\bar \g_0, \Dt \Dt_0, \a \bar \a\g_0\bar \g_0}(\Theta_{\Dt \g_0\bar \g_0, \Dt\Dt_0}, \Theta_{\Dt\Dt_0,\a \bar \a \g_0\bar \g_0}).
\label{eq:1-handle-hypercube-involving-H1-unpacked}
\end{equation}

\begin{rem}
\label{rem:1-handle-many-choices}
There are many choices of morphisms in the diagram in equation~\eqref{eq:1-handle-hypercube-involving-H1}. Any two choices of top-degree generators $\Theta_{\Dt \g_0 \bar \g_0, \Dt \Dt_0}$ are homologous. Similarly any two choices of $\Theta_{\Dt \g_0 \bar \g_0, \a \bar\a \g_0 \bar \g_0}$ and $\Theta_{\Dt \Dt_0, \a \bar \a \g_0 \bar \g_0}$ are homologous. Unlike these morphisms, it is not the case that any two choices of diagonal morphism $\eta=\eta_{\Dt \g_0\bar \g_0, \a \bar \a \g_0 \bar \g_0}$ are homologous. Indeed this chain lies in the top degree of $\CF^-(\#^{g+1} S^1\times S^2)$, and hence there two choices which are not homologous. Similarly, it is not the case that any two choices for the curve $\tau$ are homologous. Nonetheless, we will show in Lemma~\ref{lem:1-handles-and-handleslides} that resulting involutive 1-handle map obtained by compressing the diagram in equation~\eqref{eq:1-handles-expanded-bottom-portion} is independent from these choices, up to chain homotopy.
\end{rem}

We now show that this is satisfiable:

\begin{lem}
\label{lem:1-handle-bottom-hypercube}There is a homogeneously graded $\eta\in \CF(\Ds\g_0\bar \g_0, \as \bar \as \g_0\bar \g_0)$ satisfying~\eqref{eq:1-handle-hypercube-involving-H1-unpacked}.
\end{lem}
\begin{proof}Any two choices of $\Ds\Ds_0$ are related by a sequence of handleslides and isotopies. Via an associativity argument, it is sufficient to show the claim for any convenient choice of $\Ds$ and $\Ds_0$. In particular, we may assume that the curves of $\Ds_0$ are disjoint from $\as$ and $\bar \as$. In particular, the relevant triple diagram $(\Sigma' \# \bar \Sigma', \Ds \g_0 \bar \g_0,\Ds \Ds_0,\as \bar \as \g_0 \bar \g_0)$ may be decomposed as a connected sum 
 \[
 (\Sigma\# \bar \Sigma,\Ds,\Ds,\as \bar \as)\#(\bT^2\# \bT^2,\g_0 \bar \g_0, \Ds_0, \g_0 \bar \g_0).
 \]
  By the stabilization result in Proposition~\ref{prop:multi-stabilization-counts}, it suffices to show the claim for the genus 2 triple $(\bT^2\# \bT^2,\g_0 \bar \g_0, \Ds_0, \g_0 \bar \g_0)$.
  
   There are two ways to verify the claim in this situation. The first strategy is to argue topologically via the following reasoning:
\begin{enumerate}
\item Topologically the triple $(\bT^2\# \bT^2,\g_0 \bar \g_0, \Ds_0, \g_0 \bar \g_0)$ is for 0-surgery on an unknot in $S^1\times S^2$.
\item By picking $\Ds_0$ appropriately, we may assume that $\tau$ runs parallel to a curve $\Dt\in  \Ds_0$. Furthermore, $\Dt$ is the meridian of the knot on which we are performing surgery, and in particular defines the dual of the surgery knot.
\item In general, if $U$ is a 0-framed unknot in $Y$, and $\mu$ is a meridian, then $\CF(W(U))=A_\mu\circ F_{1}$, where $F_1$ is the 1-handle map. This is easily verified using a genus 1 diagram and a stabilization result for triangles.
\item By combining the ideas above, the proof is complete.
\end{enumerate}
Alternately, one may explicitly perform the computation by counting holomorphic triangles. Write $\theta^{\pm}$ (resp. $\bar{\theta}{}^{\pm}$) for the intersection points of $\g_0$ with its translate (resp. $\bar \g_0$ with its translate). In Figure~\ref{fig:4}, we show the four index 0 triangle classes which can contribute.
Two of the classes are in $\pi_2(\Theta^+,\Theta^+,\theta^+|\bar{\theta}{}^-)$ and two of the classes are in $\pi_2(\Theta^+,\Theta^+,\theta^-|\bar{\theta}{}^+)$.
If $\psi_1$ and $\psi_2$ denote the two classes in $\pi_2(\Theta^+,\Theta^+,\theta^+|\bar{\theta}{}^-)$, we claim
\begin{equation}
\# \cM(\psi_1)+\# \cM(\psi_2)\equiv 1\mod 2, \label{eq:sum-of-classes-is-1}
\end{equation}
and similarly for the two classes in $\pi_2(\Theta^+,\Theta^+,\theta^-|\bar{\theta}{}^+)$. 
This claim is verified via a Gromov compactness argument which is illustrated and explained in Figure~\ref{fig:13}.
Hence $f_{\g_0\bar \g_0, \Dt_0, \g_0 \bar \g_0}(\Theta^+,\Theta^+)=\theta^+|\bar{\theta}^-+\theta^-|\bar{\theta}^+$, which is also easily seen to be the action $\tau$ on $\theta^+|\bar \theta^+$. 

\end{proof}

 \begin{figure}[ht]
	\centering
	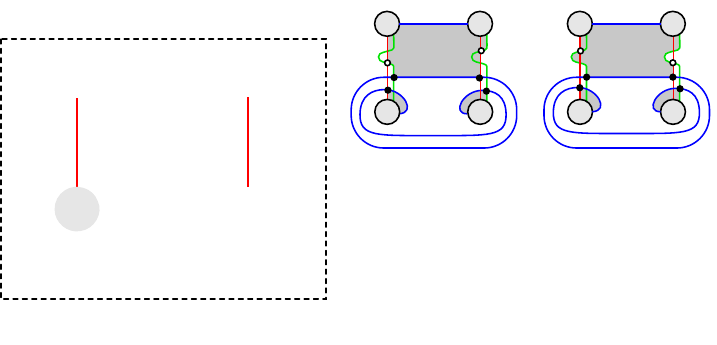
	\caption{Verifying equation~\eqref{eq:1-handle-hypercube-involving-H1-unpacked} by counting triangles. Left: The diagram $(\bT^2\# \bT^2, \g_0\bar\g_0, \Dt_0)$. Right: The four classes of index 0 holomorphic triangles which contribute to the triangle map $f_{\g_0\bar \g_0,\Dt_0,\g_0\bar\g_0}(\Theta^+,\Theta^+)$. 
	}\label{fig:4}
\end{figure}
 \begin{figure}[ht]
	\centering
	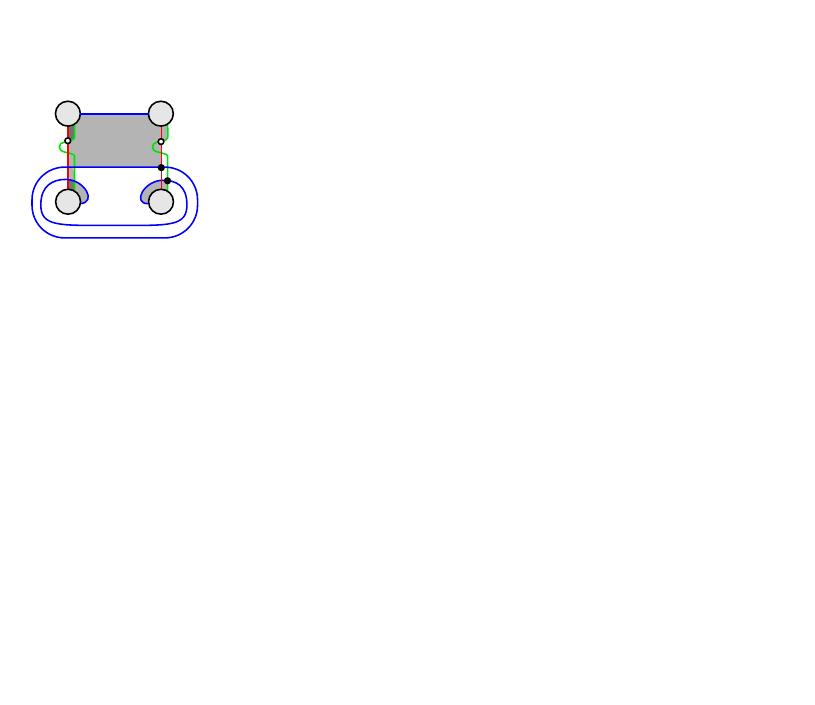
	\caption{A Gromov compactness argument which shows that the total count of holomorphic triangles of index 0 in $\pi_2(\Theta^+,\Theta^+,\theta^+|\bar{\theta}{}^-)$ is 1. In the left-most columns are index 1 classes of triangles. On the other columns of a given row, we show all decompositions of that class into an index 0 triangle class and an index 1 disk class (such that both classes have nonnegative multiplicities). Configurations $A$ and $E$ have the same total count as the $\theta^+|\bar{\theta}{}^-$ coefficient of our map. Configuration $B$ has total count 1. Configurations $C$ and $D$ have the same count. Gromov compactness gives $\# A+\# B+\# C\equiv 0$ and $\# D+\# E\equiv 0$, which quickly gives Equation~\eqref{eq:sum-of-classes-is-1}.}\label{fig:13}
\end{figure}

Finally, we verify that the bottom-most face of~\eqref{eq:1-handles} is a hypercube. Note that this face may be expanded into the diagram shown in~\eqref{eq:1-handles-expanded-bottom-portion}. In light of this, it suffices to show the following relation (note the similarity to Equation ~\eqref{eq:add-basepoint-hypercube}):
\begin{lem}\label{lem:formal-cube-1-handle} 
As maps from $\CF(\as \bar \bs \bar \g_0, \as \bar \as \bar \g_0)$ to $\CF(\as \bar \bs \bar \g_0, \as \bar \as \bar \g_0)$, there is an equality
\[
F_3^{\g_0,\g_0}\circ B_\tau \circ F_1^{\g_0,\g_0}=\id.
\]
\end{lem}
\begin{proof}Compare \cite{ZemGraphTQFT}*{Lemma~7.10}. The relation follows from the form of the differential after attaching a 1-handle. Namely, it follows from \cite{ZemGraphTQFT}*{Proposition~8.5} that the differential satisfies
\[
\d(\xs\times \theta^+_{\g_0,\g_0})= \d(\xs)\otimes \theta^+_{\g_0,\g_0}.
\]
Furthermore, the proof shows that the holomorphic curves which go from $\xs\times \theta^+_{\g_0,\g_0}$ to $\ys\times \theta^-_{\g_0,\g_0}$ for some $\xs$ and $\ys$ all have $\xs=\ys$ and have domain consisting of one of the two bigons contained in the 1-handle region. The 1-chain $\tau$ intersects the boundary of exactly one of these bigons, and hence $B_{\tau}(\xs\times \theta_{\g_0,\g_0}^+)= \xs\times \theta_{\g_0,\g_0}^-+\zs\otimes \theta_{\g_0,\g_0}^+$, for some $\zs\in \CF(\as \bar \bs, \as \bar \as)$. The stated claim follows immediately.
\end{proof}

\subsection{Well-definedness of the 1-handle maps}

In this section, we prove that the stabilization maps are compatible with the maps for changing the alpha and beta curves via handleslides or isotopies, or changing the choice of doubling curves $\Ds$ and $\Ds_0$. We also prove that the involutive 1-handle maps are independent of the choice of the curve $\tau\subset \Sigma'\# \bar \Sigma'$, as well as the choice of chains used to construct the hypercubes in~\eqref{eq:1-handles}.

As a first step, we verify independence from the curve $\tau$:

\begin{lem}
\label{lem:1-handle-map-ind-tau}
 Suppose that $\tau$ and $\tau'$ are two closed curves on $\Sigma'\# \bar \Sigma'$ obtained by doubling  properly embedded arcs on $\Sigma'\setminus N(w)$ which intersect $\g_0$ transversely in a single point. Then there are choices of the additional auxiliary data in the construction so that the models of the 1-handle maps in Equation~\eqref{eq:1-handles} are chain homotopic.
\end{lem} 
\begin{proof} In the compression of Equation~\eqref{eq:1-handles} the choice of $\tau$ only affects the diagonal map. There are two ways in which it appears. The first is via the choice of $\eta$, and the second is a summand which has a factor of
\begin{equation}
F_3^{\a \g_0,\a \g_0}\circ f_{\a \bar \b \g_0\bar \g_0}^{\Dt \g_0\bar \g_0\to \a \bar \a \g_0\bar \g_0\xrightarrow{\tau} \a \bar \a \g_0\bar \g_0}\circ F_{1}^{\g_0\bar \g_0, \g_0\bar \g_0}.\label{eq:1-handle-map-tau-contribution}
\end{equation}
We consider this latter summand first. The map in Equation~\eqref{eq:1-handle-map-tau-contribution} counts holomorphic triangles weighted by $\# \tau \cap \d_{\a \bar \a \g_0\bar \g_0}(\psi)$. If $\psi\in \pi_2(\xs,\ys,\zs)$ is a class of triangles counted in the above composition, then both the $\xs$ and $\ys$ (the incoming intersection points) have $\theta_{\g_0,\g_0}^+$ as factors. By the stabilization result of Proposition~\ref{prop:multi-stabilization-counts}, if such a class has a representative, then $\zs$ must have a factor of $\theta_{\g_0,\g_0}^+$ as well. However, such a term would evaluate trivially under $F_3^{\a\g_0,\a \g_0}$. Hence, Equation~\eqref{eq:1-handle-map-tau-contribution} vanishes.

We now consider the first summand, which involves the choice of $\eta$. Let $\tau$ and $\tau'$ be two choices of curves, as above, and let $\eta_\tau$ be some choice of diagonal chain for the hypercube from Equation~\eqref{eq:1-handle-hypercube-involving-H1}, when $\tau$ is used. It is straightforward to see that $[\tau]-[\tau']\in H_1(\Sigma'\# \bar \Sigma')$ lies in the span of the $\Ds$ curves.  Let $C$ be a 2-chain on $\Sigma$ such that $\d \Sigma$ is $\tau-\tau'$ plus a linear combination of small parallel pushoffs of the $\Ds$ curves. 

We define the 2-chain
\[
\eta_{\tau'}:=\eta_\tau+H_C(\Theta_{\Dt \g_0\bar \g_0, \a \bar \a \g_0\bar \g_0}).
\]
In the above, $H_C$ denotes the $\bF[U]$-linear map whose value on an intersection point $\xs\in \bT_{\Dt\g_0\bar \g_0}\cap \bT_{\a \bar \a \g_0\bar \g_0}$ is the sum of the multiplicities of $C$ at each factor of $\xs$. It is straightforward to see that $\eta_{\tau'}$ is a valid choice of diagonal for the hypercube in Equation~\eqref{eq:1-handle-hypercube-involving-H1} which uses $\tau'$.
Since each summand of $H_C(\Theta_{\Dt \g_0\bar \g_0, \a \bar \a \g_0\bar \g_0})$ has $\theta_{\g_0,\g_0}^+$ as a factor, the same argument as above shows that
\[
F_3^{\g_0,\g_0}\circ f_{\a \bar \b\g_0\bar \g_0, \Dt \g_0\bar \g_0, \a \bar \a \g_0\bar \g_0}(F_1^{\g_0,\g_0}(\xs), H_C(\Theta_{\Dt \g_0\bar \g_0, \a \bar \a \g_0\bar \g_0}))=0,
\]
so the diagonal maps computed using $\eta_\tau$ and $\eta_{\tau'}$ coincide.
\end{proof}

\begin{lem}\label{lem:1-handles-and-handleslides} Suppose that $\cH=(\Sigma,\as,\bs)$ and $\cH'=(\Sigma,\as',\bs')$ are two Heegaard diagrams for $Y$ with the same underlying Heegaard surface, such that $\as'$ is obtained from $\as$ by an elementary handleslide or isotopy. We assume $\bs'$ is obtained similarly from $\bs$. Let $\frD$ and $\frD'$ denote doubling enhancements of $\cH$ and $\cH'$ with choices of almost complex structures and doubling curves $\Ds$ and $\Ds'$.  Let $\bS_0$ denote a 0-sphere which is embedded in $\Sigma\setminus (\as\cup \bs \cup \as'\cup \bs')$ and let $\Sigma(\bS_0)$ denote the result of attaching a 1-handle along $\bS_0$ to $\Sigma$. Let $\Ds_0$ and $\Ds_0'$ denote extra choices of doubling arcs, as in the construction of the 1-handle map, and let $\frD(\bS_0)$ and $\frD'(\bS_0)$ denote the doubling enhancements of $(\Sigma(\bS_0),\as\g_0,\bs\g_0)$ and $(\Sigma(\bS_0),\as'\g_0,\bs'\g_0)$ by using the new doubling curves $\Ds_0$ and $\Ds_0'$, respectively. Let $\tau$ be a single choice of symmetric 1-cycle on $\Sigma(\bS_0)\# \bar \Sigma(\bS_0)$ which intersects $\g_0$ transversely at a single point. Finally, let $\CFI(W(\bS_0))$ and $\CFI(W(\bS_0))'$ denote the involutive 1-handle maps, computed with these two choices of data. Then
\begin{equation}
\Psi_{\frD(\bS_0)\to \frD'(\bS_0)}\circ \CFI(W(\bS_0))\simeq \CFI(W(\bS_0))'\circ \Psi_{\frD\to \frD'}.
\label{eq:transition-map-commutes-1-handles}
\end{equation}
\end{lem} 
\begin{proof} 
The proof is to enlarge the diagram in~\eqref{eq:1-handles} into a 3-dimensional hyperbox by adding an extra direction coming out of the page. One face will correspond to the map $\CFI(W(\bS_0))$ and an opposite face will correspond to the map $\CFI(W(\bS_0))'$. It will be straightforward to unpack the hypercube relations to obtain~\eqref{eq:transition-map-commutes-1-handles}. We presently describe the construction of the hyperbox.

We extend the top-most hypercube of~\eqref{eq:1-handles} to the hyperbox shown in equation~\eqref{eq:hyperbox-1-handles-top}. The unlabeled arrows are holomorphic triangle maps. There are no morphisms of length 2 or 3. The hypercube relations follow from the stabilization formula of Proposition~\ref{prop:multi-stabilization-counts}. 
\begin{equation}
\begin{tikzcd}[
	column sep={1.6cm,between origins},
	row sep=1cm,
	fill opacity=.7,
	text opacity=1,
	execute at end picture={
	\foreach \Nombre in  {A,B,C,D}
  {\coordinate (\Nombre) at (\Nombre.center);}
\fill[opacity=0.1] 
  (A) -- (B) -- (C) -- (D) -- cycle;
}]
\CF(\as ,\bs)
	\ar[dr]
	\ar[ddd, "F_1^{\bar \b,\bar \b}"]
	\ar[rr,"F_1^{\g_0,\g_0}"]
&[.7 cm]
&\CF(\as\g_0,\bs \g_0)
	\ar[dr]
	\ar[ddd, "F_1^{\bar \b \bar\g_0, \bar \b \bar\g_0}"]
&[.7 cm]
\\
&[.7 cm]
 |[alias=A]|\CF(\as',\bs)
 	\ar[rr,line width=2mm,dash,color=white,opacity=.7]
	\ar[rr,"F_1^{\g_0,\g_0}"]
&\,
&[.7 cm]
 |[alias=B]|
\CF (\as' \g_0, \bs\g_0)
	\ar[ddd,swap, pos=.35, "F_1^{\bar \b'\bar \g_0, \bar \b\bar \g_0}"]
\\
\\
\CF(\as \bar \bs, \bs \bar \bs)
	\ar[dr]
	\ar[rr, "F_1^{\g_0\bar \g_0,\g_0\bar \g_0}"]
&[.7 cm]\,&
\CF(\as \bar \bs \g_0 \bar \g_0, \bs \bar \bs \g_0 \bar \g_0)
	\ar[dr]
\\
& [.7 cm]
|[alias=D]|
\CF(\as' \bar \bs', \bs \bar \bs)
 	\ar[from=uuu,line width=2mm,dash,color=white,opacity=.7]
	\ar[from=uuu, "F_1^{\bar \b', \bar \b}", pos=.35]
	\ar[rr, "F_1^{\g_0\bar \g_0, \g_0\bar \g_0}"]
&
&[.7 cm] 
|[alias=C]|
\CF(\as' \bar \bs'\g_0\bar \g_0, \bs \bar \bs\g_0\bar \g_0)
\end{tikzcd}
\hspace{-1.7cm}
\begin{tikzcd}[
	column sep={1.6cm,between origins},
	row sep=1cm,
	fill opacity=.7,
	text opacity=1,
	execute at end picture={
	\foreach \Nombre in  {A,B,C,D}
  {\coordinate (\Nombre) at (\Nombre.center);}
\fill[opacity=0.1] 
  (A) -- (B) -- (C) -- (D) -- cycle;
}]
 |[alias=A]|
\CF(\as',\bs)
	\ar[dr]
	\ar[ddd, "F_1^{\bar \b', \bar \b}"]
	\ar[rr, "F_1^{\g_0,\g_0}"]
&[.7 cm]
&
 |[alias=B]|
  \CF(\as' \g_0, \bs \g_0)
	\ar[rd]
	\ar[ddd, "F_1^{\bar \b'\bar \g_0,\bar \b \bar \g_0}"]
&[.7 cm]
\\
&[.7 cm]
\CF(\as',\bs')
 	\ar[rr,line width=2mm,dash,color=white,opacity=.7]
	\ar[rr,"F_1^{\g_0,\g_0}"]
&\,
&[.7 cm]
\CF (\as'\g_0,\bs'\g_0)
	\ar[ddd, "F_1^{\bar \b' \bar \g_0, \bar \b'\bar \g_0} " ]
\\
\\
|[alias=D]|
\CF(\as'\bar \bs', \bs \bar \bs)
	\ar[dr]
	\ar[rr,"F_1^{\g_0\bar \g_0,\g_0 \bar\g_0}"]
&[.7 cm]
&
|[alias=C]|
\CF(\as' \bar \bs' \g_0 \bar \g_0, \bs \bar \bs \g_0 \bar \g_0)
	\ar[dr]
\\
& [.7 cm]
\CF(\as'\bar \bs', \bs' \bar \bs')
 	\ar[from=uuu,line width=2mm,dash,color=white,opacity=.7]
	\ar[from=uuu, "F_1^{\bar \b',\bar \b'}", pos=.4]
	\ar[rr, "F_1^{\g_0\bar \g_0,\g_0 \bar\g_0}"]
&
&[.7 cm] 
\CF(\as' \bar \bs' \g_0 \bar \g_0, \bs' \bar \bs'\g_0 \bar \g_0)
\end{tikzcd}
\label{eq:hyperbox-1-handles-top}
\end{equation}

The hypercubes obtained by extending the two small left-most cubes of~\eqref{eq:1-handles} are similarly defined to be hyperboxes of stabilizations.

We now consider the right-most cubes of~\eqref{eq:1-handles} (i.e. the two cubes which have length 2-maps). The top-most cube is easily extended by building hypercubes of attaching curves where each length 1 chain is a cycle generating the top degree of $\HF^-$. We leave the details to the reader.

We now consider the bottom-most cube of~\eqref{eq:1-handles} which has a length 2 map, and we describe its extension to a 3-dimensional hyperbox. (Recall that this hypercube involved more choices than usual, such as the choice of the chain $\eta$ and the curve $\tau$; see Remark~\ref{rem:1-handle-many-choices}.) The construction is obtained by pairing hypercubes of attaching curves. The main technical challenge in extending the 2-dimensional cube in~\eqref{eq:1-handle-hypercube-involving-H1} is building a 3-dimensional hypercube of attaching curves with the following form:
\begin{equation}
\begin{tikzcd}[
	column sep={1.3cm,between origins},
	row sep=.8cm,
	labels=description,
	fill opacity=.7,
	text opacity=1,
	]
\Ds \g_0 \bar \g_0
	\ar[dr]
	\ar[ddd]
	\ar[rr]
	\ar[dddrr,dashed, "\eta"]
	\ar[ddddrrr,dotted]
&[.7 cm]
& \Ds \Ds_0
	\ar[rd]
	\ar[ddd]
	\ar[ddddr,dashed]
&[.7 cm]
\\
&[.7 cm]
\Ds' \g_0 \bar \g_0
 	\ar[rr,line width=2mm,dash,color=white,opacity=.7]
	\ar[rr]
&\,
&[.7 cm]
\Ds' \Ds_0'
	\ar[ddd]
 	\ar[from=ulll,line width=2mm,dash,color=white,opacity=.7]
	\ar[from=ulll,dashed]
\\
\\
\as \bar \as \g_0 \bar \g_0
	\ar[dr]
	\ar[rr,"\tau"]
	\ar[drrr,dashed, "h"]
&[.7 cm]\,&
\as \bar \as \g_0 \bar \g_0
	\ar[dr]
\\
& [.7 cm]
\as' \bar \as' \g_0 \bar \g_0
 	\ar[from=uuu,line width=2mm,dash,color=white,opacity=.7]
	\ar[from=uuu]
	\ar[rr, "\tau"]
 	\ar[from=uuuul,line width=2mm,dash,color=white,opacity=.7]
 	\ar[from=uuuul,dashed]
&
&[.7 cm] 
\as'\bar \as' \g_0 \bar \g_0
 	\ar[from=uuull,line width=2mm,dash,color=white,opacity=.7]
	\ar[from=uuull,dashed, "\eta'"]
\end{tikzcd}
\label{eq:hypercube-attaching-curves-1-handle-w-tau}
\end{equation}
In equation~\eqref{eq:hypercube-attaching-curves-1-handle-w-tau},  the chains $\eta$ and $\eta'$ are the ones used to construct $\CFI(W(\bS_0))$ and $\CFI(W(\bS_0))'$, as constructed in Lemma~\ref{lem:1-handle-bottom-hypercube}. In particular, these morphisms are fixed, and similarly the top length 2 morphism has also already been chosen in the previous step, so we assume that chain is fixed. Note that the two morphisms from $\as \bar \as \g_0\bar\g_0$ to $\as'\bar \as' \g_0\bar \g_0$ are forced to coincide, since $\as'$ is obtained from $\as$ by an elementary handleslide, so there is a single top degree generator of $\CF(\as'\bar \as'\g_0\bar \g_0, \as \bar \as\g_0\bar \g_0)$, which is $\Theta_{\a'\bar \a', \a \bar \a}\otimes \theta_{\g_0\bar \g_0, \g_0\bar \g_0}^+$.

Assuming for the moment the existence of the hypercube in~\eqref{eq:hypercube-attaching-curves-1-handle-w-tau}, one extends the bottom-most cube of~\eqref{eq:1-handles} by constructing two hypercubes. The first hypercube is obtained by pairing the back face of~\eqref{eq:hypercube-attaching-curves-1-handle-w-tau} with the hypercube
\[
\begin{tikzcd}[column sep=2cm]
\as \bar \bs \ar[r, "\Theta_{\a \bar \b, \a'\bar \b'}"]& \as'\bar \bs'
\end{tikzcd}.
\]
The second hypercube is obtained by pairing the entirety of ~\eqref{eq:hypercube-attaching-curves-1-handle-w-tau} with the 0-dimensional hypercube $\as' \bar \bs'$.

We now prove the existence of the hypercube in~\eqref{eq:hypercube-attaching-curves-1-handle-w-tau}.  The existence of length 2 chains in Equation~\eqref{eq:hypercube-attaching-curves-1-handle-w-tau} along the top, left and right faces is straightforward. 

We now consider the bottom face of~\eqref{eq:hypercube-attaching-curves-1-handle-w-tau}. The desired relation is
\[
A_\tau(\Theta_{\a \bar \a \bar \g_0, \a'\bar \a' \g_0\bar \g_0})+B_\tau(\Theta_{\a \bar \a \bar \g_0, \a'\bar \a' \g_0\bar \g_0})=\d (h)
\]
where $h$ denotes the diagonal chain. Since $\tau$ is a closed curve, we have an equality $A_\tau=B_{\tau}$, the above equation is satisfied by $h=0$. (Compare Equation~\eqref{eq:telescope-A-lambda} for arcs instead of closed curves). 
As well will see, setting $h=0$ might not allow the length 3 hypercube relation to be satisfied, so instead we use the chain
\begin{equation}
h=\veps\cdot \Theta_{\a\bar \a \g_0\bar \g_0,\a'\bar \a'\g_0\bar \g_0},
\label{eq:length-2-map-1-handle-T}
\end{equation}
where $\veps\in \bF$ is to be determined.

Finally, we claim that we may make a choice of the constant $\veps$ in equation~\eqref{eq:length-2-map-1-handle-T}, and also pick a length 3 chain for the cube in~\eqref{eq:hypercube-attaching-curves-1-handle-w-tau} so that the hypercube relations are satisfied.  To see this, let $C\in \CF(\Ds \g_0 \bar \g_0, \as'\bar \as' \g_0 \bar \g_0)$ denote the sum of all compositions featuring in the length 3 hypercube relations, except for the length 3 map itself. Here, we momentarily set $\veps=0$ in~\eqref{eq:length-2-map-1-handle-T}. The hypercube relations on the proper faces force $C$ to be a cycle. We note however that $C$ has the same grading as the top degree element generator of $\HF^-(\Ds\g_0\bar \g_0, \as'\bar \as' \g_0\bar \g_0)$. In particular, it either represents the top degree element or is null-homologous. If it is null-homologous, we set $\veps=0$ in~\eqref{eq:hypercube-attaching-curves-1-handle-w-tau} and pick any primitive of $C$ in the appropriate grading to be the length 3 map in our hypercube. If $C$ represents the top degree generator, then we pick $\veps=1$, and let the length 3 chain be any primitive of
\[
f_{\Dt\g_0\bar \g_0, \a \bar \a \g_0\bar \g_0, \a'\bar \a' \g_0\bar \g_0}(\Theta_{\Dt\g_0\bar \g_0, \a \bar \a \g_0\bar \g_0},\Theta_{\a \bar \a \g_0\bar \g_0, \a'\bar \a' \g_0\bar \g_0})+C.
\]
The hypercube relations are clearly satisfied.

Finally, it remains to extend the bottom-most cube of Figure~\ref{eq:1-handles}. To extend this hypercube, it is convenient to consider the expanded description in~\eqref{eq:1-handles-expanded-bottom-portion}. In fact, the only hypercube which requires explanation to extend is the cube
\[
\begin{tikzcd}[
labels=description,
row sep=1cm,
column sep=1cm,
fill opacity=.7,
text opacity=1
]
\CF(\as \bar \bs \bar \g_0, \as \bar \as \bar \g_0) 
	\ar[r, "F_1^{\g_0, \g_0}"]
	\ar[d,equal]
&
\CF(\as \bar \bs \g_0\bar \g_0, \as \bar \as \g_0\bar \g_0)
	\ar[r, "B_\tau"]
&
\CF(\as \bar \bs \g_0 \bar \g_0, \as \bar \as \g_0\bar \g_0)
	\ar[d, "F_3^{\g_0,\g_0}"]
\\
\CF(\as \bar \bs \bar \g_0, \as \bar \as \bar \g_0)
	\ar[rr, equal]
&&
\CF(\as \bar \bs \bar \g_0, \as \bar \as \bar \g_0)
\end{tikzcd}
\]
We define our 3-dimensional extension of this cube to have trivial length 3 map. The hypercube relations for the 3-dimensional extension of this cube are proven somewhat similarly to the proof of Proposition~\ref{prop:simple-expansion-hypercube}. The only non-trivial length 2 map in the above diagram is in the top subface. This subface has the following form
\[
\begin{tikzcd}[labels=description,column sep=1.3cm, row sep=1cm]
\CF(\as \bar \bs \bar \g_0,\as \bar \as\bar \g_0)
	\ar[r, "F_1^{\g_0,\g_0}"]
	\ar[d]
&
\CF(\as \bar \bs\g_0 \bar \g_0, \as \bar \as\g_0\bar \g_0)
	\ar[r, "B_\tau"] 
	\ar[d]
	\ar[dr,dashed]& 
\CF(\as \bar \bs\g_0 \bar \g_0, \as \bar \as\g_0\bar \g_0)
	\ar[d]\\
\CF(\as' \bar \bs' \bar \g_0,\as \bar \as\bar \g_0) 
	\ar[r, "F_1^{\g_0,\g_0}"]
	\ar[d]
&\CF(\as' \bar \bs'\g_0 \bar \g_0, \as \bar \as\g_0\bar \g_0)
	\ar[r, "B_{\tau}"]
	\ar[d]
	\ar[dr,dashed]
&
\CF(\as' \bar \bs'\g_0 \bar \g_0, \as \bar \as\g_0\bar \g_0)
	\ar[d]\\
\CF(\as' \bar \bs' \bar \g_0,\as' \bar \as'\bar \g_0)
	\ar[r, "F_1^{\g_0,\g_0}"]&
\CF(\as' \bar \bs'\g_0 \bar \g_0, \as' \bar \as'\g_0\bar \g_0)
	\ar[r, "B_{\tau}"]&
 \CF(\as' \bar \bs'\g_0 \bar \g_0, \as' \bar \as'\g_0\bar \g_0)
\end{tikzcd}
\]
It is sufficient to show that the compression of this diagram composes trivially with the map $F_3^{\g_0, \g_0}$. The key tool we will use is that the triangle maps, when applied to elements of the form $(\xs\times\theta_{\g_0,\g_0}^+)\otimes (\ys\times \theta_{\g_0,\g_0}^+)$ have only summands of the form $\zs\times \theta_{\g_0,\g_0}^+$. This fact follows from Proposition~\ref{prop:multi-stabilization-counts}. The top diagonal map counts holomorphic triangles which are weighted by the quantity $\# (\d_{\a\bar \a \g_0\bar \g_0}(\psi)\cap \tau)\in \bF$. Similarly, the bottom triangle map is a sum of holomorphic triangle maps with certain weights. The special inputs for the length 1 triangle maps all have tensor factors of $\theta_{\g_0,\g_0}^+$, as does the special length 2 element used for the bottom diagonal map by equation~\eqref{eq:length-2-map-1-handle-T}. Hence, in compression of the diagram, the only holomorphic triangles which are counted have only tensor factors of $\theta_{\g_0,\g_0}^+$, and hence so does their output. In particular, the output vanishes when composed with $F_3^{\g_0,\g_0}$.
\end{proof}

\subsection{3-handles}

We define the 3-handle map by dualizing the construction of the 1-handle map. Concretely, this amounts to compressing the diagram shown in Equation~\eqref{eq:3-handles}.

\begin{equation}
\begin{tikzcd}[
labels=description,
row sep=1cm,
column sep=2.6cm,
fill opacity=.7,
text opacity=1
]
\CF(\as\g_0,\bs\g_0)
	\ar[d, "F_{1}^{\bar{\b}\bar \g_0,\bar{\b}\bar \g_0}"]
	\ar[rr, "F_3^{\g_0,\g_0}"]
&
&
\CF(\as,\bs)
	\ar[d,"F_{1}^{\bar{\b},\bar{\b}}"]
\\
\CF(\as \bar{\bs} \g_0\bar \g_0, \bs \bar{\bs} \g_0\bar \g_0)
	\ar[r, "B_\tau"]
	\ar[d,"f^{\b \bar{\b} \g_0\bar \g_0\to \Dt\Dt_0}_{\a \bar{\b}\g_0 \bar \g_0}"]
	\ar[dr,dashed]
& 
\CF(\as \bar \bs \g_0 \bar \g_0, \bs  \bar \bs \g_0  \bar \g_0)
	\ar[r,"F_3^{\g_0\bar \g_0, \g_0\bar \g_0}"]
	\ar[d, "f^{\b \bar \b\g_0 \bar \g_0\to \Dt \g_0 \bar \g_0}_{\a \bar \b \g_0 \bar \g_0}"]
& \CF(\as\bar \bs, \bs \bar \bs)
	\ar[d,"f^{\b \bar \b \to \Dt }_{\a \bar \b }"]
\\
\CF(\as \bar \bs \g_0 \bar \g_0,\Ds \Ds_0)
	\ar[r, "f_{\a \bar \g_0\bar \g_0}^{\Dt\Dt_0\to \Dt \g_0\bar \g_0}"]
	\ar[d, "f^{\Dt\to \a  \bar{\a}}_{\a  \bar{\b}}"]
	\ar[dr,dashed]
&
 \CF(\as \bar \bs  \g_0 \bar \g_0, \Ds \g_0 \bar \g_0)
	\ar[r, "F_3^{\g_0\bar \g_0, \g_0\bar \g_0}"]
	\ar[d, "f_{\a \bar \b \g_0 \bar \g_0}^{\Dt \g_0 \bar \g_0\to \a \bar \a \g_0 \bar \g_0}"]
& \CF(\as\bar{\bs},\Ds)
	\ar[d, "f_{\a \bar \b}^{\Dt\to \a \bar \a}"]
\\
\CF(\as  \bar{\bs} \g_0\bar \g_0, \as  \bar{\as} \g_0\bar \g_0)
	\ar[r, equal]
	\ar[d, "F_3^{\a \g_0,\a \g_0}"]
&
\CF(\as \bar \bs \g_0 \bar \g_0, \as \bar \as \g_0 \bar \g_0)
	\ar[r, "F_3^{\g_0\bar \g_0, \g_0\bar \g_0}"]
&
\CF(\as \bar \bs, \as \bar \as)
	\ar[d, "F_3^{\a,\a}"]
\\
\CF(\bar{\bs}\bar \g_0, \bar{\as} \bar \g_0)
	\ar[rr, "F_3^{\bar \g_0, \bar \g_0}"]
&
&
\CF(\bar{\bs},\bar{\as})
\end{tikzcd}
\label{eq:3-handles}
\end{equation}

\subsection{Attaching 1-handles near the basepoint}

We may obtain a slightly simpler formula for the 1-handle map when we perform the attachment near the basepoint $w$.

\begin{lem}\label{lem:1-handles-simple} Suppose $\bS_0\subset Y$ is a 0-sphere and $Y$ is connected. We have an isomorphism of groups
\[
\CFI^-(Y(\bS_0))\iso \CFI^-(Y)\otimes_{\bF} \langle \theta^+,\theta^-\rangle.
\]
With respect to this decomposition, the involutive 1-handle map takes the form
 \[
 \CFI(W(\bS_0))(\xs)=\xs\times \theta^+,
 \]
 extended equivariantly over $\bF[U,Q]/Q^2$. The 3-handle map takes the form $\xs\times \theta^+\mapsto 0$ and $\xs\times \theta^-\mapsto \xs$, extended $\bF[U,Q]/Q^2$-equivariantly.
\end{lem}
\begin{proof}We focus on the 1-handle map. We pick our diagram so that, before doubling, there is a special genus 1 region near the basepoint, and so the doubled arcs have the configuration shown in Figure~\ref{fig:4} in the double of this region. We form the compression of~\eqref{eq:1-handles} by first compression horizontally, then vertically.

We first show that for suitable almost complex structures, the diagonal maps make trivial contribution to compression of the diagram. We consider the top-most diagonal map. This map counts holomorphic quadrilaterals on the connected sum
\[
(\as \bar \bs, \bs \bar \bs, \Ds, \Ds)\# (\g_0\bar \g_0,\g_0\bar \g_0,\g_0\bar \g_0,\Ds_0).
\]
Note that each Heegaard subdiagram of the right-handle quadruple has vanishing differential (since it has the minimal number of generators possible; see Figure~\ref{fig:4}). Hence the stabilization results for holomorphic quadrilaterals from Proposition~\ref{prop:multi-stabilization-counts} imply that 
\[
h_{\a \bar \b \g_0\bar \g_0}^{\b \bar \b \g_0\bar \g_0\to \Dt \g_0\bar \g_0\to \Dt \Dt_0}\circ F_1^{\g_0 \bar \g_0, \g_0\bar \g_0}=F_1^{\g_0\bar \g_0, \Dt_0} \circ h_{\a\bar \b}^{\b \bar \b \to \Dt\to \Dt}.
\]
By the small translate theorem for holomorphic quadrilaterals \cite[Proposition 11.5]{HHSZExact}, the map $h_{\a\bar \b}^{\b \bar \b \to \Dt\to \Dt}$ vanishes.

We now consider the bottom-most diagonal map in~\eqref{eq:1-handles}. Because of our choice of diagram, the chain $\eta$ in~\eqref{eq:1-handle-hypercube-involving-H1} may be chosen to vanish, by the same model computation we used to verify Lemma~\ref{lem:1-handle-bottom-hypercube}.  Hence, the dashed length 2 map is the sum
\begin{equation}
h^{\Dt \g_0\bar \g_0\to \Dt \Dt_0\to \a \bar \a \g_0\bar \g_0}_{\a \bar \b \g_0\bar \g_0}
+h_{\a \bar \b \g_0\bar \g_0}^{\Dt \g_0\bar \g_0\to \a \bar \a \g_0\bar \g_0\xrightarrow{\tau} \a \bar \a \g_0\bar \g_0}
\label{eq:holomorphic-quadrilaterals-1-handle-diagonal}
\end{equation}
Here, $h_{\a \bar \b \g_0\bar \g_0}^{\Dt \g_0\bar \g_0\to \a \bar \a \g_0\bar \g_0\xrightarrow{\tau} \a \bar \a \g_0\bar \g_0}$ denotes the map which counts holomorphic triangles on the diagram $(\Sigma'\# \bar \Sigma', \as \bar \bs \g_0, \bar \g_0, \Ds \g_0\bar \g_0, \as \bar \as \g_0\bar \g_0)$ which are weighted by an extra factor of $\d_{\a \bar \a \g_0\bar \g_0} D(\psi)\cdot \tau$.

 We claim that both of these maps in~\eqref{eq:holomorphic-quadrilaterals-1-handle-diagonal} vanish when precomposed with $F_1^{\g_0\bar \g_0,\g_0\bar \g_0}$ and post-composed with $F_3^{\a \g_0,\a \g_0}$. Consider the first quadrilateral map. It counts curves on the diagram
\[
(\as \bar \bs, \Ds, \Ds, \as \bar \as)\# (\g_0\bar \g_0, \g_0\bar \g_0, \Ds_0, \g_0\bar \g_0 ).
\]
We can use Proposition~\ref{prop:multi-stabilization-counts} to destabilize the quadruple, noting that the evaluation of the right-hand triple on a degree 0 tree of $C_*^{\cell}(K_4)$ is $\theta^+|\bar \theta^-+\theta^-|\bar \theta^+$ by Lemma~\ref{lem:1-handle-bottom-hypercube}. Hence $\theta_{\g_0,\g_0}^-$ factor of the composition is
\[
h_{\a \bar \b}^{ \Dt \to \Dt\to \a \bar \a}(\xs)\otimes \theta_{\bar \g_0, \bar \g_0}^+.
\]
This expression vanishes by the small translate theorem for quadrilaterals \cite[Proposition 11.5]{HHSZExact}, since $\Ds$ appears twice.

Finally, we argue that the triangle counting map $h_{\a \bar \b \g_0\bar \g_0}^{\Dt \g_0\bar \g_0\to \a \bar \a \g_0\bar \g_0\xrightarrow{\tau} \a \bar \a \g_0\bar \g_0}$ also makes trivial contribution. In this case, the triangle count occurs on 
\[
(\as \bar \bs, \Ds, \as \bar \as)\#(\g_0\bar \g_0, \g_0\bar \g_0, \g_0\bar \g_0).
\]
Proposition~\ref{prop:multi-stabilization-counts} implies there are no triangles with $\theta_{\g_0,\g_0}^-$ as output when only terms involving $\theta_{\g_0,\g_0}^+$ are the input, so this triangle map makes trivial contribution to the compressed diagram.

The argument for 3-handles follows because the 3-handle map is defined dually to the 1-handle map, so the above argument immediately translates after taking duals.
\end{proof}

\subsection{Commutation of 1-handles}

\begin{lem}
\label{lem:commute-1/3-handles}
 Suppose that $\bS$ and $\bS'$ are two embedded and disjoint spheres in $Y$, whose dimensions lie in $\{0,2\}$. Then
 \[
\CFI(W(Y(\bS),\bS'))\circ \CFI(W(Y,\bS))\simeq \CFI(W(Y(\bS'),\bS))\circ
\CFI(W(Y,\bS')).
\]
\end{lem}

Lemma~\ref{lem:commute-1/3-handles} may be proven by using the model computations from Lemma~\ref{lem:1-handles-simple}. Alternatively, it may be proven by building a 3-dimensional hyperbox which has two faces corresponding to the definition of the 1-handle maps from~\eqref{eq:1-handles}.

%% file: fig6.pdf_tex
\begingroup%
  \makeatletter%
  \providecommand\color[2][]{%
    \errmessage{(Inkscape) Color is used for the text in Inkscape, but the package 'color.sty' is not loaded}%
    \renewcommand\color[2][]{}%
  }%
  \providecommand\transparent[1]{%
    \errmessage{(Inkscape) Transparency is used (non-zero) for the text in Inkscape, but the package 'transparent.sty' is not loaded}%
    \renewcommand\transparent[1]{}%
  }%
  \providecommand\rotatebox[2]{#2}%
  \newcommand*\fsize{\dimexpr\f@size pt\relax}%
  \newcommand*\lineheight[1]{\fontsize{\fsize}{#1\fsize}\selectfont}%
  \ifx\svgwidth\undefined%
    \setlength{\unitlength}{344.0537343bp}%
    \ifx\svgscale\undefined%
      \relax%
    \else%
      \setlength{\unitlength}{\unitlength * \real{\svgscale}}%
    \fi%
  \else%
    \setlength{\unitlength}{\svgwidth}%
  \fi%
  \global\let\svgwidth\undefined%
  \global\let\svgscale\undefined%
  \makeatother%
  \begin{picture}(1,0.48044154)%
    \lineheight{1}%
    \setlength\tabcolsep{0pt}%
    \put(0,0){\includegraphics[width=\unitlength,page=1]{fig6.pdf}}%
    \put(0.10773927,0.17683469){\makebox(0,0)[t]{\lineheight{1.25}\smash{\begin{tabular}[t]{c}\rotatebox{90}{B}\end{tabular}}}}%
    \put(0,0){\includegraphics[width=\unitlength,page=2]{fig6.pdf}}%
    \put(0.34601803,0.37489373){\makebox(0,0)[t]{\lineheight{1.25}\smash{\begin{tabular}[t]{c}\rotatebox{-90}{E}\end{tabular}}}}%
    \put(0.34361594,0.17683469){\makebox(0,0)[t]{\lineheight{1.25}\smash{\begin{tabular}[t]{c}\rotatebox{90}{E}\end{tabular}}}}%
    \put(0,0){\includegraphics[width=\unitlength,page=3]{fig6.pdf}}%
    \put(0.10773927,0.37489373){\makebox(0,0)[t]{\lineheight{1.25}\smash{\begin{tabular}[t]{c}\rotatebox{-90}{B}\end{tabular}}}}%
    \put(0.41518908,0.31626021){\makebox(0,0)[lt]{\lineheight{1.25}\smash{\begin{tabular}[t]{l}$w$\end{tabular}}}}%
    \put(0.69288916,0.3941126){\makebox(0,0)[lt]{\lineheight{1.25}\smash{\begin{tabular}[t]{l}$w$\end{tabular}}}}%
    \put(0,0){\includegraphics[width=\unitlength,page=4]{fig6.pdf}}%
    \put(0.69371832,0.1769641){\makebox(0,0)[lt]{\lineheight{1.25}\smash{\begin{tabular}[t]{l}$w$\end{tabular}}}}%
    \put(0.95856042,0.3941126){\makebox(0,0)[lt]{\lineheight{1.25}\smash{\begin{tabular}[t]{l}$w$\end{tabular}}}}%
    \put(0,0){\includegraphics[width=\unitlength,page=5]{fig6.pdf}}%
    \put(0.95945779,0.17696423){\makebox(0,0)[lt]{\lineheight{1.25}\smash{\begin{tabular}[t]{l}$w$\end{tabular}}}}%
    \put(0,0){\includegraphics[width=\unitlength,page=6]{fig6.pdf}}%
    \put(0.60599785,0.012668){\makebox(0,0)[t]{\lineheight{1.25}\smash{\begin{tabular}[t]{c}$\pi_2(\Theta^+,\Theta^+,\theta^+|\bar{\theta}{}^-)$\end{tabular}}}}%
    \put(0.87305606,0.01162199){\makebox(0,0)[t]{\lineheight{1.25}\smash{\begin{tabular}[t]{c}$\pi_2(\Theta^+,\Theta^+,\theta^-|\bar{\theta}{}^+)$\end{tabular}}}}%
  \end{picture}%
\endgroup%

%% file: fig13.pdf_tex
\begingroup%
  \makeatletter%
  \providecommand\color[2][]{%
    \errmessage{(Inkscape) Color is used for the text in Inkscape, but the package 'color.sty' is not loaded}%
    \renewcommand\color[2][]{}%
  }%
  \providecommand\transparent[1]{%
    \errmessage{(Inkscape) Transparency is used (non-zero) for the text in Inkscape, but the package 'transparent.sty' is not loaded}%
    \renewcommand\transparent[1]{}%
  }%
  \providecommand\rotatebox[2]{#2}%
  \newcommand*\fsize{\dimexpr\f@size pt\relax}%
  \newcommand*\lineheight[1]{\fontsize{\fsize}{#1\fsize}\selectfont}%
  \ifx\svgwidth\undefined%
    \setlength{\unitlength}{390.75116664bp}%
    \ifx\svgscale\undefined%
      \relax%
    \else%
      \setlength{\unitlength}{\unitlength * \real{\svgscale}}%
    \fi%
  \else%
    \setlength{\unitlength}{\svgwidth}%
  \fi%
  \global\let\svgwidth\undefined%
  \global\let\svgscale\undefined%
  \makeatother%
  \begin{picture}(1,0.86572238)%
    \lineheight{1}%
    \setlength\tabcolsep{0pt}%
    \put(0,0){\includegraphics[width=\unitlength,page=1]{fig13.pdf}}%
    \put(0.2180203,0.67932316){\makebox(0,0)[lt]{\lineheight{1.25}\smash{\begin{tabular}[t]{l}$w$\end{tabular}}}}%
    \put(0,0){\includegraphics[width=\unitlength,page=2]{fig13.pdf}}%
    \put(0.71424956,0.57103542){\makebox(0,0)[lt]{\lineheight{1.25}\smash{\begin{tabular}[t]{l}$w$\end{tabular}}}}%
    \put(0,0){\includegraphics[width=\unitlength,page=3]{fig13.pdf}}%
    \put(0.47432696,0.78216721){\makebox(0,0)[lt]{\lineheight{1.25}\smash{\begin{tabular}[t]{l}$w$\end{tabular}}}}%
    \put(0,0){\includegraphics[width=\unitlength,page=4]{fig13.pdf}}%
    \put(0.47432696,0.57103548){\makebox(0,0)[lt]{\lineheight{1.25}\smash{\begin{tabular}[t]{l}$w$\end{tabular}}}}%
    \put(0,0){\includegraphics[width=\unitlength,page=5]{fig13.pdf}}%
    \put(0.71424967,0.78216727){\makebox(0,0)[lt]{\lineheight{1.25}\smash{\begin{tabular}[t]{l}$w$\end{tabular}}}}%
    \put(0,0){\includegraphics[width=\unitlength,page=6]{fig13.pdf}}%
    \put(0.95417111,0.78216716){\makebox(0,0)[lt]{\lineheight{1.25}\smash{\begin{tabular}[t]{l}$w$\end{tabular}}}}%
    \put(0,0){\includegraphics[width=\unitlength,page=7]{fig13.pdf}}%
    \put(0.95697938,0.57103531){\makebox(0,0)[lt]{\lineheight{1.25}\smash{\begin{tabular}[t]{l}$w$\end{tabular}}}}%
    \put(0,0){\includegraphics[width=\unitlength,page=8]{fig13.pdf}}%
    \put(0.32283484,0.22556915){\makebox(0,0)[lt]{\lineheight{1.25}\smash{\begin{tabular}[t]{l}$w$\end{tabular}}}}%
    \put(0,0){\includegraphics[width=\unitlength,page=9]{fig13.pdf}}%
    \put(0.58139597,0.12744404){\makebox(0,0)[lt]{\lineheight{1.25}\smash{\begin{tabular}[t]{l}$w$\end{tabular}}}}%
    \put(0,0){\includegraphics[width=\unitlength,page=10]{fig13.pdf}}%
    \put(0.82249972,0.34448178){\makebox(0,0)[lt]{\lineheight{1.25}\smash{\begin{tabular}[t]{l}$w$\end{tabular}}}}%
    \put(0,0){\includegraphics[width=\unitlength,page=11]{fig13.pdf}}%
    \put(0.82205252,0.13227123){\makebox(0,0)[lt]{\lineheight{1.25}\smash{\begin{tabular}[t]{l}$w$\end{tabular}}}}%
    \put(0,0){\includegraphics[width=\unitlength,page=12]{fig13.pdf}}%
    \put(0.58092996,0.34382653){\makebox(0,0)[lt]{\lineheight{1.25}\smash{\begin{tabular}[t]{l}$w$\end{tabular}}}}%
    \put(0,0){\includegraphics[width=\unitlength,page=13]{fig13.pdf}}%
    \put(0.27855486,0.83899839){\makebox(0,0)[lt]{\lineheight{1.25}\smash{\begin{tabular}[t]{l}$A$\end{tabular}}}}%
    \put(0.52756439,0.83899839){\makebox(0,0)[lt]{\lineheight{1.25}\smash{\begin{tabular}[t]{l}$B$\end{tabular}}}}%
    \put(0.77464963,0.83899839){\makebox(0,0)[lt]{\lineheight{1.25}\smash{\begin{tabular}[t]{l}$C$\end{tabular}}}}%
    \put(0.39377684,0.39944768){\makebox(0,0)[lt]{\lineheight{1.25}\smash{\begin{tabular}[t]{l}$D$\end{tabular}}}}%
    \put(0.64022388,0.39944768){\makebox(0,0)[lt]{\lineheight{1.25}\smash{\begin{tabular}[t]{l}$E$\end{tabular}}}}%
  \end{picture}%
\endgroup%

%% file: section-5-2handles.tex
\section{2-handles} \label{sec:2-handles}
In this section, we describe the map for a framed link $\bL\subset Y$, with a self-conjugate $\Spin^c$ structure $\frs$ on $W(Y,\bL)$ (or more generally, a family $\frS\subset \Spin^c(W(Y,\bL))$ which is closed under conjugation). As in Section \ref{sec:1-and-3-handles}, we may omit the path and framing from our notation.

The existence of a map for 2-handles was described by Hendricks and Manolescu \cite{HMInvolutive}, though they were not able to prove invariance of the construction. The construction we present in this section should be thought of as a systematic version of the construction therein using the doubling operation.

\subsection{The construction}

Suppose $\bL$ is a framed link in $Y$. We assume that the framing is Morse, so that there is a corresponding 2-handle cobordism $W(Y,\bL)$. We pick a bouquet for $\bL$, in the sense of \cite{OSTriangles}*{Definition~4.1}, as well as a Heegaard triple $(\Sigma,\as',\as,\bs)$ subordinate to this bouquet \cite{OSTriangles}*{Definition~4.2}.

\begin{rem}
 Like the 2-handle maps from \cite{OSTriangles}, our 2-handle maps also depend on a bouquet for the link $\bL$. Independence from the bouquet will be proven in Lemma~\ref{lem:well-defined-2handle}.
 \end{rem}

We first construct several auxiliary hypercubes. We firstly have two hypercubes of alpha attaching circles
\begin{equation}
\begin{tikzcd}[column sep=1.7cm]
\as \ar[r, "\Theta_{\a', \a}"] & \as'
\end{tikzcd}
 \quad
 \text{and}
 \quad
\begin{tikzcd}[column sep=1.7cm]
\as \bar \bs \ar[r, "\Theta_{\a' \bar \b, \a \bar \b}"] & \as'\bar \bs.
\end{tikzcd}
\label{eq:easy-aux-hypercubes-2-handles}
\end{equation}
and a hypercube of beta attaching circles
\begin{equation}
\begin{tikzcd}[column sep=1.7cm]
\bs\bar \bs \ar[r, "\Theta_{\b\bar \b, \Dt}"]& \Ds
\end{tikzcd}
\label{eq:easy-aux-hypercubes-2-handles-beta}
\end{equation}

There is a unique choice of top degree cycle in the hypercubes in Equation~\eqref{eq:easy-aux-hypercubes-2-handles}, while the cycle $\Theta_{\b\bar \b,\Dt}$ in Equation~\eqref{eq:easy-aux-hypercubes-2-handles-beta} is only well-defined up to addition of a boundary.

 Next we construct a more complicated hypercube of beta attaching curves:
\begin{equation}
\begin{tikzcd}[
labels=description,
row sep=1cm,
column sep=3cm,
fill opacity=.7,
text opacity=1
]
\Ds 
	\ar[r, "\Theta_{\Dt,\a'\bar\a'}"]
	\ar[d, "\Theta_{\Dt, \a\bar \a}"]
	\ar[dr,dashed, "\lambda_{\Dt,\a\bar\a'}"]
&
\as'\bar \as'
	\ar[d, "\Theta_{\a'\bar\a', \a \bar \a'}"]
\\
\as \bar \as
	\ar[r, "\Theta_{\a \bar \a,\a \bar \a'}"]
&
\as \bar \as',
\end{tikzcd}
\label{eq:second-auxiliary-hypercube-2-handle}
\end{equation}
We now prove that the hypercube in~\eqref{eq:second-auxiliary-hypercube-2-handle} may be constructed:
\begin{lem}\label{lem:lambda-cube} If $\Theta_{\Dt,\a'\bar \a'}$, $\Theta_{\a'\bar \a', \a \bar \a'}$, $\Theta_{\Dt,\a \bar \a}$ and $\Theta_{\a \bar \a, \a \bar \a'}$ are cycles which represent the top degree elements of their respective $\HF^-$ groups, then
\[
\left[f_{\Dt, \a'\bar \a', \a \bar \a'}(\Theta_{\Dt, \a'\bar \a'}, \Theta_{\a'\bar \a', \a \bar \a'})\right]=\left[f_{\Dt, \a \bar \a, \a \bar \a'}(\Theta_{\Dt, \a \bar \a},\Theta_{\a \bar \a, \a \bar \a'})\right],
\]
where brackets denote the induced element in homology.
In particular, $\lambda_{\Dt,\a \bar \a'}\in \CF(\Ds,\as\bar \as')$ may be chosen so that~\eqref{eq:second-auxiliary-hypercube-2-handle} is a hypercube of attaching curves.
\end{lem}
\begin{proof}
Both triples $(\Sigma\# \bar \Sigma,\Ds, \as'\bar \as', \as \bar \as')$ and $(\Sigma\# \bar \Sigma,\Ds, \as\bar \as, \as \bar \as')$ represent the same cobordism, topologically. Namely, they represent a cobordism from $(S^1\times S^2)^{\# g}$ to $(S^1\times S^2)^{g-|\bL|}$ which is surgery on a $|\bL|$ component link, each component of which is an $S^1$-fiber of a $S^1\times S^2$-summand.
\end{proof}
Since $\left[f_{\Dt, \a'\bar \a', \a \bar \a'}(\Theta_{\Dt, \a'\bar \a'}, \Theta_{\a'\bar \a', \a \bar \a'})\right]$ represents a top-degree class of $\CF(S^1\times S^2)^{g-|\bL|}$, we have that any two choices of $\lambda_{\Dt,\a\bar\a'}$ satisfying (\ref{eq:second-auxiliary-hypercube-2-handle}) are homologous.

Now suppose that $\frS\subset \Spin^c(W(Y,\bL))$ is a finite set which is closed under conjugation. We now define the 2-handle map for $\bL$, which we denote by $\CFI(W(Y,\bL),\frS)$, using the hypercubes in~\eqref{eq:easy-aux-hypercubes-2-handles}, ~\eqref{eq:easy-aux-hypercubes-2-handles-beta} and~\eqref{eq:second-auxiliary-hypercube-2-handle}. We define the map $\CFI(W(Y,\bL),\frS)$ to be the compression of the hyperbox shown in equation~\eqref{eq:hyperbox-2-handles}.  A priori, $\CFI(W(Y,\bL), \frS)$ depends on the choices in its construction; for well-definedness see Section \ref{subsec:independence-from-simple-handleslides}. 
\begin{equation}
\begin{tikzcd}[
labels=description,
row sep=1cm,
column sep=3cm,
fill opacity=.7,
text opacity=1
]
\CF(\as,\bs)
	\ar[rr, "f_{\a\to \a'}^{\b}"]
	\ar[d, "F_1^{\bar \b, \bar \b}"]
&&
\CF(\as', \bs)
	\ar[d, "F_1^{\bar \b, \bar \b}"]
\\
\CF(\as\bar\bs, \bs \bar \bs)
	\ar[rr, "f_{\a \bar \b\to \a' \bar \b}^{\b \bar \b}"]
	\ar[d,"f_{\a \bar \b}^{\b \bar \b\to \Dt}"]
	\ar[drr,dashed, "h_{\a\bar \b\to \a'\bar \b}^{\b\bar \b\to \Dt}"]
&&
\CF(\as '\bar \bs, \bs \bar \bs)
	\ar[d,"f_{\a' \bar \b}^{\b \bar \b\to \Dt}"]
\\
\CF(\as\bar \bs, \Ds)
	\ar[rr,"f_{\a\bar \b\to \a'\bar \b}^{\Dt}"]
	\ar[d,equal]
	\ar[drr,dashed," h_{\a\bar\b\to \a'\bar \b}^{\Dt\to \a'\bar \a'}"]
&
&
\CF(\as' \bar \bs, \Ds)
	\ar[d,"f_{\a'\b}^{\Dt\to \a'\bar \a'}"]
\\
\CF(\as \bar \bs, \Ds)
	\ar[d,"f_{\a \bar{\b}}^{\Dt\to \a \bar \a}"]
	\ar[dr,dashed, " H_{\a\bar \b}^{\Dt\to \a \bar\a'} "]
	\ar[r, "f_{\a \bar \b}^{\Dt\to \a'\bar \a'} "]
&
\CF(\as \bar \bs, \as'\bar \as')
	\ar[r, "f_{\a \bar \b\to \a'\bar \b}^{\a'\bar \a'}"]
	\ar[d, "f_{\a \bar \b}^{\a' \bar \a'\to \a \bar \a'}"]
&
\CF(\as' \bar \bs, \as' \bar \as')
	\ar[d, "F_3^{\a',\a'}"]
\\
\CF(\as \bar \bs, \as \bar \as)
	\ar[r, "f_{\a \bar \b}^{\a\bar \a\to \a \bar \a'}"]
	\ar[d,"F_3^{\a,\a}"]
&
\CF(\as \bar \bs, \as \bar \as')
	\ar[r, "F_3^{\a, \a}"]
	\ar[d, "F_3^{\a,\a}"]
&
\CF( \bar \bs, \bar \as')
	\ar[d,equal]
\\
\CF(\bar \bs, \bar \as)
	\ar[r,"f_{\bar \b}^{\bar \a\to \bar \a'}"]
&
\CF(\bar \bs, \bar \as')
	\ar[r,equal]
&
\CF(\bar \bs, \bar \as')
\end{tikzcd}
\label{eq:hyperbox-2-handles}
\end{equation}

\begin{lem}
 If we compress each horizontal level in~\eqref{eq:hyperbox-2-handles}, then the resulting diagram is a hyperbox of chain complexes.
\end{lem}
\begin{proof} It suffices to show that each face is a hypercube of chain complexes. The subcubes with length 2 maps are all obtained by pairing hypercubes of attaching curves, so the claim is immediate for these faces. The top-most and bottom-most subcubes are hypercubes of stabilization, so the claim is immediate. The only remaining cube of interest is the second-bottom-most cube on the right. The hypercube relations amount to the relation
\begin{equation}
F_{3}^{\a',\a'}\circ f_{\a \bar \b\to \a'\bar \b}^{\a' \bar \a'}=F_3^{\a,\a}\circ f_{\a \bar \b}^{\a'\bar \a'\to \a \bar \a'}.
\label{eq:2-handle-map-relation-with-3-handles}
\end{equation}
This is proven as follows. The Heegaard triples for these triangle maps are $(\Sigma, \as',\as,\as')\#(\bar \Sigma, \bar \bs, \bar \bs, \bar \as')$ and $(\Sigma, \as, \as', \as)\#(\bar \Sigma, \bar \bs, \bar \as', \bar \as')$. We see that the first triple is a stabilization of the small isotopy triple $(\bar \Sigma, \bar \bs, \bar \bs, \bar \as')$, and similarly the second triple is also a stabilization of a Heegaard triple for a small isotopy. Using the stabilization results of Proposition~\ref{prop:multi-stabilization-counts} and the small triangle theorem \cite{HHSZExact}*{Proposition~11.1}, we obtain ~\eqref{eq:2-handle-map-relation-with-3-handles}, completing the proof.
\end{proof}

\begin{rem}
 If $\frS\subset \Spin^c(W(Y,\bL))$ is a set of $\Spin^c$ structures which is closed under the conjugation action, then a cobordism map may be defined by summing the maps above for all $\frs\in \frS$. If $\frS$ is infinite, one may need to use coefficients in $\bF[[U]]$ to obtain a well defined map.
\end{rem}

\subsection{Independence from simple handleslides}\label{subsec:independence-from-simple-handleslides}

We now prove that the construction above commutes with the maps for simple handleslides and changes of the doubling arcs. We phrase our result in enough generality that they also imply that the 2-handle maps are independent of the choice of bouquet for $\bL$, and are independent from handleslides amongst the components of $\bL$.

Suppose that $(\Sigma,\as',\as,\bs)$ is a Heegaard triple subordinate to a bouquet of $\bL\subset Y$. Suppose that $(\Sigma,\as'_H,\as_H,\bs_H)$ is another Heegaard triple, such that $\as'_H$, $\as_H$ and $\bs_H$ are obtained from $\as',$ $\as$ and $\bs$ (respectively) by elementary handleslide equivalences, in the sense of Section~\ref{sec:elementary-equivalences}. We assume, furthermore that each of the following sets admits a unique top degree intersection point, which is a cycle in its respective Floer complex:
\[
 \bT_{\a_H}\cap \bT_{\a}, \quad \bT_{\a_H'}\cap \bT_{\a'}, \quad \bT_{\a_H'}\cap \bT_{\a}, \quad \text{and} \quad \bT_{\b}\cap \bT_{\b_H}.
 \]
We also suppose that we pick two sets of doubling arcs, $\Ds$ and $\Ds_H$. We require no particular relation between $\Ds$ and $\Ds_H$, except that the diagram $(\Sigma\# \bar \Sigma,\Ds,\Ds_H,w)$ be admissible.

\begin{prop}\label{prop:2-handles-handle-slides} Suppose that $(\Sigma,\as',\as,\bs)$ and $(\Sigma,\as'_H,\as_H,\bs_H)$ are as above.
Let $\CFI(W(\bL))$ and $\CFI(W(\bL))^H$ denote the 2-handle maps, computed with the triples $(\Sigma,\as',\as,\bs)$ and $(\Sigma,\as'_H,\as_H,\bs_H)$, and the doubling arcs $\Ds$ and $\Ds_H$, respectively. Let $\frD$ denote the doubling data for $Y$ determined by $(\Sigma,\as,\bs)$, $\Ds$, and some choices of almost-complex structures. Let $\frD_H$ denote the doubling data determined by $(\Sigma,\as_H,\bs_H)$, $\Ds_H$, and some choices of almost-complex structures. Let $\frD'$ and $\frD_H'$ denote the analogous data for $Y(\bL)$. Then
\[
\Psi_{\frD'\to \frD_H'} \circ \CFI(W(\bL))\simeq \CFI(W(\bL))^H\circ \Psi_{\frD\to \frD_H}.
\]
Here, $\Psi_{\frD\to \frD_H}$ and $\Psi_{\frD'\to \frD'_H}$ denote the transition maps for elementary handleslides and changes of the doubling data, and $\simeq$ denotes $\bF[U,Q]/Q^2$-equivariant chain homotopy.
\end{prop}
\begin{proof} Mirroring the construction of the maps $\Psi_{\frD\to \frD_H}$, we construct for each subcube of Equation~\eqref{eq:hyperbox-2-handles} a hyperbox of size $(2,1,1)$. The final coordinate corresponds to the involution. If we collapse this axis of the cube, we obtain a hyperbox of size $(2,1)$. If we compress this hyperbox, we obtain the following hypercube, which realizes the homotopy commutation in the statement.
\[
\begin{tikzcd}[
labels=description,
row sep=1cm,
column sep=3cm,
fill opacity=.7,
text opacity=1
]
\CFI(\frD)
	\ar[dr, "\Psi_{\frD\to \frD_H}"]
	\ar[r,"F_{W(\bL)}"]
	\ar[drr,dashed]
& 
\CFI(\frD')
	\ar[dr,  "\Psi_{\frD'\to \frD'_H}"]
&
\\
&\CFI(\frD_H) 
	\ar[r,"F_{W(\bL)}^H"]
&
\CFI(\frD'_H) 
\end{tikzcd}
\]

   The top-most hyperbox is obtained by extending the top-most hypercube of Equation~\eqref{eq:hyperbox-2-handles} out of the page, and is obtained by stacking two hypercubes of stabilization. We leave the details to the reader.

We now consider the hypercube in~\eqref{eq:hyperbox-2-handles} which is second from the top. The corresponding hyperbox is obtained by pairing hypercubes, as follow. We think of a hyperbox of size $(2,1,1)$ as being obtained by stacking two hypercubes. The first hypercube is obtained by constructing and pairing the following hypercubes of attaching curves, where the length-$1$ edges are obtained using the unique top-degree generators specified earlier; the $2$-faces are filled as in Lemma \ref{lem:lambda-cube}:
\[
\begin{tikzcd}[
labels=description,
row sep=1cm,
column sep=2cm,
fill opacity=.7,
text opacity=1
]
\as \bar \bs
	\ar[r]
	\ar[drr,dashed]
	\ar[dr]
&
\as'\bar \bs
	\ar[dr]
&
\,
\\
&
\as_H \bar \bs_H
	\ar[r]
&
\as_H' \bar \bs_H
\end{tikzcd}
\quad \text{and} \quad
\begin{tikzcd}[
labels=description,
row sep=1cm,
column sep=3cm,
fill opacity=.7,
text opacity=1
]
\bs \bar \bs\ar[d]\\
\Ds
\end{tikzcd}.
\]
The next hypercube is obtained by constructing and pairing the following hypercubes of attaching curves:
\[
\begin{tikzcd}[
labels=description,
row sep=1cm,
column sep=2cm,
fill opacity=.7,
text opacity=1
]
\as_H\bar \bs_H \ar[r]
&
\as_H'\bar \bs_H
\end{tikzcd}
\quad 
\text{and}
\quad
\begin{tikzcd}[
labels=description,
row sep=1cm,
column sep=1.2cm,
fill opacity=.7,
text opacity=1
]
\bs \bar \bs
	\ar[d]
	\ar[dr]
	\ar[ddr,dashed]
&\,\\
\Ds
	\ar[dr]
& \bs_H \bar \bs_H
	\ar[d]
\\
& \Ds_H
\end{tikzcd}
\]

The remaining hypercubes which have a length 2 map are constructed by pairing hypercubes of attaching curves in a straightforward manner. We remark that the central-most hypercube of~\eqref{eq:hyperbox-2-handles} is obtained by a small rearrangement of the pairing of an alpha hypercube with a beta hypercube. The corresponding 3-dimensional hyperbox is obtained by performing the analogous rearrangement to two hypercubes obtained by pairing.

Finally, we consider the three bottom-most hypercubes of \eqref{eq:hyperbox-2-handles} which do not have a length-$2$ map. We consider the right hypercube in the second to bottom-most level.

\begin{equation}
\begin{tikzcd}[
	column sep={1.6cm,between origins},
	row sep=1cm,
	labels=description,
	fill opacity=.7,
	text opacity=1,
	execute at end picture={
	\foreach \Nombre in  {A,B,C,D}
  {\coordinate (\Nombre) at (\Nombre.center);}
\fill[opacity=0.1] 
  (A) -- (B) -- (C) -- (D) -- cycle;
}]
\CF(\as \bar \bs, \as'\bar \as')
	\ar[dr]
	\ar[ddd]
	\ar[rr]
&[.7 cm]
&\CF(\as' \bar \bs, \as'\bar \as')
	\ar[dr]
	\ar[ddd, "F_3^{\a',\a'}"]
&[.7 cm]
\\
&[.7 cm]
 |[alias=A]|\CF(\as_H \bar \bs_H, \as'\bar \as')
 	\ar[rr,line width=2mm,dash,color=white,opacity=.7]
	\ar[rr]
&\,
&[.7 cm]
 |[alias=B]|
\CF (\as'_H \bar \bs_H, \as'\bar \as')
	\ar[ddd, "F_3^{\a_H',\a'}"]
 	\ar[from=ulll,line width=2mm,dash,color=white,opacity=.7]
	\ar[from=ulll,dashed]
\\
\\
\CF(\as \bar \bs, \as \bar \as')
	\ar[dr]
	\ar[rr, "F_3^{\a,\a}"]
&[.7 cm]\,&
\CF(\bar \bs, \bar \as)
	\ar[dr]
\\
& [.7 cm]
|[alias=D]|
\CF(\as_H \bar \bs_H, \as \bar \as')
 	\ar[from=uuu,line width=2mm,dash,color=white,opacity=.7]
	\ar[from=uuu]
	\ar[rr, "F_3^{\a_H,\a}"]
 	\ar[from=uuuul,line width=2mm,dash,color=white,opacity=.7]
	\ar[from=uuuul,dashed]
&
&[.7 cm] 
|[alias=C]|
\CF(\bar \bs_H, \bar \as')
\end{tikzcd}
\hspace{-1.2cm}
\begin{tikzcd}[
	column sep={1.6cm,between origins},
	row sep=1cm,
	labels=description,
	fill opacity=.7,
	text opacity=1,
	execute at end picture={
	\foreach \Nombre in  {A,B,C,D}
  {\coordinate (\Nombre) at (\Nombre.center);}
\fill[opacity=0.1] 
  (A) -- (B) -- (C) -- (D) -- cycle;
}]
 |[alias=A]|
\CF(\as_H \bar \bs_H, \as'\bar \as')
	\ar[dr]
	\ar[ddd]
	\ar[rr]
&[.7 cm]
&
 |[alias=B]|
  \CF(\as'_H \bar \bs_H, \as'\bar \as')
	\ar[rd]
	\ar[ddd, "F_3^{\a_H',\a'}"]
&[.7 cm]
\\
&[.7 cm]
\CF(\as_H\bar \bs_H, \as'_H \bar \as'_H)
 	\ar[rr,line width=2mm,dash,color=white,opacity=.7]
	\ar[rr]
&\,
&[.7 cm]
\CF (\as_H' \bar \bs_H, \as'_H\bar \as'_H)
	\ar[ddd, "F_3^{\a_H',\a_H'}"]
 	\ar[from=ulll,line width=2mm,dash,color=white,opacity=.7]
	\ar[from=ulll,dashed]
\\
\\
|[alias=D]|
\CF(\as_H \bar \bs_H, \as \bar \as')
	\ar[dr]
	\ar[rr, "F_3^{\a_H, \a}"]
&[.7 cm]&
|[alias=C]|
\CF(\bar \bs_H, \bar \as')
	\ar[dr]
\\
& [.7 cm]
\CF(\as_H \bar \bs_H, \as_H \bar \as_H')
 	\ar[from=uuu,line width=2mm,dash,color=white,opacity=.7]
	\ar[from=uuu]
	\ar[rr, "F_3^{\a_H,\a_H}"]
 	\ar[from=uuuul,line width=2mm,dash,color=white,opacity=.7]
	\ar[from=uuuul,dashed]
&
&[.7 cm] 
\CF(\bar \bs_H, \bar \as_H')
\end{tikzcd}
\label{eq:hyperbox-2-handle-bottom-right}
\end{equation}
In~\eqref{eq:hyperbox-2-handle-bottom-right}, all of the length 1 and 2 maps which are unlabeled are holomorphic triangle maps or quadrilateral maps. There are no length 3 maps.

The length 2 hypercube relations for the diagrams in~\eqref{eq:hyperbox-2-handle-bottom-right} are verified as follows. For the faces which have length 2 maps, the hypercube relations are obvious. For the other faces, the length 2 relations are proven similarly to the proof of~\eqref{eq:2-handle-map-relation-with-3-handles}.

It remains to prove the length 3 relations for the diagrams in ~\eqref{eq:hyperbox-2-handle-bottom-right}. To verify these, we note that each summand in the length 3 relation is a composition of a sum of quadrilateral maps followed by a  3-handle map. We claim that each summand vanishes identically. To see this, consider the length 2 map along the top face of the left cube of~\eqref{eq:hyperbox-2-handle-bottom-right}. This map is the sum of two quadrilateral maps, which take place on the diagrams $(\Sigma, \as'_H,\as',\as,\as')\#(\bar \Sigma,\bar \bs_H, \bar \bs, \bar \bs, \bar \as')$ and $(\Sigma, \as'_H,\as_H,\as,\as')\#(\bar \Sigma,\bar \bs_H, \bar \bs_H, \bar \bs, \bar \as') $. In particular, both diagrams are algebraically rigid and admissible multi-stabilizations. By our result on stabilizations in Proposition~\ref{prop:multi-stabilization-counts}, we see that 
\[
F_3^{\a_H',\a_H}\circ h_{\a \bar \b\to \a'\bar \b\to \a_H'\bar \b_H}^{\a' \bar \a'}=h_{\bar \b \to \bar \b\to \bar \b_H}^{\bar \a'}\circ F_3^{\a, \a'},
\]
however $h_{\bar \b \to \bar \b\to \bar \b_H}^{\bar \a'}=0$ by the small translate theorem for quadrilaterals \cite{HHSZExact}*{Proposition~11.5}, since $\bar \bs$ is repeated twice in the diagram. The same argument works for all of the other summands, verifying the hypercube relations for the diagrams in~\eqref{eq:hyperbox-2-handle-bottom-right}.

One now constructs hypercubes for the bottom two faces of~\eqref{eq:hyperbox-2-handles}. The construction of these hyperboxes is similar to the ones constructed above, and we leave the details to the reader.

There is one final hyperbox to construct, corresponding to the lowest level of Figure~\ref{def:transition-map-elementary-handleslide}. This hyperbox is constructed easily by pairing hypercubes of attaching curves and rearranging the result slightly. We leave the details to the reader. 
\end{proof}

\subsection{The composition law}

In this section, we prove the composition law for cobordisms with only 2-handles.

\begin{prop}
\label{prop:composition-law}
 Suppose that $\bL_1$ and $\bL_2$ are two framed links in $Y$, and $\frs_1\in \Spin^c(W(Y,\bL_1))$ and $\frs_2\in \Spin^c(W(Y(\bL_1),\bL_2))$ are two self-conjugate $\Spin^c$ structures. Then
 \[
\CFI(W(Y(\bL_1),\bL_2),\frs_2)\circ \CFI(W(Y,\bL_1),\frs_1)\simeq \CFI(W(Y,\bL_1\cup \bL_2),\frS(\frs_1,\frs_2))
 \]
 where $\frS(\frs_1,\frs_2)$ is the set of $\Spin^c$ structures on $W(Y,\bL_1\cup \bL_2)$ which restrict to $\frs_1$ and $\frs_2$.
\end{prop}
\begin{proof}
The proof of Proposition~\ref{prop:composition-law} is to construct a 3-dimensional hypercube with the following properties:
\begin{enumerate}
\item One direction corresponds to the involution.
\item One face corresponds to the cobordism map $\CFI(W(Y,\bL_1),\frs_1)$.
\item One face corresponds to the cobordism map $\CFI(W(Y(\bL_1),\bL_2), \frs_2)$.
\item One face corresponds to the identity map from $\CFI(Y)$ to itself.
\item One face corresponds to the map $\CFI(W(Y,\bL_1\cup \bL_2), \frS(\frs_1,\frs_2))$. 
\end{enumerate}
If we view the 3-dimensional hypercube as forming a 2-dimensional hypercube involving the $\CFI$ complexes, then the corresponding shape is as follows:
\[
\begin{tikzcd}[
labels=description,
row sep=1cm,
column sep=3cm,
fill opacity=.7,
text opacity=1
]
\CFI(Y)
	\ar[r, "F_{W(Y,\bL_1),\frs_1}"]
	\ar[dr,"F_{W(Y,\bL_1\cup \bL_2),\frS(\frs_1,\frs_2)}"]
	\ar[drr,dashed]
& 
\CFI(Y(\bL_1))
	\ar[dr, "F_{W(Y(\bL_1),\bL_2),\frs_2}"]
&
\\
&\CFI(Y(\bL_1,\bL_2)) 
	\ar[r, "\id"]
&
\CFI(Y(\bL_1,\bL_2))
\end{tikzcd}
\]

As a first step, pick a Heegaard quadruple $(\Sigma,\as'',\as',\as,\bs)$ satisfying the following:
\begin{enumerate}
\item $(\Sigma,\as',\as,\bs)$ is subordinate to a bouquet of the link $\bL_1\subset Y$.
\item $(\Sigma,\as'',\as',\bs)$ is subordinate to a bouquet of the link $\bL_2\subset Y(\bL_1)$.
\item $(\Sigma,\as'',\as,\bs)$ is subordinate to a bouquet of the link $\bL_1\cup \bL_2\subset Y$.
\end{enumerate}
Let $\Theta_{\a',\a}$, $\Theta_{\a'',\a'}$ and $\Theta_{\a'',\a}$ be the canonical generators. As in Ozsv\'{a}th and Szab\'{o}'s original proof for $\HF^-$, an easy model computation implies that the following diagram is a hypercube of alpha attaching curves:
\begin{equation}
\begin{tikzcd}[
labels=description,
row sep=1cm,
column sep=3cm,
fill opacity=.7,
text opacity=1
]
 \as
 \ar[r, "\Theta_{\a',\a}"]
 \ar[d, "\Theta_{\a'', \a}"]
 &
 \as'
 \ar[d, "\Theta_{\a'',\a'}"]
 \\
\as''
	\ar[r, "1"]
&
 \as''
 \end{tikzcd}
 \label{eq:hypercube-composition-law-non-inv}
\end{equation}
Pairing with the 0-dimensional hypercube $\bs$ and restricting to the $\Spin^c$ structures in the statement gives the non-involutive version of Proposition~\ref{prop:composition-law}. Note that we can stabilize~\eqref{eq:hypercube-composition-law-non-inv} to obtain the following hypercube of attaching curves on $\Sigma\# \bar \Sigma$:
\begin{equation}
\begin{tikzcd}[
labels=description,
row sep=1cm,
column sep=3cm,
fill opacity=.7,
text opacity=1
]
 \as\bar \bs
 \ar[r, "\Theta_{\a'\bar \b,\a\bar \b}"]
 \ar[d, "\Theta_{\a''\bar \b, \a\bar \b}"]
 &
 \as'\bar\bs
 \ar[d, "\Theta_{\a''\bar \b,\a'\bar \b}"]
 \\
\as''\bar \bs
	\ar[r, "1"]
&
 \as''\bar \bs
 \end{tikzcd}
 \label{eq:hypercube-composition-law-non-inv-stabilized}
\end{equation}
Our hyperbox for Proposition~\ref{prop:composition-law} will be of size $(1,1,6)$, which we organize into 6 hypercubes $\scC_1,\dots, \scC_6$, which we stack to obtain a hyperbox. The cubes $\scC_1$ and $\scC_2$ are shown in Figure~\ref{fig:composition-law-C1C2}.
\begin{figure}[ht!]
\[
\begin{tikzcd}[
	column sep={2cm,between origins},
	row sep=1.2cm,
	labels=description,
	fill opacity=.7,
	text opacity=1,
execute at end picture={
	\foreach \Nombre in  {A,B,C,D}
  	{\coordinate (\Nombre) at (\Nombre.center);}
	\fill[opacity=0.1] 
  (A) -- (B) -- (C) -- (D) -- cycle;}	
	]
\CF(\as,\bs)
	\ar[dr, "f_{\a\to \a''}^{\b}"]
	\ar[ddd, "F_1^{\bar \b, \bar \b}"]
	\ar[rr, "f_{\a\to \a'}^{\b}"]
&[.7 cm]
&\CF(\as',\bs)
	\ar[rd, "f_{\a\to \a''}^{\b}"]
	\ar[ddd,"F_1^{\bar \b, \bar \b}"]
&[.7 cm]
\\
&[.7 cm]
\CF(\as'',\bs)
 	\ar[rr,line width=2mm,dash,color=white,opacity=.7]
	\ar[rr, "\id"]
&\,
&[.7 cm]
\CF (\as'', \bs)
	\ar[ddd,"F_1^{\bar \b, \bar \b}"]
 	\ar[from=ulll,line width=2mm,dash,color=white,opacity=.7]
	\ar[from=ulll,dashed, " h_{\a\to \a'\to \a''}^{\b}" ]
\\
\\
|[alias=A]|\CF(\as\bar \bs, \bs \bar \bs)
	\ar[dr, "f_{\a \bar \b\to \a'' \bar \b}^{\b \bar \b}"]
	\ar[rr, "f_{\a \bar \b\to \a'\bar \b}^{\b \bar \b}"]
	\ar[drrr,dashed, "h_{\a\bar \b \to \a'\bar \b\to \a''\bar \b}^{\b\bar\b}"]
&[.7 cm]\,&
|[alias=B]|\CF(\as' \bar \bs, \bs \bar \bs)
	\ar[dr, "f_{\a'\bar \b\to\a'' \bar \b}^{\b \bar \b}"]
\\
& [.7 cm]
|[alias=D]|\CF(\as''\bar \bs, \bs \bar \bs)
 	\ar[from=uuu,line width=2mm,dash,color=white,opacity=.7]
	\ar[from=uuu,"F_1^{\bar \b, \bar \b}"]
	\ar[rr, "\id"]
&
&[.7 cm] 
|[alias=C]|\CF(\as''\bar \bs, \bs \bar \bs)
\end{tikzcd}
\]
\[
\begin{tikzcd}[
	column sep={2cm,between origins},
	row sep=1.2cm,
	labels=description,
	fill opacity=.7,
	text opacity=1,
	execute at end picture={
	\foreach \Nombre in  {A,B,C,D}
  	{\coordinate (\Nombre) at (\Nombre.center);}
	\fill[opacity=0.1] 
  (A) -- (B) -- (C) -- (D) -- cycle;}		
	]
|[alias=A]|\CF(\as\bar \bs,\bs \bar \bs)
	\ar[dr, "f_{\a\bar\b\to \a'' \bar \b}^{\b \bar \b}"]
	\ar[ddd, "f_{\a \bar \b}^{\b \bar \b\to \Dt}"]
	\ar[rr, "f_{\a\bar \b\to \a'\bar \b}^{\b\bar \b}"]
	\ar[dddrr,dashed]
&[.7 cm]
&|[alias=B]|\CF(\as'\bar \bs,\bs\bar \bs)
	\ar[rd, "f_{\a'\bar \b\to \a''\bar \b}^{\b\bar \b}"]
	\ar[ddd,"f_{\a'\bar \b}^{\b \bar \b\to \Dt}"]
	\ar[ddddr,dashed]
&[.7 cm]
\\
&[.7 cm]
|[alias=D]|\CF(\as''\bar \bs,\bs\bar \bs)
 	\ar[rr,line width=2mm,dash,color=white,opacity=.7]
	\ar[rr, "\id"]
&\,
&[.7 cm]
|[alias=C]|\CF (\as''\bar \bs, \bs\bar \bs)
	\ar[ddd,"f_{\a''\bar \b}^{\b \bar \b\to \Dt}"]
 	\ar[from=ulll,line width=2mm,dash,color=white,opacity=.7]
	\ar[from=ulll,dashed, " h_{\a\bar \b\to\a'\bar \b\to \a''\bar \b}^{\b\bar \b}" ]
\\
\\
\CF(\as\bar \bs, \Ds)
	\ar[dr, "f_{\a \bar \b\to \a''\bar \b}^{\Dt}"]
	\ar[rr, "f_{\a \bar \b\to \a'\bar \b}^{\Dt}"]
	\ar[drrr,dashed, "h_{\a\bar \b \to \a'\bar \b\to \a''\bar \b}^{\b\bar\b}"]
&[.7 cm]\,&
\CF(\as' \bar \bs,\Ds)
	\ar[dr, "f_{\a'\bar \b\to\a'' \bar \b}^{\Dt}"]
\\
& [.7 cm]
\CF(\as''\bar \bs, \Ds)
 	\ar[from=uuu,line width=2mm,dash,color=white,opacity=.7]
	\ar[from=uuu,"f_{\a \bar \b}^{\b\bar \b\to \Dt}"]
	\ar[rr, "\id"]
 	\ar[from=uuuul,line width=2mm,dash,color=white,opacity=.7]
 	\ar[from=uuuul,dashed]
&
&[.7 cm] 
\CF(\as''\bar \bs, \Ds)
\end{tikzcd}
\]
\caption{The hypercubes $\scC_1$ and $\scC_2$. The gray faces are stacked.}
\label{fig:composition-law-C1C2}
\end{figure}

The second hypercube in Figure~\ref{fig:composition-law-C1C2} also has a length 3 map which is not drawn. In the figure, the top hypercube is a hypercube of stabilization, and the bottom hypercube is obtained by pairing the hypercube in~\eqref{eq:hypercube-composition-law-non-inv-stabilized} with the 1-dimensional hypercube 
\[
\begin{tikzcd}[column sep=2cm, labels=description]
\bs \bar \bs\ar[r, "\Theta_{\b \bar \b, \Dt}"]& \Ds.
\end{tikzcd}
\]

We now move to the construction of the hypercube $\scC_3$ and $\scC_4$, starting by constructing some auxiliary hypercubes. In Figure~\ref{fig:auxiliary-beta-hypercube-2-handle-map}, we construct the 3-dimensional analog of the hypercube of beta-attaching curves in \eqref{eq:second-auxiliary-hypercube-2-handle}. We next consider the hypercube of chain complexes obtained by pairing the cube in~\eqref{eq:hypercube-composition-law-non-inv-stabilized} with the 1-dimensional cube 
\[
\begin{tikzcd}[column sep=2cm, labels=description]
\Ds\ar[r, "\Theta_{\Dt, \a''\bar \a''}"]& \as'' \bar \as''.
\end{tikzcd}
\]
\begin{figure} 
\begin{tikzcd}[
	column sep={1.3cm,between origins},
	row sep=.8cm,
	labels=description,
	fill opacity=.7,
	text opacity=1,
	]
\Ds
	\ar[dr]
	\ar[ddd]
	\ar[rr]
	\ar[dddrr,dashed]
	\ar[ddddrrr,dotted]
&[.7 cm]
& \as' \bar \as'
	\ar[rd]
	\ar[ddd]
&[.7 cm]
\\
&[.7 cm]
\as'' \bar \as''
 	\ar[rr,line width=2mm,dash,color=white,opacity=.7]
	\ar[rr]
&\,
&[.7 cm]
\as' \bar \as''
	\ar[ddd]
 	\ar[from=ulll,line width=2mm,dash,color=white,opacity=.7]
	\ar[from=ulll,dashed]
\\
\\
\as \bar \as
	\ar[dr]
	\ar[rr]
&[.7 cm]\,&
\as \bar \as'
	\ar[dr]
\\
& [.7 cm]
\as \bar \as''
 	\ar[from=uuu,line width=2mm,dash,color=white,opacity=.7]
	\ar[from=uuu]
	\ar[rr, "1"]
 	\ar[from=uuuul,line width=2mm,dash,color=white,opacity=.7]
 	\ar[from=uuuul,dashed]
&
&[.7 cm] 
\as \bar \as''
\end{tikzcd}
\caption{An auxiliary hypercube of $\beta$-attaching curves used in the construction of $\scC_4$.}
\label{fig:auxiliary-beta-hypercube-2-handle-map}
\end{figure}
\noindent We reorganize the resulting cube slightly, to obtain the the diagram shown in Figure~\ref{fig:C3}. Although the diagram in the figure is not technically a hypercube, since the bottom face is subdivided, if we compress the bottom face, we obtain a hypercube of chain complexes, which we call $\scC_3$.
\begin{figure}[ht]
\[
\begin{tikzcd}[
	column sep={1.8cm,between origins},
	row sep={1.5cm,between origins},
	labels=description,
	fill opacity=.7,
	text opacity=1,
	]
\CF(\as\bar \bs,\Ds)
	\ar[ddd, equal]
	\ar[rrrr, "f_{\a\bar \b\to \a'\bar \b}^{\Dt}"]
	\ar[ddrr, "f_{\a\bar \b\to \a'' \bar \b}^{\Dt}"]
	\ar[dddddrrrrrr,dotted]
&&&&
\CF(\as'\bar \bs, \Ds)
	\ar[ddrr, "f_{\a'\bar \b \to \a''\bar \b}^{\Dt}"]
	\ar[ddd,equal]
	\ar[dddddrr,dashed, "h_{\a'\bar \b\to \a''\bar \b}^{\Dt\to \a''\bar \a''}" ]
&&\,
\\
\\
&&\CF(\as''\bar \bs,\Ds)
	\ar[rrrr,equal]
&&&&
\CF(\as'' \bar \bs, \Ds)
 	\ar[from=uullllll,line width=2mm,dash,color=white,opacity=.7]
	\ar[from=uullllll, sloped,"h_{\a\bar \b\to \a'\bar \b \to \a'' \bar \b}^{\Dt}",dashed]
\\
\CF(\as\bar \bs, \Ds)
	\ar[dr, "f_{\a \bar \b}^{\Dt\to \a''\bar \a''}"]
	\ar[rrrr, "f_{\a \bar \b\to \a'\bar \b}^{\Dt}"]
	\ar[drrrrr,dashed, sloped, "h_{\a\bar \b \to \a'\bar \b}^{\Dt\to \a'' \bar \a''}", pos=.7]
&&&&
\CF(\as'\bar \bs,\Ds)
	\ar[dr, "f_{\a'\bar \b}^{\Dt\to \a''\bar\a''}"]
\\
&
\CF(\as\bar \bs, \as'' \bar \as'')
	\ar[rr, equal]
	\ar[dr, "f_{\a\bar \b\to \a''\bar \b}^{\a''\bar \a''}"]
&& 
\CF(\as\bar \bs, \as''\bar \as'')
	\ar[rr, "f_{\a\bar \b\to \a'\bar \b}^{\a''\bar \a''}"]
	\ar[dr, "f_{\a \bar \b\to \a''\bar \b}^{\a''\bar \a''}"]
	\ar[drrr,dashed, "h_{\a \bar \b\to \a'\bar \b\to \a''\bar \b}^{\a''\bar \a''}"]
&&
\CF(\as'\bar \bs, \as'' \bar \as'')
	\ar[dr, "f_{\a'\bar \b\to \a''\bar \b}^{\a''\bar \a''}"]
\\
\,&&
\CF(\as'' \bar \bs, \as'' \bar \as'')
	\ar[rr,equal]
 	\ar[from=uuu,line width=2mm,dash,color=white,opacity=.7]
	\ar[from=uuu,"f_{\a''\bar \b}^{\Dt\to \a''\bar \a''}", pos=.6]
 	\ar[from=uuuuull,line width=2mm,dash,color=white,opacity=.7]
	\ar[from=uuuuull,"h_{\a\bar \b\to\a''\bar \b}^{\Dt\to \a''\bar \a''}",dashed]
&&
\CF(\as'' \bar \bs, \as'' \bar \as'')
	\ar[rr,equal]
&&
\CF(\as'' \bar \bs, \as'' \bar \as'')
 	\ar[from=uuu,line width=2mm,dash,color=white,opacity=.7]
	\ar[from=uuu,"f_{\a''\bar \b}^{\Dt\to \a''\bar \a''}"]
\end{tikzcd}
\]
\caption{The diagram $\scC_3$.}
\label{fig:C3}
\end{figure}

Next, we pair the top face of Figure~\ref{fig:auxiliary-beta-hypercube-2-handle-map} with the cube
\[
\begin{tikzcd}[labels=description, column sep=2cm]
\as\bar \bs \ar[r,"\Theta_{\a\bar \b, \a'\bar \b}"] &\as' \bar \bs,
\end{tikzcd}
\]
and modify the resulting hypercube to obtain the back-most portion of the diagram shown on the top of Figure~\ref{fig:C456}. This diagram is not technically a hypercube, since some of the faces are subdivided. Nonetheless, we can form a hypercube $\scC_4$ as follows. We view the diagram as being obtained by stacking two box-like diagrams (one on the back, and one on the front). Some of the faces of these boxes are subdivided, so we compress these faces. This gives two cube-diagrams. The hypercube relations for these cubes are straightforward to verify. For the back-most cube, they follow from the hypercube relations for the original cube described above. For the front cube, the hypercube relations are straightforward to verify directly, since all of the triangle maps involve stabilizations of Heegaard triples for small isotopies of the attaching curves. Hence, Proposition~\ref{prop:multi-stabilization-counts} may be used to destabilize the relevant counts, and then the small translate theorem for triangles \cite{HHSZExact}*{Proposition~11.1} may be used to identify the destabilized maps with nearest point maps.

We now describe the construction of the hypercube $\scC_5$. We first pair the 0-dimensional hypercube $\as \bar \bs$ with the hypercube in Figure~\ref{fig:auxiliary-beta-hypercube-2-handle-map}. Then we rearrange, and obtain the diagram shown in the middle of Figure~\ref{fig:C456}.

There is a final hyperbox involving 3-handles, which we compress to obtain the hypercube $\scC_6$, shown on the bottom of Figure~\ref{fig:C456}. Stacking and compressing the hypercubes $\scC_1,\dots, \scC_6$ proves the statement.
\end{proof}


\begin{figure}[p]
 \[
 \scC_4:
\begin{tikzcd}[
	column sep={1.8cm,between origins},
	row sep=.5cm,
	labels=description,
	fill opacity=.7,
	text opacity=1,
	execute at end picture={
	\foreach \Number in  {A,B}
		      {\coordinate (\Number) at (\Number.center);}
	\filldraw[black] (A) node[rectangle,fill=white,fill opacity=.9, text opacity=1]{$\CF(\as \bar \bs, \as''\bar \as'') $};
	\filldraw[black] (B) node[rectangle,fill=white,fill opacity=.9, text opacity=1]{$ \CF(\as''\bar \bs,\as''\bar \as'')$};
		}
	]
\CF(\as\bar \bs,\Ds)
	\ar[ddd, equal]
	\ar[rrrr, ]
	\ar[dr]
	\ar[ddddrrrrr,dotted]
	\ar[dddrrrr,dashed]
&&&&
\CF(\as'\bar \bs, \Ds)
	\ar[dr,]
	\ar[ddd,]
	\ar[ddddr,dashed,  ]
&&\,
\\
&|[alias=A]| \phantom{\CF(\as \bar \bs, \as''\bar \as'')}
	\ar[dr]
	\ar[rr,equal]
&&
\CF(\as \bar \bs, \as'' \bar \as'')
	\ar[rr]
	\ar[ddd]
	\ar[dddrr,dashed]
&&
\CF(\as'\bar \bs, \as'' \bar \as'')
	\ar[dr]
	\ar[ddd]
	\ar[from=ulllll,dashed]
\\
&&|[alias=B]|
\phantom{\CF(\as''\bar \bs,\as''\bar \as'')}
	\ar[rrrr,equal,crossing over]
&&&&
\CF(\as'' \bar \bs, \as''\bar \as'')
	\ar[from=ulllll,dashed,crossing over]
\\
\CF(\as\bar \bs, \Ds)
	\ar[dr]
	\ar[rr]
	\ar[drrr,dashed]
&&
\CF(\as \bar \bs, \as'\bar \as')
	\ar[rr]
	\ar[dr]
	\ar[drrr,dashed]
&&
\CF(\as'\bar \bs,\as'\bar \as')
	\ar[dr]
\\
&
\CF(\as\bar \bs, \as'' \bar \as'')
	\ar[rr]
	\ar[dr]
	\ar[from=uuu, crossing over, equal]
&& 
\CF(\as\bar \bs, \as'\bar \as'')
	\ar[rr]
 	\ar[from=uuu,line width=2mm,dash,color=white,opacity=1,shorten <= 1.5cm]
	\ar[from=uuu,shorten <= 1.5cm]
&&
\CF(\as'\bar \bs, \as' \bar \as'')
	\ar[dr, "F_3^{\a',\a'}"]
\\
\,&&
\CF(\as'' \bar \bs, \as'' \bar \as'')
	\ar[rrrr,"F_3^{\a'',\a''}"]
 	\ar[from=uuu,line width=2mm,dash,color=white,opacity=.7]
	\ar[from=uuu,equal]
&&
&&
\CF(\bar \bs, \bar \as'')
 	\ar[from=uuu,line width=2mm,dash,color=white,opacity=.7]
	\ar[from=uuu,"F_3^{\a'',\a''}"]
\end{tikzcd}
\]
\[
\scC_5:
\begin{tikzcd}[
	column sep={1.8cm,between origins},
	row sep=.5cm,
	labels=description,
	fill opacity=.7,
	text opacity=1,
	execute at end picture={
		  \foreach \Number in  {A,B}
		      {\coordinate (\Number) at (\Number.center);}
		\filldraw[black] (A) node[rectangle,fill=white,fill opacity=.9, text opacity=1]{$\CF(\as \bar \bs, \as''\bar \as'') $};
		\filldraw[black] (B) node[rectangle,fill=white,fill opacity=.9, text opacity=1]{$\CF(\as''\bar \bs,\as''\bar \as'')$};
		}
	]
\CF(\as\bar \bs,\Ds)
	\ar[ddd]
	\ar[rr ]
	\ar[dr]
	\ar[dddrr,dashed]
	\ar[ddddrrr,dotted]
&&
\CF(\as \bar \bs, \as'\bar \as')
	\ar[dr]
	\ar[rr]
	\ar[ddd,shift left]
	\ar[ddddr,dashed]
&&
\CF(\as'\bar \bs, \as'\bar \as')
	\ar[dr,]
	\ar[ddd,"F_3^{\a',\a'}"]
&&\,
\\
&|[alias=A]|\phantom{\CF(\as \bar \bs, \as''\bar \as'')}
	\ar[dr]
	\ar[rr,crossing over]
	\ar[ddd]
&&
\CF(\as \bar \bs, \as' \bar \as'')
	\ar[rr,crossing over]
	\ar[ddd]
	\ar[from=ulll,dashed,crossing over]
&&
\CF(\as'\bar \bs, \as' \bar \as'')
	\ar[dr, "F_3^{\a',\a'}"]
	\ar[ddd, "F_3^{\a',\a'}"]
	\ar[from=ulll,dashed,crossing over]
\\
&&
|[alias=B]|\phantom{\CF(\as''\bar \bs,\as''\bar \as'')}
	\ar[rrrr,crossing over,"F_3^{\a'',\a''}", pos=.3]
&&&&
\CF(\bar \bs, \bar \as'')
\\
\CF(\as\bar \bs, \as\bar \as)
	\ar[dr]
	\ar[rr]
	\ar[drrr,dashed]
&&
\CF(\as \bar \bs, \as\bar \as')
	\ar[rr, "F_3^{\a,\a}", pos=.7]
	\ar[dr]
&&
\CF(\bar \bs,\bar \as')
	\ar[dr]
\\
&
\CF(\as\bar \bs, \as \bar \as'')
	\ar[rr,equal]
	\ar[dr, "F_3^{\a,\a}"]
	\ar[from=uuu, crossing over]
	\ar[from=uuuul,crossing over,dashed]
&& 
\CF(\as\bar \bs, \as\bar \as'')
	\ar[rr, "F_3^{\a,\a}"]
 	\ar[from=uuu,line width=2mm,dash,color=white,opacity=1, shorten <= 2cm]
 	 \ar[from=uuu, shorten <= 2cm]
&&
\CF(\bar \bs,\bar \as'')
	\ar[dr, equal]
\\
\,&&
\CF(\bar \bs,\bar \as'')
	\ar[rrrr,equal]
 	\ar[from=uuu,line width=2mm,dash,color=white,opacity=.7]
	\ar[from=uuu,"F_3^{\a'',\a''}", pos=.85]
&&
&&
\CF(\bar \bs, \bar \as'')
 	\ar[from=uuu,line width=2mm,dash,color=white,opacity=.7]
	\ar[from=uuu,equal]
\end{tikzcd}
\]
\[
\scC_6:
\begin{tikzcd}[
	column sep={1.8cm,between origins},
	row sep=.5cm,
	labels=description,
	fill opacity=.7,
	text opacity=1,
		execute at end picture={
		  \foreach \Number in  {A}
		      {\coordinate (\Number) at (\Number.center);}
		\filldraw[black] (A) node[rectangle,fill=white,fill opacity=.9, text opacity=1]{$\CF(\bar \bs,\bar \as'')$};
		}]
\CF(\as\bar \bs,\as \bar \as)
	\ar[ddd, "F_3^{\a, \a}"]
	\ar[rr]
	\ar[dr]
&&
\CF(\as \bar \bs, \as\bar \as')
	\ar[dr]
	\ar[rr, "F_3^{\a,\a}"]
	\ar[ddd, "F_3^{\a,\a}",shift left]
&&
\CF(\bar \bs, \bar \as')
	\ar[dr,]
	\ar[ddd,equal]
&&\,
\\
& \CF(\as \bar \bs, \as\bar \as'')
	\ar[dr, "F_3^{\a,\a}"]
	\ar[rr,crossing over, equal]
&&
\CF(\as \bar \bs, \as\bar \as'')
	\ar[rr, "F_3^{\a,\a}"]
	\ar[ddd, "F_3^{\a,\a}"]
 	\ar[from=ulll,line width=2mm,dash,color=white,opacity=.7]
	\ar[from=ulll,dashed]
&&
\CF(\bar \bs, \bar \as'')
	\ar[dr,equal]
	\ar[ddd,equal]
\\
&&
|[alias=A]|\phantom{\CF(\bar \bs,\bar \as'')}
	\ar[rrrr,crossing over,equal]
&&&&
\CF(\bar \bs, \bar \as'')
\\
\CF(\bar \bs,\bar \as)
	\ar[dr]
	\ar[rr]
	\ar[drrr,dashed]
&&
\CF(\bar \bs, \bar \as')
	\ar[rr, equal]
	\ar[dr]
&&
\CF(\bar \bs,\bar \as')
	\ar[dr]
\\
&
\CF(\bar \bs, \bar \as'')
	\ar[rr,equal]
	\ar[dr,equal]
	\ar[from=uuu,crossing over, "F_3^{\a,\a}"]
&& 
\CF(\bar \bs,\bar \as'')
	\ar[rr, equal]
 	\ar[from=uuu,line width=2mm,dash,color=white,opacity=1, shorten <= 1.7cm]
	\ar[from=uuu,shorten <= 1.7cm]
&&
\CF(\bar \bs,\bar \as'')
	\ar[dr, equal]
\\
\,&&
\CF(\bar \bs,\bar \as'')
	\ar[rrrr,equal]
 	\ar[from=uuu,line width=2mm,dash,color=white,opacity=.7,shift right=.3mm]
	\ar[from=uuu,equal, shift right=.3mm]
&&
&&
\CF(\bar \bs, \bar \as'')
 	\ar[from=uuu,line width=2mm,dash,color=white,opacity=.7]
	\ar[from=uuu,equal]
\end{tikzcd}
\]
\caption{The diagrams used to construct $\scC_4$, $\scC_5$ and $\scC_6$}
\label{fig:C456}
\end{figure}

\subsection{More on commutations of handles}

In this section, we prove that 1-handle and 3-handle maps may be commuted with 2-handle maps, when the corresponding topological handle manipulation is also possible:

\begin{prop}\label{prop:commute-1handles/2-handles}
 Suppose that $\bL$ is a framed 1-dimensional link in $Y$, and $\bS$ is a 0- or 2-dimensional sphere in $Y\setminus \bL$. Let $\frS\subset \Spin^c (W(Y,\bL))$ be a set of $\Spin^c$ structures which is closed under conjugation, and let $\frS'\subset \Spin^c (W(Y(\bS),\bL))$ be the corresponding set. Then
\[
\CFI(W(Y(\bL),\bS))\circ \CFI(W(Y,\bL),\frS)\simeq \CFI(W(Y(\bS),\bL),\frS')\circ \CFI(W(Y,\bS)).
\]
\end{prop}
\begin{proof}
 The proof is similar to the proof of Proposition~\ref{prop:composition-law}. We will build a 3-dimensional hypercube, four of whose faces correspond to the four maps $\CFI(W(Y(\bL),\bS))$,  $\CFI(W(Y,\bL),\frS)$,  $\CFI(W(Y(\bS),\bL),\frS')$, and $\CFI(W(Y,\bS))$. The hypercube relations for the compression of this hyperbox translate exactly to the chain homotopy for the commutation described in the statement. 
 
 The hypercube realizing the commutation will be constructed as the compression of a hyperbox of size $(1,1,5)$, which we view as five hypercubes, stacked on top of each other. We write $\scB_1,\cdots \scB_5$ for these hypercubes.
 
The top-most hypercubes, $\scB_1$ and $\scB_2$, are shown in Figure~\ref{fig:B1B2}. They are constructed by stacking hypercubes constructed by pairing hypercubes of attaching curves, as well as hypercubes of stabilization.

\begin{figure}[H]
\[
\scB_1:\hspace{-.5cm}
  \begin{tikzcd}[
  	column sep={2.1cm,between origins},
  	row sep=1cm,
  	labels=description,
  	fill opacity=.7,
  	text opacity=1,
	execute at end picture={
	\foreach \Nombre in  {A,B,C,D}
  	{\coordinate (\Nombre) at (\Nombre.center);}
	\fill[opacity=0.1] 
  (A) -- (B) -- (C) -- (D) -- cycle;}	
  	]
  \CF(\as,\bs)
  	\ar[dr,"f_{\a\to \a'}^{\b}"]
  	\ar[ddd, "F_1^{\bar \b,\bar \b}"]
  	\ar[rr, "F_1^{\g_0,\g_0}"]
  &[.7 cm]
  &\CF(\as\g_0,\bs\g_0)
  	\ar[rd, "f_{\a\g_0\to \a'\g_0}^{\b\g_0}"]
  	\ar[rr,equal]
  	\ar[ddd, "F_1^{\bar\b\bar \g_0,\bar\b\bar \g_0}"]
  &[.7 cm]
  &
  \CF(\as\g_0,\bs \g_0)
  	\ar[dr, "f_{\a\g_0\to \a'\g_0}^{\b\g_0}"]
  	\ar[ddd,"F_1^{\bar \b \bar \g_0, \bar \b \bar \g_0}"]
  &[.7 cm]
  \\
  &[.7 cm]
\CF(\as',\bs)
   	\ar[rr,line width=2mm,dash,color=white,opacity=.7]
  	\ar[rr," F_1^{\g_0,\g_0}"]
  &\,
  &[.7 cm]
   \CF(\as'\g_0,\bs\g_0)
    \ar[rr,line width=2mm,dash,color=white,opacity=.7]
   	\ar[rr, equal]
  &&
\CF (\as'\g_0,\bs\g_0)
  	\ar[ddd,"F_1^{\bar \b \bar \g_0, \bar \b \bar \g_0}"]
  \\
  \\
|[alias=A]| \CF(\as \bar \bs, \bs \bar\bs)
  	\ar[dr, "f_{\a\bar \b\to \a'\bar \b}^{\b\bar \b}"]
  	\ar[rr, "F_1^{\g_0\bar \g_0,\g_0\bar \g_0}"]
  &[.7 cm]\,&
  \CF(\as \bar \bs \g_0\bar \g_0, \bs \bar \bs \g_0 \bar \g_0)
  	\ar[dr, "f_{\a\bar \b\g_0\bar \g_0\to \a'\bar \b \g_0\bar \g_0}^{\b \bar \b \g_0\bar \g_0}"]
  	\ar[rr,equal]
  &[.7 cm]\,
  &
|[alias=B]|  \CF(\as \bar \bs \g_0 \bar \g_0, \bs \bar \bs \g_0 \bar \g_0)
  	\ar[dr,"f_{\a\bar \b\g_0\bar \g_0\to \a'\bar \b \g_0\bar \g_0}^{\b \bar \b \g_0\bar \g_0}"]
  \\
  & [.7 cm]
|[alias=D]|  \CF(\as'\bar \bs, \bs \bar \bs)
   	\ar[from=uuu,line width=2mm,dash,color=white,opacity=.7]
  	\ar[from=uuu,"F_1^{\bar \b,\bar \b}"]
  	\ar[rr, "F_1^{\g_0\bar \g_0,\g_0\bar \g_0}"]
  &
  &[.7 cm] 
  \CF(\as'\bar \bs \g_0\bar \g_0, \bs \bar \bs \g_0\bar \g_0)
  	\ar[rr,equal]
 	\ar[from=uuu, crossing over, "F_1^{\bar \b \bar \g_0, \bar \b \bar \g_0}"]
  	&&
|[alias=C]| \CF(\as'\bar \bs\g_0\bar \g_0,\bs \bar \bs \g_0\bar \g_0)
  \end{tikzcd}
  \]
  \[
  \scB_2:\hspace{-.5cm}
  \begin{tikzcd}[
  	column sep={2.1cm,between origins},
  	row sep=1cm,
  	labels=description,
  	fill opacity=.7,
  	text opacity=1,
	execute at end picture={
	\foreach \Nombre in  {A,B,C,D}
  	{\coordinate (\Nombre) at (\Nombre.center);}
	\fill[opacity=0.1] 
  (A) -- (B) -- (C) -- (D) -- cycle;}	
  	]
 |[alias=A]| \CF(\as\bar \bs,\bs\bar \bs)
  	\ar[dr]
  	\ar[ddd]
  	\ar[rr, "F_1^{\g_0\bar \g_0,\g_0\bar \g_0}"]
  &[.7 cm]
  &\CF(\as\bar\bs \g_0\bar \g_0,\bs\bar \bs\g_0\bar \g_0)
  	\ar[rd]
  	\ar[ddd]
  	\ar[rr,equal]
  	\ar[ddddr,dashed]
  &[.7 cm]
  &
|[alias=B]|  \CF(\as\bar \bs\g_0\bar \g_0,\bs\bar \bs \g_0\bar \g_0)
  	\ar[dr]
  	\ar[ddd]
    \ar[ddddr,dashed]
  &[.7 cm]
  \\
  &[.7 cm]
|[alias=D]|\CF(\as'\bar \bs,\bs \bar \bs)
   	\ar[rr,line width=2mm,dash,color=white,opacity=.7]
  	\ar[rr," F_1^{\g_0 \bar \g_0,\g_0 \bar \g_0}"]
  &\,
  &[.7 cm]
   \CF(\as'\bar \bs\g_0\bar \g_0,\bs\bar \bs\g_0\bar \g_0)
  	\ar[ddd]
    \ar[rr,line width=2mm,dash,color=white,opacity=.7]
   	\ar[rr, equal]
  &&
|[alias=C]|   \CF(\as'\bar \bs\g_0\bar \g_0,\bs\bar \bs\g_0\bar \g_0)
  	\ar[ddd]
  \\
  \\
 \CF(\as \bar \bs,\Ds)
  	\ar[dr,]
  	\ar[rr, "F_1^{\g_0\bar \g_0,\g_0\bar \g_0}"]
  &[.7 cm]\,&
  \CF(\as \bar \bs \g_0\bar \g_0, \Ds \g_0 \bar \g_0)
  	\ar[dr,]
  	\ar[rr,]
  	\ar[drrr,dashed] 
  &[.7 cm]\,
  &
  \CF(\as \bar \bs \g_0 \bar \g_0, \Ds \Ds_0)
  	\ar[dr]
  \\
  & [.7 cm]
  \CF(\as'\bar \bs, \Ds)
   	\ar[from=uuu,line width=2mm,dash,color=white,opacity=.7]
  	\ar[from=uuu,]
  	\ar[rr,"F_1^{\g_0\bar \g_0, \g_0\bar \g_0}"]
   	\ar[from=uuuul,line width=2mm,dash,color=white,opacity=.7]  	
  	\ar[from=uuuul,dashed]
  &
  &[.7 cm] 
  \CF(\as'\bar \bs \g_0\bar \g_0, \Ds\g_0\bar \g_0)
  	\ar[rr]
  	\ar[from=uuu,crossing over]
  	&&
 \CF(\as'\bar \bs\g_0\bar \g_0,\Ds \Ds_0)
  \end{tikzcd}
\]

  \caption{The hypercubes $\scB_1$ (top) and $\scB_2$ (bottom) are obtained by compressing these two diagrams. The unlabeled arrows are all holomorphic triangle or quadrilateral maps.    There is also a length 3 arrow in the right subcube for $\scB_2$ which is not drawn.}
  \label{fig:B1B2}
  \end{figure}

The diagram used to construct $\scB_3$ is shown in Figure~\ref{fig:B3}.
\begin{figure}[H]
\[
\begin{tikzcd}[
	column sep={2.2cm,between origins},
	row sep={1.5cm,between origins},
	fill opacity=.7,
	text opacity=1,
execute at end picture={
 \foreach \Number in  {A,B}
    {\coordinate (\Number) at (\Number.center);}
	\filldraw[black] (A) node[rectangle,fill=white, fill opacity=.9, text opacity=1]{$\CF(\as' \bar \bs, \Ds)$};
	\filldraw[black] (B) node[rectangle,fill=white,fill opacity=.9, text opacity=1]{$\CF(\as'\bar \bs \g_0\bar \g_0, \Ds \g_0\bar \g_0)$};
	}
	]
\CF(\as\bar \bs,\Ds)
	\ar[ddd,equal]
	\ar[rr, "F_1^{\g_0\bar \g_0,\g_0\bar \g_0}"]
	\ar[ddrr]
&&
\CF(\as \bar \bs \g_0 \bar \g_0, \Ds\g_0\bar \g_0)
	\ar[rr]
	\ar[ddd,dash, shift left=2mm]
	\ar[ddd,dash, shift left=1.3mm]
	\ar[ddrr]
	\ar[dddddrr,dashed]
&&
\CF(\as\bar \bs \g_0 \bar \g_0, \Ds \Ds_0)
	\ar[ddrr]
	\ar[ddd,shift left=2mm,dash]
	\ar[ddd,shift left=1.3mm, dash]
	\ar[dddddrr,dashed]
&&\,
\\
\\
&&|[alias=A]|\phantom{\CF(\as' \bar \bs, \Ds)}
	\ar[rr, "F_1^{\g_0\bar \g_0,\g_0\bar \g_0}"]
&&
|[alias=B]|\phantom{\CF(\as'\bar \bs \g_0\bar \g_0, \Ds \g_0\bar \g_0)}
	\ar[rr]
&&
\CF(\as'\bar \bs \g_0 \bar \g_0, \Ds\Ds_0)
 	\ar[from=uullll,line width=2mm,dash,color=white,opacity=.7]
	\ar[from=uullll,dashed]
\\
\CF(\as\bar \bs, \Ds)
	\ar[dr]
	\ar[rr,"F_1^{\g_0\bar \g_0, \g_0\bar \g_0}"]
&&
\CF(\as\bar \bs \g_0 \bar \g_0, \Ds \g_0\bar \g_0)
	\ar[rr]
	\ar[dr]
	\ar[drrr,dashed]
&&
\CF(\as \bar \bs \g_0 \bar \g_0, \Ds \Ds_0)
	\ar[dr]
\\
&
\CF(\as\bar \bs, \as' \bar \as')
	\ar[rr, "F_1^{\g_0\bar\g_0,\g_0\bar\g_0}"]
	\ar[dr]
&& 
\CF(\as\bar \bs \g_0 \bar \g_0, \as' \bar \as'\g_0\bar \g_0)
	\ar[rr,"B_\tau"]
	\ar[dr]
	\ar[drrr,dashed]
&&
\CF(\as \bar \bs \g_0 \bar \g_0, \as'\bar \as' \g_0 \bar \g_0)
	\ar[dr]
\\
&&
\CF(\as'\bar \bs, \as'\bar \as')
	\ar[rr,"F_1^{\g_0\bar \g_0, \g_0\bar \g_0}"]
 	\ar[from=uuu,line width=2mm,dash,color=white,opacity=.7]
	\ar[from=uuu]
 	\ar[from=uuuuull,line width=2mm,dash,color=white,opacity=.7]
	\ar[from=uuuuull,dashed]
&&
\CF(\as'\bar \bs \g_0 \bar \g_0, \as'\bar \as' \g_0\bar \g_0)
	\ar[rr, "B_\tau"]
	\ar[from=uuu,line width=2mm,dash,color=white,opacity=.7]
	\ar[from=uuu]
&&
\CF(\as'\bar \bs \g_0 \bar \g_0, \as'\bar \as' \g_0 \bar \g_0)
 	\ar[from=uuu,line width=2mm,dash,color=white,opacity=.7]
	\ar[from=uuu]
 	\ar[from=uuull,line width=2mm,dash,color=white,opacity=.7]
	\ar[from=uuull,dashed]
\end{tikzcd}
\]
\caption{The diagram used to construct the hypercube $\scB_3$.}
\label{fig:B3}
\end{figure}

The hypercube $\scB_4$ is displayed in Figure~\ref{fig:B4}. The construction is as follows. The back-left cube is a hypercube of stabilization for the map $F_1^{\g_0\bar \g_0, \g_0\bar \g_0}$. The back-right cube is constructed by pairing the 0-dimensional hypercube $\as \bar \bs \g_0\bar \g_0$ with a 3-dimensional generalized hypercube of beta attaching curves which has $\tau$ as one of the length 1 morphisms. The construction is similar to the definition of the 1-handle map, and the proof of well definedness is similar to Lemma~\ref{lem:1-handles-and-handleslides}. The front portion is verified to satisfy the hypercube relations (after compressing the top face) by observing that the length 2 map composes trivially with the 3-handle map $F_3^{\a'\g_0,\a'\g_0}$ by applying our stabilization  result from Proposition~\ref{prop:multi-stabilization-counts} and applying the small translate theorems for triangles and quadrilaterals \cite[Propositions 11.1 and 11.5]{HHSZExact}.

Our final hypercube $\scB_5$ is shown in Figure~\ref{fig:B5}. Stacking and compressing these hypercubes completes the proof.
\end{proof}

\begin{figure}[H]
\[
\begin{tikzcd}[
	column sep={2.3cm,between origins},
	row sep={1.8cm,between origins},
	fill opacity=.7,
	text opacity=1,
	execute at end picture={
	\foreach \Nombre in  {A,B,C,D}
  	{\coordinate (\Nombre) at (\Nombre.center);}
	\fill[opacity=0.1] 
  (A) -- (B) -- (C) -- (D) -- cycle;}	
	]
\CF(\as\bar \bs,\Ds)
	\ar[dd]
	\ar[rr, "F_1^{\g_0\bar \g_0,\g_0\bar \g_0}"]
	\ar[dr]
&&
\CF(\as \bar \bs \g_0 \bar \g_0, \Ds\g_0\bar \g_0)
	\ar[rr]
	\ar[dd, shift left=1mm]
	\ar[dr]
	\ar[dddr,dashed]
	\ar[ddrr,dashed]
&&
\CF(\as\bar \bs \g_0 \bar \g_0, \Ds \Ds_0)
	\ar[dr]
	\ar[dd,shift left=1mm]
	\ar[dddr,dashed]
&&\,
\\
&
 |[alias=A]|\CF(\as \bar \bs, \as'\bar \as')
	\ar[rr, "F_1^{\g_0\bar \g_0,\g_0\bar \g_0}"]
&&
\CF(\as \bar \bs \g_0\bar \g_0, \as'\bar \as' \g_0\bar \g_0)
	\ar[rr, "B_\tau"]
	\ar[ddrr,dashed]
&&
 |[alias=B]|\CF(\as \bar \bs \g_0\bar \g_0, \as'\bar \as' \g_0\bar \g_0)
 	\ar[from=ulll,line width=2mm,dash,color=white,opacity=.7]
	\ar[from=ulll,dashed]
	\ar[dd]
\\[.5cm]
\CF(\as\bar \bs, \as \bar \as)
	\ar[dr]
	\ar[rr,"F_1^{\g_0\bar \g_0, \g_0\bar \g_0}"]
&&
\CF(\as \bar \bs \g_0 \bar \g_0, \as \bar \as \g_0 \bar \g_0)
	\ar[rr, "B_\tau"]
	\ar[dr]
	\ar[drrr,dashed]
&&
\CF(\as \bar \bs \g_0 \bar \g_0,\as \bar \as \g_0 \bar \g_0)
	\ar[dr]
\\
&
 |[alias=D]|\CF(\as\bar \bs, \as \bar \as')
	\ar[rr, "F_1^{\g_0\bar \g_0, \g_0\bar \g_0}"]
	\ar[from=uu,line width=2mm,dash,color=white,opacity=.7]
	\ar[from=uu]
 	\ar[from=uuul,line width=2mm,dash,color=white,opacity=.7]
	\ar[from=uuul,dashed]
&& 
\CF(\as\bar \bs \g_0 \bar \g_0, \as \bar \as'\g_0\bar \g_0)
	\ar[rr,"B_\tau"]
 	\ar[from=uu,line width=2mm,dash,color=white,opacity=.7]
	\ar[from=uu]
&&
 |[alias=C]|\CF(\as \bar \bs \g_0 \bar \g_0, \as\bar \as' \g_0 \bar \g_0)
\end{tikzcd}
\]
\[
\begin{tikzcd}[
	column sep={2.3cm,between origins},
	row sep={1.8cm,between origins},
	fill opacity=.7,
	text opacity=1,
	execute at end picture={
	\foreach \Nombre in  {A,B,C,D}
  	{\coordinate (\Nombre) at (\Nombre.center);}
	\fill[opacity=0.1] 
  (A) -- (B) -- (C) -- (D) -- cycle;},
	]
 |[alias=A]|\CF(\as \bar \bs, \as'\bar \as')
	\ar[rr, "F_1^{\g_0\bar \g_0,\g_0\bar \g_0}"]
&&
\CF(\as \bar \bs \g_0\bar \g_0, \as'\bar \as' \g_0\bar \g_0)
	\ar[rr, "B_\tau"]
	\ar[ddrr,dashed]
	\ar[dd]
&&
 |[alias=B]|\CF(\as \bar \bs \g_0\bar \g_0, \as'\bar \as' \g_0\bar \g_0)
	\ar[dd]
&\,
\\
&\CF(\as' \bar \bs, \as'\bar \as')
	\ar[rr,crossing over, "F_1^{\g_0\bar \g_0, \g_0\bar \g_0}"]
	\ar[from=ul]
&&
\CF(\as'\bar \bs \g_0\bar \g_0, \as'\bar \as' \g_0\bar \g_0)
	\ar[rr,crossing over,pos=.3, "B_\tau"]
	\ar[from=ul]
&&
\CF(\as'\bar \bs \g_0 \bar \g_0, \as'\bar \as' \g_0 \bar \g_0)
 	\ar[from=ulll,line width=2mm,dash,color=white,opacity=.7]
	\ar[from=ulll,dashed]
	\ar[from=ul]
\\[.5cm]
 |[alias=D]|\CF(\as\bar \bs, \as \bar \as')
	\ar[rr, "F_1^{\g_0\bar \g_0, \g_0\bar \g_0}"]
	\ar[dr, "F_3^{\a,\a}"]
	\ar[from=uu,line width=2mm,dash,color=white,opacity=.7]
	\ar[from=uu]
&& 
\CF(\as\bar \bs \g_0 \bar \g_0, \as \bar \as'\g_0\bar \g_0)
	\ar[rr,"B_\tau"]
&&
 |[alias=C]|\CF(\as \bar \bs \g_0 \bar \g_0, \as\bar \as' \g_0 \bar \g_0)
	\ar[dr, "F_3^{\a \g_0,\a \g_0}"]
\\
\,&
\CF(\bar \bs, \bar \as')
	\ar[rrrr,"F_1^{\bar \g_0,\bar \g_0}"]
	\ar[from=uu, "F_3^{\a',\a'}",crossing over, pos=.3]
&&
&&
\CF(\bar \bs \bar \g_0, \bar \as' \bar \g_0)
 	\ar[from=uu,line width=2mm,dash,color=white,opacity=.7]
	\ar[from=uu, "F_3^{\a'\g_0,\a'\g_0}"]
\end{tikzcd}
\]
\caption{The diagram used to build the hypercube $\scB_4$. We stack the diagrams along the gray face. There is a length 3 map in the right portion of top-most diagram which is not drawn.}
\label{fig:B4}
\end{figure}

\begin{figure}[H]
\[
\begin{tikzcd}[
	column sep={2.2cm,between origins},
	row sep={1.1cm,between origins},
	fill opacity=.7,
	text opacity=1,
	]
\CF(\as \bar \bs, \as \bar \as)
	\ar[ddd, "F_3^{\a,\a}"]
	\ar[rr, "F_1^{\g_0\bar \g_0,\g_0\bar \g_0}"]
	\ar[dr]
&&
\CF(\as \bar \bs \g_0 \bar \g_0, \as \bar \as\g_0\bar \g_0)
	\ar[rr, "B_\tau"]
	\ar[dr]
&&
\CF(\as\bar \bs \g_0 \bar \g_0, \as \bar \as \g_0 \bar \g_0)
	\ar[dr]
	\ar[ddd,shift left=1mm]
&&\,
\\
&
\CF(\as \bar \bs, \as \bar \as')
	\ar[rr, "F_1^{\g_0\bar\g_0,\g_0\bar \g_0}"]
&&
\CF(\as \bar \bs \g_0\bar \g_0, \as\bar \as' \g_0\bar \g_0)
	\ar[rr, "B_\tau",crossing over, pos=.3]
&&
\CF(\as \bar \bs \g_0\bar \g_0, \as\bar \as' \g_0\bar \g_0)
 	\ar[from=ulll,line width=2mm,dash,color=white,opacity=.7]
	\ar[from=ulll,dashed]
\\
&&\CF(\bar \bs, \bar \as')
	\ar[rrrr, "F_1^{\bar \g_0, \bar \g_0}", pos=.45, crossing over]
	\ar[from=ul, "F_3^{\a,\a}"]
&&
&&
\CF(\bar \bs \bar \g_0, \bar \as' \bar \g_0)
	\ar[from=ul, "F_3^{\a,\a}"]
\\
\CF(\bar \bs, \bar \as)
	\ar[ddrr]
	\ar[rrrr,"F_1^{\bar \g_0,\bar \g_0}", pos=.4]
&&
&&
\CF(\bar \bs \bar \g_0, \bar \as \bar \g_0)
	\ar[ddrr]
&&
\\
\,
\\
\,
&
\,&
\CF(\bar \bs, \bar \as')
	\ar[rrrr,"F_1^{\bar \g_0,\bar \g_0}"]
 	\ar[from=uuu,line width=2mm,dash,color=white,opacity=.7]
	\ar[from=uuu, equal]
&&
&&
\CF(\bar \bs \bar \g_0, \bar \as' \bar \g_0)
 	\ar[from=uuu,line width=2mm,dash,color=white,opacity=.7]
	\ar[from=uuu,equal]
\end{tikzcd}
\]
\caption{The diagram used to construct the hypercube $\scB_5$}
\label{fig:B5}
\end{figure}

%% file: section-6a-stabilizations.tex
\section{Naturality maps for stabilizations}\label{sec:stable}

In this section, we construct our naturality maps for stabilizations of the Heegaard surface, and prove some basic properties. Our definition of the naturality map for a stabilization is simple: we define the map for a stabilization as the composition of a 1-handle cobordism map with the 2-handle cobordism map for a 2-handle which topologically cancels the 1-handle.

 The definition gives some diagrammatic flexibility for computing the stabilization map. In particular, for a stabilization we require two points $p_0$ and $p_1$ to be chosen on $\Sigma\setminus (\as \cup \bs)$. We then add a tube to $\Sigma$ with feet at $p_0$ and $p_1$. We add one new alpha curve and one beta curve. The beta curve is a meridian of the tube, and the alpha curve is obtained by concatenating a longitude of the curve with an embedded curve on $\Sigma\setminus \as$. If $\Sigma$ is embedded in $Y$, then we may view the construction as being determined by picking an embedded path on $\Sigma\setminus \as$ which connects $p_0$ and $p_1$, and having the tube run parallel to this path. The new alpha circle runs parallel to this arc. We define the destabilization map as the composition of a 2-handle map for a 0-framed unknot, followed by the 3-handle cobordism map.
 
 Note that this stabilization operation naturally determines a 3-ball in $Y$ which intersects $\Sigma$ in a disk. We say two stabilizations are \emph{disjoint} if there are two choices of corresponding 3-balls in $Y$ which are disjoint. 
 
 We write $\sigma$ for the stabilization map, and $\tau$ for the destabilization map. The main results of this section are summarized in the following proposition.
 
 \begin{prop}\label{prop:stabilization-naturality}\,
 \label{prop:properties-stabilization}
 \begin{enumerate}
 \item\label{stabilization-properties-1}The stabilization and destabilization maps commute with the transition maps for elementary equivalences of the doubling-enhanced Heegaard diagram.
 \item\label{stabilization-properties-2} If $\sigma$ and $\sigma'$ are the maps for two disjoint stabilizations, then
 \[
\sigma\circ \sigma'\simeq \sigma'\circ \sigma. 
 \]
 \item\label{stabilization-properties-3} If $\sigma$ and $\tau$ are stabilizations and destabilizations for the same 3-ball, then
 \[
\sigma\circ \tau\simeq \id\quad \text{and} \quad \tau\circ \sigma\simeq \id. 
 \]
 \end{enumerate}
 \end{prop}

Our proof of Parts \eqref{stabilization-properties-2} and \eqref{stabilization-properties-3} goes by way of computing basepoint-adjacent stabilizations and destabilizations (note that by handlesliding the original alpha and beta curves, such a configuration may always be achieved):

\begin{lem}
\label{lem:stabilization-near-basepoint}
 Suppose that $p_1$ and $p_0$ are two points on $\Sigma\setminus (\as\cup \bs)$ which both lie in the same component as the basepoint $w$. For a suitable choice of doubling curves on the stabilized diagram and suitably degenerated choices of almost complex structures, we have
 \[
\sigma\simeq (\xs\mapsto \xs\times c)\quad \text{and} \quad \tau\simeq(\xs\times c\mapsto \xs), 
 \]
 where $\{c\}=\a_0\cap \b_0$, and $(\xs\mapsto \xs\times c)$ denotes the $\bF[U,Q]/Q^2$-equivariant map which sends $\xs$ to $\xs\times c$, and similarly for $(\xs\times c\mapsto \xs)$.
\end{lem}
\begin{proof}Consider $\sigma$ (the proof for $\tau$ is essentially the same). By definition, $\sigma$ is the composition of the 1-handle cobordism map, followed by the map for a canceling 2-handle. By Lemma~\ref{lem:1-handles-simple}, for a suitable choice of almost complex structure, the 1-handle map takes the form $\xs\mapsto \xs\times \theta^+$, extended $\bF[U,Q]/Q^2$-equivariantly. We will show that the 2-handle cobordism map is chain homotopic to a map which sends $\xs\times \theta^+$ to $\xs\times c$.

We recall that the 2-handle map is defined in Equation~\eqref{eq:hyperbox-2-handles}. We claim that the diagonal map of the compression is trivial. To see this, we make the following subclaims about the diagram in~\eqref{eq:hyperbox-2-handles} in the present context:
\begin{enumerate}
\item \label{model-comp-stabilization-1} The top-most length 2 map vanishes on elements of the form $\xs\times \theta_{\b_0,\b_0}^+\times \Theta_{\bar \b, \bar \b}^+\times \theta_{\bar \b_0,\bar \b_0}^+$, for $\xs\in \bT_{\a}\cap \bT_{\b}$.
\item\label{model-comp-stabilization-2} The middle length 2 map composes trivially with the 3-handle map $F_3^{\a  \b_0, \a  \b_0}$ (labeled $F_3^{\a',\a'}$ in Equation~\eqref{eq:hyperbox-2-handles}), and the bottom-most length 2 map composes trivially with the map $F_3^{\a \a_0, \a \a_0}$ (labeled $F_3^{\a,\a}$ in Equation~\eqref{eq:hyperbox-2-handles}).
\end{enumerate}

We consider claim~\eqref{model-comp-stabilization-1} first. There is a genus 2 portion of the Heegaard quadruple corresponding to the stabilization, which takes the following form:
\[
\cQ_0=(\bT^2\# \bT^2, \a_0 \bar \b_0, \b_0 \bar \b_0, \b_0 \bar \b_0, \Ds_0).
\]
We observe that
\[
\widehat{f}_{\a_0\bar \b_0, \b_0\bar \b_0, \Dt_0}(\widehat{f}_{\a_0\bar \b_0,\b_0\bar \b_0, \b_0\bar \b_0}(\Theta_{\a_0\bar \b_0, \b_0\bar \b_0}^+, \Theta_{\b_0\bar \b_0, \b_0\bar \b_0}^+), \Theta_{\b_0\bar \b_0, \Dt_0}^+)=\Theta_{\a_0\bar \b_0, \Dt_0}.
\]
This may be seen since the inner triangle map corresponds to the naturality map for a small perturbation of the curves $\b_0\bar \b_0$, while the outer triangle map corresponds to a factor of the doubling model of the involution for the diagram $(\bT^2, \a_0,\b_0)$. Hence, by applying Proposition~\ref{prop:multi-stabilization-counts}, we obtain the equality
\[
h_{\a \b_0 \bar \b \bar \b_0\to \a \a_0 \bar \b \bar \b_0}^{\b \b_0 \bar \b \bar \b_0\to \Dt\Dt_0}(\xs\times \theta_{\b_0,\b_0}^+\times \Theta_{\bar \b, \bar \b}^+\times \theta_{\bar \b_0,\bar \b_0}^+)=h_{\a \bar \b \to \a \bar \b}^{\b \bar \b \to \Dt}(\xs\times \theta_{\b_0,\b_0}^+)\otimes \Theta_{\a_0\bar \b_0, \Dt_0}.
\]
which vanishes by the small translate theorem for holomorphic quadrilaterals \cite[Proposition 11.5]{HHSZExact}.

Next, we consider claim~\eqref{model-comp-stabilization-2}, from above. There are two quadrilateral maps which appear in this claim. The proof for both maps is essentially the same, so we focus on the map
\[
h_{\a \b_0\bar \b \bar \b_0\to \a \a_0 \bar \b\bar \b_0}^{\Dt\Dt_0\to \a  \a_0\bar \a \bar \a_0}.
\]
This is the middle length 2 map in~\eqref{eq:hyperbox-2-handles}. We wish to show that this map composes trivially with $F_3^{\a',\a'}$. To see this, we consider the genus 2 Heegaard quadruple
 \[
\cQ_1= (\bT^2\# \bT^2, \a_0 \bar \a_0, \Dt_0, \b_0\bar \b_0, \a_0\bar \b_0).
 \]
 This is an admissible, algebraically rigid multi-stabilizing quadruple. Furthermore, we claim that
\begin{equation}
\widehat{f}_{\a_0\bar \a_0, \b_0\bar \b_0, \a_0\bar \b_0}(\widehat{f}_{\a_0\bar \a_0,\Dt_0, \b_0\bar \b_0}(\Theta^+_{\a_0\bar \a_0, \b_0\bar \b_0},\Theta^+_{\Dt_0, \b_0\bar \b_0}),\Theta^+_{\b_0\bar \b_0, \a_0\bar \b_0})=\Theta^-_{\a_0\bar \a_0, \a_0\bar \b_0}.
\label{eq:model-computation-stabilization-2-handle}
\end{equation}
The above equation is verified by noting that the inner triangle map may be interpreted as the cobordism map for a 2-handle attachment along a fiber of $S^1\times S^2$, while the outer triangle map may be interpreted as a 2-handle attached along a 0-framed unknot. The fact that $\Theta^-_{\a_0\bar \a_0, \a_0\bar \b_0}$ is the output may be viewed as a consequence of the above topological description. 

Combining~\eqref{eq:model-computation-stabilization-2-handle} with Proposition~\ref{prop:multi-stabilization-counts}, we conclude that
\begin{equation}
h_{\a \b_0\bar \b \bar \b_0\to \a \a_0 \bar \b\bar \b_0}^{\Dt\Dt_0\to \a  \a_0\bar \a \bar \a_0}(\xs \times \Theta^+_{\b_0\bar \b_0,\Dt})= h_{\a \bar \b \to \a \bar \b}^{\Dt \to \a \bar \a}(\xs)\times \Theta^-_{\a_0\bar \b_0, \a_0 \bar \a_0}+\sum_{\zs\in \bT_{\a\bar \b}\cap \bT_{\a \bar \a}} C_{\zs}\cdot \zs\times \Theta_{\a_0\bar \b_0,\a_0\bar \a_0}^+,
\label{eq:destabilize-quadrilatal-model-comp-stabilization}
\end{equation}
for some $C_{\zs}\in \bF[U]$. Equation ~\eqref{eq:destabilize-quadrilatal-model-comp-stabilization} vanishes once composed with $F_{3}^{\a \a_0,\a \a_0}$, by the small translate theorem for holomorphic quadrilaterals \cite[Proposition 11.5]{HHSZExact}. 

The composition of the bottom-most length 2 map from Equation~\eqref{eq:hyperbox-2-handles} with $F_3^{\a \b_0,\a\b_0}$ vanishes in our present setting by a nearly identical argument.
\end{proof}

We can now prove Proposition~\ref{prop:properties-stabilization}:

\begin{proof}[Proof of Proposition~\ref{prop:properties-stabilization}] Claim~\eqref{stabilization-properties-1} follows immediately from the fact that 1-handle, 2-handle and 3-handle cobordism maps commute (up to chain homotopy) with the transition maps for elementary handleslides.

Claims~\eqref{stabilization-properties-2} and~\eqref{stabilization-properties-3} follow from Lemma~\ref{lem:stabilization-near-basepoint}, since by~\eqref{stabilization-properties-1}, we may perform stabilizations near the basepoint, and then observe that the formulas given therein clearly satisfy the stated relations.
\end{proof}

%% file: section-7-8-handleswaps.tex
	\section{Handleswaps} \label{sec:handleswap}
	
In this section, we prove handleswap invariance.

\subsection{Simple handleswaps}

Of central importance to the approach to naturality in \cite{JTNaturality}  are the following loops in the graph of Heegaard diagrams for a 3-manifold $Y$:

\begin{define} \label{def:simple-handleswap}
	Suppose $Y$ is a 3-manifold. A \emph{simple handleswap loop}  is a triangle of Heegaard diagrams
	\[
	\begin{tikzcd}
	&\cH_1\arrow[rd,"e"] & \\
	\cH_3 \arrow[ur,"g"] && \cH_2 \arrow[ll,"f"]
	\end{tikzcd}
	\]
	which satisfies the following:
	\begin{enumerate}
		\item $\cH_i $ is an embedded Heegaard diagram for $Y$, which has the same Heegaard surface for each $i\in \{1,2,3\}$.
		\item $e$ is an $\a$-equivalence, $f$ is a $\b$-equivalence, and $g$
		is a diffeomorphism.
		\item	The diagrams $\cH_i$ decompose as connected sums of a fixed Heegaard diagram $\cH=(\Sigma,\as_u,\bs_u)$ with genus two diagrams $\cG_i=(\Sigma_0,\as_i,\bs_i)$. Furthermore, in the punctured genus two surface $(\Sigma\# \Sigma_0)\setminus \Sigma$, the above triangle is diffeomorphic to the  triangle in Figure~\ref{fig:handlewap-elem}.
	\end{enumerate}
	\end{define}
	
	In particular, in the above definition, $\as_1$ consists of two closed curves $\a_1$ and $\a_2$, while $\bs_1$ consists of two closed curves $\b_1$ and $\b_2$. The arrow $e$ corresponds to handlesliding $\a_1$ over $\a_2$, while the arrow $f$ corresponds to handlesliding $\b_1$ over $\b_2$.

 \begin{figure}[H]
 	\centering
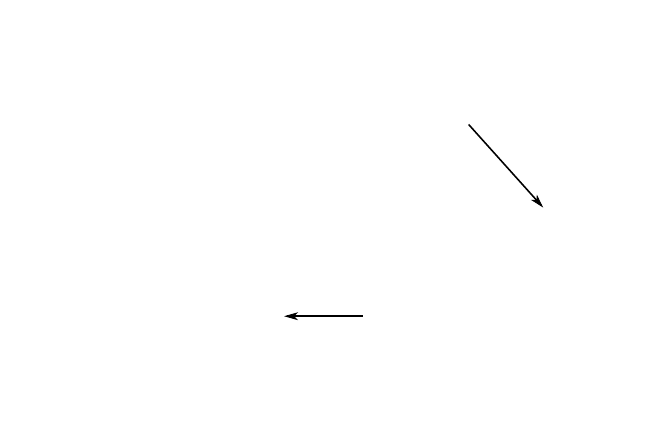
 	\caption{A simple handleswap.}\label{fig:handlewap-elem}
 \end{figure}

Recall that associated to handleslides and diffeomorphisms, we have defined morphisms, $\CFI(e)$, $\CFI(f)$ and $\CFI(g)$, which are well-defined up to $\bF[U,Q]/Q^2$-equivariant homotopy equivalence.

\begin{thm}
\label{thm:handleswap}
$
\CFI(g)\circ \CFI(f)\circ \CFI(e)\simeq \Id_{\CFI(\cH_1)}.$

\end{thm}

The proof of Theorem~\ref{thm:handleswap} occupies the subsequent sections.

\subsection{Reduction to basepoint-adjacent handleswaps}

A simple handleswap need not occur in a region of the Heegaard diagram near a basepoint. However we now show that for the purposes of naturality, it is sufficient to consider handleswaps where the special genus 2 region is adjacent to the basepoint. We note that this fact was also observed by Sungkyung Kang. 

\begin{lem}
Suppose that $F$ is a weak Heegaard invariant (in the sense of \cite{JTNaturality}*{Definition~2.24}) which satisfies all of the conditions of a strong Heegaard invariant (\cite{JTNaturality}*{Definition~3.23}) except possibly handleswap invariance. If $F$ has no monodromy around basepoint-adjacent handleswaps, then it has no monodromy around arbitrary handleswaps.
\end{lem}
\begin{proof} 
Let $p\in \Sigma$ denote the connected sum point. Pick an arc $\lambda$ which connects $p$ to the basepoint $w$. We proceed by induction on the number of intersections of $(\as_u\cup \bs_u) \cap \lambda.$ The base case is the hypothesis. If $\cH$ is arbitrary, there is an equivalence $j$ between $\cH$ and a diagram $\cH'$ (obtained by handlesliding an alpha or beta curve over $p$) such that on $\cH'$, there is one fewer intersection of $(\as_u \cup \bs_u)\cap \lambda$. If $\cH'_1,$ $\cH_2'$ and $\cH_3'$ denote the diagrams obtained by connect summing the standard genus 2 diagrams from Figure~\ref{fig:handlewap-elem}, there is an analogous simple handleswap loop involving $\cH_1'$, $\cH_2'$ and $\cH_3'$. Write $e'$, $f'$ and $g'$ for the analogous equivalences. We have a diagram of equivalences
\[
\begin{tikzcd}[labels=description, column sep=.7cm, row sep=.9cm]
&&&\cH_1
	\ar[dddrrr, "e"]
	\ar[d, "j_1"]
&&&\\
&&&\cH_1'
	\ar[dr, "e'"]
\\
&&\cH_3'
	\ar[ur, "g'"]
	\ar[from=dll, "j_3"]
&& \cH_2'
	\ar[ll, "f'"]
	\ar[drr, "j_2"]
\\[-.3cm]
\cH_3
	\ar[uuurrr, "g"]
&&&&&&
\cH_2
	\ar[llllll, "f"]	
\end{tikzcd}
\]
The three outer rectanges are obtained from distinguished squares by inverting some of the arrows, which are isomorphsims. By hypothesis, $F$ commutes around these squares. Our inductive hypothesis is that  $F$ has no monodromy around the central triangle. This clearly implies $F$ has no monodromy around the outer triangle, completing the proof.
\end{proof}

\subsection{Proof of Theorem~~\ref{thm:handleswap}}

Let us introduce some notation.  For the curves in the handleswap region, we write $\as_0$, $\as_0^H$, $\bs_0$ and $\bs_0^H$. We write $\as^H$ for $\as_u\cup \as_0^H$, and similarly for $\bs_0^H$. Hence, for the three diagrams appearing in the handleswap loop, we write
\[
\cH_1=(\Sigma\# \Sigma_0,\as,\bs), \quad \cH_2=(\Sigma\# \Sigma_0, \as^H, \bs),\quad \text{and} \quad \cH_3=(\Sigma\# \Sigma_0, \as^H, \bs^H).
\]

Similarly, let $\Ds_u$ be a collection of doubling curves on $\Sigma\# \bar \Sigma$. Let $\Ds_0$ be choice of doubling curves on $\Sigma_0\# \bar \Sigma_0$, as shown in Figure~\ref{fig:involutive-handleswap}. Write $\Ds$ for $\Ds_u\cup \Ds_0$. 

We write $\frD$ for $\cH$, decorated with $\Ds_u$. We pick a collection of four doubling curves $\Ds_0$ on $\Sigma_0\# \bar \Sigma_0$. We decorate $\cH_1,$ $\cH_2$, $\cH_3$ with $\Ds$ to form the doubling enhanced diagrams $\frD_1$, $\frD_2$, and $\frD_3$, respectively. 

\begin{figure}[ht]
\centering
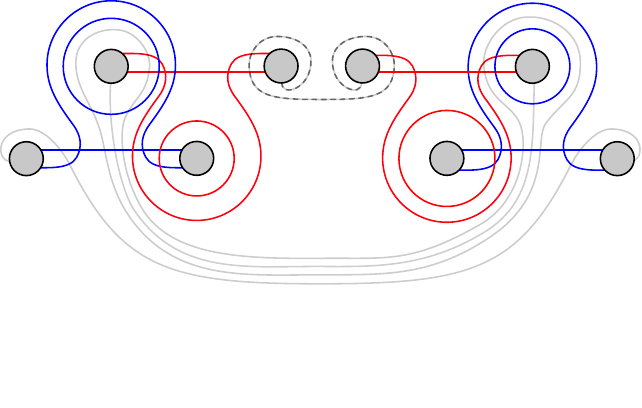
\caption{The double of the handleswap region. The left side is $\Sigma$, and the right is $\bar \Sigma$.}
\label{fig:involutive-handleswap}
\end{figure}

We note that for $i\in \{1,2,3\}$, there are canonical isomorphisms 
\begin{equation}
\CFI(\frD)\iso \CFI(\frD_i), \label{eq:canonical-iso}
\end{equation}
 given by $\xs\mapsto \ve{x}\times  \ve{c}_i,$
extended $\bF[U,Q]/Q^2$-equivariantly. Here, we write $\ve{c}_i$ for the unique intersection points
\[
\ve{c}_1\in \bT_{\a_0}\cap \bT_{\b_0}, \quad  \ve{c}_2\in \bT_{\a_0^H}\cap \bT_{\b_0}, \quad \text{and} \quad \ve{c}_3\in \bT_{\a_0^H}\cap \bT_{\b_0^H}.
\]

In fact, we have the following:
\begin{lem}
\label{lem:admissibility-handleswap-1}
The genus 2 diagram shown in Figure~\ref{fig:involutive-handleswap} is weakly admissible. Furthermore, each of the following subdiagrams has vanishing differential.
\begin{enumerate}
\item $(\Sigma_0\# \bar \Sigma_0,\bs_0\bar \bs_0, \Ds_0,w)$
\item $(\Sigma_0\# \bar \Sigma_0,\bs^H_0\bar \bs^H_0, \Ds_0,w)$.
\item $(\Sigma_0\# \bar \Sigma_0,\Ds_0, \as_0 \bar \as_0,w)$.
\item $(\Sigma_0\# \bar \Sigma_0,\Ds_0, \as^H_0\bar \as^H_0,w)$.
\item $(\Sigma_0\# \bar \Sigma_0,\as_0 \bar \bs_0, \Ds_0,w)$.
\item $(\Sigma_0\# \bar \Sigma_0,\as^H_0 \bar \bs^H_0, \Ds_0,w)$.
\end{enumerate}
\end{lem}
\begin{proof} Admissibility is a straightforward computation which we leave to the reader. To check that each of the stated complexes has vanishing differential, it suffices to check that each has the same number of generators as its homology. The first four subdiagrams represent a complex for $\#^2 S^1\times S^2$ and also have exactly four generators. Similarly the final two subdiagrams represent $S^3$, and each have a unique intersection point.
\end{proof}

By definition,
\[
\CFI(e)=\Psi_{\frD_1\to \frD_2}\quad \text{and} \quad \CFI(f)=\Psi_{\frD_2\to \frD_3},
\]
where $\Psi_{\frD_i\to \frD_{i+1}}$ is obtained by compressing a hyperbox as in Figure~\ref{def:transition-map-elementary-handleslide}. Computing the composition $\Psi_{\frD_2\to \frD_3}\circ \Psi_{\frD_1\to \frD_2}$ is somewhat inefficient, however, since the involutive transition maps naturally allow for changing both the alpha and the beta curves simultaneously. Instead, we will consider the map $\Psi_{\frD_1\to \frD_3}$ obtained by by compressing the hyperbox in Figure~\ref{def:transition-map-elementary-handleslide}. (See Equation~\eqref{eq:Psi_1->3}, below). By Theorem~\ref{thm:naturality-fixed-Sigma},
\begin{equation}
\CFI(f)\circ \CFI(e)=\Psi_{\frD_2\to \frD_3}\circ \Psi_{\frD_1\to \frD_2}\simeq \Psi_{\frD_1\to \frD_3}.
\label{eq:f-e-composed}
\end{equation}

To prove Theorem~\ref{thm:handleswap}, we also need to understand the diffeomorphism map $\CFI(g)$. We note that the diffeomorphism $g$ does not fix $\Ds_0$. Hence,  by definition, this map is the composition
\begin{equation}
\CFI(g):=\Psi_{g(\frD_3)\to \frD_1}\circ T_g,
\label{eq:CFI-g-def}
\end{equation}
which we explain presently. The map $T_g$ is the tautological map which sends an intersection point $\xs\in \bT_{\a \a_0^H}\cap \bT_{\b \b_0^H}$ to its image in $\bT_{\a \a_0}\cap \bT_{\b \b_0}$ under $g$, extended $\bF[U,Q]/Q^2$-equivariantly. Also, $g(\frD_3)$ is obtained by enhancing the Heegaard diagram $\cH_1$ with the doubling curves $\Ds_u\cup \Ds_0'$, where $\Ds_0'=g(\Ds_0)$. The map $\Psi_{g(\frD_3)\to \frD_1}$ is built by compressing a hyperbox similar to the one in Figure~\ref{def:transition-map-elementary-handleslide} which realizes the change of doubling curves.

 With respect to the canonical isomorphisms of~\eqref{eq:canonical-iso}, the map $T_g$ is clearly equal to the identity map.  Theorem~\ref{thm:handleswap} follows from  ~\eqref{eq:f-e-composed} and~\eqref{eq:CFI-g-def} and the following result.
 \begin{prop}\label{prop:handleswap-factors=identity}
  With respect to the canonical isomorphisms in~\eqref{eq:canonical-iso}, the maps $\Psi_{\frD_1\to \frD_3}$ and $\Psi_{g(\frD_3)\to \frD_1}$ are chain homotopic to the identity map.
 \end{prop}

\begin{proof}
 We focus first on the computation of $\Psi_{\frD_1\to \frD_3}$. By definition, this map is obtained by compressing the following hyperbox:
\begin{equation}
\begin{tikzcd}[labels=description, row sep=1cm]
\CF(\as,\bs)
	\ar[r]
	\ar[d, "F_{1}^{\bar{\b},\bar{\b}}"]
& \CF(\as^H,\bs)
	\ar[r]
	\ar[d, "F_{1}^{\bar{\b}^H,\bar{\b}}"]
& \CF(\as^H,\bs^H)
	\ar[d,"F_{1}^{\bar{\b}^H,\bar{\b}^H}"]
\\
\CF(\as \bar{\bs}, \bs \bar{\bs})
	\ar[r]
	\ar[d]
	\ar[dr,dashed]
& \CF(\as^H\bar{\bs}^H, \bs \bar{\bs})
	\ar[r]
	\ar[d]
	\ar[dr,dashed]
& \CF(\as^H \bar{\bs}^H, \bs^H \bar{\bs}^H)
	\ar[d]
\\
\CF(\as \bar \bs,\Ds)
	\ar[r]
	\ar[d]
	\ar[dr,dashed]
& \CF(\as^H\bar \bs^H,\Ds)
	\ar[r]
	\ar[d]
	\ar[dr,dashed]
& \CF(\as^H\bar \bs^H,\Ds)
	\ar[d]
\\
\CF(\as\bar \bs, \as \bar \as)
	\ar[r]
	\ar[d, "F_3^{\a,\a }"]
&
\CF(\as^H \bar \bs^H, \as \bar \as)
	\ar[r]
	\ar[d, "F_3^{\a^H,\a }"]
&
 \CF(\as^H \bar{\bs}^H, \as^H \bar{\as}^H)
	\ar[d, "F_3^{\a^H,\a^H}"]
\\
\CF(\bar{\bs}, \bar{\as})
	\ar[r]
	\ar[drr,dashed]
	\ar[d]
& \CF(\bar{\bs}^H,\bar{\as})
	\ar[r]
& \CF(\bar{\bs}^H,\bar{\as}^H)
	\ar[d]
\\
\CF(\bar \bs,\bar \as)
	\ar[r]
&
\CF(\bar \bs, \bar \as^H)
	\ar[r]
&
\CF(\bar \bs^H,\bar \as^H)
\end{tikzcd}
\label{eq:Psi_1->3}
\end{equation}
We note that in the above diagram, each polygon counting map occurs on a connected sum. One summand is $\Sigma$ or $\Sigma\# \bar \Sigma$ while the other summand is the genus 2 subsurface $\Sigma_0$, or its double. We will argue by viewing the diagram as a multi-stabilization and applying Lemma~\ref{lem:admissibility-handleswap-1} and Proposition~\ref{prop:multi-stabilization-counts}.

 There is a subtlety in that the almost complex structures used to  construct the transition map Equation~\eqref{eq:Psi_1->3} are not naturally adapted to the genus four stabilization region where the handleswap occurs. Instead, we must first argue that we can change the above composition to one which is naturally adapted to the handleswap. To do so, we consider the closed curves on $\Sigma\# \bar \Sigma$ shown in Figure~\ref{fig:12}. We label these as $c_+,$ $c_-$, $c_{\delta}$ and $c_0$. These are as follows:
 \begin{enumerate}
 \item $c_0$ is a meridian of the connected sum tube used in the doubling operation.
 \item $c_+$ is a closed curve on $\Sigma$ which bounds the handleswap region and the connected sum point.
 \item $c_-$ is a symmetric closed curve on $\bar \Sigma$, and bounds the handleswap region as well as a the connected sum point.
 \item $c_\delta$ is a closed curve on $\Sigma\# \bar \Sigma$ which bounds the union of both handleswap regions. Also, $c_{\delta}$ is disjoint from all $\Ds$ curves.
 \end{enumerate}

\begin{figure}[ht]
\centering
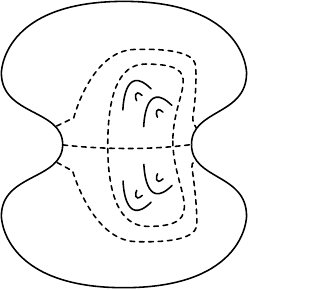
\caption{Curves along which we degenerate the almost complex structure.}
\label{fig:12}
\end{figure}

Write $J_{c_0}$ for an almost complex structure or family of almost complex structures which is nodal along $c_0$, and use similar notation to define almost complex structures degenerated on another curve or collections of pairwise disjoint curves on $\Sigma\# \bar \Sigma$.

We introduce some shorthand. Write $\CF(L_1),\dots, \CF(L_6)$ for the six horizontal levels of Equation~\eqref{eq:Psi_1->3}. (These are themselves 1-dimensional hyperboxes of chain complexes). By using change of almost complex structure hypercubes, similar to \cite{HHSZExact}*{Section~13}, we may assume that the complexes and maps in Equation~\eqref{eq:Psi_1->3} are computed with respect to almost complex structures which are \emph{maximally pinched} along $c_0$, $c_+$ and $c_-$, in the sense of \cite{HHSZExact}*{Section~8.3}. This gives the map $\Psi_{\frD_1\to \frD_3}$, as defined in the earlier sections of the present paper. On the other hand, using similar techniques, we may use change of almost complex structure hypercubes to relate this model to one which is maximally pinched along $c_+$, $c_-$ and $c_\delta$. This gives a hyperbox which we schematically indicate by the following diagram:
\begin{equation}
\begin{tikzcd}[labels=description, row sep=.8cm, column sep=4cm]
\CF_{J_{c_+}}(L_1)
	\ar[r, "\id"]
	\ar[d, "F_{1}"]
& \CF_{J_{c_+}}(L_1)
	\ar[d, "F_{1}"]
\\
\CF_{J_{c_+,c_-,c_0}}(L_2)
	\ar[r,"\Psi_{J_{c_+,c_-,c_0}\to J_{c_+,c_-,c_\delta}}"]
	\ar[d]
	\ar[dr,dashed]
& \CF_{J_{c_+,c_-,c_\delta}}(L_2)
	\ar[d]
\\
\CF_{J}(L_3)
	\ar[r, "\Psi_{J\to J_{c_\delta}}"]
	\ar[d]
	\ar[dr,dashed]
& \CF_{J_{c_\delta}}(L_3)
	\ar[d]
\\
\CF_{J_{c_+,c_-,c_0}}(L_4)
	\ar[r, "\Psi_{J_{c_+,c_-,c_0}\to J_{c_+,c_-,c_\delta}}"]
	\ar[d, "F_3"]
&
\CF_{J_{c_+,c_-,c_\delta}}(L_4)
	\ar[d, "F_3"]
\\
\CF_{J_{c_-}}(L_5)
	\ar[r,"\id"]
	\ar[dr,dashed]
	\ar[d]
& \CF_{J_{c_-}}(L_5)
	\ar[d]
\\
\CF_{J_{c_-}}(L_6)
	\ar[r, "\id"]
&
\CF_{J_{c_-}}(L_6)
\end{tikzcd}
\label{eq:Psi_1->3-cx-structures}
\end{equation}

The hyperbox in Equation~\eqref{eq:Psi_1->3-cx-structures} realizes a chain homotopy between the model of $\Psi_{\frD_1\to \frD_3}$, computed with respect to almost complex structures which are maximally pinched along $c_+,$ $c_-$ and $c_0$, and a model of $\Psi_{\frD_1\to \frD_3}$ which is computed using almost complex structures which are maximally pinched along $c_+,$ $c_-$ and $c_\delta$.

In the model of $\Psi_{\frD_1\to \frD_3}$ appearing along the right hand side of Equation~\eqref{eq:Psi_1->3-cx-structures}, the holomorphic polygon counts all occur on a multi-stabilization of the diagram on $\Sigma$ or $\Sigma\# \bar \Sigma$. The stabilization results for triangles imply that the two length 1 maps along the top of ~\eqref{eq:Psi_1->3} compose to give a map which is intertwined with the identity under the canonical isomorphisms $\CF(\cH)\iso \CF(\cH_i)$ (similar to~\eqref{eq:canonical-iso}). (This is the content of \cite{JTNaturality}*{Propositions~9.31, 9.32}). 

It remains to show that none of the diagonal maps make non-trivial contribution to the compression of~\eqref{eq:Psi_1->3}. To this end, we make the following claims, which together imply that the compression has trivial diagonal map:
\begin{enumerate}
\item\label{compression-claim-1} The compression of the top two levels of~\eqref{eq:Psi_1->3} has trivial diagonal map.
\item\label{compression-claim-2} The compression of the third and fourth levels of~\eqref{eq:Psi_1->3} has trivial diagonal map.
\item\label{compression-claim-3} The bottom-most level of ~\eqref{eq:Psi_1->3} has trivial diagonal map.
\end{enumerate}
Consider first the claim~\eqref{compression-claim-3}. The diagonal map is a quadrilateral counting map on the connected sum $(\Sigma,\as_u,\as_u,\bs_u,\bs_u)\#(\Sigma_0,\as_0^H,\as_0,\bs_0,\bs_0^H)$. Applying the stabilization result from Proposition~\ref{prop:multi-stabilization-counts} to the setting of holomorphic quadrilaterals, we may destabilize the map to the diagram $(\Sigma,\as_u,\as_u,\bs_u,\bs_u)$. By the small translate theorem for quadrilaterals \cite[Proposition 11.5]{HHSZExact}, this map vanishes.

Next, consider~\eqref{compression-claim-1}. For this claim, we observe that the stabilization results for holomorphic triangles imply that the length 1 holomorphic triangle maps preserve the subspaces spanned by the top-degree generator (and furthermore, act tensorially on these subspaces). Furthermore, the same argument as for claim~\eqref{compression-claim-3} implies that the holomorphic quadrilateral maps vanish on elements which have a factor which is the top degree generator. Since the top degree generator is the output of the 1-handle maps on the top row, this implies that the compression has trivial diagonal map. Claim~\eqref{compression-claim-2} follows from an analogous argument. This completes the proof of the claim about $\Psi_{\frD_1\to \frD_3}$.

We now consider the claim about $\Psi_{g(\frD_3)\to \frD_1}$. This map may be computed by compressing the following diagram
\begin{equation}
\begin{tikzcd}[labels=description, row sep=1cm]
\CF(\as,\bs)
	\ar[r]
	\ar[d, "F_{1}^{\bar{\b},\bar{\b}}"]
& \CF(\as,\bs)
	\ar[r]
	\ar[d, "F_{1}^{\bar{\b},\bar{\b}}"]
& \CF(\as,\bs)
	\ar[d,"F_{1}^{\bar{\b},\bar{\b}}"]
\\
\CF(\as \bar{\bs}, \bs \bar{\bs})
	\ar[r]
	\ar[d]
	\ar[dr,dashed]
& \CF(\as\bar{\bs}, \bs \bar{\bs})
	\ar[r]
	\ar[d]
	\ar[dr,dashed]
& \CF(\as \bar{\bs}, \bs \bar{\bs})
	\ar[d]
\\
\CF(\as \bar \bs,g(\Ds))
	\ar[r]
	\ar[d]
	\ar[dr,dashed]
& \CF(\as\bar \bs,g(\Ds))
	\ar[r]
	\ar[d]
	\ar[dr,dashed]
& \CF(\as\bar \bs,\Ds)
	\ar[d]
\\
\CF(\as\bar \bs, \as \bar \as)
	\ar[r]
	\ar[d, "F_3^{\a,\a }"]
&
\CF(\as \bar \bs, \as \bar \as)
	\ar[r]
	\ar[d, "F_3^{\a,\a }"]
&
 \CF(\as \bar{\bs}, \as \bar{\as})
	\ar[d, "F_3^{\a,\a}"]
\\
\CF(\bar{\bs}, \bar{\as})
	\ar[r]
	\ar[drr,dashed]
	\ar[d]
& \CF(\bar{\bs},\bar{\as})
	\ar[r]
& \CF(\bar{\bs},\bar{\as})
	\ar[d]
\\
\CF(\bar \bs,\bar \as)
	\ar[r]
&
\CF(\bar \bs, \bar \as)
	\ar[r]
&
\CF(\bar \bs,\bar \as)
\end{tikzcd}
\label{eq:g(D)->D}
\end{equation}
The description in~\eqref{eq:g(D)->D} is not the most convenient for our purposes. Instead, we describe a sequence of doubling curves $g(\Ds_0)=\Ds_0^0,\dots, \Ds_0^6=\Ds_0$ on $\Sigma_0\# \bar \Sigma_0$, such that each $\Ds_0^i$ is obtained from $\Ds_0^{i-1}$ by a handleslide. Furthermore, we have the following properties for each $i\in \{0,\dots 6\}$:
\begin{enumerate}
\item The diagram
\[
(\Sigma_0\# \bar \Sigma_0,\as_0 \bar \as_0, \bs_0 \bar \bs_0, \Ds_0^i,\Ds_0^{i+1},w)
\]
is weakly admissible.
\item Each of the complexes $\CF(\as_0 \bar \bs_0, \Ds_0^i)$, $\CF(\as \bar \as, \Ds_0^i)$, $\CF(\bs \bar \bs, \Ds_0^i)$, and $\CF(\Ds_0^i,\Ds_0^{i+1})$ has vanishing differential.
\end{enumerate}
The curves $\Ds_0^0,\dots \Ds_0^6$ are shown in Figure~\ref{fig:handleswaps-Delta}. We leave it to the reader to verify the above claims. (Compare Lemma~\ref{lem:admissibility-handleswap-1}).

We compute the map $\Psi_{g(\frD_3)\to \frD_1}$ as a composition of 6 maps, corresponding to the transition maps for changing $\Ds_0^i$ to $\Ds_0^{i+1}$, each of which is computed by compressing a hyperbox similar to~\eqref{eq:g(D)->D}. The same argument as for $\Psi_{\frD_1\to \frD_3}$ implies that each of these maps is intertwined with the identity map by the canonical isomorphisms in ~\eqref{eq:canonical-iso}. The proof is complete.
\end{proof}

\begin{figure}[p]
\centering
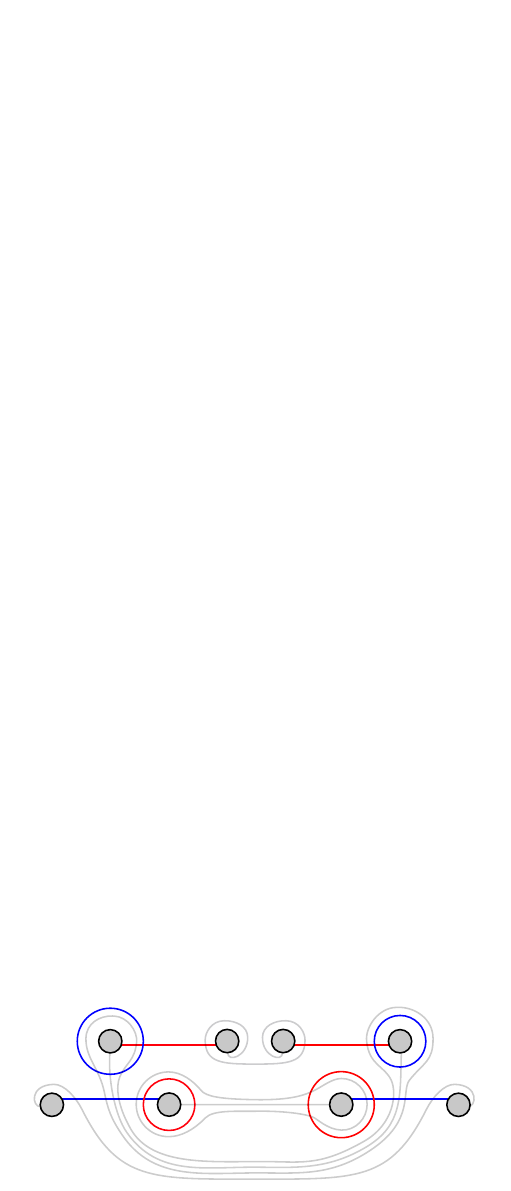
\caption{The curves $g(\Ds_0)=\Ds_0^0,\dots, \Ds_0^6=\Ds_0$, which relate $g(\Ds_0)$ and $\Ds_0$ by a sequence of handleslides. Each arrow indicates the subsequent handleslide.}
\label{fig:handleswaps-Delta}
\end{figure}

%% file: handleswap-elem.pdf_tex
\begingroup%
  \makeatletter%
  \providecommand\color[2][]{%
    \errmessage{(Inkscape) Color is used for the text in Inkscape, but the package 'color.sty' is not loaded}%
    \renewcommand\color[2][]{}%
  }%
  \providecommand\transparent[1]{%
    \errmessage{(Inkscape) Transparency is used (non-zero) for the text in Inkscape, but the package 'transparent.sty' is not loaded}%
    \renewcommand\transparent[1]{}%
  }%
  \providecommand\rotatebox[2]{#2}%
  \newcommand*\fsize{\dimexpr\f@size pt\relax}%
  \newcommand*\lineheight[1]{\fontsize{\fsize}{#1\fsize}\selectfont}%
  \ifx\svgwidth\undefined%
    \setlength{\unitlength}{312.54025222bp}%
    \ifx\svgscale\undefined%
      \relax%
    \else%
      \setlength{\unitlength}{\unitlength * \real{\svgscale}}%
    \fi%
  \else%
    \setlength{\unitlength}{\svgwidth}%
  \fi%
  \global\let\svgwidth\undefined%
  \global\let\svgscale\undefined%
  \makeatother%
  \begin{picture}(1,0.64587086)%
    \lineheight{1}%
    \setlength\tabcolsep{0pt}%
    \put(0,0){\includegraphics[width=\unitlength,page=1]{handleswap-elem.pdf}}%
    \put(0.78439196,0.40507839){\makebox(0,0)[lt]{\lineheight{1.25}\smash{\begin{tabular}[t]{l}$e$\end{tabular}}}}%
    \put(0.4875111,0.17343721){\makebox(0,0)[lt]{\lineheight{1.25}\smash{\begin{tabular}[t]{l}$f$\end{tabular}}}}%
    \put(0.21215317,0.40541472){\makebox(0,0)[rt]{\lineheight{1.25}\smash{\begin{tabular}[t]{r}$g$\end{tabular}}}}%
    \put(0,0){\includegraphics[width=\unitlength,page=2]{handleswap-elem.pdf}}%
    \put(0.03146421,0.08060178){\makebox(0,0)[lt]{\lineheight{1.25}\smash{\begin{tabular}[t]{l}$\mathrm{E}$\end{tabular}}}}%
    \put(0.26695272,0.08060142){\makebox(0,0)[t]{\lineheight{1.25}\smash{\begin{tabular}[t]{c}\reflectbox{$\mathrm{E}$}\end{tabular}}}}%
    \put(0.15290136,0.20365315){\makebox(0,0)[t]{\lineheight{1.25}\smash{\begin{tabular}[t]{c}$\mathrm{B}$\end{tabular}}}}%
    \put(0.37944643,0.20365288){\makebox(0,0)[t]{\lineheight{1.25}\smash{\begin{tabular}[t]{c}\reflectbox{$\mathrm{B}$}\end{tabular}}}}%
    \put(0,0){\includegraphics[width=\unitlength,page=3]{handleswap-elem.pdf}}%
    \put(0.62042489,0.08060186){\makebox(0,0)[t]{\lineheight{1.25}\smash{\begin{tabular}[t]{c}$\mathrm{E}$\end{tabular}}}}%
    \put(0.84780278,0.08060159){\makebox(0,0)[t]{\lineheight{1.25}\smash{\begin{tabular}[t]{c}\reflectbox{$\mathrm{E}$}\end{tabular}}}}%
    \put(0.73377715,0.20365328){\makebox(0,0)[t]{\lineheight{1.25}\smash{\begin{tabular}[t]{c}$\mathrm{B}$\end{tabular}}}}%
    \put(0.96032197,0.20365298){\makebox(0,0)[t]{\lineheight{1.25}\smash{\begin{tabular}[t]{c}\reflectbox{$\mathrm{B}$}\end{tabular}}}}%
    \put(0,0){\includegraphics[width=\unitlength,page=4]{handleswap-elem.pdf}}%
    \put(0.32190218,0.41467539){\makebox(0,0)[lt]{\lineheight{1.25}\smash{\begin{tabular}[t]{l}$\mathrm{E}$\end{tabular}}}}%
    \put(0.55739068,0.41467498){\makebox(0,0)[t]{\lineheight{1.25}\smash{\begin{tabular}[t]{c}\reflectbox{$\mathrm{E}$}\end{tabular}}}}%
    \put(0.44333924,0.53772681){\makebox(0,0)[t]{\lineheight{1.25}\smash{\begin{tabular}[t]{c}$\mathrm{B}$\end{tabular}}}}%
    \put(0.66988445,0.53772649){\makebox(0,0)[t]{\lineheight{1.25}\smash{\begin{tabular}[t]{c}\reflectbox{$\mathrm{B}$}\end{tabular}}}}%
    \put(0,0){\includegraphics[width=\unitlength,page=5]{handleswap-elem.pdf}}%
    \put(0.58323388,0.55328233){\color[rgb]{1,0,0}\makebox(0,0)[lt]{\lineheight{1.25}\smash{\begin{tabular}[t]{l}$\a_1$\end{tabular}}}}%
    \put(0.62084367,0.45958707){\color[rgb]{1,0,0}\makebox(0,0)[lt]{\lineheight{1.25}\smash{\begin{tabular}[t]{l}$\a_2$\end{tabular}}}}%
    \put(0.49571834,0.60204842){\color[rgb]{0,0,1}\makebox(0,0)[lt]{\lineheight{1.25}\smash{\begin{tabular}[t]{l}$\b_2$\end{tabular}}}}%
    \put(0.3940452,0.44932983){\color[rgb]{0,0,1}\makebox(0,0)[lt]{\lineheight{1.25}\smash{\begin{tabular}[t]{l}$\b_1$\end{tabular}}}}%
    \put(0,0){\includegraphics[width=\unitlength,page=6]{handleswap-elem.pdf}}%
  \end{picture}%
\endgroup%

%% file: handleswap-colortest.pdf_tex
\begingroup%
  \makeatletter%
  \providecommand\color[2][]{%
    \errmessage{(Inkscape) Color is used for the text in Inkscape, but the package 'color.sty' is not loaded}%
    \renewcommand\color[2][]{}%
  }%
  \providecommand\transparent[1]{%
    \errmessage{(Inkscape) Transparency is used (non-zero) for the text in Inkscape, but the package 'transparent.sty' is not loaded}%
    \renewcommand\transparent[1]{}%
  }%
  \providecommand\rotatebox[2]{#2}%
  \newcommand*\fsize{\dimexpr\f@size pt\relax}%
  \newcommand*\lineheight[1]{\fontsize{\fsize}{#1\fsize}\selectfont}%
  \ifx\svgwidth\undefined%
    \setlength{\unitlength}{307.70460104bp}%
    \ifx\svgscale\undefined%
      \relax%
    \else%
      \setlength{\unitlength}{\unitlength * \real{\svgscale}}%
    \fi%
  \else%
    \setlength{\unitlength}{\svgwidth}%
  \fi%
  \global\let\svgwidth\undefined%
  \global\let\svgscale\undefined%
  \makeatother%
  \begin{picture}(1,0.63624254)%
    \lineheight{1}%
    \setlength\tabcolsep{0pt}%
    \put(0,0){\includegraphics[width=\unitlength,page=1]{handleswap-colortest.pdf}}%
    \put(0.03144345,0.37633449){\makebox(0,0)[lt]{\lineheight{1.25}\smash{\begin{tabular}[t]{l}$\mathrm{E}$\end{tabular}}}}%
    \put(0.30693224,0.37633407){\makebox(0,0)[t]{\lineheight{1.25}\smash{\begin{tabular}[t]{c}\reflectbox{$\mathrm{E}$}\end{tabular}}}}%
    \put(0.17350167,0.52518973){\makebox(0,0)[t]{\lineheight{1.25}\smash{\begin{tabular}[t]{c}$\mathrm{B}$\end{tabular}}}}%
    \put(0.43853431,0.52518942){\makebox(0,0)[t]{\lineheight{1.25}\smash{\begin{tabular}[t]{c}\reflectbox{$\mathrm{B}$}\end{tabular}}}}%
    \put(0.56552004,0.52518979){\makebox(0,0)[t]{\lineheight{1.25}\smash{\begin{tabular}[t]{c}$\mathrm{C}$\end{tabular}}}}%
    \put(0.83056775,0.52518989){\makebox(0,0)[t]{\lineheight{1.25}\smash{\begin{tabular}[t]{c}\reflectbox{$\mathrm{C}$}\end{tabular}}}}%
    \put(0.69714033,0.3763341){\makebox(0,0)[t]{\lineheight{1.25}\smash{\begin{tabular}[t]{c}$\mathrm{D}$\end{tabular}}}}%
    \put(0.96301149,0.37633381){\makebox(0,0)[t]{\lineheight{1.25}\smash{\begin{tabular}[t]{c}\reflectbox{$\mathrm{D}$}\end{tabular}}}}%
    \put(0,0){\includegraphics[width=\unitlength,page=2]{handleswap-colortest.pdf}}%
    \put(0.752069,0.01976699){\makebox(0,0)[rt]{\lineheight{1.25}\smash{\begin{tabular}[t]{r}$\Ds_0$\end{tabular}}}}%
    \put(0.75167138,0.11941274){\makebox(0,0)[rt]{\lineheight{1.25}\smash{\begin{tabular}[t]{r}$\as_0,$ $\as_0^H$, $\bar \as_0$, and $\bar \as_0^H$\end{tabular}}}}%
    \put(0.75167136,0.07178608){\makebox(0,0)[rt]{\lineheight{1.25}\smash{\begin{tabular}[t]{r}$\bs_0,$ $\bs_0^H$, $\bar \bs_0$, and $\bar \bs_0^H$\end{tabular}}}}%
    \put(0,0){\includegraphics[width=\unitlength,page=3]{handleswap-colortest.pdf}}%
    \put(0.09140026,0.25939174){\makebox(0,0)[lt]{\lineheight{1.25}\smash{\begin{tabular}[t]{l}$w$\end{tabular}}}}%
  \end{picture}%
\endgroup%

%% file: fig12.pdf_tex
\begingroup%
  \makeatletter%
  \providecommand\color[2][]{%
    \errmessage{(Inkscape) Color is used for the text in Inkscape, but the package 'color.sty' is not loaded}%
    \renewcommand\color[2][]{}%
  }%
  \providecommand\transparent[1]{%
    \errmessage{(Inkscape) Transparency is used (non-zero) for the text in Inkscape, but the package 'transparent.sty' is not loaded}%
    \renewcommand\transparent[1]{}%
  }%
  \providecommand\rotatebox[2]{#2}%
  \newcommand*\fsize{\dimexpr\f@size pt\relax}%
  \newcommand*\lineheight[1]{\fontsize{\fsize}{#1\fsize}\selectfont}%
  \ifx\svgwidth\undefined%
    \setlength{\unitlength}{157.16056521bp}%
    \ifx\svgscale\undefined%
      \relax%
    \else%
      \setlength{\unitlength}{\unitlength * \real{\svgscale}}%
    \fi%
  \else%
    \setlength{\unitlength}{\svgwidth}%
  \fi%
  \global\let\svgwidth\undefined%
  \global\let\svgscale\undefined%
  \makeatother%
  \begin{picture}(1,0.88286572)%
    \lineheight{1}%
    \setlength\tabcolsep{0pt}%
    \put(0,0){\includegraphics[width=\unitlength,page=1]{fig12.pdf}}%
    \put(0.39598671,0.44998143){\makebox(0,0)[lt]{\lineheight{1.25}\smash{\begin{tabular}[t]{l}$c_0$\end{tabular}}}}%
    \put(0.60949483,0.64031009){\makebox(0,0)[lt]{\lineheight{1.25}\smash{\begin{tabular}[t]{l}$c_+$\end{tabular}}}}%
    \put(0.61074652,0.23787373){\makebox(0,0)[lt]{\lineheight{1.25}\smash{\begin{tabular}[t]{l}$c_-$\end{tabular}}}}%
    \put(0.32940176,0.53815504){\makebox(0,0)[rt]{\lineheight{1.25}\smash{\begin{tabular}[t]{r}$c_\delta$\end{tabular}}}}%
    \put(0,0){\includegraphics[width=\unitlength,page=2]{fig12.pdf}}%
    \put(0.7930438,0.64567999){\makebox(0,0)[lt]{\lineheight{1.25}\smash{\begin{tabular}[t]{l}$\Sigma$\end{tabular}}}}%
    \put(0.77819703,0.26226407){\makebox(0,0)[lt]{\lineheight{1.25}\smash{\begin{tabular}[t]{l}$\bar\Sigma$\end{tabular}}}}%
    \put(0,0){\includegraphics[width=\unitlength,page=3]{fig12.pdf}}%
  \end{picture}%
\endgroup%

%% file: handleswaps_delta-2.pdf_tex
\begingroup%
  \makeatletter%
  \providecommand\color[2][]{%
    \errmessage{(Inkscape) Color is used for the text in Inkscape, but the package 'color.sty' is not loaded}%
    \renewcommand\color[2][]{}%
  }%
  \providecommand\transparent[1]{%
    \errmessage{(Inkscape) Transparency is used (non-zero) for the text in Inkscape, but the package 'transparent.sty' is not loaded}%
    \renewcommand\transparent[1]{}%
  }%
  \providecommand\rotatebox[2]{#2}%
  \newcommand*\fsize{\dimexpr\f@size pt\relax}%
  \newcommand*\lineheight[1]{\fontsize{\fsize}{#1\fsize}\selectfont}%
  \ifx\svgwidth\undefined%
    \setlength{\unitlength}{245.19125481bp}%
    \ifx\svgscale\undefined%
      \relax%
    \else%
      \setlength{\unitlength}{\unitlength * \real{\svgscale}}%
    \fi%
  \else%
    \setlength{\unitlength}{\svgwidth}%
  \fi%
  \global\let\svgwidth\undefined%
  \global\let\svgscale\undefined%
  \makeatother%
  \begin{picture}(1,2.33147234)%
    \lineheight{1}%
    \setlength\tabcolsep{0pt}%
    \put(0,0){\includegraphics[width=\unitlength,page=1]{handleswaps_delta-2.pdf}}%
    \put(0.11339931,0.07797559){\makebox(0,0)[lt]{\lineheight{1.25}\smash{\begin{tabular}[t]{l}$w$\end{tabular}}}}%
    \put(0,0){\includegraphics[width=\unitlength,page=2]{handleswaps_delta-2.pdf}}%
    \put(0.06146759,2.06481315){\makebox(0,0)[lt]{\lineheight{1.25}\smash{\begin{tabular}[t]{l}$\Ss{\mathrm{E}}$\end{tabular}}}}%
    \put(0.29933017,2.06481267){\makebox(0,0)[t]{\lineheight{1.25}\smash{\begin{tabular}[t]{c}\reflectbox{$\Ss{\mathrm{E}}$}\end{tabular}}}}%
    \put(0.18412353,2.1933806){\makebox(0,0)[t]{\lineheight{1.25}\smash{\begin{tabular}[t]{c}$\Ss{\mathrm{B}}$\end{tabular}}}}%
    \put(0.43821232,2.19338034){\makebox(0,0)[t]{\lineheight{1.25}\smash{\begin{tabular}[t]{c}\reflectbox{$\Ss{\mathrm{B}}$}\end{tabular}}}}%
    \put(0.54785423,2.19338075){\makebox(0,0)[t]{\lineheight{1.25}\smash{\begin{tabular}[t]{c}$\Ss{\mathrm{C}}$\end{tabular}}}}%
    \put(0.81458285,2.19338082){\makebox(0,0)[t]{\lineheight{1.25}\smash{\begin{tabular}[t]{c}\reflectbox{$\Ss{\mathrm{C}}$}\end{tabular}}}}%
    \put(0.69937916,2.06481275){\makebox(0,0)[t]{\lineheight{1.25}\smash{\begin{tabular}[t]{c}$\Ss{\mathrm{D}}$\end{tabular}}}}%
    \put(0.92893722,2.06481265){\makebox(0,0)[t]{\lineheight{1.25}\smash{\begin{tabular}[t]{c}\reflectbox{$\Ss{\mathrm{D}}$}\end{tabular}}}}%
    \put(0.09488253,2.04110758){\makebox(0,0)[lt]{\lineheight{1.25}\smash{\begin{tabular}[t]{l}$w$\end{tabular}}}}%
    \put(0,0){\includegraphics[width=\unitlength,page=3]{handleswaps_delta-2.pdf}}%
    \put(0.11323555,1.66438701){\makebox(0,0)[lt]{\lineheight{1.25}\smash{\begin{tabular}[t]{l}$w$\end{tabular}}}}%
    \put(0,0){\includegraphics[width=\unitlength,page=4]{handleswaps_delta-2.pdf}}%
    \put(0.11323557,1.339877){\makebox(0,0)[lt]{\lineheight{1.25}\smash{\begin{tabular}[t]{l}$w$\end{tabular}}}}%
    \put(0,0){\includegraphics[width=\unitlength,page=5]{handleswaps_delta-2.pdf}}%
    \put(0.11323557,1.04200551){\makebox(0,0)[lt]{\lineheight{1.25}\smash{\begin{tabular}[t]{l}$w$\end{tabular}}}}%
    \put(0,0){\includegraphics[width=\unitlength,page=6]{handleswaps_delta-2.pdf}}%
    \put(0.11323557,0.73551726){\makebox(0,0)[lt]{\lineheight{1.25}\smash{\begin{tabular}[t]{l}$w$\end{tabular}}}}%
    \put(0,0){\includegraphics[width=\unitlength,page=7]{handleswaps_delta-2.pdf}}%
    \put(0.11323559,0.42930054){\makebox(0,0)[lt]{\lineheight{1.25}\smash{\begin{tabular}[t]{l}$w$\end{tabular}}}}%
    \put(0,0){\includegraphics[width=\unitlength,page=8]{handleswaps_delta-2.pdf}}%
    \put(0.02495551,2.0019176){\makebox(0,0)[lt]{\lineheight{1.25}\smash{\begin{tabular}[t]{l}$g(\Ds_0)$\end{tabular}}}}%
    \put(0.03039013,0.03108642){\makebox(0,0)[lt]{\lineheight{1.25}\smash{\begin{tabular}[t]{l}$\Ds_0$\end{tabular}}}}%
    \put(0,0){\includegraphics[width=\unitlength,page=9]{handleswaps_delta-2.pdf}}%
  \end{picture}%
\endgroup%

%% file: section-10-final-overview.tex
\section{Completing the proofs of naturality and functoriality} \label{sec:final-overview}

In this section, we complete our proofs of naturality and functoriality.

\subsection{Naturality}

If $\frD$ and $\frD'$ are two arbitrary doubling-enhanced Heegaard diagrams (equipped with a framing at the basepoint), we define our transition map $\Psi_{\frD\to \frD'}$ by composing the maps for an arbitrary sequence stabilizations, elementary equivalences of attaching curves, and pointed isotopies of the Heegaard diagram in $(Y,w)$.

We now prove Theorem \ref{thm:final-naturality}, showing that this transition map $\Psi_{\frD\to \frD'}$ defined above is independent up to chain homotopy from the sequence of Heegaard moves.

\begin{proof}[Proof of Theorem \ref{thm:final-naturality}] By \cite{JTNaturality}*{Theorem~2.38} and as in Section \ref{sec:naturality}, it suffices to show that the transition maps we have defined satisfy the axioms of a strong Heegaard invariant \cite{JTNaturality}*{Definition~2.32}.  We verify these axioms presently.

The \emph{functoriality} axiom asserts two claims. Firstly, it asserts that we have well defined morphisms associated to handleslide equivalences of alpha curves and morphisms associated to handleslide equivalences of beta curves (for a fixed Heegaard surface). This part of the axiom follows from Theorem~\ref{thm:naturality-fixed-Sigma}. (We remind the reader that unlike in the setting of ordinary Heegaard Floer homology, it is not particularly natural to define an involutive transition map which changes just the alpha curves, or just the beta curves, since the transition map is defined by compressing hyperboxes as in~\eqref{def:transition-map-elementary-handleslide}. Such hyperboxes necessarily require  changing both the alpha and beta curves simultaneously).

The second part of the functoriality axiom asserts that the morphisms for stabilizations are inverse to the morphisms for destabilizations. This is proven in Proposition~\ref{prop:stabilization-naturality}. 

The next axiom from \cite{JTNaturality}*{Definition~3.32} is the \emph{commutativity} axiom, concerning commutativity of the distinguished rectangles from \cite{JTNaturality}*{Definition~2.29}. In our present case, we verify that there is no monodromy around the distinguished rectangles in Proposition~\ref{prop:simply-connected}.
 In our present setting, we may rephrase the distinguished rectangles from \cite{JTNaturality} as diagrams of embedded, doubling enhanced Heegaard diagrams with the following shape, satisfying one of five configurations:
\[
\begin{tikzcd}
\frD_1
	\ar[r, "e"]
	\ar[d, "f"]	
	 &
\frD_2
	\ar[d, "g"]
\\
\frD_3
 \ar[r, "h"] & \frD_4
\end{tikzcd}
\]
We presently enumerate the five configurations, and also prove that there is no monodromy:
\begin{enumerate}
\item $e$, $f$, $g$ and $h$ are handleslide equivalences. In our setting, we must consider changes of both of the attaching curves, as well as changes of the doubling datum. Note these are also referred to as \emph{strong equivalences} in the literature, e.g.~in \cite{MOIntegerSurgery}. The diagram commutes up to chain homotopy by Theorem~\ref{thm:naturality-fixed-Sigma}. We note that in \cite{JTNaturality}, the stated rectangle has the property that $e$ and $h$ are $\alpha$-equivalences, while $f$ and $g$ are $\beta$-equivalences. However the rectangle therein is also a rectangle of \emph{isotopy diagrams}, so we must consider changes of both the alpha and beta curves along each edge.
\item $e$ and $h$ are handleslide equivalences of attaching curves and doubling curves, while $f$ and $g$ are both stabilizations. Commutativity follows from Proposition~\ref{prop:stabilization-naturality}.
\item $e$ and $h$ are handleslide equivalences, while $f$ and $g$ are diffeomorphisms preserving the framing $\xi$. Commutativity around such rectangles is tautological.
\item $e$, $f$, $g$ and $h$ are all stabilizations. Furthermore, there are disjoint 3-balls $B_1$ and $B_2$, such that $e$ and $h$ correspond to a stabilization in $B_1$, while $f$ and $g$ correspond to a stabilization in $B_2$. Commutativity around such rectangles follows from Proposition~\ref{prop:stabilization-naturality}.
\item The maps $e$ and $h$ are stabilizations, while $f$ and $g$ are diffeomorphisms preserving the framing. Furthermore, the stabilization ball for $h$ is the image of the stabilization ball for $e$ under the diffeomorphism $g$. Commutativity around such rectangles is tautological.
\end{enumerate}

The remaining axioms of \cite{JTNaturality}*{Definition~3.32}, continuity and handleswap invariance, were verified in  Proposition~\ref{prop:continuity} and Theorem~\ref{thm:handleswap}.
\end{proof}

\subsection{Functoriality}

We now finish defining our cobordism maps and sketch the proof of Theorem \ref{thm:well-defined-cobordism-map}, showing that these maps are well-defined. Our proof is modeled on \cite{OSTriangles}. Since our approach is standard, we only provide an overview and sketch the points at which our construction deviates from \cite{OSTriangles}.

We begin with the construction, and will shortly sketch a proof of invariance. Suppose $W$ is a cobordism from $Y_1$ to $Y_2$, and assume basepoints in $Y_1$ and $Y_2$, as well as a framed path $\g$ connecting them, have been chosen. We assume that $W$, $Y_1$, $Y_2$ are connected. We pick a Morse function $f\colon W\to [0,1]$. We may assume that the indices of the critical points of $f$ are non-decreasing, and we assume all critical values are distinct and that $f$ has critical points only of index 1, 2 and 3. Next, we pick a gradient like vector field $v$ such that our path $\g$ is a flow-line. A choice of gradient like vector field gives a diffeomorphism
\[
W\iso W(\bS_2)\cup W(\bS_1)\cup W(\bS_0),
\]
where $\bS_i$ are framed $i$-dimensional links and $\bS_0\subset Y_1$, $\bS_1\subset Y_1(\bS_0)$, and $\bS_2\subset Y_1(\bS_0)(\bS_1)$. If $\frs\in \Spin^c(W)$ is self-conjugate, we define our cobordism map as the composition
\[
\CFI(W,\xi, \frs)=\CFI(W(\bS_2))\circ \CFI(W(\bS_1), \frs|_{W(\bS_1)})\circ \CFI(W(\bS_0)).
\]

We now prove Theorem \ref{thm:well-defined-cobordism-map}. As a first step, we show that the 2-handle map is well-defined:

\begin{lem}\label{lem:well-defined-2handle}
 Our 2-handle map $\CFI(W(\bS_1),\frs|_{W(\bS_1)})$ is independent of the choice of diagram subordinate to a particular bouquet. Furthermore, it is also invariant from the choice of bouquet, as well as handleslides amongst the components of $\bS_1$.
\end{lem}
\begin{proof}
Ozsv\'{a}th and Szab\'{o} \cite{OSTriangles}*{Lemma~4.5} described five moves which relate any two bouquets of a fixed link. Invariance under the first four of these moves follows immediately from invariance of our map from handleslides amongst the curves in our Heegaard triple, Proposition~\ref{prop:2-handles-handle-slides}. The final move described by Ozsv\'{a}th and Szab\'{o} is stabilization of the Heegaard triple. Invariance from this move is a formal consequence of our definition of the stabilization map. Indeed we defined the naturality map for stabilizations as the composition of the 1-handle map, followed by a canceling 2-handle map. Both the 1-handle map and map for a canceling 2-handle map commute with the 2-handle map $\CFI(W(\bS_1), \frs|_{W(\bS_1)})$ by Propositions~\ref{prop:composition-law} and ~\ref{prop:commute-1handles/2-handles}. Hence our 2-handle map is independent of the choice of diagram subordinate to a fixed bouquet.

Next, we note that independence from the choice of bouquet, as well as independence from handleslides amongst the components of $\bS_1$ both follow from independence of the map from handleslides of the attaching curves, which we proved in Proposition~\ref{prop:2-handles-handle-slides}. Compare \cite{OSTriangles}*{Lemma~4.8}.
\end{proof}

\begin{proof}[Proof of Theorem~\ref{thm:well-defined-cobordism-map}]
Firstly, we verify that the construction is independent of the choice of gradient-like vector field $v$ for $f$. Subsequently we sketch that the cobordism map is invariant from $f$.

The space of gradient like vector fields for a fixed Morse function is connected. Since we have already decomposed $W$ along level sets which separate critical points of different indices, the codimension 1 singularities of a path of gradient like vector fields are handleslides amongst link components of the same indices. For critical points of index 1 or 3, invariance under these moves can instead be proven by reordering the attachment of the handles using the composition law for 1-handles and 3-handles, Proposition~\ref{lem:commute-1/3-handles}, since a handleslide is the same as an isotopy after attaching one of the handles. For 2-handles, we cannot add additional level sets because of the $\Spin^c$ decomposition, so instead handleslide invariance is proven directly in Lemma~\ref{lem:well-defined-2handle}.

We now consider invariance under $f$. By standard Cerf theory, \cite{CerfStratification} \cite{KirbyCalculus} any two Morse functions $f_0$ and $f_1$ with critical points ordered monotonically may be connected with a 1-parameter family $(f_{t})_{t\in [0,1]}$ whose critical points are also ordered monotonically, except at finitely many points where a birth-death singularity occurs. Furthermore, if $f_0$ and $f_1$ have no index 0 or 4 critical points, then each $f_t$ may also be chosen to have no index 0 or 4 critical points. We now show that the maps are invariant under index 1/2 handle cancellations, as well as index 2/3 handle cancellations. Invariance under such handle cancellation is essentially automatic from the definition of our naturality map for stabilizations. Indeed, consider the case that $W$ has a Morse function with only index 2 critical points, whose descending manifolds intersect $Y_0$ in a framed link $\bS_1$. Let $\bS_0\subset Y_0$ be a 0-sphere, and let $\bK$ be a framed knot in $Y_0(\bS_0)$ which intersects the co-core of $\bS_0$ in exactly 1-point. Assume $\bS_1$ is disjoint from $\bS_0$ and $\bK$. Using the composition law of Propositions~\ref{prop:composition-law}, we obtain
\[
\CFI(W(\bS_1\cup \bK))\circ \CFI(W(\bS_0))\simeq \CFI(W(\bS_1))\circ \CFI(W(\bK))\circ \CFI(W(\bS_0)):=\CFI(W(\bS_1))\circ \sigma,
\]
where $\sigma$ is the naturality map for stabilization (and we are omitting $\Spin^c$ structures from the notation). Invariance under index 2/3 handle cancelations follows entirely analogously. The proof is complete.
\end{proof}

%% file: section-duality.tex
\section{The cobordism map for $S^2\times S^2$}
 \label{sec:example}

In this section we compute the cobordism map for $S^2\times S^2$. 
 We prove the following:
\begin{prop} 
Let $W$ denote $S^2\times S^2$, with two 4-balls removed. Let $\frs$ denote the unique self-conjugate $\Spin^c$ structure on $W$. Then the cobordism map
\[
\CFI(W,\frs)\colon \bF[U,Q]/Q^2\to \bF[U,Q]/Q^2
\]
is multiplication by $Q$.
\end{prop}
Using the composition law, we may decompose $W$ as two 2-handle cobordisms, $W=W_2\circ W_1$. The first 2-handle is attached along a 0-framed unknot in $S^3$ to form $W_1$. The second 2-handle is attached along an $S^1$ fiber of $S^1\times S^2$ to form $W_2$. Note that, in both cases, one end of the cobordism is a copy of $S^3$, on which the map $\Phi$ is nullhomotopic. This implies that the choice of framings of the basepoints of the three-manifolds and along choices of paths connecting them do not affect the final computation, and we therefore omit them. 

We start with the standard Heegaard triple $(\Sigma,\as',\as,\bs)$ for $W_1$. In Figure~\ref{fig:s2s2-1}, we show the diagram $(\Sigma\# \bar{\Sigma},\as'\bar{\bs},\as\bar{\bs},\bs\bar{\bs},\Ds)$. Therein, the top half of the figure is $\Sigma$, and the bottom half is $\bar \Sigma$. This figure also encodes $(\Sigma,\as',\as,\bs)$.

The first quadrilateral counting map in Equation~\eqref{eq:hyperbox-2-handles} is $h^{\b\bar{\b}\to \Delta}_{\a\bar{\b}\to \a'\bar{\b}}$.  The Heegaard quadruple $(\Sigma\# \bar \Sigma, \as' \bar \bs, \as \bar \bs, \bs \bar \bs ,\Ds)$ is depicted in Figure \ref{fig:s2s2-1}. This diagram is weakly admissible.  The input for $h^{\b\bar{\b}\to \Delta}_{\a\bar{\b}\to \a'\bar{\b}}$ from $\CF(\bs\bar{\bs},\Ds)$ is $\Theta^+_{\b\bar{\b},\Delta}$, and the input from $\CF(\as'\bar{\bs},\as\bar{\bs})$ is $\Theta^+_{\a'\bar{\b},\a\bar{\b}}$.  We need only evaluate $h^{\b\bar{\b}\to \Delta}_{\a\bar{\b}\to \a'\bar{\b}}(\Theta_{\a\bar{\b},\b\bar{\b}}^+)$, since $F_1^{\bar{\b},\bar{\b}}(\Theta_{\a,\b})=\Theta_{\a\bar{\b},\b\bar{\b}}^+$, and for grading reasons $h^{\b\bar{\b}\to \Delta}_{\a\bar{\b}\to \a'\bar{\b}}(\Theta_{\a\bar{\b},\b\bar{\b}}^+)$ is a multiple of $\Theta_{\a'\bar{\b},\Delta}^+$. Hence, we need only consider holomorphic quadrilaterals in
\[
\pi_2(\Theta^+_{\a'\bar{\b},\a\bar{\b}},\Theta_{\a\bar{\b},\b\bar{\b}}^+,\Theta^+_{\b\bar{\b},\Delta},\Theta_{\a'\bar{\b},\Delta}^+)
\]
There is an index $-1$ domain (shaded). This class has a holomorphic representative for a unique conformal class of rectangle.  Indeed, for each slit length (along the alpha and beta arcs in the interior of $P_2$)  there is a map of a varying conformal class of rectangle to the domain of $P_2$.  Meanwhile, there is a holomorphic rectangle mapping to $P_1$ for a unique conformal structure.  

 \begin{figure}[H]
	\centering
	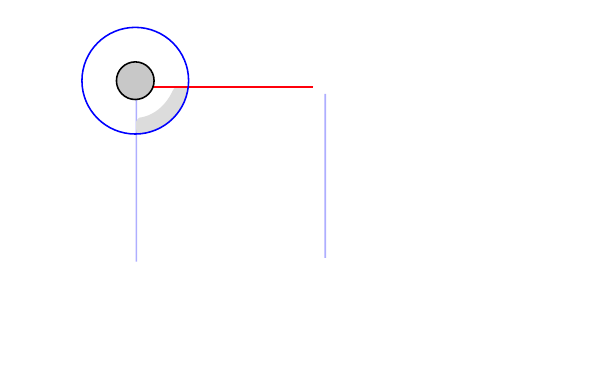
	\caption{The Heegaard quadruple $(\Sigma\# \bar{\Sigma},\as'\bar{\bs},\as\bar{\bs},\bs\bar{\bs},\Ds)$.  The generators of interest are $\Theta^+_{\a'\bar{\b},\a\bar{\b}}=\{\phi,\phi^+\}$, $\Theta_{\a\bar{\b},\b\bar{\b}}^+=\{\theta,\theta^+\}$, $\Theta^+_{\b\bar{\b},\Delta}=\{\tau,\tau^+\}$ and $\Theta_{\a'\bar{\b},\Delta}^+=\{z,z^+\}$. }\label{fig:s2s2-1}
\end{figure}

By inspection, there are no other positive domains and so
\[
h^{\b\bar{\b}\to \Delta}_{\a\bar{\b}\to \a'\bar{\b}}(\Theta_{\a\bar{\b},\b\bar{\b}}^+)=\Theta_{\a'\bar{\b},\Delta}^+.
\]  

For $h^{\Delta\to \a'\bar{\a}'}_{\a\bar{\b}\to \a'\bar{\b}}$ we need only consider summands in the image containing $\xs=\{x_0^-,x_1^+\}$, since grading considerations imply that the only other generator that can appear in the image of $h^{\Delta\to \a'\bar{\a'}}_{\a\bar{\b}\to \a'\bar{\b}}\circ f^{\b\bar{\b}\to\Delta}_{\a\bar{\b}}\circ F_1^{\bar{\b},\bar{\b}}$ is $\{x_0^+,x_1^-\}$, which is annihilated by $F_3^{\a',\a'}$.

 \begin{figure}[H]
	\centering
	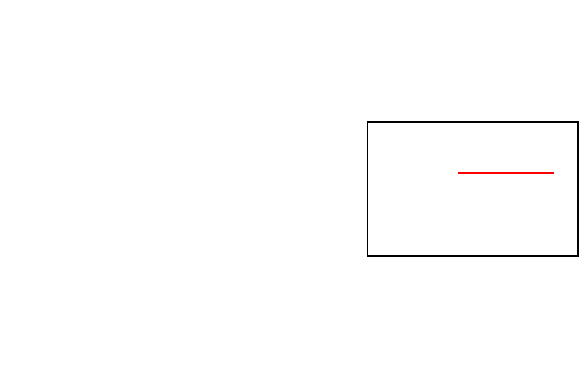
	\caption{The Heegaard quadruple $(\Sigma,\as'\bar{\bs},\as\bar{\bs},\Ds,\as'\bar{\as}')$ in the definition of $h^{\Delta\to \a'\bar{\a'}}_{\a\bar{\b}\to \a'\bar{\b}}$, with some intersection points labeled.  The generators of interest are $\Theta^+_{\Delta,\a'\bar{\a}'}=\{\tau^+,\tau\}$,  $\Theta^+_{\a\bar{\b},\Delta}=\{\theta,\theta'\}$,  $\Theta^+_{\a'\bar{\b},\a\bar{\b}}=\{ \phi^+,\phi\}$ and $\xs=\{x_0^-,x_1^+\}$. 	}\label{fig:s2s2-2}
\end{figure}

 \begin{figure}[H]
	\centering
	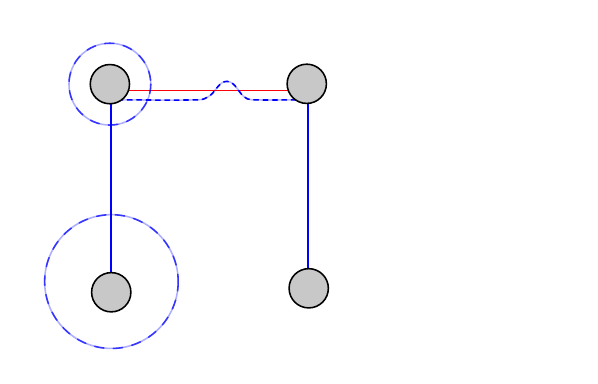
	\caption{The Heegaard quadruple $(\Sigma,\as\bar{\bs},\Ds,\as'\bar{\as}',\as\bar{\as}')$ in the definition of $H^{\Dt\to \a\bar{\a}'}_{\a\bar{\b}}$, with some intersection points labeled.  The generators of interest are $\Theta^+_{\Delta,\a'\bar{\a}'}=\{\tau^+,\tau\}$,  $\Theta^+_{\a'\bar{\a}',\a\bar{\a}'}=\xs_2:=\{x,x_2\}$,  $\ys=\{ \phi_2,\phi_3\}$ and $\Theta_{\a\bar{\b},\Delta}=\{\theta,\theta'\}$. 	}\label{fig:s2s2-3}
\end{figure}

One checks directly that there are no positive domains among those generators for $h^{\Delta\to \a'\bar{\a'}}_{\a\bar{\b}\to \a'\bar{\b}}$ as well as $H^{\Delta\to \a\bar{\a}'}_{\a\bar{\b}}$.

It follows that the map $F_{W_1,\frs|_{W_1}}$ takes the form
\[\left(
\begin{tikzcd}[column sep=1cm]
\CFI(S^3)
	\ar[r, "F_{W_1,\frs|_{W_1}}"]
&
\CFI(S^1\times S^2)
\end{tikzcd}\right)
\simeq 
\begin{tikzcd}[column sep=2cm, row sep=1cm, labels=description]
\CF(S^3)
	\ar[r, "1\mapsto \Theta^-"]
	\ar[d, "Q(1+\iota)"]
	\ar[dr,dashed, "1\mapsto Q\Theta^+"]
&
\CF(S^1\times S^2)
	\ar[d, "Q(1+\iota)"]
\\
Q\cdot\CF(S^3)
	\ar[r, "1\mapsto \Theta^-"]
	&
Q\cdot\CF(S^1\times S^2)
\end{tikzcd}
\]

We now consider the $2$-handle cobordism $W$ from $S^1\times S^2\to S^3$.  This is determined by the Heegaard triple $(\Sigma,\as'',\as',\bs)$, where $\as',\bs$ are as before, and $\as''$ is a copy of $\as$.  Let $\theta^-_{\a',\b}$ be the lower-degree generator of $\CF^-(\Sigma,\as',\bs)$ and $\theta^+_{\bar\b,\bar\b}$ be the top-degree generator of $\CF^-(\bar{\Sigma},\bar\bs,\bar\bs)$.  We need to calculate $h^{\b\bar\b\to \Delta}_{\a'\bar\b\to \a''\bar\b}(\theta^-_{\a',\b}\times \theta^+_{\bar\b,\bar\b})$, $h^{\Delta,\a''\bar{\a}''}_{\a'\bar{\b}\to\a''\bar{\b}}(\Theta^-_{\a'\bar{\b},\Delta})$, and $H^{\Delta\to\a'\bar{\a}''}_{\a'\bar{\b}}(\Theta^-_{\a'\bar{\b},\Delta})$.
  
We consider diagrams for each of those quadruples.  Indeed, the first quadruple is represented in Figure \ref{fig:s2s2-1} with the role of the $\as'\bar{\bs}$ and $\as\bar{\bs}$ curves exchanged.  A direct calculation shows there are no positive domains for $\pi_2(\Theta^+_{\a''\bar{\b},\a'\bar{\b}},\theta^-_{\a',\b}\times \theta^+_{\bar\b,\bar\b},\Theta^+_{\b\bar{\b},\Delta},\Theta_{\a''\bar{\b},\Delta})$ so $h^{\b\bar\b\to \Delta}_{\a'\bar\b\to \a''\bar\b}(\theta^-_{\a',\b}\times\theta^+_{\bar\b,\bar\b})=0$.  Entirely similar remarks apply for $h^{\Delta,\a''\bar{\a}''}_{\a'\bar{\b}\to\a''\bar{\b}}(\Theta^-_{\a'\bar{\b},\Delta})$, and $H^{\Delta\to\a'\bar{\a}''}_{\a'\bar{\b}}(\Theta^-_{\a'\bar{\b},\Delta})$.
Computing as above, we obtain that
\[\left(
\begin{tikzcd}[column sep=1cm]
\CFI(S^1\times S^2)
	\ar[r, "F_{W_2,\frs|_{W_2}}"]
&
\CFI(S^3)
\end{tikzcd}\right)
\simeq 
\begin{tikzcd}[column sep=2cm, row sep=1cm, labels=description]
\CF(S^1\times S^2)
	\ar[r, "\Theta^+\mapsto 1{,} \Theta^-\mapsto 0"]
	\ar[d, "Q(1+\iota)"]
&
\CF(S^3)
	\ar[d, "Q(1+\iota)"]
\\
Q\cdot\CF(S^1\times S^2)
	\ar[r, "\Theta^+\mapsto 1{,} \Theta^-\mapsto 0"]
	&
Q\cdot\CF(S^3)
\end{tikzcd}
\]

Composing the $\CFI(W_2,\frs_2)\circ \CFI(W_1,\frs_1)$, we obtain multiplication by $Q$.

\begin{rem} \label{rem:duality}
 The cobordism $W_1$ is obtained topologically by turning the cobordism $W_2$ around and reversing its orientation. The maps $\CFI(W_1,\frs_1)$ and $\CFI(W_2,\frs_2)$ become dual over $\bF[U,Q]/Q^2$ after performing a change of basis, corresponding to a nontrivial automorphism on $\CFI(S^1\times S^2)$. Compare \cite{OSTriangles}*{Theorem~3.5}. Note that our definition of the 2-handle map in Equation~\eqref{eq:hyperbox-2-handles} does not naturally dualize to another 2-handle map. We hope to investigate duality more precisely in future work.
\end{rem}

%% file: s2s2-1.pdf_tex
\begingroup%
  \makeatletter%
  \providecommand\color[2][]{%
    \errmessage{(Inkscape) Color is used for the text in Inkscape, but the package 'color.sty' is not loaded}%
    \renewcommand\color[2][]{}%
  }%
  \providecommand\transparent[1]{%
    \errmessage{(Inkscape) Transparency is used (non-zero) for the text in Inkscape, but the package 'transparent.sty' is not loaded}%
    \renewcommand\transparent[1]{}%
  }%
  \providecommand\rotatebox[2]{#2}%
  \newcommand*\fsize{\dimexpr\f@size pt\relax}%
  \newcommand*\lineheight[1]{\fontsize{\fsize}{#1\fsize}\selectfont}%
  \ifx\svgwidth\undefined%
    \setlength{\unitlength}{290.40410042bp}%
    \ifx\svgscale\undefined%
      \relax%
    \else%
      \setlength{\unitlength}{\unitlength * \real{\svgscale}}%
    \fi%
  \else%
    \setlength{\unitlength}{\svgwidth}%
  \fi%
  \global\let\svgwidth\undefined%
  \global\let\svgscale\undefined%
  \makeatother%
  \begin{picture}(1,0.62697945)%
    \lineheight{1}%
    \setlength\tabcolsep{0pt}%
    \put(0,0){\includegraphics[width=\unitlength,page=1]{s2s2-1.pdf}}%
    \put(0.22355225,0.48467837){\makebox(0,0)[t]{\lineheight{1.25}\smash{\begin{tabular}[t]{c}$\mathrm{B}$\end{tabular}}}}%
    \put(0,0){\includegraphics[width=\unitlength,page=2]{s2s2-1.pdf}}%
    \put(0.77940752,0.26527754){\makebox(0,0)[rt]{\lineheight{1.25}\smash{\begin{tabular}[t]{r}$\bs\bar{\bs}$\end{tabular}}}}%
    \put(0,0){\includegraphics[width=\unitlength,page=3]{s2s2-1.pdf}}%
    \put(0.53864115,0.15506167){\makebox(0,0)[t]{\lineheight{1.25}\smash{\begin{tabular}[t]{c}\reflectbox{$\mathrm{E}$}\end{tabular}}}}%
    \put(0,0){\includegraphics[width=\unitlength,page=4]{s2s2-1.pdf}}%
    \put(0.53551156,0.48467806){\makebox(0,0)[t]{\lineheight{1.25}\smash{\begin{tabular}[t]{c}\reflectbox{$\mathrm{B}$}\end{tabular}}}}%
    \put(0,0){\includegraphics[width=\unitlength,page=5]{s2s2-1.pdf}}%
    \put(0.7806939,0.31898174){\color[rgb]{0,0,0}\makebox(0,0)[rt]{\lineheight{1.25}\smash{\begin{tabular}[t]{r}$\as\bar{\bs}$\end{tabular}}}}%
    \put(0.78077323,0.21142679){\makebox(0,0)[rt]{\lineheight{1.25}\smash{\begin{tabular}[t]{r}$\Ds$\end{tabular}}}}%
    \put(0,0){\includegraphics[width=\unitlength,page=6]{s2s2-1.pdf}}%
    \put(0.23326762,0.37526262){\color[rgb]{0,0,0}\makebox(0,0)[lt]{\lineheight{1.25}\smash{\begin{tabular}[t]{l}$\tau$\end{tabular}}}}%
    \put(0.2291801,0.43952274){\color[rgb]{0,0,0}\makebox(0,0)[lt]{\lineheight{1.25}\smash{\begin{tabular}[t]{l}$z$\end{tabular}}}}%
    \put(0.31512912,0.4651174){\color[rgb]{0,0,0}\makebox(0,0)[lt]{\lineheight{1.25}\smash{\begin{tabular}[t]{l}$\theta$\end{tabular}}}}%
    \put(0.28469888,0.49125084){\color[rgb]{0,0,0}\makebox(0,0)[rt]{\lineheight{1.25}\smash{\begin{tabular}[t]{r}$\phi$\end{tabular}}}}%
    \put(0.29546052,0.39991845){\color[rgb]{0,0,0}\makebox(0,0)[lt]{\lineheight{1.25}\smash{\begin{tabular}[t]{l}$P_1$\end{tabular}}}}%
    \put(0,0){\includegraphics[width=\unitlength,page=7]{s2s2-1.pdf}}%
    \put(0.77946827,0.3697019){\makebox(0,0)[rt]{\lineheight{1.25}\smash{\begin{tabular}[t]{r}$\as'\bar{\bs}$\end{tabular}}}}%
    \put(0,0){\includegraphics[width=\unitlength,page=8]{s2s2-1.pdf}}%
    \put(0.04946732,0.49689362){\makebox(0,0)[rt]{\lineheight{1.25}\smash{\begin{tabular}[t]{r}$\Sigma$\end{tabular}}}}%
    \put(0.04710794,0.11185006){\makebox(0,0)[rt]{\lineheight{1.25}\smash{\begin{tabular}[t]{r}$\bar\Sigma$\end{tabular}}}}%
    \put(0,0){\includegraphics[width=\unitlength,page=9]{s2s2-1.pdf}}%
    \put(0.11766021,0.18876956){\color[rgb]{0,0,0}\makebox(0,0)[rt]{\lineheight{1.25}\smash{\begin{tabular}[t]{r}$\tau^+$\end{tabular}}}}%
    \put(0.16492943,0.14913234){\color[rgb]{0,0,0}\makebox(0,0)[lt]{\lineheight{1.25}\smash{\begin{tabular}[t]{l}$z^+$\end{tabular}}}}%
    \put(0.12930972,0.10709164){\color[rgb]{0,0,0}\makebox(0,0)[rt]{\lineheight{1.25}\smash{\begin{tabular}[t]{r}$\theta^+$\end{tabular}}}}%
    \put(0.16568156,0.035559){\color[rgb]{0,0,0}\makebox(0,0)[rt]{\lineheight{1.25}\smash{\begin{tabular}[t]{r}$\phi^+$\end{tabular}}}}%
    \put(0.10443538,0.04475413){\color[rgb]{0,0,0}\makebox(0,0)[rt]{\lineheight{1.25}\smash{\begin{tabular}[t]{r}$P_2$\end{tabular}}}}%
    \put(0,0){\includegraphics[width=\unitlength,page=10]{s2s2-1.pdf}}%
    \put(0.21520176,0.16520227){\makebox(0,0)[lt]{\lineheight{1.25}\smash{\begin{tabular}[t]{l}$\mathrm{E}$\end{tabular}}}}%
  \end{picture}%
\endgroup%

%% file: s2s2-2.pdf_tex
\begingroup%
  \makeatletter%
  \providecommand\color[2][]{%
    \errmessage{(Inkscape) Color is used for the text in Inkscape, but the package 'color.sty' is not loaded}%
    \renewcommand\color[2][]{}%
  }%
  \providecommand\transparent[1]{%
    \errmessage{(Inkscape) Transparency is used (non-zero) for the text in Inkscape, but the package 'transparent.sty' is not loaded}%
    \renewcommand\transparent[1]{}%
  }%
  \providecommand\rotatebox[2]{#2}%
  \newcommand*\fsize{\dimexpr\f@size pt\relax}%
  \newcommand*\lineheight[1]{\fontsize{\fsize}{#1\fsize}\selectfont}%
  \ifx\svgwidth\undefined%
    \setlength{\unitlength}{277.74639738bp}%
    \ifx\svgscale\undefined%
      \relax%
    \else%
      \setlength{\unitlength}{\unitlength * \real{\svgscale}}%
    \fi%
  \else%
    \setlength{\unitlength}{\svgwidth}%
  \fi%
  \global\let\svgwidth\undefined%
  \global\let\svgscale\undefined%
  \makeatother%
  \begin{picture}(1,0.66306939)%
    \lineheight{1}%
    \setlength\tabcolsep{0pt}%
    \put(0,0){\includegraphics[width=\unitlength,page=1]{s2s2-2.pdf}}%
    \put(0.77078234,0.30866571){\makebox(0,0)[rt]{\lineheight{1.25}\smash{\begin{tabular}[t]{r}$\Ds$\end{tabular}}}}%
    \put(0,0){\includegraphics[width=\unitlength,page=2]{s2s2-2.pdf}}%
    \put(0.77122282,0.4084549){\makebox(0,0)[rt]{\lineheight{1.25}\smash{\begin{tabular}[t]{r}$\as'\bar{\bs}$\end{tabular}}}}%
    \put(0,0){\includegraphics[width=\unitlength,page=3]{s2s2-2.pdf}}%
    \put(0.77097085,0.36136289){\color[rgb]{0,0,0}\makebox(0,0)[rt]{\lineheight{1.25}\smash{\begin{tabular}[t]{r}$\as\bar{\bs}$\end{tabular}}}}%
    \put(0.77078234,0.25430644){\makebox(0,0)[rt]{\lineheight{1.25}\smash{\begin{tabular}[t]{r}$\as'\bar{\as}'$\end{tabular}}}}%
    \put(0,0){\includegraphics[width=\unitlength,page=4]{s2s2-2.pdf}}%
    \put(0.52052322,0.15657629){\makebox(0,0)[t]{\lineheight{1.25}\smash{\begin{tabular}[t]{c}\reflectbox{$\mathrm{E}$}\end{tabular}}}}%
    \put(0.1823177,0.5103676){\makebox(0,0)[t]{\lineheight{1.25}\smash{\begin{tabular}[t]{c}$\mathrm{B}$\end{tabular}}}}%
    \put(0.51716402,0.5103672){\makebox(0,0)[t]{\lineheight{1.25}\smash{\begin{tabular}[t]{c}\reflectbox{$\mathrm{B}$}\end{tabular}}}}%
    \put(0,0){\includegraphics[width=\unitlength,page=5]{s2s2-2.pdf}}%
    \put(0.18721395,0.40060429){\color[rgb]{0,0,0}\makebox(0,0)[lt]{\lineheight{1.25}\smash{\begin{tabular}[t]{l}$\tau$\end{tabular}}}}%
    \put(0.31284249,0.51554522){\color[rgb]{0,0,0}\makebox(0,0)[lt]{\lineheight{1.25}\smash{\begin{tabular}[t]{l}$\theta$\end{tabular}}}}%
    \put(0.24644411,0.51730933){\color[rgb]{0,0,0}\makebox(0,0)[rt]{\lineheight{1.25}\smash{\begin{tabular}[t]{r}$\phi$\end{tabular}}}}%
    \put(0.11821476,0.56218039){\color[rgb]{0,0,0}\makebox(0,0)[lt]{\lineheight{1.25}\smash{\begin{tabular}[t]{l}$x_0^-$\end{tabular}}}}%
    \put(0.24523324,0.56176865){\color[rgb]{0,0,0}\makebox(0,0)[rt]{\lineheight{1.25}\smash{\begin{tabular}[t]{r}$x_0^+$\end{tabular}}}}%
    \put(0,0){\includegraphics[width=\unitlength,page=6]{s2s2-2.pdf}}%
    \put(0.17246627,0.14501742){\makebox(0,0)[lt]{\lineheight{1.25}\smash{\begin{tabular}[t]{l}$\mathrm{E}$\end{tabular}}}}%
    \put(0,0){\includegraphics[width=\unitlength,page=7]{s2s2-2.pdf}}%
    \put(0.06702865,0.21143781){\color[rgb]{0,0,0}\makebox(0,0)[rt]{\lineheight{1.25}\smash{\begin{tabular}[t]{r}$\tau^-$\end{tabular}}}}%
    \put(0.12023731,0.0171297){\color[rgb]{0,0,0}\makebox(0,0)[rt]{\lineheight{1.25}\smash{\begin{tabular}[t]{r}$\phi^+$\end{tabular}}}}%
    \put(0.29870567,0.21137476){\color[rgb]{0,0,0}\makebox(0,0)[lt]{\lineheight{1.25}\smash{\begin{tabular}[t]{l}$\tau^+$\end{tabular}}}}%
    \put(0.12271294,0.08320443){\color[rgb]{0,0,0}\makebox(0,0)[lt]{\lineheight{1.25}\smash{\begin{tabular}[t]{l}$x_1^+$\end{tabular}}}}%
    \put(0.25043338,0.08320443){\color[rgb]{0,0,0}\makebox(0,0)[rt]{\lineheight{1.25}\smash{\begin{tabular}[t]{r}$x_1^-$\end{tabular}}}}%
    \put(0.18970174,0.24549636){\color[rgb]{0,0,0}\makebox(0,0)[lt]{\lineheight{1.25}\smash{\begin{tabular}[t]{l}$\theta'$\end{tabular}}}}%
    \put(0,0){\includegraphics[width=\unitlength,page=8]{s2s2-2.pdf}}%
  \end{picture}%
\endgroup%

%% file: s2s2-3.pdf_tex
\begingroup%
  \makeatletter%
  \providecommand\color[2][]{%
    \errmessage{(Inkscape) Color is used for the text in Inkscape, but the package 'color.sty' is not loaded}%
    \renewcommand\color[2][]{}%
  }%
  \providecommand\transparent[1]{%
    \errmessage{(Inkscape) Transparency is used (non-zero) for the text in Inkscape, but the package 'transparent.sty' is not loaded}%
    \renewcommand\transparent[1]{}%
  }%
  \providecommand\rotatebox[2]{#2}%
  \newcommand*\fsize{\dimexpr\f@size pt\relax}%
  \newcommand*\lineheight[1]{\fontsize{\fsize}{#1\fsize}\selectfont}%
  \ifx\svgwidth\undefined%
    \setlength{\unitlength}{286.95052366bp}%
    \ifx\svgscale\undefined%
      \relax%
    \else%
      \setlength{\unitlength}{\unitlength * \real{\svgscale}}%
    \fi%
  \else%
    \setlength{\unitlength}{\svgwidth}%
  \fi%
  \global\let\svgwidth\undefined%
  \global\let\svgscale\undefined%
  \makeatother%
  \begin{picture}(1,0.64143319)%
    \lineheight{1}%
    \setlength\tabcolsep{0pt}%
    \put(0,0){\includegraphics[width=\unitlength,page=1]{s2s2-3.pdf}}%
    \put(0.17410819,0.13716755){\makebox(0,0)[lt]{\lineheight{1.25}\smash{\begin{tabular}[t]{l}$\mathrm{E}$\end{tabular}}}}%
    \put(0.51649365,0.14322524){\makebox(0,0)[t]{\lineheight{1.25}\smash{\begin{tabular}[t]{c}\reflectbox{$\mathrm{E}$}\end{tabular}}}}%
    \put(0.18379907,0.49124407){\makebox(0,0)[t]{\lineheight{1.25}\smash{\begin{tabular}[t]{c}$\mathrm{B}$\end{tabular}}}}%
    \put(0.51318917,0.49124371){\makebox(0,0)[t]{\lineheight{1.25}\smash{\begin{tabular}[t]{c}\reflectbox{$\mathrm{B}$}\end{tabular}}}}%
    \put(0,0){\includegraphics[width=\unitlength,page=2]{s2s2-3.pdf}}%
    \put(0.7767526,0.3525061){\makebox(0,0)[rt]{\lineheight{1.25}\smash{\begin{tabular}[t]{r}$\Ds$\end{tabular}}}}%
    \put(0,0){\includegraphics[width=\unitlength,page=3]{s2s2-3.pdf}}%
    \put(0.77652433,0.25284414){\makebox(0,0)[rt]{\lineheight{1.25}\smash{\begin{tabular}[t]{r}$\as\bar{\as}'$\end{tabular}}}}%
    \put(0,0){\includegraphics[width=\unitlength,page=4]{s2s2-3.pdf}}%
    \put(0.77805453,0.40351265){\color[rgb]{0,0,0}\makebox(0,0)[rt]{\lineheight{1.25}\smash{\begin{tabular}[t]{r}$\as\bar{\bs}$\end{tabular}}}}%
    \put(0.77675264,0.3027278){\makebox(0,0)[rt]{\lineheight{1.25}\smash{\begin{tabular}[t]{r}$\as'\bar{\as}'$\end{tabular}}}}%
    \put(0,0){\includegraphics[width=\unitlength,page=5]{s2s2-3.pdf}}%
    \put(0.19396439,0.41136592){\color[rgb]{0,0,0}\makebox(0,0)[lt]{\lineheight{1.25}\smash{\begin{tabular}[t]{l}$\tau$\end{tabular}}}}%
    \put(0.06593603,0.20243032){\color[rgb]{0,0,0}\makebox(0,0)[rt]{\lineheight{1.25}\smash{\begin{tabular}[t]{r}$\tau^-$\end{tabular}}}}%
    \put(0.29707234,0.49710057){\color[rgb]{0,0,0}\makebox(0,0)[rt]{\lineheight{1.25}\smash{\begin{tabular}[t]{r}$\theta$\end{tabular}}}}%
    \put(0.12859577,0.01677068){\color[rgb]{0,0,0}\makebox(0,0)[rt]{\lineheight{1.25}\smash{\begin{tabular}[t]{r}$\phi_4$\end{tabular}}}}%
    \put(0.2966442,0.20336381){\color[rgb]{0,0,0}\makebox(0,0)[lt]{\lineheight{1.25}\smash{\begin{tabular}[t]{l}$\tau^+$\end{tabular}}}}%
    \put(0.12898296,0.08036874){\color[rgb]{0,0,0}\makebox(0,0)[lt]{\lineheight{1.25}\smash{\begin{tabular}[t]{l}$x_2$\end{tabular}}}}%
    \put(0.24800155,0.08036874){\color[rgb]{0,0,0}\makebox(0,0)[rt]{\lineheight{1.25}\smash{\begin{tabular}[t]{r}$x_1$\end{tabular}}}}%
    \put(0.19173576,0.2461033){\color[rgb]{0,0,0}\makebox(0,0)[lt]{\lineheight{1.25}\smash{\begin{tabular}[t]{l}$\theta$\end{tabular}}}}%
    \put(0.25213963,0.01677068){\color[rgb]{0,0,0}\makebox(0,0)[lt]{\lineheight{1.25}\smash{\begin{tabular}[t]{l}$\phi_3$\end{tabular}}}}%
    \put(0.25215151,0.44980148){\color[rgb]{0,0,0}\makebox(0,0)[lt]{\lineheight{1.25}\smash{\begin{tabular}[t]{l}$x$\end{tabular}}}}%
    \put(0.40510178,0.49644215){\color[rgb]{0,0,0}\makebox(0,0)[lt]{\lineheight{1.25}\smash{\begin{tabular}[t]{l}$\phi_1$\end{tabular}}}}%
    \put(0.35686153,0.49518281){\color[rgb]{0,0,0}\makebox(0,0)[rt]{\lineheight{1.25}\smash{\begin{tabular}[t]{r}$\phi_2$\end{tabular}}}}%
    \put(0,0){\includegraphics[width=\unitlength,page=6]{s2s2-3.pdf}}%
  \end{picture}%
\endgroup%

%% file: section-knots.tex
\section{Knots and links} \label{sec:knots}

In this section, we describe how to adapt the naturality and functoriality results of the previous sections to the case of knots and links, and in particular prove Theorems \ref{thm:link-naturality} and \ref{thm:link-functoriality}.

If $L$ is a link, in this section we write $\cCFLI(L)$ for the pair $(\cCFL(L), \iota_L)$. Analogously to the case of knots, our construction requires a choice of orientation on $L$, so there are potentially $2^{\ell}$ different potential models for $\iota_L$. To simplify the notation, we assume that an orientation is fixed, and we consider only the two orientations which are either coherent, or opposite to our preferred orientation. We write $\iota_{L,+}$ and $\iota_{L,-}$ for these models of the link involution.

In analogy to our construction in the case of closed 3-manifolds, we will describe several expanded models of the link involution $\tilde{\iota}_{L,\pm}$ in Section~\ref{sec:expanded-involution-knots}. To streamline the presentation, we only define transition maps with respect to the expanded models (unlike in the case of closed 3-manifolds, where we defined both non-expanded and expanded models). Since we only consider the expanded models of the involution in this section, we will write $\frD$ for the expanded, doubling enhanced link diagram we consider in Section~\ref{sec:expanded-involution-knots}. This notation departs from previous sections, where we wrote $\tilde \frD$ for expanded models of the involution. Similarly, we write $\Psi_{ \frD_1\to\frD_2}$ for the expanded models of the transition maps in this section, instead of $\tilde \Psi_{\tilde \frD_1\to \tilde \frD_2}$.

In this section, we show how to adapt the techniques of the earlier sections to prove the following naturality theorem, which implies Theorem \ref{thm:link-naturality}.

\begin{thm}
\label{thm:naturality-knots-links}
Suppose that $L$ is a link in a 3-manifold.
\begin{enumerate}
\item The transition maps $ \Psi_{ \frD\to \frD}$ are well-defined up to homotopies of morphisms of $\iota_L$-complexes.
\item $\Psi_{ \frD\to  \frD}\simeq \id_{\cCFLI( \frD)}$.
\item If $\frD_1$, $ \frD_2$ and $\frD_3$ are three expanded, doubling enhanced Heegaard link diagrams for $(Y,L)$, then
\[
\Psi_{\frD_2\to\frD_3}\circ  \Psi_{ \frD_1\to \frD_2}\simeq \Psi_{ \frD_1\to \frD_3}. 
\]
\end{enumerate}
\end{thm}

Additionally, we prove the following functoriality theorem, which is an expanded version of Theorem \ref{thm:link-functoriality}.

\begin{thm}
\label{thm:functoriality}
 Suppose that $(W,\Sigma)$ is a link cobordism from $(Y_1,L_1)$ to $(Y_2,L_2)$. Suppose further that $\Sigma$ consists of a collection of annuli, each with one boundary component in $Y_1$, and one boundary component in $Y_2$. Suppose further that each annulus is decorated with two parallel longitudinal  arcs, $\frs\in W$ is a $\Spin^c$ structure such that $\bar \frs=\frs+\PD[\Sigma]$. Then the cobordism map $\cCFLI(W,\Sigma,\frs)$ defined via the doubling model is an enhanced $\iota_L$-homomorphism, and is well-defined up to enhanced $\iota_L$-homotopy.
\end{thm}

\subsection{Expanded models of the knot and link involutions}
\label{sec:expanded-involution-knots}

In this section, we describe our expanded model of the knot and link involution, similar to the expanded 3-manifold involution which appeared in Section~\ref{sec:expanded-def}. This is the model we use to prove Theorem~\ref{thm:naturality-knots-links}.

The construction is as follows. Suppose that $L$ is a link in $Y$, and  $(\Sigma,\as,\bs,\ws,\zs)$ is a Heegaard link diagram for $(Y,L)$ such that each component of $L$ contains exactly two basepoints. We pick an auxiliary point $p\in \Sigma\setminus (\as\cup \bs \cup \{w,z\})$, at which we will form the connected sum of $\Sigma$ and $\bar \Sigma$.

We write $c$ for the curve which is a meridian of the connected sum tube, as in the 3-manifold case. There is a natural 1-handle map
\[
F_1^{c\bar \b,c \bar \b}\colon \cCFL(\Sigma,\as,\bs,\ws,\zs)\to \cCFL(\Sigma\# \bar \Sigma, \as c \bar \bs, \bs c\bar \bs,\ws\cup \bar \zs, \zs\cup \bar \ws).
\]

We pick a collection of properly embedded and pairwise disjoint arcs $\delta_1,\dots, \delta_{2g+2|L|-1}$ on $\Sigma\setminus N(p)$, such that after cutting $\Sigma\setminus N(p)$ along $\delta_1,\dots, \delta_{2g+2|L|-1}$, we are left with $2|L|$ punctured disks. We assume that each disk contains exactly one basepoint from $\ws\cup \zs$.  Equivalently, we may assume that $\delta_1,\dots, \delta_{2g+2|L|-1}$ forms a basis of $H_1(\Sigma\setminus (\ws\cup \zs), p)$. We double the arcs $\delta_1,\dots, \delta_{2g+2|L|-1}$ to obtain the attaching curves $\Ds$. See Figure~\ref{fig:11} for an example.

The complex $\cCFL(\Sigma\# \bar \Sigma, \bs c \bar \bs, \Ds, \ws\cup \bar \zs, \zs \cup \bar \ws)$ represents an $|L|$-component unlink in $\#^g S^1\times S^2$, where each component contains 4 basepoints. There are two canonical homology classes, represented by cycles $\Theta^{\ws}_{\b c \bar \b, \Dt}$ and $\Theta^{\zs}_{\b c \bar \b, \Dt}$. These are the generators of the top degree of homology with respect to the gradings $\gr_{\ws}$ and $\gr_{\zs}$, respectively. See \cite{ZemCFLTQFT}*{Lemma~3.7}. We define 
\[
f_{\a c \bar \b}^{\b c \bar \b\to\Dt;\zs}(-)=f_{\a c \bar \b, \b c \bar \b, \Dt}(-,\Theta^{\ws}_{\b c \bar \b ,\Dt})
\]
and we similarly define $f_{\a c \bar \b}^{\b c \bar \b \to \Dt;\ws}$ to be the holomorphic triangle map with special input $\Theta_{\b c \bar \b, \Dt}^{\zs}$. (The reason that the $\zs$-map has the $\ws$-generator as its input is due to the interpretation of the map in terms of decorated link cobordisms from \cite{ZemCFLTQFT}).

 Similarly, there are two distinguished classes $\Theta_{\Dt,\a c \bar \a}^{\ws},\Theta_{\Dt,\a c \bar \a}^{\zs}$ in $\cCFL(\Sigma\# \bar \Sigma, \Ds, \as c \bar \as, \ws\cup \bar \zs, \zs\cup \bar \ws)$.

We define the following two expanded models of the knot involution:
\[
\begin{split}
\tilde\iota_{L,+}&=\eta_K\circ F_{3}^{\a c, \a c} \circ f_{\a c \bar \a}^{\Dt\to \a c \bar \a;\ws} \circ f_{\a c \bar \a}^{\b c \bar \b \to \Dt;\zs}\circ F_1^{c \bar \b, c \bar \b}\\
\tilde\iota_{L,-}&=\eta_K\circ F_{3}^{\a c, \a c} \circ f_{\a c \bar \a}^{\Dt\to \a c \bar \a;\zs} \circ f_{\a c \bar \a}^{\b c \bar \b \to \Dt;\ws}\circ F_1^{c \bar \b, c \bar \b}.
\end{split}
\]

\begin{figure}[h]
\centering
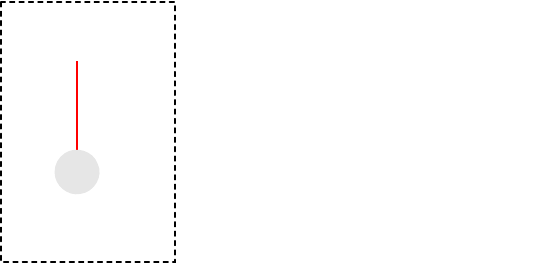
\caption{A doubly pointed Heegaard knot diagram with a special point $p$ (left), and the double (right).}
\label{fig:11}
\end{figure}

\begin{lem}
\label{lem:expanded-transition-links}
 The maps $\tilde \iota_{L,-}$ and $\tilde \iota_{L,+}$ are homotopic to the maps $\iota_{L,-}$ and $\iota_{L,+}$, respectively, which were defined in Section~\ref{sec:doubling-knots}.
\end{lem}
\begin{proof}
We focus on the case when $L$ is a doubly pointed knot $K$, since the proof for links is not substantially different.

The proof is to identify each of the above compositions with a cobordism map from \cite{ZemCFLTQFT}. The map $F_1^{c \bar \b, c \bar \b}$ is equal to a composition of $g$ 4-dimensional 1-handle maps, as well as the birth cobordism map for adding an unknot component. The map $F_3^{\a c, \a c}$ is similarly equal to the composition of $g$ 3-handle maps, as well as a death cobordism map which deletes an unknot.

 We now consider the triangle maps.   The map $f_{\a c \bar \a}^{\b c \bar \b\to \Dt;\zs}$ is the composition of $g$ 4-dimensional 2-handle maps (as in the ordinary doubling model of the involution), as well as the saddle cobordism map which attaches a band which has both of its ends in a component of $K\setminus \{w,z\}$ which is oriented from $z$ to $w$ (i.e. which lies in the beta handlebody). Furthermore, the band is contained in the type-$\zs$ subsurface of the link cobordism. The map $f_{\a c \bar \a}^{\Dt\to \a c \bar \a;\ws}$ has a similar description, except that the band is instead in the type-$\ws$ subsurface. 
 
 To make the identifications in the main statement, we argue as follows. By canceling the 1-handles and 2-handles, we may identify the composition $f_{\a c \bar \a}^{\b c \bar \b\to \Dt;\zs}\circ F_1^{c \bar \b, c \bar \b}$ with the composition of a birth cobordism, followed by a $\ws$-saddle map. By \cite{ZemCFLTQFT}*{Proposition~8.5}, we may identify this composition with the quasi-stabilization map $S^+_{\bar z, \bar w}$, from \cite{ZemCFLTQFT}*{Section~4} and \cite{ZemQuasi}. Analogously, we may identify the composition $F_3^{\a c, \a c}\circ f_{\a c \bar \a}^{\Dt\to \a c \bar \a;\ws}$ with $T^-_{z,w}$. The composition $T^-_{z,w}\circ S^+_{\bar z, \bar w}$ is homotopic to the basepoint moving map which sends $w$ to $\bar z$ and $z$ to $\bar w$, both moving in the positive direction by \cite{ZemCFLTQFT}*{Section~4.4}. See Figure~\ref{fig:9}. 
\end{proof}

\begin{figure}[h]
\centering
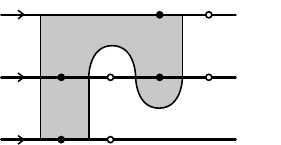
\caption{Decorations on a cylindrical link cobordism $[0,1]\times K$ corresponding to a composition of a quasi-stabilization followed by a quasi-destabilization.}
\label{fig:9}
\end{figure}

\begin{rem} The compositions $\eta_K\circ F_{3}^{\a c, \a c} \circ f_{\a c \bar \a}^{\Dt\to \a c \bar \a;\zs} \circ f_{\a c \bar \a}^{\b c \bar \b \to \Dt;\zs}\circ F_1^{c \bar \b, c \bar \b}$ and $\eta_K\circ F_{3}^{\a c, \a c} \circ f_{\a c \bar \a}^{\Dt\to \a c \bar \a;\ws} \circ f_{\a c \bar \a}^{\b c \bar \b \to \Dt;\ws}\circ F_1^{c \bar \b, c \bar \b}$ are null-homotopic. Indeed the same argument as above identifies the two maps with the link cobordism maps for a cylinder which contains a null-homotopic loop in the dividing set. Such a cobordism map induces the trivial map, since it may be identified with one of the compositions $S_{w,z}^- S_{w,z}^+$ or $T_{w,z}^- T_{w,z}^+$, which both vanish by \cite{ZemCFLTQFT}*{Lemma~4.13}.
\end{rem}

\begin{rem}
 As previously mentioned, if $L$ has $\ell$-components, there are $2^\ell$ natural models of the link involution, depending on how we twist each link component. The expanded model described above can also be used to realize any of these models.  We consider the diagram $(\Sigma\# \bar \Sigma, \bs c \bar \bs, \Ds, \ws\cup \bar \zs, \zs\cup \bar \ws)$. Suppose that $\mathbb{O}$ is a map from the set of link components of $L$ to the set $\{+,-\}$ (corresponding to a choice of whether to twist in the positive or negative direction). We let $\bZ(\bO)$  consist of the basepoints from $\ws\cup \bar \zs$, for components where $\bO=-$, as well as the basepoints from $\zs\cup \bar \ws$ for components where $\bO=+$. Let $\bW(\bO)$ denote the complement. Then each collection $\bZ(\bO)$ and $\bW(\bO)$ determine Maslov gradings on $\cCFL(\Sigma\# \bar \Sigma, \bs c \bar \bs, \Ds, \ws\cup \bar \zs, \zs\cup \bar \ws)$, and we let $\Theta^{\bZ(\bO)}$ and $\Theta^{\bW(\bO)}$ be the corresponding top-degree generators. A straightforward extension of Lemma~\ref{lem:expanded-transition-links} shows that the involution $\tilde \iota_{L,\bO}$ is computed by using $\Theta^{\bZ(\bO)}$ for the first triangle map, and $\Theta^{\bW(\bO)}$ for the second.
\end{rem}

\subsection{Naturality for knots and links}
\label{sec:naturality-knots-and-links}

Most of the ideas from the previous sections carry over to define transition maps and cobordism maps for the model of involutive link Floer homology described in Section~\ref{sec:expanded-involution-knots}.

 The main new technical subtlety is that the model of involutive link Floer homology in Section~\ref{sec:expanded-involution-knots} has an additional piece of data: the choice of connected sum point $p$. If $p$ and $p'$ are two choices of special points on $\Sigma$, a choice of path on $\Sigma$ from $p$ to $p'$ can be used to relate the two models where the connected sum is taken at $p$ or $p'$, by using a point-pushing diffeomorphism along this path. In principle, a transition map defined in this manner might depend on the choice of path. To handle this issue, we introduce further expansions of the doubling model, as follows.
 
 Suppose that $\ve{p}=\{p_1,\dots, p_n\}$ is a collection of marked points on $\Sigma\setminus (\as\cup \bs\cup \ws\cup \zs)$. We may form an expanded model of the involution, as follows. Firstly, we consider the surface $\Sigma\#_{\ve{p}} \bar \Sigma$, where we attach a connected sum tube at each point in $\ve{p}$. Let $\ve{c}=\{c_1,\dots, c_n\}$ denote a collection of meridians of the tubes. We pick properly embedded arcs $\delta_1,\dots, \delta_{2g+2|L|+|\ve{p}|-2}$, which form a basis of $H_1(\Sigma\setminus (\ve{w}\cup \ve{z}), \ve{p})$. It is straightforward to see that if $\Ds$ is constructed by doubling $\delta_1,\dots, \delta_{2g+2|L|+|\ve{p}|-2}$, then $(\Sigma\#_{\ve{p}} \bar \Sigma, \as \ve{c} \bar \bs, \Ds, \ws\cup \bar \zs, \zs\cup \bar \ws)$ is a diagram for $(Y,L)$, where each component of $L$ is given four basepoints.
 
 Given a collection of points $\ve{p}$ as well as doubling arcs, as above, the construction from Section~\ref{sec:expanded-involution-knots} adapts verbatim to produce a model of the link involution. It remains to relate the models for different choices of $\ve{p}$. 
 
Suppose $\cH=(\Sigma,\as,\bs,\ws,\zs)$ is a Heegaard link diagram, and $\ve{p}$ is a subset of $\Sigma\setminus (\as\cup \bs\cup \ws\cup \zs)$, and $q$ is a point on $\Sigma\setminus (\as\cup \bs \cup \ws \cup \zs \cup \ve{p})$. Suppose that  $\frD_{\ve{p}}$ and $\frD_{\ve{p}\cup \{q\}}$ are doubling enhancements of $\cH$, which use the connected sum points $\ve{p}$ and $\ve{p}\cup \{q\}$, respectively. We now define \emph{expansion} and \emph{contraction} morphisms
\[
\scE_{q}\colon \cCFKI(\frD_{\ve{p}})\to \cCFKI(\frD_{\ve{p}\cup \{q\}})\quad \text{and} \quad \scC_{q}\colon \cCFKI(\frD_{\ve{p}\cup \{q\}})\to \cCFKI(\frD_{\ve{p}}).
\]

We define the map $\scE_q$ in Equation~\eqref{eq:def-E_p}. Our definition of $\scE_{q}$ is similar to our definition of the 1-handle map, and $\scC_{q}$ is defined similarly to the 3-handle map.
\begin{equation}
\scE_q:= \begin{tikzcd}[column sep=2cm, row sep=2cm, labels=description]
\cCFK(\as,\bs)
	\ar[rr, "\id"]
	\ar[d,"F_1^{c_p\bar \b,c_p\bar \b}"]
&&
\cCFK(\as,\bs)
	\ar[d, "F_1^{c_p c_q \bar \b,c_p c_q\bar \b}"]
\\
\cCFK(\as\ve{c}_{\ve{p}}\bar \bs,\bs\ve{c}_{\ve{p}}\bar \bs)
	\ar[r, "F_1^{c_q,c_q}"]
	\ar[d, "f_{\a c_p \bar \b }^{\b c_p \bar \b\to \Dt;\zs}"]
&
\cCFK(\as\ve{c}_{\ve{p}}c_{q}\bar \bs,\bs\ve{c}_{\ve{p}}c_q\bar \bs)
	\ar[r, "\id"]
	\ar[d, "f_{\a c_p c_q \bar \b}^{\b c_p c_q \bar \b\to \Dt c_q;\zs}"]
	\ar[dr,dashed]
&
\cCFK(\as\ve{c}_{\ve{p}}c_{q}\bar \bs,\bs\ve{c}_{\ve{p}}c_q\bar \bs)
	\ar[d, "f_{\a c_p c_q \bar \b}^{\b c_p c_q \bar \b\to \Dt';\zs} "]
\\
\cCFK(\as\ve{c}_{\ve{p}}\bar \bs,\Ds)
	\ar[r, "F_1^{c_q,c_q}"]
	\ar[d,"f_{\a c_p \bar \b}^{\Dt\to \a c_p \bar \a;\ws}"]
&
\cCFK(\as\ve{c}_{\ve{p}}c_q\bar \bs,\Ds c_q)
	\ar[r, "f_{\a c_p c_q \bar \b}^{\Dt c_q\to \Dt'}"]
	\ar[d, "f_{\a c_p c_q \bar \b}^{\Dt c_q\to \a c_p c_q \bar \a;\ws}"]
	\ar[dr,dashed]
&
\cCFK(\as\ve{c}_{\ve{p}}c_q\bar \bs,\Ds')
	\ar[d,"f_{\a c_p c_q \bar \b}^{\Dt'\to \a c_p c_q \bar \a;\ws}"]
\\
\cCFK(\as\ve{c}_{\ve{p}}\bar \bs,\as \ve{c}_{\ve{p}} \bar \as)
	\ar[r, "F_1^{c_q,c_q}"]
	\ar[d,"F_3^{a c_p,a c_p}"]
&
\cCFK(\as\ve{c}_{\ve{p}}c_q\bar \bs,\as \ve{c}_{\ve{p}} c_q \bar \as)
	\ar[r, "B_\tau"]
&
\cCFK(\as\ve{c}_{\ve{p}}c_q\bar \bs,\as \ve{c}_{\ve{p}} c_q \bar \as)
	\ar[d, "F_3^{\a c_p c_q, \a c_p c_q}"]
\\
\cCFK(\bar \bs, \bar \as)
	\ar[rr, "\id"]
&
&
\cCFK(\bar \bs, \bar \as)
\end{tikzcd}
\label{eq:def-E_p}
\end{equation}
In the definition of $\scE_p$, we write $\tau$ for a closed curve on $\Sigma\#_{\ve{p}\cup \{q\}} \bar \Sigma$ which is symmetric with respect to the $\Z_2$-symmetry, and has a single transverse intersection with $c_q$.

\begin{lem}\,
\begin{enumerate}
\item\label{eq:E-C-1} The maps $\scE_{q}$ and $\scC_{q}$ commute with the transition maps for handleslides and isotopies of the attaching curves. Furthermore, $\scE_q$ and $\scC_q$ are independent of the choice of $\tau$, up to chain homotopy.
\item\label{eq:E-C-2} If $p\neq q$, then 
\[
[\scE_p,\scC_q]\simeq 0,\quad [\scE_p,\scE_q]\simeq 0\quad \text{and} \quad [\scC_p,\scC_q]\simeq 0.
\]
\item\label{eq:E-C-3} $\scE_p \circ \scC_p\simeq\id$ and $\scC_p\circ \scE_p\simeq \id$.
\item\label{eq:E-C-4} If $\phi$ is a surface isotopy of $\Sigma$ which moves $p$ to $q$ and fixes all of the other connected sum points and basepoints, then
\[
\phi_*\simeq \scC_p\circ \scE_q.
\]
\end{enumerate}
\end{lem}
\begin{proof}
The proof of part~\eqref{eq:E-C-1} is essentially the same as for the 1-handle and 3-handle maps. See Lemmas~\ref{lem:1-handle-map-ind-tau} and~\ref{lem:1-handles-and-handleslides}.

The proofs of claims~\eqref{eq:E-C-2} and \eqref{eq:E-C-3} amount to constructing five hyperboxes (one for each relation). All of these hyperboxes have a similar flavor, so we describe the hyperbox for the relation $\scC_q\circ \scE_q\simeq \id$, and leave the remaining hyperboxes to the reader.

The hyperbox we build will correspond to the hypercube
\[
\begin{tikzcd}[labels=description]\cCFKI(\frD_{\ve{p}})\ar[r, "\scE_q"] \ar[d, "\id"] \ar[dr,dashed]& \cCFKI(\frD_{\ve{p}\cup \{q\}}) \ar[d, "\scC_q"]\\
\cCFKI(\frD_{\ve{p}})\ar[r, "\id"]& \cCFKI(\frD_{\ve{p}})
\end{tikzcd}
\]
We build four constituent hyperboxes, each corresponding to one of the four levels of Equation~\eqref{eq:def-E_p}. The highest level is the following diagram:
\[
\begin{tikzcd}[
	column sep={2cm,between origins},
	row sep=.8cm,
	labels=description,
	fill opacity=.7,
	text opacity=1,
	]
\cCFK(\as,\bs)
	\ar[dr, "\id"]
	\ar[ddd, "F_1^{c_p\bar \b,c_p \bar \b}"]
	\ar[rr, "\id"]
&[.7 cm]
&\cCFK(\as,\bs)
	\ar[rd, "\id"]
	\ar[ddd,"F_1^{c_pc_q\bar \b,c_pc_q \bar \b}"]
&[.7 cm]
\\
&[.7 cm]
\cCFK(\as,\bs)
 	\ar[rr,line width=2mm,dash,color=white,opacity=.7]
	\ar[rr, "\id"]
&\,
&[.7 cm]
\cCFK(\as, \bs)
	\ar[ddd,"F_1^{c_p\bar \b,c_p \bar \b}"]
\\
\\
\cCFK(\as\ve{c}_{\ve{p}}\bar \bs, \bs \ve{c}_{\ve{p}} \bar \bs)
	\ar[dr, "\id"]
	\ar[rr, "F_1^{c_q,c_q}"]
&[.7 cm]\,&
\cCFK(\as \ve{c}_{\ve{p}} c_q \bar \bs, \bs \ve{c}_{\ve{p}} c_q \bar \bs)
	\ar[dr, "F_3^{c_q,c_q}\circ B_\tau"]
\\
& [.7 cm]
\cCFK(\as \ve{c}_{\ve{p}}\bar \bs, \bs\ve{c}_{\ve{p}} \bar \bs)
 	\ar[from=uuu,line width=2mm,dash,color=white,opacity=.7]
	\ar[from=uuu,"F_1^{\bar \b, \bar \b}"]
	\ar[rr, "\id"]
&
&[.7 cm] 
\cCFK(\as \ve{c}_{\ve{p}}\bar \bs, \bs \ve{c}_{\ve{p}}\bar \bs)
\end{tikzcd}
\]
The bottom-most level is very similar to the above, so we omit it.

We now consider the hyperbox corresponding to the second level from the top in Equation~\eqref{eq:def-E_p}:
\begin{equation}
\begin{tikzcd}[
	column sep={2.2cm,between origins},
	row sep={1.5cm,between origins},
	labels=description,
	fill opacity=.8,
	text opacity=1,
	execute at end picture={
	  \foreach \Number in  {A,B,C,D,E}
	      {\coordinate (\Number) at (\Number.center);}
	\filldraw[black] (A) node[rectangle,fill=white, text opacity=1]{$\cCFK(\as \ve{c}_{\ve{p}} \bar \bs, \bs \ve{c}_{\ve{p}} \bar \bs)$};
	\filldraw[black] (B) node[rectangle,fill=white, text opacity=1]{$\cCFK(\as \ve{c}_{\ve{p}} c_q \bar \bs, \bs \ve{c}_{\ve{p}} c_q \bar \bs)$};
	\filldraw[black] (C) node[rectangle,fill=white, text opacity=1]{$\cCFK(\as \ve{c}_{\ve{p}} c_q \bar \bs, \bs \ve{c}_{\ve{p}} c_q \bar \bs)$};
	\filldraw[black] (D) node[rectangle,fill=white, text opacity=1]{$\cCFK(\as \ve{c}_{\ve{p}} \bar \bs, \bs \ve{c}_{\ve{p}} \bar \bs)$};
	\filldraw[black] (E) node[rectangle,fill=white, text opacity=1]{$\cCFK(\as \ve{c}_{\ve{p}} \bar \bs, \bs \ve{c}_{\ve{p}} \bar \bs) $};
	}
	]
|[alias=A]|\phantom{\cCFK(\as \ve{c}_{\ve{p}} \bar \bs, \bs \ve{c}_{\ve{p}} \bar \bs)}
	\ar[ddd]
	\ar[rr, "F_1^{c_q,c_q}"]
	\ar[ddrr, "\id"]
&&
\cCFK(\as \ve{c}_{\ve{p}}c_{q}\bar \bs, \bs \ve{c}_{\ve{p}}c_{q} \bar \bs)
	\ar[rr, "\id"]
	\ar[ddd, shift right]
	\ar[dr, "B_\tau"]
	\ar[ddddr,dashed]
	\ar[dddrr,dashed]
&&
\cCFK(\as \ve{c}_{\ve{p}}c_{q}\bar \bs, \bs \ve{c}_{\ve{p}}c_{q} \bar \bs)
	\ar[dr, "B_\tau"]
	\ar[ddd]
	\ar[ddddr,dashed]
&&\,
\\
&&&
|[alias=B]|\phantom{\cCFK(\as \ve{c}_{\ve{p}} c_q \bar \bs, \bs \ve{c}_{\ve{p}} c_q \bar \bs)}
	\ar[rr, "\id", pos=.2]
	\ar[ddd]
	\ar[dr, "F_3^{c_q,c_q}"]
&&
|[alias=C]|\phantom{\cCFK(\as \ve{c}_{\ve{p}} c_q \bar \bs, \bs \ve{c}_{\ve{p}} c_q \bar \bs)}
	\ar[dr, "F_3^{c_q,c_q}"]
	\ar[ddd]
\\
&&|[alias=D]|\phantom{\cCFK(\as \ve{c}_{\ve{p}} \bar \bs, \bs \ve{c}_{\ve{p}} \bar \bs)}
	\ar[rr, "\id",pos=.7]
&&
|[alias=E]|\phantom{\cCFK(\as \ve{c}_{\ve{p}} \bar \bs, \bs \ve{c}_{\ve{p}} \bar \bs)}
	\ar[rr, "\id",pos=.7]
&&
\cCFK(\as \ve{c}_{\ve{p}} \bar \bs, \bs \ve{c}_{\ve{p}} \bar \bs)
\\
\cCFK(\as \ve{c}_{\ve{p}} \bar \bs, \Ds)
	\ar[ddrr, "\id"]
	\ar[rr, "F_1^{c_q,c_q}"]
&&
\cCFK(\as \ve{c}_{\ve{p}} c_q \bar \bs, \Ds c_q)
	\ar[rr]
	\ar[dr, "B_\tau"]
	\ar[drrr,dashed]
&&
\cCFK(\as \ve{c}_{\ve{p}} c_q \bar \bs, \Ds')
	\ar[dr]
\\
&
&& 
\cCFK(\as \ve{c}_{\ve{p}} c_q \bar\bs, \Ds c_q )
	\ar[rr, "\id"]
	\ar[dr, "F_3^{c_q,c_q}"]
&&
\cCFK(\as \ve{c}_{\ve{p}} c_q \bar \bs, \Ds c_q)
	\ar[dr, "F_3^{c_q,c_q}"]
\\
\,&&
\cCFK(\as \ve{c}_{\ve{p}} \bar \bs, \Ds)
	\ar[rr,"\id"]
 	\ar[from=uuu,line width=2mm,dash,color=white,opacity=.7]
	\ar[from=uuu, pos=.6]
&&
\cCFK(\as \ve{c}_{\ve{p}} \bar \bs, \Ds)
	\ar[rr,"\id"]
 	\ar[from=uuu,line width=2mm,dash,color=white,opacity=.7, shift left]
	\ar[from=uuu, shift left]
&&
\cCFK(\as \ve{c}_{\ve{p}} \bar \bs, \Ds)
 	\ar[from=uuu,line width=2mm,dash,color=white,opacity=.7]
	\ar[from=uuu]
\end{tikzcd}
\label{eq:main-hypercube-C_q-E_q=id}
\end{equation}
In Equation~\eqref{eq:main-hypercube-C_q-E_q=id}, each unlabeled arrow is a triangle or quadrilateral map. The back right hypercube also has a length 3 map, which is not labeled. We now explain the back-right hypercube. This hypercube is obtained by pairing $\as \ve{c}_{\ve{p}} c_q \bar \bs$ with the following hypercube of beta attaching curves:
\begin{equation}
\begin{tikzcd}[
	column sep={2cm,between origins},
	row sep=.8cm,
	labels=description,
	fill opacity=.7,
	text opacity=1,
	]
\bs \ve{c}_{\ve{p}} c_q \bar \bs
	\ar[dr, "\tau"]
	\ar[ddd, "\Theta_{\b c_p c_q \bar \b, \Dt c_q}^{\zs}"]
	\ar[rr, "1"]
	\ar[dddrr,dashed]
	\ar[ddddrrr,dotted]
&[.7 cm]
&\bs \ve{c}_{\ve{p}} c_q \bar \bs
	\ar[rd, "\tau"]
	\ar[ddd,"\Theta_{\b c_p c_q \bar \b, \Dt'}^{\zs}"]
	\ar[ddddr,dashed]
&[.7 cm]
\\
&[.7 cm]
\bs \ve{c}_{\ve{p}} c_q \bar \bs
 	\ar[rr,line width=2mm,dash,color=white,opacity=.7]
	\ar[rr, "1"]
&\,
&[.7 cm]
\bs \ve{c}_{\ve{p}} c_q \bar \bs
	\ar[ddd,"\Theta_{\b c_p c_q \bar \b, \Dt c_q}^{\zs}"]
\\
\\
\Ds c_q
	\ar[dr, "\tau"]
	\ar[rr, "\Theta_{\Dt c_q,\Dt'}", pos=.4]
	\ar[drrr,dashed, "\eta"]
&[.7 cm]\,&
\Ds'
	\ar[dr, "\Theta_{\Dt',\Dt c_q}"]
\\
& [.7 cm]
\Ds c_q
 	\ar[from=uuu,line width=2mm,dash,color=white,opacity=.7]
	\ar[from=uuu,"\Theta_{\b c_p c_q \bar \b, \Dt c_q}^{\zs}"]
	\ar[rr, "1"]
&
&[.7 cm] 
\Ds c_q
\end{tikzcd}
\label{eq:3-d-hyperucbe-[C,E]}
\end{equation}
The above hypercube of beta attaching curves is constructed as follows. Length 2 chains are constructed by adapting Lemma~\ref{lem:1-handle-bottom-hypercube}. Once these are chosen, a length 3 chain may be chosen by the standard cube-filling procedure. Namely, if $C$ is the sum of all other terms of the length 3 relation, we know $\d C=0$ because of the hypercube relations on the other proper faces, and $C$ is in $\cCFK(\bs \ve{c}_{\ve{p}} c_q \bar \bs, \Ds c_q)$, a complex for an unlink in a connected sum of several copies of $S^1\times S^2$, of $\gr_{\zs}$-grading equal to that of the top degree generator of homology. By potentially adding a copy of the top degree generator to the 2-chain labeled $\eta$, in Equation~\eqref{eq:3-d-hyperucbe-[C,E]}, the hypercube relations may be satisfied.

The hyperbox for the third level of Equation~\eqref{eq:def-E_p} is constructed very similarly. Indeed, if we switch the roles of the alpha and beta curves, then the hypercube for the third level of Equation~\eqref{eq:def-E_p} is obtained by dualizing Equation~\eqref{eq:main-hypercube-C_q-E_q=id}. Stacking these hyperboxes and compressing gives the relation $\scC_q\circ \scE_q\simeq \id$. 

One important point to check is that the diagonal maps on the second and third hypercubes can be chosen to match. For example, to build the hypercube in Equation~\eqref{eq:main-hypercube-C_q-E_q=id}, we may potentially have to add the top degree generator to $\eta$ so that the hypercube relations are satisfiable. A similar issue occurs in constructing the length 2 map on the top face of the third hypercube as well. Nonetheless, as in the proof of Lemma~\ref{lem:1-handles-and-handleslides}, addition of the top degree generator to $\eta$ in Equation~\eqref{eq:3-d-hyperucbe-[C,E]} has no effect on the diagonal map in the compression of Equation~\eqref{eq:main-hypercube-C_q-E_q=id}. Hence, the compression of Equation~\eqref{eq:main-hypercube-C_q-E_q=id} may be stacked on top of the corresponding hypercube below it.

We now consider part~\eqref{eq:E-C-4}. This follows from the first claims, as we now explain. Let $\lambda$ be a path from $p$ to $q$. Let $\phi_{q\to p}$ denote the surface isotopy which moves $q$ to $p$ along $\lambda$, and is supported in a small neighborhood of $\lambda$. It is sufficient to show that
\[
\phi_{q\to p}\circ \scC_p\circ \scE_q\simeq \id.
\]
Let $x_0$ be a point in the interior of $\lambda$. We compute as follows, using the previous relations:
\[
\begin{split}
\phi_{q\to p}\circ \scC_p\circ \scE_q&\simeq  \phi_{q\to p}\circ \scC_p\circ \scC_{x_0} \circ \scE_{x_0}\circ  \scE_q\\
&\simeq \phi_{q\to p} \circ \scC_{x_0} \circ\scE_{q} \circ \scC_{p}\circ \scE_{x_0}\\
&\simeq  \scC_{x_0} \circ \phi_{q\to p} \circ\scE_{q} \circ \scC_{p}\circ \scE_{x_0}\\
&\simeq \scC_{x_0} \circ \scE_{p} \circ \scC_{p}\circ \scE_{x_0}\\
&\simeq \id.
\end{split}
\]
Note that commuting $\phi_{q\to p}$ past $\scC_{x_0}$ in the second line is possible since the diffeomorphism $\phi_{q\to p}$ is isotopic to one which fixes $x_0$. Next, $\phi_{q\to p}\circ \scE_q\simeq \scE_p$ by naturality of the expansion and contraction maps with respect to changes of attaching curves, which is part~\eqref{eq:E-C-1} of the statement.  This establishes part~\eqref{eq:E-C-4}.
\end{proof}

\subsection{Functoriality for knots and links}

In this section, we describe several aspects of the proof of Theorem~\ref{thm:functoriality}. Firstly, to construct cobordism maps for $(W,\Sigma)\colon (Y_1,L_1)\to (Y_2,L_2)$ when $\Sigma$ is a collection of annuli such that each component of $\Sigma$ intersects $Y_1$ and $Y_2$ non-trivially, it is sufficient to describe cobordism maps on link Floer homology for 4-dimensional handles attached to the complement of a link. 

Cobordism maps for 4-dimensional handles attached in the complement of a link are defined similarly to the maps for cobordisms between closed 3-manifolds (cf. \cite{ZemCFLTQFT}*{Section~5}). The 1-handle and 3-handle maps are defined as in Section~\ref{sec:1-and-3-handles}, and the 2-handle maps are defined similarly to Section~\ref{sec:2-handles}. One caveat is that there is an asymmetry in the definition of the 2-handle maps with respect to the basepoints, in that we choose the 2-handle maps to count holomorphic triangles which represent classes satisfying $\frs_{\ws}(\psi)=\frs$. Note that by \cite{ZemAbsoluteGradings}*{Lemma~3.8}, we have
\[
(\frs_{\ws}-\frs_{\zs})(\psi)=\PD[\Sigma].
\]
In particular, the final hypercube we use in constructing the cobordism maps for our 2-handle map will have the following form:
\[
\begin{tikzcd}[column sep=3cm, labels=description]
\cCFL(\bar \Sigma, \bar \bs, \bar \as)
	\ar[r, "f_{\bar \b;\frs}^{\bar \a\to \bar \a'}"]
	\ar[d, "\eta_L"]
&
\cCFL(\bar \Sigma, \bar \bs, \bar \as')
	\ar[d,"\eta_L"]
\\
\cCFL(\Sigma,\as, \bs)
	\ar[r, "f_{\a\to \a', \bar \frs+\PD[\Sigma]}^{\b}"]
&
\cCFL(\Sigma,\as',\bs)
\end{tikzcd}
\]

%% file: fig11.pdf_tex
\begingroup%
  \makeatletter%
  \providecommand\color[2][]{%
    \errmessage{(Inkscape) Color is used for the text in Inkscape, but the package 'color.sty' is not loaded}%
    \renewcommand\color[2][]{}%
  }%
  \providecommand\transparent[1]{%
    \errmessage{(Inkscape) Transparency is used (non-zero) for the text in Inkscape, but the package 'transparent.sty' is not loaded}%
    \renewcommand\transparent[1]{}%
  }%
  \providecommand\rotatebox[2]{#2}%
  \newcommand*\fsize{\dimexpr\f@size pt\relax}%
  \newcommand*\lineheight[1]{\fontsize{\fsize}{#1\fsize}\selectfont}%
  \ifx\svgwidth\undefined%
    \setlength{\unitlength}{259.45304559bp}%
    \ifx\svgscale\undefined%
      \relax%
    \else%
      \setlength{\unitlength}{\unitlength * \real{\svgscale}}%
    \fi%
  \else%
    \setlength{\unitlength}{\svgwidth}%
  \fi%
  \global\let\svgwidth\undefined%
  \global\let\svgscale\undefined%
  \makeatother%
  \begin{picture}(1,0.48613243)%
    \lineheight{1}%
    \setlength\tabcolsep{0pt}%
    \put(0,0){\includegraphics[width=\unitlength,page=1]{fig11.pdf}}%
    \put(0.14287017,0.15221802){\makebox(0,0)[t]{\lineheight{1.25}\smash{\begin{tabular}[t]{c}\rotatebox{90}{B}\end{tabular}}}}%
    \put(0,0){\includegraphics[width=\unitlength,page=2]{fig11.pdf}}%
    \put(0.14287017,0.41485805){\makebox(0,0)[t]{\lineheight{1.25}\smash{\begin{tabular}[t]{c}\rotatebox{-90}{B}\end{tabular}}}}%
    \put(0,0){\includegraphics[width=\unitlength,page=3]{fig11.pdf}}%
    \put(0.28168968,0.28386859){\makebox(0,0)[lt]{\lineheight{1.25}\smash{\begin{tabular}[t]{l}$p$\end{tabular}}}}%
    \put(0,0){\includegraphics[width=\unitlength,page=4]{fig11.pdf}}%
    \put(0.5381514,0.15241869){\makebox(0,0)[t]{\lineheight{1.25}\smash{\begin{tabular}[t]{c}\rotatebox{90}{B}\end{tabular}}}}%
    \put(0,0){\includegraphics[width=\unitlength,page=5]{fig11.pdf}}%
    \put(0.85412639,0.41505872){\makebox(0,0)[t]{\lineheight{1.25}\smash{\begin{tabular}[t]{c}\rotatebox{-90}{E}\end{tabular}}}}%
    \put(0.85094112,0.15241869){\makebox(0,0)[t]{\lineheight{1.25}\smash{\begin{tabular}[t]{c}\rotatebox{90}{E}\end{tabular}}}}%
    \put(0,0){\includegraphics[width=\unitlength,page=6]{fig11.pdf}}%
    \put(0.5381514,0.41505872){\makebox(0,0)[t]{\lineheight{1.25}\smash{\begin{tabular}[t]{c}\rotatebox{-90}{B}\end{tabular}}}}%
    \put(0,0){\includegraphics[width=\unitlength,page=7]{fig11.pdf}}%
    \put(0.70559371,0.4477453){\color[rgb]{1,0,0}\makebox(0,0)[lt]{\lineheight{1.25}\smash{\begin{tabular}[t]{l}$c$\end{tabular}}}}%
    \put(0.21216524,0.18735477){\makebox(0,0)[lt]{\lineheight{1.25}\smash{\begin{tabular}[t]{l}$w$\end{tabular}}}}%
    \put(0.19625528,0.32859264){\makebox(0,0)[lt]{\lineheight{1.25}\smash{\begin{tabular}[t]{l}$z$\end{tabular}}}}%
    \put(0.61385736,0.19264826){\makebox(0,0)[lt]{\lineheight{1.25}\smash{\begin{tabular}[t]{l}$w$\end{tabular}}}}%
    \put(0.61626756,0.32567973){\makebox(0,0)[lt]{\lineheight{1.25}\smash{\begin{tabular}[t]{l}$z$\end{tabular}}}}%
    \put(0.81309659,0.21433961){\makebox(0,0)[lt]{\lineheight{1.25}\smash{\begin{tabular}[t]{l}$\bar w$\end{tabular}}}}%
    \put(0.78553995,0.32567973){\makebox(0,0)[rt]{\lineheight{1.25}\smash{\begin{tabular}[t]{r}$\bar z$\end{tabular}}}}%
  \end{picture}%
\endgroup%

%% file: fig9.pdf_tex
\begingroup%
  \makeatletter%
  \providecommand\color[2][]{%
    \errmessage{(Inkscape) Color is used for the text in Inkscape, but the package 'color.sty' is not loaded}%
    \renewcommand\color[2][]{}%
  }%
  \providecommand\transparent[1]{%
    \errmessage{(Inkscape) Transparency is used (non-zero) for the text in Inkscape, but the package 'transparent.sty' is not loaded}%
    \renewcommand\transparent[1]{}%
  }%
  \providecommand\rotatebox[2]{#2}%
  \newcommand*\fsize{\dimexpr\f@size pt\relax}%
  \newcommand*\lineheight[1]{\fontsize{\fsize}{#1\fsize}\selectfont}%
  \ifx\svgwidth\undefined%
    \setlength{\unitlength}{134.42343539bp}%
    \ifx\svgscale\undefined%
      \relax%
    \else%
      \setlength{\unitlength}{\unitlength * \real{\svgscale}}%
    \fi%
  \else%
    \setlength{\unitlength}{\svgwidth}%
  \fi%
  \global\let\svgwidth\undefined%
  \global\let\svgscale\undefined%
  \makeatother%
  \begin{picture}(1,0.51685688)%
    \lineheight{1}%
    \setlength\tabcolsep{0pt}%
    \put(0,0){\includegraphics[width=\unitlength,page=1]{fig9.pdf}}%
    \put(0.21662823,0.04690072){\makebox(0,0)[t]{\lineheight{1.25}\smash{\begin{tabular}[t]{c}$w$\end{tabular}}}}%
    \put(0.39369258,0.04690072){\makebox(0,0)[t]{\lineheight{1.25}\smash{\begin{tabular}[t]{c}$z$\end{tabular}}}}%
    \put(0.21909281,0.26601591){\makebox(0,0)[t]{\lineheight{1.25}\smash{\begin{tabular}[t]{c}$w$\end{tabular}}}}%
    \put(0.39615717,0.26601591){\makebox(0,0)[t]{\lineheight{1.25}\smash{\begin{tabular}[t]{c}$z$\end{tabular}}}}%
    \put(0.56979098,0.49395817){\makebox(0,0)[t]{\lineheight{1.25}\smash{\begin{tabular}[t]{c}$\bar z$\end{tabular}}}}%
    \put(0.74737403,0.49395817){\makebox(0,0)[t]{\lineheight{1.25}\smash{\begin{tabular}[t]{c}$\bar w$\end{tabular}}}}%
    \put(0.56801723,0.26601591){\makebox(0,0)[t]{\lineheight{1.25}\smash{\begin{tabular}[t]{c}$\bar z$\end{tabular}}}}%
    \put(0.74560028,0.26601591){\makebox(0,0)[t]{\lineheight{1.25}\smash{\begin{tabular}[t]{c}$\bar w$\end{tabular}}}}%
    \put(0,0){\includegraphics[width=\unitlength,page=2]{fig9.pdf}}%
    \put(0.91158928,0.11927067){\makebox(0,0)[lt]{\lineheight{1.25}\smash{\begin{tabular}[t]{l}$S^+_{\bar w, \bar z}$\end{tabular}}}}%
    \put(0.91109008,0.34802537){\makebox(0,0)[lt]{\lineheight{1.25}\smash{\begin{tabular}[t]{l}$T^-_{z,w}$\end{tabular}}}}%
    \put(0.7312473,0.03275223){\makebox(0,0)[lt]{\lineheight{1.25}\smash{\begin{tabular}[t]{l}$K$\end{tabular}}}}%
  \end{picture}%
\endgroup%

%% file: section-surgery-formula.tex
\section{2-handle maps and the knot surgery formula}
\label{sec:knot-surgery}

In this section, we compute another family of involutive cobordism maps. If $K\subset S^3$, and $n\in \Z$, we will compute the map for the 2-handle cobordism from $S_{2n}^3(K)$ to $S^3$. Our description will be in terms of the knot surgery formula of \cite{HHSZExact}, and is in a similar spirit to the non-involutive description of the 2-handle map of Ozsv\'{a}th and Szab\'{o} \cite{OSIntegerSurgeries}*{Theorem~1.1}.

We suppose that $K\subset S^3$ is a knot, and we consider the 2-handle cobordism $W'_m(K)$ from $S^3_m(K)$ to $S^3$. This cobordism is $\Spin$ if and only if $m$ is even. Furthermore, if $m=2n$, then the unique self-conjugate $\Spin^c$ structure on $W'_{2n}(K)$ is the one which satisfies $\langle c_1(\frs),\Sigma\rangle =0$, where $\Sigma$ is obtained by capping $K$ with the core of the 2-handle. This $\Spin^c$ structure restricts to $[n]\in \Spin^c(S^3_{2n}(K))$ under the isomorphism $\Spin^c(S^3_{2n}(K))\iso \Z/2n$.

Our formula is phrased in terms of the mapping cone formula of Ozsv\'{a}th and Szab\'{o} \cite{OSIntegerSurgeries} and its involutive analog from \cite{HHSZExact}. We review these constructions briefly. We write $A_s(K)$ for the subcomplex of $\CFK^\infty(K)$ generated by elements $\scU^i \scV^j \ve{x}$ where $A(\ve{x})+j-i=0$ and $j\ge -s$ and $i\ge 0$. We write $B_s(K)$ for the subspace generated by similar monomials, with no restriction on $j$. We write $\tilde{B}_s(K)$ for the subspace generated by monomials which satisfy $j\ge -s$, but with no restriction on $i$. We write $v_s\colon A_s\to B_s$ and $\tilde{v}_s\colon A_s\to \tilde{B}_s$ for the canonical inclusions. We recall from \cite{OSIntegerSurgeries} that the mapping cone formula requires a choice of homotopy equivalence from $\tilde{B}_s(K)$ to $B_{s+m}(K)$. We write $\frF_m$ for this map. We write $\bA(K)$ for the direct product over $s$ of $A_s(K)$, tensored with the power series ring $\bF[[U]]$. Similarly $\bB(K)$ denotes the direct product of the complexes $B_s(K)$, tensored with $\bF[[U]]$. 

We recall that Ozsv\'{a}th and Szab\'{o} proved that if $K\subset S^3$ and $m\in \Z$, then
\[
\ve{\CF}^-(S_m^3(K))\iso \Cone(v+h_m\colon \bA(K)\to \bB(K)),
\]
where $h_m=\frF_m\tilde v$. In \cite{HHSZExact}, we extended this to the involutive setting, and proved that
\[
\ve{\CFI}^-(S^3_m(K))\simeq \begin{tikzcd}[column sep=2cm]\bA(K)\ar[dr,dashed, "H_m\tilde v"]\ar[r, "v+h_m"]\ar[d, "Q(\id+\iota_{\bA})"]& \bB(K)\ar[d, "Q(\id+\iota_{\bB})"]\\
Q\bA(K)\ar[r, "v+h_m"]& Q \bB(K)
\end{tikzcd}
\]

The map $\iota_{\bA} \colon A_s \rightarrow A_{-s}$ is given by $\iota_{\bA} = U^s \iota_K$ and the map $\iota_{\bB}:B_s \rightarrow B_{-s+m}$ is given by $\frF_m U^s \iota_K$. The map $H_m$ is defined by picking a map $H_m\colon \tilde{B}_s(K)\to B_{-s}(K)$ satisfying
\begin{equation} \label{eqn:H}
U^s\iota_K+\frF_m U^s \iota_K \frF_m=[\d, H_m].
\end{equation}
Indeed, any two choices of map $H_m$ satisfying this equation are themselves related by appropriate homotopies. See \cite{HHSZExact}*{Section~3.5}.

 With the above notation in place, we now state our formula for the 2-handle map. Once again, because one end of the cobordism is the three-sphere, the formula is not affected by choices of framings, and we omit them from the notation. 
\begin{thm}\label{thm:knot-surgery-cobordism-map} Let $\frs$ be the unique self-conjugate $\Spin^c$ structure on the 2-handle cobordism $W'_{2n}(K)$ from $S^3_{2n}(K)$ to $S^3$. Write $\mathit{BI}_n(K):=\Cone(Q(\id+\iota_{\bB})\colon B_n(K)\to Q B_n(K))$, and let
 \[
J\colon \bXI_n(K)\to \BI_n(K)
\]
be the map $J:=v_n \Pi^A_n+Q\Pi_n^B\iota_{\bX}$, where $\Pi_s^A$ is projection onto $A_s\oplus Q A_s$ and $\Pi_s^B$ is similar. Then there is a hypercube of $\bF[U,Q]/Q^2$-equivariant maps
\[
\begin{tikzcd}[row sep=2cm, column sep=3cm, labels=description]
\CFI^-(S_{2n}^3(K)) \ar[dr,dashed] \ar[d,"\simeq"]\ar[r, "{\CFI(W'_{2n}(K),\frs)}"]& \CFI^-(S^3) \ar[d, "\simeq"]\\
\bXI_{2n}(K)\ar[r, "J"]& \BI_n(K)
\end{tikzcd}
\]
\end{thm}
\begin{rem} In Theorem~\ref{thm:knot-surgery-cobordism-map}, $\iota_{\bX}$ denotes the involution on the mapping cone, i.e. $\iota_{\bX}=\iota_{\bA}+\iota_{\bB}+H_{2n} \tilde v$. Note $\iota_{\bA}$ vanishes when composed with $\Pi_n^B$, and hence makes no contribution to the map $J$. 
\end{rem}

\begin{example} Consider the case $K=U$ and $n=0$, so that $W'_{2n}(K)$ is the cobordism $W_2$ from $S^1\times S^2$ to $S^3$ considered in Section~\ref{sec:example}. We recall that $\bX_0(K)$ may be taken to be the cone of $A_0(K) \xrightarrow{v+h_0} B_0(K)$ \cite[Section 4.8]{OSIntegerSurgeries}. In \cite[Section 22.9]{HHSZExact} we show that $\CFI^-(S_{0}^3(K)) \simeq \bXI_0(K)$, to wit the complex
\[\begin{tikzcd}[column sep=2cm]A_0(K)\ar[dr,dashed, "H_0\tilde v"]\ar[r, "v+h_0"]\ar[d, "Q(\id+\iota_{\bA})"]& B_0(K)\ar[d, "Q(\id+\iota_{\bB})"]\\
QA_0(K)\ar[r, "v+h_m"]& Q B_0(K).
\end{tikzcd}\]
With this in mind, $BI_0(K) = B_0(K) \oplus QB_0(K)$ with appropriate gradings and differentials. We take $A_0(U)$ and $B_0(U)$ to be appropriately-graded copies of $\bF[U]$. For clarity let $a$ be the generator of $A_0(U)$ and $b$ be the generator of $B_0(U)$. Choosing $H_0 \equiv 0$ satisfies the formula~\ref{eqn:H}. With this in mind, the map $J$ sends $J(a) = b$ and $J(b) = Qb$, and correspondingly $J(Qa) = Qb$ and $J(Qb)=0$. This differs from the map associated to the cobordism computed in Section~\ref{sec:example} by a change of basis; setting $\Theta^{+}=a$ and $\Theta^{-} = a+Qb$ recovers the map as described there.
\end{example}

We will focus on proving the case that $n\neq 0$. The case that $n=0$ is similar.

We begin by considering the non-involutive analog of the above theorem, which is similar to \cite{OSIntegerSurgeries}*{Theorem~1.1}.  Ozsv\'{a}th and Szab\'{o} constructed a map from $\ve{\CF}^-(S^3_{2n}(K))$ to $\Cone(\ve{\CF}^-(S^3_{2n+m}(K))\to \underline{\ve{\CF}}^-(S^3))$, which we think of as a hypercube
\begin{equation}
\begin{tikzcd}
\ve{\CF}^-(S^3_{2n}(K))\ar[d] \ar[dr,dashed, "k"]\\
\ve{\CF}^-(S^3_{2n+m}(K))\ar[r]& \underline{\ve{\CF}}^-(S^3).
\label{eq:surgery-quasi-iso-main}
\end{tikzcd}
\end{equation}
 In the above, $\underline{\ve{\CF}}^-(S^3)$ denotes a version of twisted Floer homology which is isomorphic to  $\ve{\CF}^-(S^3)\otimes\bF[T]/(T^m-1)$.
 
 We may consider instead a version of the above hypercube with twisted coefficients lying in $\bF[T]$.   We will write $\tilde{\ve{\CF}}{}^-(S^3)$ for this version. We can also construct an analog of Equation~\eqref{eq:surgery-quasi-iso-main} over $\bF[T]$, though the resulting hypercube has a different shape:
 \begin{equation}
  \begin{tikzcd}
  \ve{\CF}^-(S^3_{2n}(K))\ar[d] \ar[dr,dashed, "k"] \ar[r, "F"]& \tilde{\ve{\CF}}{}^-(S^3)\ar[d, "M"]\\
  \ve{\CF}^-(S^3_{2n+m}(K))\ar[r, "G"]& \tilde{\ve{\CF}}{}^-(S^3).
  \label{eq:surgery-quasi-iso-pre-quotient}
  \end{tikzcd}
\end{equation}
  In Equation~\eqref{eq:surgery-quasi-iso-pre-quotient}, $F$ denotes 
  \[
  F=\sum_{\frs\in \Spin^c(W_{2n}'(K))} T^{A_{W_{2n}',\Sigma}(\frs)}\CF(W_{2n}'(K),\frs)
  \]
where
 \[
 A_{W_{2n}',\Sigma}(\frs):=\frac{\langle c_1(\frs),\widehat \Sigma\rangle +2n}{2}.
 \]
Additionally,
  \[
  M=\sum_{s\in 2\Z+1} U^{m(s^2-1)/8}T^{m(s+1)/2}.
  \]

Note that $M=0$ if we set $T^m=1$. Hence, the hypercube in Equation~\eqref{eq:surgery-quasi-iso-main} is obtained from Equation~\eqref{eq:surgery-quasi-iso-pre-quotient}   if we set $T^m=1$ and remove the top right corner of the hypercube.
 
The formula for $F$ is derived as follows. Ozsv\'{a}th and Szab\'{o} proved Equation~\eqref{eq:surgery-quasi-iso-main} by considering a Heegaard quadruple $(\Sigma,\as_3,\as_2,\as_1,\bs,w,z)$, with a special genus 1 surgery region, as in \cite{OSIntegerSurgeries}*{Section~3}. In the diagrams $(\Sigma,\as_1,\bs)$ and $(\Sigma,\as_2,\bs)$, the point $w$ and $z$ are immediately adjacent. The diagram $(\Sigma,\as_1,\bs,w)$ represents $S^3_{2n}(K)$. The diagram $(\Sigma,\as_2,\bs,w)$ represents $S^3_{2n+m}(K)$. The diagram $(\Sigma,\as_3,\bs,w,z)$ represents $(S^3,K)$. One obtains the hypercube in Equation~\eqref{eq:surgery-quasi-iso-pre-quotient} by considering the degenerations of holomorphic quadrilaterals. The quantity $U^{m(s^2-1)/8}T^{m(s+1)/2}$ arises from a model count of holomorphic triangles on the diagram $(\Sigma,\as_3,\as_2,\as_1,w,z)$ (see \cite{OSIntegerSurgeries}*{Section~3}).

We now fix $\delta\gg 0$, and we set $U^\delta=0$. Write $\CF^\delta$ for $\CF^-/U^\delta$, and make similar notation for maps. We assume that $m\gg 0$, so the only terms of $F$ with $U$-power less than $\delta$ have $s\in \{1,-1\}$. Hence
\[
F^\delta=(1+T^{m})\sum_{
\\ \frs\in \Spin^c(W_{2n}'(K))} T^{A_{W_{2n}',\Sigma}(\frs)} \CF^\delta(W_{2n}'(K),\frs).
\]

 On $D(-m,1)$, $\Spin^c$ structures may be identified with $2\Z+1$, where $s\in 2\Z+1$ corresponds to the $\Spin^c$ structure $\frt$ satisfying
\[
\langle c_1(\frt),S^2\rangle=s\cdot m.
\]
We are most interested in the $\Spin^c$ structures corresponding to $s=\pm 1$, which are the $\Spin^c$ structures with maximal square. Note that the composition of the natural cobordisms from $S_{2n}^3(K)$ to $S^3_{2n+m}(K)$ (which has an extra boundary component $L(m,1)$) and the cobordism from $S^3_{2n+m}(K)$ to $S^3$ is $W'_{2n}(K)\# D(-m,1)$. We identify $\Spin^c$ structures on $W_{2n}'(K)\#D(-m,1)$ as pairs $\frs\#[s]$ where $\frs\in \Spin^c(W'_{2n}(K))$ and $s\in 2\Z+1$.

We recall the notation of Ozsv\'{a}th and Szab\'{o} for $\Spin^c$ structures on $W_{2n}'(K)$. They write $\frx_s$ and $\fry_s$ for the $\Spin^c$ structures which satisfy
\[
\langle c_1(\frx_s),\Sigma\rangle +2n=2s\quad \text{and} \quad \langle c_1(\fry_s),\Sigma\rangle -2n=2s,
\]
respectively.

There are also analogous $\Spin^c$ structures on $W_{2n+m}'(K)$, for which we write $\frX_s$ and $\frY_s$. The $\Spin^c$ structure $\frX_s$ satisfies $\langle c_1(\frX_s),\Sigma\rangle +2n+m=2s$ and $\frY_s$ is similar. We write
\[
G_{\frX_s},G_{\frY_s}\colon \ve{\CF}^-(S_{2n+m}^3(K))\to \ve{\CF}^-(S^3)
\]
for the corresponding cobordism maps.

If $s\in \Z$, we may define a projection map $\Pi_{s}\colon \tilde{\CF}{}^{\delta} (S^3)\to \CF^\delta(S^3)$ by reading off only the component $T^s$. Using this map, we may form a hypercube
\[
 \begin{tikzcd}
 \CF^\delta(S^3_{2n}(K),[s])
 	\ar[d] 
 	\ar[dr,dashed, "k"]
 	\ar[r, "F"]
& \tilde{\CF}{}^\delta(S^3)
 	\ar[d, "M"]
 	\ar[r, "\Pi_s M"]
&\CF^\delta(S^3)
	\ar[d, "\id"]
	\\
\CF^\delta(S^3_{2n+m}(K))
	\ar[r, "G"]
&\tilde{\CF}{}^\delta(S^3)
	\ar[r, "\Pi_s"]
&\CF^\delta(S^3).
 \end{tikzcd}
\] 
If we compress the above hyperbox, we obtain the following diagram:
\[
 \begin{tikzcd}[column sep=3cm]
 \CF^\delta(S^3_{2n}(K),[s])
 	\ar[d]
 	\ar[dr,dashed, "\Pi_s k"]
 	\ar[r, "{\CF(W'_{2n},\frx_s)}"]
&
\CF^\delta(S^3)
	\ar[d, "\id"]
\\
\CF^\delta(S^3_{2n+m}(K))
	\ar[r, "G_{\frX_s}"]
&
\CF^\delta(S^3)
 \end{tikzcd}
\]
If we instead use $\Pi_{s+2n+m}$, we obtain a hypercube relating the maps $\CF(W_{2n}',\fry_s)$ and $G_{\frY_s}$.

The next step of the knot surgery formula is to construct a hypercube of the following form:
\begin{equation}
\begin{tikzcd}[column sep=1.5cm] \CF^\delta(S_{2n+m}^3(K)) \ar[r, "G_{\frX_s}+G_{\frY_s}"] \ar[d, "\Gamma"] \ar[dr,dashed,"j"]& \underline{\CF}^\delta(S^3) \ar[d, "\theta_w"]\\
\bA^\delta (K) \ar[r, "v+h_{2n}"]& \bB^\delta (K)
\end{tikzcd}
\label{eq:cone-cobordism->AB}
\end{equation}
The map $\Gamma$ is a holomorphic triangle counting map. The map $\theta_w$ is the map for trivializing the twisted coefficients defining $\underline{\CF}^\delta(S^3)$. 

 The above hypercube is built by gluing together hypercubes with the following shape:
\[
\begin{tikzcd} \CF^\delta(S_{2n+m}^3(K)) \ar[r, "G_{\frX_s}"] \ar[d, "\Gamma"]& \underline{\CF}^\delta(S^3) \ar[d, "\theta_w"]\\
\bA^\delta (K) \ar[r, "v"]& \bB^\delta (K)
\end{tikzcd}
\quad\text{and}\quad 
\begin{tikzcd} \CF^\delta(S_{2n+m}^3(K)) \ar[r, "G_{\frY_s}"] \ar[d, "\Gamma"] \ar[dr,dashed,"j"]& \underline{\CF}^\delta(S^3) \ar[d, "\theta_w"]\\
\bA^\delta (K)\ar[r, "h_{2n}"]& \bB^\delta (K)
\end{tikzcd}
\]

We may combine the above hypercubes into a hyperbox of the following shape:
\begin{equation}
\begin{tikzcd}[column sep=3cm,labels=description]
\CF^\delta(S^3_{2n}(K),[s])
	\ar[d, "{\CF^\delta(W_{2n,m})}"]
	\ar[dr, "k"]
	\ar[drr,dashed, "k_{\frx_s\#[- 1]}"]
	\ar[rr, "{\CF(W_{2n}',\frx_s)}"] 
&[-2.5cm]
& \CF^\delta(S^3)\ar[d, "\id"]\\
\bigg(\CF^\delta(S^3_{2n+m}(K))
	\ar[d, "\Gamma"]
	\ar[dr,"j"]
	\ar[r, "G"]
	\ar[rr, bend right=10, "G_{\frX_s}\Pi_{[s]}"]
 &\underline{\CF}^\delta(S^3) \ar[d, "\theta_w"]\bigg)
 &
 \CF^\delta(S^3)\ar[d,"\theta_w"]
 \\
\bigg(\bA^\delta(K)\ar[r, "v+h_{2n}"]
	\ar[rr, bend right=10, "v\Pi^A_s"]
	&
 \bB^\delta(K)\bigg)& B_n^\delta(K)
\end{tikzcd}
\label{eq:non-involutive-cobordism-comp-1}
\end{equation}
In the above, $\Pi_{[s]}$ denotes projection of $\CF^\delta(S^3_{2n+m}(K)$ onto the $\Spin^c$ structure identified with $[s]\in \Z/(2n+m)$.
Also, we view each parenthesized group as being a single point in the set $\bE(1,2)\iso \{0,1\}\times \{0,1,2\}$.

Symmetrically, there is another hypercube of the following form:
\begin{equation}
\begin{tikzcd}[column sep=3cm,labels=description]
\CF^\delta(S^3_{2n}(K),[s])
	\ar[d, "{\CF(W_{2n,m})}"]
	\ar[dr, "k"]
	\ar[drr,dashed, "k_{\fry_s\#[1]}"]
	\ar[rr, "{\CF(W_{2n}',\fry_s)}"] 
&[-2.5cm]
& \CF^\delta(S^3)\ar[d, "\id"]
\\
\bigg(\CF^\delta(S^3_{2n+m}(K))
	\ar[d, "\Gamma"]
	\ar[dr,"j"]
	\ar[drr,dashed, "\Pi_{s+2n+m}^B j"]
 	\ar[r, "G"]
 	\ar[rr, bend right=10, "G_{\frY_s}\Pi_{[s]}"]
 &\underline{\CF}^\delta(S^3) \ar[d, "\theta_w"]\bigg)
 & \CF^\delta(S^3)\ar[d,"\theta_w"] \\
\bigg(\bA^\delta(K)\ar[r, "v+h_{2n}"] \ar[rr, bend right=10, "h_{2n}\Pi^A_{s}"]& \bB^\delta(K)\bigg)& B_n^\delta(K)
\end{tikzcd}
\label{eq:non-involutive-cobordism-comp-2}
\end{equation}

Equations~\eqref{eq:non-involutive-cobordism-comp-1} and~\eqref{eq:non-involutive-cobordism-comp-2} can be used to compute the cobordism maps for $W_{2n}'(K)$, very similarly to Ozsv\'{a}th and Szab\'{o}'s \cite{OSIntegerSurgeries}*{Theorem~1.1}.

In particular, the fundamental objects to construct are the hypercubes in Equations~\eqref{eq:surgery-quasi-iso-main} and~\eqref{eq:surgery-quasi-iso-pre-quotient}. The involutive extension of Equation~\eqref{eq:surgery-quasi-iso-main} is one of the main theorems of \cite{HHSZExact}. We now describe how to extend Equation~\eqref{eq:surgery-quasi-iso-pre-quotient} to an involutive hypercube:
\begin{prop}\label{prop:tilde-hypercubes} There is a hypercube of the form: 
\[
  \begin{tikzcd}
  \ve{\CFI}^-(S^3_{2n}(K))\ar[d] \ar[dr,dashed, "k"] \ar[r, "F"]& \tilde{\ve{\CFI}}{}^-(S^3)\ar[d, "M"]\\
  \ve{\CFI}^-(S^3_{2n+m}(K))\ar[r, "G"]& \tilde{\ve{\CFI}}{}^-(S^3).
  \end{tikzcd}
\]
Furthermore, $G$ and $F$ are the involutive cobordism maps summed over all $\Spin^c$ structures (as defined in this paper) and $M$ is the map $\sum_{s\in 2\Z+1} U^{m(s^2-1)/8}T^{m(s+1)/2}\cdot \id+EQ$. The map $E$ decomposes into a sum
\[
E=\sum_{s\in 2\Z+1} E_s
\]
where each $E_s$ is of homogeneous grading $-m(s^2-1)/4+1$.
\end{prop}

\begin{proof}The proof is mostly by a straightforward reinterpretation of the techniques of \cite{HHSZExact}. There is one point where our description in \cite{HHSZExact} is not convenient for our present purpose, which is in the first subcube of the central hypercube. We recall the Heegaard quintuple $(\gs'',\gs',\gs,\ds,\ds')$ used to define the first central hypercube. There is a natural hypercube, obtained by pairing two hypercubes (in \cite{HHSZExact}, we referred to this hypercube as the first \emph{auxiliary} hypercube). Over $\bF[T]$, this hypercube takes the following form: 
\[
\begin{tikzcd}[
	column sep={2cm,between origins},
	row sep=1cm,
	labels=description,
	fill opacity=.7,
	text opacity=1,
	]
\ve{\CF}^-(\gs,\ds)
	\ar[dr, "f_{\g\to \g'}^{\dt}"]
	\ar[ddd, "f_{\g}^{\dt\to \dt'}"]
	\ar[rr, "f_{\g\to \g''}^{\dt}"]
	\ar[dddrr,dashed, "h_{\g\to \g''}^{\dt\to \dt'}"]
	\ar[drrr,dashed, "h_{\g\to \g'\to \g''}^{\dt}"]
&[.7 cm]
&\tilde{\ve{\CF}}{}^-(\gs'',\ds)
	\ar[dr, "M"]
	\ar[ddd,"f_{\g''}^{\dt\to \dt'}"]
&[.7 cm]
\\
&[.7 cm]
\ve{\CF}^-(\gs',\ds)
 	\ar[rr,line width=2mm,dash,color=white,opacity=.7]
	\ar[rr, "f_{\g'\to \g''}^{\dt}"]
&\,
&[.7 cm]
\tilde{\ve{\CF}}{}^-(\gs'',\ds)
	\ar[ddd,"f_{\g''}^{\dt\to \dt'}"]
\\
\\
\tilde{\ve{\CF}}{}^-(\gs,\ds')
	\ar[dr, "f_{\g\to \g'}^{\dt'}"]
	\ar[rr, "f_{\g\to \g''}^{ \dt'}", pos=.4]
	\ar[drrr,dashed, "h_{\g\to\g'\to \g''}^{\dt'}"]
&[.7 cm]\,&
\tilde{\ve{\CF}}{}^-(\gs'',\ds')
	\ar[dr, "M"]
\\
& [.7 cm]
\tilde{\ve{\CF}}{}^-(\gs',\ds')
 	\ar[from=uuu,line width=2mm,dash,color=white,opacity=.7]
	\ar[from=uuu,"f_{\g'}^{\dt\to \dt'}"]
	\ar[rr, "f_{\g'\to \g''}^{\dt'}"]
	\ar[from=uuuul,line width=2mm,dash,color=white,opacity=.7]
	\ar[from=uuuul,dashed, "h_{\g\to \g'}^{\dt\to \dt'}"]
&
&[.7 cm] 
\tilde{ \ve{\CF}}{}^-(\gs'',\ds')
	 \ar[from=uuull,line width=2mm,dash,color=white,opacity=.7]
	\ar[from=uuull,"h_{\g'\to \g''}^{\dt\to \dt'}", dashed, pos=.3]
\end{tikzcd}
\]
The length 3 map (not shown) is the pentagon counting map $p_{\g\to \g'\to \g''}^{\dt\to \dt'}$.

We recall that the diagram $(\gs',\ds'')$ admits a 3-handle map to a diagram $\cH_0$, and that both $(\gs,\ds')$ and $(\gs',\ds')$ are genus 1 stabilizations of $\cH_0$. Furthermore, the maps $f_{\g\to \g''}^{\dt'}$ and $f_{\g'\to \g''}^{\dt'}$ are the maps for surgering on a 0-framed unknot. Since the relevant Heegaard triples are both stabilizations, both of these maps admit left inverses, which are a 3-handle map followed by a stabilization. Write $\sigma$ for the stabilization map which has $(\gs,\ds')$ as its codomain, and write $\sigma'$ for the stabilization map which has $(\gs',\ds')$ as its codomain.

As a replacement for the diagram $\cC_{\textrm{cen}}^{(1)}$ in \cite{HHSZExact}, we build the following diagram:
\[
\begin{tikzcd}[
	column sep={2.5cm,between origins},
	row sep=1.4cm,
	labels=description,
	fill opacity=.7,
	text opacity=1,
	]
\ve{\CF}^-(\gs,\ds)
	\ar[dr, "f_{\g\to \g'}^{\dt}"]
	\ar[ddd, "\id"]
	\ar[rr, "f_{\g\to \g''}^{\dt}"]
	\ar[dddrr,dashed, "\sigma F_3 h_{\g\to \g''}^{\dt\to \dt'}"]
&[.7 cm]
&\tilde{\ve{\CF}}{}^-(\gs'',\ds)
	\ar[dr, "M"]
	\ar[ddd,"\sigma F_3 f_{\g''}^{\dt\to \dt'}"]
	\ar[ddddr,dashed, "\sigma' F_3 h_1 \sigma F_3 f_{\g''}^{\dt\to \dt'}"]
&[.7 cm]
\\
&[.7 cm]
\ve{\CF}^-(\gs',\ds)
 	\ar[rr,line width=2mm,dash,color=white,opacity=.7]
	\ar[rr, "f_{\g'\to \g''}^{\dt}"]
&\,
&[.7 cm]
\tilde{\ve{\CF}}{}^-(\gs'',\ds)
	\ar[ddd,"\sigma' F_3 f_{\g''}^{\dt\to \dt'}"]
 	\ar[from=lllu,line width=2mm,dash,color=white,opacity=.7]
  	\ar[from=lllu,dashed, "h_{\g\to \g'\to \g''}^{\dt}"]
\\
\\
\ve{\CF}^-(\gs,\ds)
	\ar[dr, "f_{\g\to \g'}^{\dt}"]
	\ar[rr, "f_{\g}^{\dt\to \dt'}", pos=.4]
	\ar[drrr,dashed, "h_{\g\to \g'}^{\dt\to \dt'}"]
&[.7 cm]\,&
\tilde{\ve{\CF}}{}^-(\gs,\ds')
	\ar[dr, "f_{\g\to \g'}^{\dt'}"]
\\
& [.7 cm]
\ve{\CF}^-(\gs',\ds)
 	\ar[from=uuu,line width=2mm,dash,color=white,opacity=.7]
	\ar[from=uuu,"\id"]
	\ar[rr, "f_{\g'}^{\dt\to \dt'}"]
&
&[.7 cm] 
\tilde{ \ve{\CF}}{}^-(\gs',\ds')
	 \ar[from=uuull,line width=2mm,dash,color=white,opacity=.7]
	\ar[from=uuull,"\sigma' F_3 h_{\g'\to \g''}^{\dt \to \dt'}", dashed, pos=.3]
\end{tikzcd}
\]
The length 3 map (not shown) is $\sigma' F_3 p_{\g\to \g'\to \g''}^{\dt\to \dt'}+\sigma'F_3 h_1 \sigma F_3 h_{\g\to \g''}^{\dt\to \dt'}$. Following the notation of \cite{HHSZExact}, the map $h_1$ denotes $h_{\g\to \g'\to \g''}^{\dt'}$. The hypercube relations are straightforward to verify. Compare \cite{HHSZExact}*{Lemma~16.4}.

The hypercubes $\cC_{\mathrm{cen}}^{(2)}$ and $\cC_{\mathrm{cen}}^{(3)}$ are more straightforward to modify, and we leave the details to the reader. Upon compressing and stacking these hypercubes, we get the diagram in the statement. It remains to verify the listed claims. The identification of $J$ and $G$ as cobordism maps is straightforward, and we leave the details to the reader. Compare \cite{HHSZExact}*{Lemma~16.14}.

We now consider the map labeled $M$ in the statement, the length one component of this map (i.e. the non-$Q$ component) may be identified with $\sum_{s\in 2\Z+1} U^{m(s^2-1)/8}T^{m(s+1)/2}$ by the lattice point counting argument in \cite{OSIntegerSurgeries}*{Section~3}. It remains to identify the $Q$-term. The $Q$-term is essentially the same as the map $E$ considered in \cite{HHSZExact}*{Section~19}. Therein, we showed that $E$ is null-homotopic over $\bF[U,T]/(T^m-1)$. Our argument used at several places that the coefficients were in $\bF[U,T]/(T^m-1)$ instead of $\bF[U,T]$, so we give an alternate argument that is sufficient for our present purposes.

The first observation is that the map $E$ decomposes over $\Spin^c(D(-m,1))$, which we have already identified with $2\Z+1$. The grading change formula is also a routine consequence of this identification, and the standard grading change formulas for the cobordism maps in Heegaard Floer homology \cite{OSIntersectionForms}. 
\end{proof}

\begin{rem}
We write $M=M_0+Q M_1$. Since $\CF^-(S^3)\simeq \bF[U]$ is supported only in even gradings, and $M_1$ may be viewed as a chain map from $\ve{\CF}^-(S^3)$ to itself which shifts the mod 2 grading by 1, we may conclude that $M_1\simeq 0$. In particular, by adding a term to the length 3 map of the hypercube, we may assume that $M_1=0$. Furthermore, fixing $\delta$ and letting $m$ be sufficiently large, we see that $E^\delta$ is homogeneously graded, so the null-homotopy may also be taken to be homogeneously graded.
\end{rem}

In \cite{HHSZExact}*{Section~21}, we also described how to extend the hypercube in Equation~\eqref{eq:cone-cobordism->AB} into the involutive setting. Using this, together with Proposition~\ref{prop:tilde-hypercubes}, we may construct the involutive analog of the hyperboxes in Equation~\eqref{eq:non-involutive-cobordism-comp-1} and~\eqref{eq:non-involutive-cobordism-comp-2}. We construct this hyperbox so that the top face has the hyperbox from Equation~\eqref{eq:non-involutive-cobordism-comp-1}, which involves the $\Spin^c$ structures $\frx_s$ and $\frX_s$ with $s=n$. There is an important subtlety in that the bottom face of this hyperbox will be the hypercube from Equation~\eqref{eq:non-involutive-cobordism-comp-2}, involves the $\Spin^c$ structures $\fry_{-n}$ and $\frY_{-n}$. (Note that $\fry_{-n}=\bar\frx_n$). 

 The front face of this hyperbox will compress to have the diagonal map $H_{2n}\tilde v$, where $H_{2n}$ is the lenth 2 map from the involutive mapping cone of \cite{HHSZExact}.

We add one more level to this hyperbox so that the top and bottom levels faces coincide. We add this level to the bottom of the hyperbox. We build this new level as a pair of hypercubes. The first hypercube has the following shape:
\[
\begin{tikzcd}[
	column sep={2cm,between origins},
	row sep=.8cm,
	labels=description,
	fill opacity=1,
	text opacity=1,
	]
\CF^\delta(S_{2n}^3(K),[n])
	\ar[dr, "{\CF(W_{2n,2n+m})}"]
	\ar[drr, "k"]
	\ar[ddd, "\id"]
	\ar[rrr, "{\CF(W_{2n}',\fry_{-n})}"]
	\ar[drrrr,dashed, "k_{\fry_{-n}\#[1]}"]
&
&[.7 cm]&\CF^\delta(S^3)
	\ar[rd, "\id"]
	\ar[ddd,"\id", pos=.65]
&[.7 cm]
\\
&[.7 cm]
\bigg(\CF^\delta(S_{2n+m}(K))
	\ar[r]
 	\ar[rrr,bend right=12,line width=2mm,dash,color=white,opacity=.7]
	\ar[rrr,bend right=12, "G_{\frY_{-n}}\Pi_{[-n]}"]
&
\underline{\CF}^\delta(S^3)\bigg)
&\,
&[.7 cm]
\CF^\delta(S^3)
\\
\\
\CF^\delta(S_{2n}^3(K),[n])
	\ar[dr, "{\CF(W_{2n,2n+m})}"]
	\ar[rrr, "{\CF(W_{2n}',\frx_n)}"]
	\ar[drr,"k"]
	\ar[drrrr,dashed, pos=.4, "k_{\frx_n\#[-1]} "]
&[.7 cm]\,&&
\CF^\delta(S^3)
	\ar[dr, "\id"]
\\
& [.7 cm]
\bigg(\CF^\delta(S_{2n+m}^3(K))
	\ar[r]
 	\ar[from=uuu,line width=2mm,dash,color=white,opacity=.7]
	\ar[from=uuu,"\id"]
	\ar[rrr, bend right=12, "G_{\frX_n}\Pi_{[n]}"]
&
\underline{\CF}^\delta(S^3)\bigg)
 	\ar[from=uuu,line width=2mm,dash,color=white,opacity=.7]
	\ar[from=uuu,"\id"]
&
&[.7 cm] 
\CF^\delta(S^3)
	\ar[from=uuu, "\id"]
 	\ar[from=lluuu,line width=2mm,dash,color=white,opacity=.7]
	\ar[from=lluuu, dashed,"\Pi_{n}"]
\end{tikzcd}
\]
In the above $\Pi_n$ denotes projection onto the $B_n$-summand of $\underline{\CF}^-(S^3)$. The hypercube relations are straightforward. 

We finally have one additional hypercube, as follows:
\[
\begin{tikzcd}[
	column sep={3cm,between origins},
	row sep=.8cm,
	labels=description,
	fill opacity=1,
	text opacity=1,
	]
\bigg(\CF^\delta(S^3_{2n+m}(K))
	\ar[dr, "\Gamma"]
	\ar[ddd, "\id"]
	\ar[r]
	\ar[rrr, bend left=12, "G_{\frY_{-n}}\Pi_{[-n]}"]
&[-.2cm]\underline{\CF}^\delta(S^3)\bigg)
	\ar[dr, "\theta_w"]
	\ar[ddd,"\id", shift right=2mm, pos=.6]
	\ar[dddrr,dashed, "\Pi_n", pos=.6]
&
&\CF^\delta(S^3)
	\ar[rd, "\theta_w"]
	\ar[ddd,"\id"]
&\\
&\bigg(\bA^\delta(K)
	\ar[r]
 	\ar[rrr,bend right=12,line width=2mm,dash,color=white,opacity=.7]
	\ar[rrr, bend right=12, "h \Pi^A_{-n}"]
&
\bB^\delta(K)\bigg)
	\ar[from=ull,crossing over,"j"]
&\,
&
B_n^\delta(K)
	\ar[from=ullll, crossing over, "\Pi^B_{n} j",dashed]
\\
\\
\bigg(\CF^\delta(S_{2n+m}^3(K))
	\ar[dr, "\Gamma"]
	\ar[r]
	\ar[drr, "j"]
	\ar[rrr,bend left=12, "G_{\frX_n} \Pi_{[n]}"]
&
\underline{\CF}^\delta(S^3)\bigg)
	\ar[dr,"\theta_w"]
\,&&
\CF^\delta(S^3)
	\ar[dr, "\theta_w"]
\\
& [.7 cm]
\bigg(\bA^\delta(K)
	\ar[r]
 	\ar[from=uuu,line width=2mm,dash,color=white,opacity=.7]
	\ar[from=uuu,"\id", pos=.4]
	\ar[rrr, bend right=12, "v\Pi^A_n"]
&
\bB^\delta(K)\bigg)
 	\ar[from=uuu,line width=2mm,dash,color=white,opacity=.7]
	\ar[from=uuu,"\id"]
&
&[.7 cm] 
B_n^\delta(K)
	\ar[from=uuu, "\id"]
 	\ar[from=lluuu,line width=2mm,dash,color=white,opacity=.7]
	\ar[from=lluuu, dashed,"\Pi^B_n", pos=.4]
\end{tikzcd}
\]
The above is easily check to be a hypercube of chain complexes, so the proof is complete.

%% file: section-twist.tex
\section{The $\SO(3)$-twist} \label{sec:twist}

In this section, we show that the diffeomorphism map associated to $\tw\in \pi_1(\SO(3))\iso \Z/2$ acts on $\CFI(Y)$ by $\Id + Q\Phi$, proving Theorem~\ref{thm:twist-map}. Our argument will actually identify the map with $Q\Phi \iota$, though this is homotopic to $Q\Phi$ on $\CFI(Y)$. In Section~\ref{sec:twisting-knots-links}, we will further consider an analogous twisting diffeomorphism map on involutive knot Floer homology, and show that it is trivial.

Before beginning with the computation, we recall several topological perspectives about $\tw$. The first is as a Dehn twist along a small $S^2$ enclosing the basepoint. A regular neighborhood of such a sphere is $S^2\times I$. We recall that $\Diff(S^2)\simeq O(3)$ so the generator of $\pi_1(\SO(3))$ gives a diffeomorphism of $S^2\times I$ which is the identity on the boundary.

There is another perspective on $\tw$, is as follows. Write $\Diff^+(Y,w)$ for the orientation preserving diffeomorphisms $\phi$ such that $\phi(w)=w$; likewise, write $\Diff^f(Y,w)$ for the orientation preserving diffeomorphisms $\phi$ such that $\phi(w)=w$ and $d_w \phi=\id$. If we fix a trivialization of $T_w Y$, we obtain  a Serre fibration
\[
\Diff^f(Y,w)\to \Diff^+(Y,w)\to \SO(3),
\]
and hence a map $\pi_1(\SO(3))\to \pi_0(\Diff^f(Y,w))$.

\subsection{Heegaard diagrams for the twist map}

We first describe the Heegaard diagrams used to define $\tw_*$. Recall that in non-involutive Heegaard Floer homology, if $\phi\colon \Sigma\to \Sigma$ is a diffeomorphism, then the action of $\phi_*$ on Heegaard Floer homology is induced by composing the maps in the following diagram:
\[
\begin{tikzcd}
\CF(\Sigma,\as,\bs)
	\ar[r, "f_{\a\to \phi(\a)}^{\b\to \phi(\b)}"]& 
\CF(\Sigma,\phi^{-1}(\as), \phi^{-1}(\bs))\ar[r, "T_\phi"]&
\CF(\Sigma,\as,\bs)
\end{tikzcd}
\]
where $f_{\a\to \phi(\a)}^{\b\to \phi(\b)}$ is a composition of holomorphic triangle maps and change of almost complex structure maps, and $T_\phi$ is the tautological map on intersection points. The diffeomorphism maps on involutive Heegaard Floer homology are defined similarly, except that we replace $\CF(\Sigma,\as,\bs)$ with a one dimensional hyperbox (whose maps of non-zero length correspond to the involution), and we replace $f_{\a\to \phi(\a)}^{\b\to \phi(\b)}$ and $T_\phi$ with morphisms of hyperboxes.

This has the following instantiation in our present setting. Let $\boldsymbol{\delta}$ be a set of doubling arcs (with boundary on $w$) and $\tilde{\boldsymbol{\delta}}$ be a set obtained by applying a boundary Dehn twist in a loop which encircles $w$. Let $\Ds$ and $\tilde{\Ds}$ be their corresponding doubled curves on $\Sigma \# \bar \Sigma$. Then $\tw_*$ is obtained by compressing the following diagram (where we are omitting an extra column corresponding to the tautological portion of the diffeomorphism map):
\begin{equation}
\begin{tikzcd}
\CF(\Sigma,\as,\bs)
	\ar[r]
	\ar[d]& 
\CF(\Sigma,\as,\bs)
	\ar[d]
\\
\CF(\Sigma\# \bar \Sigma,\as \bar \bs, \bs \bar \bs)
	\ar[r]
	\ar[dr,dashed]
	\ar[d]&
\CF(\Sigma\# \bar \Sigma,\as \bar \bs, \bs \bar \bs)
	\ar[d]
\\
\CF(\Sigma\# \bar \Sigma,\as \bar \bs, \Ds)
	\ar[r]
	\ar[d]
	\ar[dr,dashed]
&
\CF(\Sigma\# \bar \Sigma,\as \bar \bs, \tilde{\Ds})
	\ar[d]
\\
\CF(\Sigma\# \bar \Sigma,\as \bar \bs, \as\bar \as)
	\ar[r]
	\ar[d]&
\CF(\Sigma\# \bar \Sigma,\as \bar\bs, \as \bar \as)
	\ar[d]
\\
\CF( \bar \Sigma,\bar \bs, \bar \as)
	\ar[r]
	&
\CF(\bar \Sigma,\bar \bs, \bar \as)
\end{tikzcd}
\label{eq:Dehn-twist-map}
\end{equation}
 	The relevant Heegaard diagrams are shown in Figure~\ref{fig:twist_fig1}.

\begin{figure}[ht]
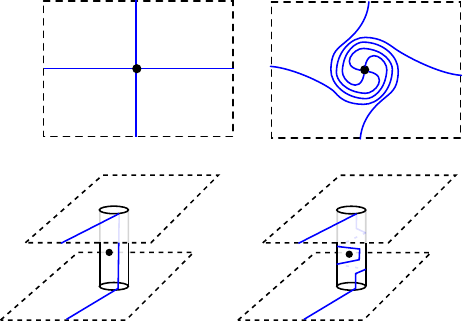
\caption{Top row, the doubling arcs $\boldsymbol{\delta}$ (left) and $\tilde{\boldsymbol{\delta}}$ (right). On the bottom row, we show the doubled diagrams, and indicate how $\tilde{\Ds}$ is obtained from $\Ds$ by moving the basepoint in a loop around the connected sum tube.}
\label{fig:twist_fig1}
\end{figure}

\begin{rem} Since $\tw_*$ can be interpreted in terms moving the basepoint around the connected sum neck on $\Sigma\# \bar \Sigma$, one might expect that the formula for the hypercube $\pi_1$-action from \cite{ZemBordered}*{Section~13} would show that $\tw_*$ is null-homotopic as a morphism of 1-dimensional hypercubes. That is, if we view $\g_*$ as a 2-dimensional diagram below, then there is a 3-dimensional hypercube taking the form shown on the right:
\[
\g_*=
\begin{tikzcd}[row sep=3cm]
 \CF(Y)
 \ar[r, "\g_*"]
 \ar[dr,dashed, "\g_*"]
 \ar[d, "\iota"]
& \CF(Y)
\ar[d, "\iota"]
\\
\CF(Y) \ar[r, "\g_*"] & \CF(Y)
\end{tikzcd}
\begin{tikzcd}[
	column sep={1.3cm,between origins},
	row sep=.8cm,
	labels=description,
	fill opacity=1,
	text opacity=1,
	]
\CF(Y)
	\ar[dr, "\id"]
	\ar[ddd, "\iota"]
	\ar[rr, "\g_*"]
	\ar[dddrr,dashed, "\g_*"]
	\ar[ddddrrr,dotted]
&[.7 cm]
&\CF(Y)
	\ar[rd, "\id"]
	\ar[ddd,"\iota"]
&[.7 cm]
\\
&[.7 cm]
\CF(Y)
 	\ar[rr,line width=2mm,dash,color=white,opacity=.7]
	\ar[rr, "\id"]
&\,
&[.7 cm]
\CF(Y)
	\ar[ddd,"\iota"]
 	\ar[from=ulll,line width=2mm,dash,color=white,opacity=.7]
	\ar[from=ulll,dashed]
\\
\\
\CF(Y)
	\ar[dr, "\id"]
	\ar[rr, "\gamma_*"]
	\ar[drrr,dashed]
&[.7 cm]\,&
\CF(Y)
	\ar[dr, "\id"]
\\
& [.7 cm]
\CF(Y)
 	\ar[from=uuu,line width=2mm,dash,color=white,opacity=.7]
	\ar[from=uuu,"\iota"]
	\ar[rr, "\id"]
&
&[.7 cm] 
\CF(Y)
\end{tikzcd}.
\]
We observe however that the techniques of \cite{ZemBordered} do not guarantee that the 3-dimensional hypercube so constructed has the same maps on the top and bottom faces of the hypercube. (This may seem surprising since the length one components of $\g_*$ are the identity on the nose). Concretely, the homotopy on the top face involves maps related to relative homology actions on stabilized diagrams, which are in general not preserved by the involution. In particular, such a hypercube does not give us an $\bF[U,Q]/Q^2$-equivariant homotopy $\id\simeq \tw_*$ on $\CFI(Y)$.
\end{rem}

\subsection{Degenerating connected sums}

Our argument depends on a somewhat complicated neck stretching argument. Our strategy is adapted from work of Ozsv\'{a}th and Szab\'{o} in the setting of bordered knot Floer homology \cite{OSBorderedHFK}. We have explored some of these ideas further in \cite{HHSZExact} and \cite{ZemBordered}.

We focus first on the case of two ordinary Heegaard diagrams $\cH=(\Sigma,\as,\bs,w)$ and $\cH'=(\Sigma',\as',\bs,',w')$, and we take the connected sum at $w$ and $w'$ (deleting one of the basepoints). 

We first degenerate the connected sum tubes by letting the neck length approach $\infty$. We obtain maps which count \emph{perfectly matched} moduli spaces. We write $\CF_{\wedge}(\cH\# \cH')$ for the resulting chain complex. 

There are homotopy equivalences
\[
\Psi^{\Sigma<\Sigma'}, \Psi^{\Sigma'<\Sigma}\colon \CF_{\wedge}(\cH\# \cH')\to \CF(\cH)\otimes_{\bF[U]} \CF(\cH').
\]

The map $\Psi^{\Sigma<\Sigma'}$ counts holomorphic disk pairs $(u,u')$ of expected dimension $0$ such that
\begin{enumerate}
\item For $[0,1]\times (-\infty,0)$, the punctures of $u$ and $u'$ consist of perfectly matched pairs; that is, each puncture of $u$ is paired with a puncture of $u'$ which has identical projection to $[0,1]\times \R$ (and vice-versa).
\item There are an even number of $u$ and $u'$ punctures on $[0,1]\times \{0\}$, the \emph{special line}. These punctures all have distinct projections to $[0,1]\times \{0\}$, and alternate between $u$ and $u'$. The left-most puncture is from $u$.
\item In $[0,1]\times (0,\infty)$, $u$ and $u'$ may have any number of punctures, and there is no constraint.
\end{enumerate}
The map $\Psi^{\Sigma'<\Sigma}$ is defined similarly, except that the left-most puncture on the special line is from $u'$.

\begin{prop} The maps $\Psi^{\Sigma<\Sigma'}$ and $\Psi^{\Sigma'<\Sigma}$ are chain maps. Furthermore, they are homotopy equivalences, with homotopy inverses given by the maps which count curves in $0$-dimensional moduli spaces with matching conditions in the regions above and below $[0,1]\times \{0\}$ reversed.
\end{prop}

A proof, modeled on Ozsv\'{a}th and Szab\'{o}'s work \cite{OSBorderedHFK}, may be found in \cite{ZemBordered}*{Section~11.2}. An extension of the connected sum maps $\Psi^{\Sigma<\Sigma'}$ to the setting of connected sums of hypercubes (under certain topological assumptions) is also proven in \cite{ZemBordered}*{Proposition~11.8}.

\subsection{Setting up the computation}

In Section~\ref{sec:naturality-knots-and-links} we considered several stabilized models of the involution where we used multiple connected sum tubes, but a single basepoint. To perform our computation, we will join $\Sigma$ and $\bar \Sigma$ using a tube at $w$ and a tube at another point $p$.  Considering this stabilized model will allow us to localize the computation.

\begin{figure}[ht]
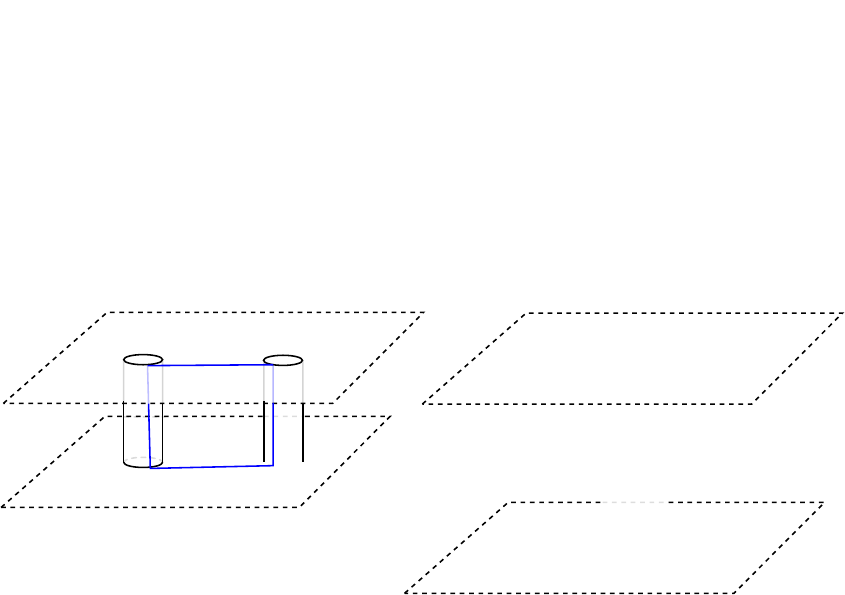
\caption{A once stabilized model of the involution. On the top, we show the doubling arcs on $\Sigma$. On the bottom left, we have an overview of the two tubes. On the bottom right, we zoom in on the $w$-tube region.}
\label{twist_fig:5}
\end{figure}

Write $c_p$ for the meridian of the $p$-tube. Note that we can get a valid Heegaard diagram by replacing $c_p$ with a meridian of the other tube. There are two natural choices, $c_w$ and $\bar c_w$, depending on which side of $w$ we put the meridian. 

We recall that the stabilized model of the involution from Section~\ref{sec:knots} takes the following form:
\[
F_3^{\a c_p, \a c_p}\circ f_{\a c_p \bar \b}^{\Dt\to (\a  c_p\bar \a)}\circ f_{\a  c_p\bar \b}^{(\b c_p \bar \b)\to \Dt}\circ  F_1^{c_p\bar \b, c_p \bar \b}.
\]

We now modify the sequence of curves we use in this Heegaard diagram, by breaking the sequence 
\[
\bs c_p \bar \bs\to \Ds\to \as c_p \bar \as
\]
into a longer sequence. This will make some holomorphic curve counts easier.

 We pick sets of doubled curves $\Ds_{1},\dots, \Ds_n$, containing $2g+1$ curves each, on $\Sigma\# \bar \Sigma$ such that the following hold:
\begin{enumerate}
\item $\Ds_i=\Ds_i^0 \cup D$, where $\Ds_i^0$ are doubled curves on the unstabilized diagram, obtained by doubling arcs with two endpoints at $w$. Here $D$ is a single curve which obtained by doubling an arc which connects $p$ to $w$.
\item $(\Sigma \# \bar \Sigma,\bs \bar \bs,\Ds_1^0)$ is algebraically rigid. (Compare \cite{HHSZExact}*{Definition~6.5}).
\item $(\Sigma \# \bar \Sigma,\as \bar \as,\Ds_n^0)$ is algebraically rigid. 
\item Each $\Ds_{i+1}^0$ is obtained from $\Ds_i^0$ by a small isotopy or a simple arcslide of the doubling arcs on $\Sigma\setminus N(p)$, not crossing over the endpoint corresponding to $D$. (In particular each $(\Sigma\# \bar \Sigma, \Ds_i^0,\Ds_{i+1}^0,w)$ is algebraically rigid). Note that this implies that $\Ds_{i+1}^0$ is obtained from $\Ds_i^0$ by an elementary handleslide.
\end{enumerate}

We will use the following model of the involution:
\[
F_3^{\a c_p, \a c_p}
	\circ 
\Psi_{(\a \bar c_w \bar \b)\to (\a c_p \bar \b)}^{(\a \bar c_w \bar \a)\to (\a c_p \bar \a)} 
	\circ
 f_{\a \bar c_w \bar \b}^{\Dt_n\to (\a \bar c_w \bar \a)}
	\circ \cdots\circ
 f_{\a \bar c_w \bar \b}^{\Dt_1\to \Dt_2} 	
 	\circ
 f_{\a \bar c_w \bar \b}^{(\b \bar c_w \bar \b)\to \Dt_1}
	 \circ
\Psi_{(\a  c_p \bar \b)\to (\a \bar c_w \bar \b )}^{(\b c_p \bar \b)\to (\b \bar c_w\bar \b)}
	\circ
F_1^{c_p\bar \b, c_p \bar \b}
\]

The motivation for using this model is that most of the holomorphic curve counts in the above diagram occur on a connected sum of the ordinary doubled diagram for $Y$ on $\Sigma\#\bar{\Sigma}$ and a genus 1 diagram where the diffeomorphism acts non-trivially. The genus 1 region is a neighborhood of $D\cup c_w$; we call it $\mathbb T$.

We now write $\tilde{\Ds}_i$ for $\tw(\Ds_i)$. If we write $\Ds_i=\Ds_i^0\cup D$, then $\tilde{\Ds}_i=\Ds_i^0\cup \tilde{D}$, where $\tilde{D}$ is obtained from $D$ by applying the twist in the torus region.

Our argument will proceed  as follows. We will view each of the holomorphic polygon counts as occurring on a connected sum of $\Sigma\# \bar \Sigma$ and $\bT$. We will build a 3-dimensional hyperbox whose front and back faces are expanded versions of the hyperbox in Equation~\eqref{eq:Dehn-twist-map}. The new dimension corresponds to the map $\Psi^{\Sigma\# \bar{\Sigma}<\bT}$. The back face will consist of the perfectly matched complexes. The front face will consist of the trivially matched complexes with some extra terms. The extra terms correspond to the $Q\Phi$ summand of $\tw_*$. 

\subsection{Initial and final hypercubes}

In this section, we extend the following hypercubes (with holomorphic curves counted with perfect matching) into a third dimension corresponding to the map $\Psi^{\Sigma\# \bar \Sigma<\bT}$:
\[
\begin{tikzcd} \CF(\as,\bs)
	\ar[r, "\id"]
	\ar[d, "F_1^{c_p \bar \b,c_p \bar \b}"]
&
\CF(\as, \bs)
	\ar[d, "F_1^{c_p \bar \b,c_p \bar \b}"]
\\
\CF(\as c_p \bar \bs)
	\ar[r, "\id"]& 
\CF(\as c_p \bar \bs)
\end{tikzcd}
\quad \begin{tikzcd}
 \CF(\as \bar c_w \bar \as,\as \bar c_w \bar \as)
	\ar[r, "\id"]
	\ar[d]
&
\CF(\as \bar c_w \bar \as,\as \bar c_w \bar \as)
	\ar[d]
\\
\CF(\as c_p \bar \as,\as c_p \bar \as)
	\ar[r, "\id"]
	\ar[d, "F_3^{\a c_p, \a c_p}"]& 
\CF(\as c_p \bar \as,\as c_p \bar \as)
	\ar[d, "F_3^{\a c_p, \a c_p}"]
	\\
\CF(\bar \bs, \bar \as) \ar[r, "\id" ] &\CF(\bar \bs, \bar \as).
\end{tikzcd}
\]

In the above, there is no diagonal term because the diffeomorphism $\tw$ fixes all of the attaching curves on the Heegaard diagram.

 \begin{lem}\label{lem:1-handle-c_w-commutation} If we stretch the almost complex structure sufficiently on the connected sum tubes, we have the equality
 \[
\Psi_{(\a  c_p \bar \b)\to (\a \bar c_w \bar \b )}^{(\b c_p \bar \b)\to (\b \bar c_w\bar \b)}\circ  F_1^{c_p \bar \b, c_p \bar \b}=F_1^{\bar c_w \bar \b, \bar c_w \bar \b}.
 \]
 \end{lem}
 \begin{proof} This follows from our stabilization result for holomorphic triangles in Proposition~\ref{prop:multi-stabilization-counts}. Topologically, it is sufficient to observe that on $\bar \Sigma$, there is a sequence of handleslides of $c_p$ over the other $\bar \bs$ curves (not crossing over $w$) which takes $c_p$ to $\bar c_w$ (i.e. handleslide $c_p$ over each $\bar \bs$ curve twice).
 \end{proof}

We now investigate the next terms in the sequence. In Figure~\ref{twist_fig:7}, we draw the diagram $(\bT, \bar c_w,D, \tilde D, \bar c_w',w)$. 

\begin{figure}[ht]
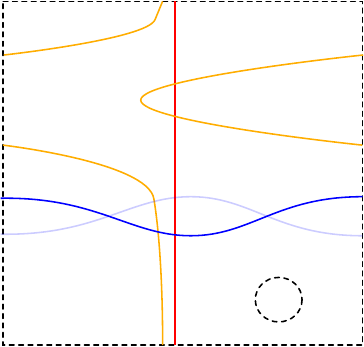
\caption{The special genus 1 region and the diagram $(\bT^2, \bar {c}_w, D, \tilde{D}, \bar{c}_w',w)$. The dashed circle is the connected sum point. Also shown is a holomorphic quadrilateral that will play a role in our argument.}
\label{twist_fig:7}
\end{figure}

Next, we observe
\[
F_3^{\a c_p, \a c_p}=F_3^{c_p,c_p}\circ F_3^{\a,\a}.
\]
We note that after surgering out the $\as$ curves, the curves $c_p$ become isotopic to $c_w$. Hence we identify
\[
F_3^{c_p,c_p}\circ F_3^{\a,\a}=F_3^{c_w,c_w}\circ F_3^{\a,\a}.
\]
In a similar manner to Lemma~\ref{lem:1-handle-c_w-commutation}, we have the strict commutation
\begin{equation}
F_3^{\a,\a}\circ\Psi_{(\a \bar c_w \bar \b)\to (\a c_p \bar \b)}^{(\a \bar c_w \bar \a)\to (\a c_p \bar \a)}
=
\Psi_{(\bar c_w \bar \b)\to (c_w\bar \b)}^{(\bar c_w \bar \a)\to (c_w \bar \a)}\circ F_3^{\a,\a}.\label{eq:commute-3-handle-c-w-c-w'}
\end{equation}

We now prove a simple index bound for the diagram $(\bT,\bar{c}_w,  c_w)$. We view this diagram as having both a basepoint $w\in \bT$, as well as a connected sum point $p\in \bT$. Note that these do not coincide. See Figure~\ref{twist_fig:8}.

\begin{figure}[ht]
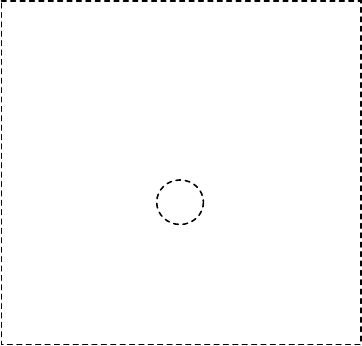
\caption{The diagram $(\bT, \bar c_w,  c_w,w)$. The dashed circle is the connected sum point $p$. }
\label{twist_fig:8}
\end{figure}

\begin{lem}
\label{lem:index-bound}
Suppose that $\phi\in \pi_2(\theta^+, \theta)$ is a class of disks on $(\bT,\bar c_w, c_w,w)$. 
\begin{enumerate}
\item Then $\mu(\phi)=2n_w(\phi)+\gr(\theta^+,\theta)$.
\item If $D(\phi)\ge 0$, then $\mu(\phi) \ge n_p(\phi)$. 
\end{enumerate}
\end{lem}
\begin{proof}The first claim is the Maslov grading formula.

We now consider the second claim. Consider first the case that $\theta=\theta^+$. Then any class $\phi=\ve{e}_{\theta}+k\cdot P+n\cdot [\bT]$, where $P$ is the periodic domain class which has $n_w=0$ and $n_p=1$. We observe $\mu(\phi)=2n$, and $n_p(\phi)=n+k$. Also, $n$, $n-k$ and $n+k$ are multiplicities of this disk class, so all must be nonnegative. This implies, in particular, that $n\ge |k|$. Ergo,
\[
\mu(\phi)=2n\ge n+k= n_p(\phi),
\]
as claimed.

We now consider the case that $\theta = \theta^-$. In this case, write $\phi_0$ for the bigon which has $n_p=1$ and write $\phi=\phi_0+k\cdot  P+n\cdot [\bT]$ for a general class. In this case, $n$, $n-k$ and $n+k+1$ are multiplicities on the diagram. Hence, $n\ge k$.  We have
\[
\mu(\phi)=1+2n\ge 1+n+k= n_p(\phi),
\]
completing the proof.
\end{proof}

We now perform a key computation:

\begin{lem}\label{lem:pair}  We have the relation
\[
F_3^{ c_w, c_w}\circ \Psi^{(\bar c_w \bar \a)\to (c_w \bar \a)}_{(\bar c_w \bar \b)\to (c_w \bar \b)}\simeq\left(\begin{array}{l} \xs\otimes \theta^-\mapsto \xs\\
\xs\otimes \theta^+\mapsto \Phi_w^{\bar\Sigma}(\xs)
\end{array} \right)\circ \Psi^{\bar\Sigma<\bT}
\]
where $\Phi_w^{\bar \Sigma}$ is the standard basepoint action on $\CF(\bar \Sigma, \bar \bs, \bar \as)$, usually denoted $\Phi_w$.
\end{lem}\label{lem:cw-cwbar}
\begin{proof} The map $\Psi^{(\bar c_w \bar \a)\to (c_w \bar \a)}_{(\bar c_w \bar \b)\to (c_w \bar \b)}$ is a composition of two triangle maps. The diagram $(\Sigma,\bar c_w, c_w, w)$ is shown in Figure~\ref{twist_fig:8}. The transition map may be viewed as handlesliding $\bar c_w$ over the connected sum point. 

We first claim that index bound which implies that the 3-handle map is well-defined implies also that 
\begin{equation}
F_3^{c_w,c_w}=F_3^{c_w,c_w}\circ \Psi^{\bar\Sigma<\bT}
\label{eq:hypercube-3-handle-preliminary}.
\end{equation}
To see this, note that $\Psi^{\bar\Sigma<\bT}$ counts holomorphic disk-pairs $(u_{\Sigma}, u_{\bT})$ on $(\bar\Sigma, \bar \bs, \bar \as)\# (\bT,c_w,c_w)$ which have a special line matching condition. Suppose that $u_{\Sigma}$ has $S$ marked points below the special line and $M$ marked points along the special line. The matching condition implies that $u_{\bT}$ has the same number of marked points below and on the special lines, respectively. If $\phi_{\bT}\in \pi_2(\theta,\theta')$ is the class of $u_{\bT}$, we observe that
\[
\mu(\phi_{\bT})=2n_w(\phi_{\bT})+\gr(\theta,\theta')\ge 2M+2S+\gr(\theta,\theta').
\]
In the composition $F_3^{c_w,c_w}\circ \Psi^{\bar\Sigma<\bT}$, only curves with $\theta'=\theta^-$ contribute. Therefore, the above shows that
\[
\mu(\phi_{\bT})\ge 2M+2S.
\]
By assumption, $\Psi^{\bar\Sigma<\bT}$ counts 0-dimensional moduli spaces, so we must have
\[
0=\mu(\phi_\Sigma)+\mu(\phi_{\bT})-2S-2M\ge \mu(\phi_\Sigma).
\]
By transversality, if $\phi_\Sigma$ admits a holomorphic representative, then $\mu(\phi_\Sigma)\ge 0$ with equality if and only if $\phi_\Sigma$ is a constant class. Hence, $F_3^{c_w,c_w}\circ \Psi^{\bar\Sigma<\bT}$ counts only curves which represent the constant class on $\Sigma$. It is easy to see by considering the diagram $(\bT,c_w,c_w)$ that this also implies that these classes are also constant on $\bT$. Hence, Equation~\eqref{eq:hypercube-3-handle-preliminary} follows.

Therefore, we may replace $F_3^{c_w,c_w}$ in the left side of the equation in the statement with $F_3^{c_w,c_w}\circ \Psi^{\bar\Sigma<\bT}$. We now consider the 1-parameter moduli spaces which would naturally be used to commute $\Psi^{\bar\Sigma<\bT}$ past both of the triangle maps used to define $\Psi^{(\bar c_w \bar \a)\to (c_w \bar \a)}_{(\bar c_w \bar \b)\to (c_w \bar \b)}$. These are the 1-dimensional moduli spaces of holomorphic triangles which have a special line which has an even number of marked points, which are $\bar\Sigma<\bT$ matched. Here, we view the triangle $\Delta$ as $[0,1]\times \R$ with the puncture $\{1\}\times \{0\}$ removed. The special lines are of the form $[0,1]\times \{s\}$, $s\in \R$. 

In these moduli spaces, there are ends corresponding to the composition
$\Psi^{(\bar c_w \bar \a)\to (c_w \bar \a)}_{(\bar c_w \bar \b)\to (c_w \bar \b)}\circ\Psi^{\bar\Sigma<\bT}$, as well as ends corresponding to a chain homotopy. There is also the possibility of an additional end where a holomorphic disk forms at the same height as the special line, and some the punctures along the special line may also degenerate into a holomorphic disk. We claim that these ends are constrained to a broken curve containing the following curves:
\begin{enumerate}[label=($u$-\arabic*), ref=$u$-\arabic*]
\item An index one disk on $(\bT, \bar c_w,  c_w)$ which has a single marked point, which is from $\bT$. \item A trivial disk on $(\bar\Sigma, \bar \bs, \bar \bs)$.
\item \label{triangle-3} A pair of triangles on $(\bT,\bar c_w,  c_w, c_w)\wedge (\bar\Sigma, \bar \as, \bar \bs,\bar \bs)$ which have marked points at the same height as the $\bar \bs$-$\bar \bs$ vertex of the triangle. Furthermore, $\bar\Sigma$ and $\bT$ marked points alternate along this line, and both the left-most and right-most punctures are from $\Sigma$.
\end{enumerate}
 See Figure~\ref{twist_fig:9}. We say that triangle pair satisfying~\eqref{triangle-3} is \emph{$(\bar\Sigma<\bT<\cdots<\bar\Sigma)$-matched}.

To establish that these are the only ends, we consider a potential degeneration where the special line on the triangle $\Delta$ has the same height as the $\bar \bs$-$\bar \bs$ puncture of $\Delta$.
 
We consider a limiting curve, and write $\psi^l_{\bar \Sigma}, \psi^l_\bT$ for the homology classes of the holomorphic triangles, and we write $\phi^r_{\bar \Sigma}$ and $\phi^r_{\bT}$ for the classes of the holomorphic disks that break off. We suppose that there are $n^r_{\bar \Sigma}$ and $n^r_\bT$ marked points of the disk classes along the special line, and $S^r=S^r_{\bar \Sigma}=S^r_{\bT}$ marked points below the special line (i.e. perfectly matched). The expected dimension of the disk class with this matching condition is
\begin{equation}
\mu(\phi^r_{\bar\Sigma})+\mu(\phi^r_{\bT})-n^r_{\bar \Sigma}-n^r_\bT+1-2 S^r.
\label{eq:expected-dimension-monster-disk}
\end{equation}
The above quantity must be 1 in a generic degeneration. 
On the other hand, since $\phi^r_{\bar \Sigma}$ is a disk class on $(\bar\Sigma, \bar \bs, \bar \bs)$, and the connected sum point is the basepoint, we have
 \[
 \mu(\phi^r_{\bar\Sigma})\ge 2n_p(\phi^r_{\bar\Sigma}) \ge 2(n^r_{\bar\Sigma}+S^r)
 \]
 by the absolute grading formula. Similarly, Lemma~\ref{lem:index-bound} implies that
 \[
  \mu(\phi^r_{\bT})\ge n^r_\bT+S^r.
 \]
   We observe that the expected dimension in Equation~\eqref{eq:expected-dimension-monster-disk} is at least
\[
n_{\bar\Sigma}^r+1+ S^r. 
\]
Hence $S^r=n_{\bar\Sigma}^r=0$ for a generic degeneration. 
Since $|n_{\bT}^r-n_{\bar\Sigma}^r|\le 1$, we must have $n_\bT^r\in \{0,1\}$. 

We therefore conclude that $\mu(\phi_{\bar\Sigma}^r)=n_p(\phi_{\bar\Sigma}^r)=0$ and $\mu(\phi_\bT^r)=1$, and $n_{\bT}^r=1$, for a generic degeneration.

 From this broken curve, we may extend the moduli space, by considering the $(\bar\Sigma<\bT<\dots<\bar\Sigma)$-matched moduli space space of triangles where the special line has height below the $\bar \bs$-$\bar \bs$ boundary puncture of the triangle $\Delta$.

 There is a natural map 
 \[
 \Psi^{\bar\Sigma<\bT<\cdots<\bar\Sigma}\colon \CF_\wedge(\bar \bs c_w,\bar \as \bar c_w)\to \CF_{\otimes}( \bar \bs c_w,\bar \as \bar c_w)
 \]
 which counts curve pairs of expected dimension 0 with the following matchings:
 \begin{enumerate}
 \item Punctures along the special line are $(\bar\Sigma<\bT<\dots< \bar\Sigma)$-matched.
 \item Punctures below the special line are perfectly matched.
 \item Punctures above the special line are trivially matched.
 \end{enumerate} 
 We claim that
 \[
\Psi^{\bar\Sigma<\bT<\cdots<\bar\Sigma}\simeq (\id\otimes \Phi_w^{\bar\Sigma})\circ  \Psi^{\bar\Sigma<\bT}.
 \]
 This is proven by splitting the special line into two special lines. The higher special line contains a single $\bar\Sigma$-marked point. The lower special line is $\bar\Sigma<\bT$ matched. Other marked points are perfectly matched below the lower special line, and trivially matched above it. (I.e. there is no change in matching as we cross the $\bar\Sigma$-matched line).

 \begin{figure}[ht]
 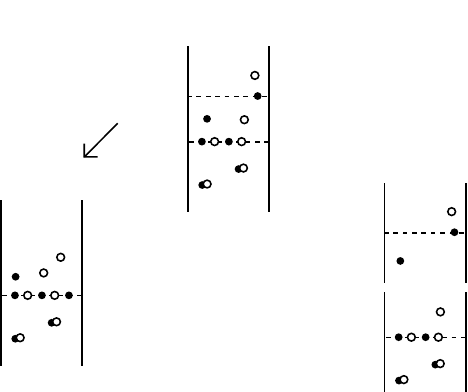
 \caption{A schematic of the homotopy $\Psi^{\bar\Sigma<\bT<\dots<\bar\Sigma}\simeq (\id\otimes \Phi^{\bar\Sigma}_w)\circ \Psi^{\bar\Sigma<\bT}$.}
 \label{twist_fig:10}
 \end{figure}

We count the ends of these 1-parameter moduli spaces with two special lines. The only non-canceling ends are as follows:
\begin{enumerate}
\item An index 1 disk breaking off while the special lines have finite, non-zero distance, giving a chain homotopy 
\item  The vertical distance between the two canceling ends approaches $0$ or $\infty$. The distance $0$ end contributes $\Psi^{\bar\Sigma<\bT<\dots<\bar\Sigma}$ while the distance $\infty$ end corresponds to the composition $(\id\otimes \Phi_w^{\bar \Sigma})\circ \Psi^{\bar\Sigma<\bT}$. 
\end{enumerate}
 Counting ends as above yields
 \[
 \Psi^{\bar\Sigma<\bT<\dots<\bar\Sigma}\simeq (\id\otimes\Phi_w^{\bar \Sigma}) \circ \Psi^{\bar\Sigma<\bT}. 
 \]

Finally, we put all the pieces together. When counting triangles with trivial matching conditions, it is straightforward to see that
\[
\Psi^{(\bar c_w \bar \a)\to (c_w \bar \a)}_{(\bar c_w \bar \b)\to (c_w \bar \b)}=\id\otimes \id,
\]
on the level of intersection points. 

Counting the ends of the above moduli spaces, we have  ends which contribute the map
\[F_3^{ c_w,  c_w}\circ \Psi^{(\bar c_w \bar \a)\to (c_w \bar \a)}_{(\bar c_w \bar \b)\to (c_w \bar \b)} \circ \Psi^{\bar\Sigma<\bT}.
\]
This is equal to
\[\left(\begin{array}{l} \xs\otimes \theta^-\mapsto \xs\\
\xs\otimes \theta^+\mapsto 0
\end{array} \right)\circ \Psi^{\bar\Sigma<\bT}
\]
Additionally, we have the contribution corresponding to extending the bigon degeneration. This gives and extra contribution of $\Psi^{\bar\Sigma<\bT}$ post-composed with
\[
\xs\otimes \theta^+\mapsto \Phi_w^{\bar\Sigma}(\xs).
\]
Summing this with the previous contribution gives the main statement.
\end{proof}

\begin{figure}[ht]
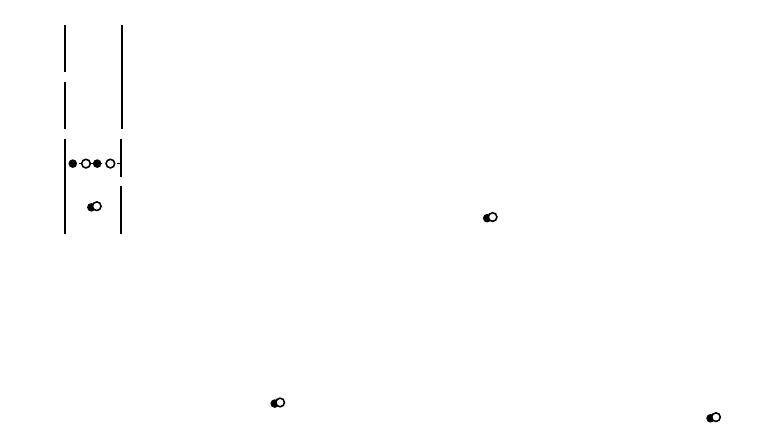
\caption{Canceling moduli space ends in Lemma~\ref{lem:cw-cwbar}. Squares denote special inputs of the holomorphic polygon maps. The curves in the square indicate 1-dimensional moduli spaces, and the arrows labeled with $\d$ indicate codimension 1 degenerations. Solid dots are $\bar\Sigma$-punctures and open dots are $\bT$-punctures. }
\label{twist_fig:9}
\end{figure}

\begin{rem} Note that the proof of Equation~\eqref{eq:hypercube-3-handle-preliminary} adapts easily to show that
\[
F_3^{\a,\a} \circ  \Psi^{\Sigma\# \bar \Sigma<\bT}=\Psi^{\bar \Sigma<\bT}\circ F_3^{\a,\a}.
\]
\end{rem}

\subsection{Central hypercubes}

 In this section, we describe how to commute $\Psi^{\Sigma\# \bar \Sigma<\bT}$ past the levels of $\tw_*$ which involve the quadrilateral maps. 
  The main result of this section is the following:

\begin{lem}
\label{lem:tensor-product-hyerpcube-middle}The map $\Psi^{\Sigma\# \bar \Sigma<\bT}$ extends to a homotopy equivalence of hyperboxes from the perfectly matched hyperbox
\[
\begin{tikzcd}
\CF(\as \bar c_w \bar \bs, \bs \bar c_w \bar \bs)
	\ar[r]
	\ar[d]
	\ar[dr,dashed]
&\CF(\as \bar c_w \bar \bs, \bs \bar c_w \bar \bs)
	\ar[d]
\\
\CF(\as \bar c_w \bar \bs, \bs D \bar \bs)
	\ar[r]
	\ar[d]
	\ar[dr,dashed]
&\CF(\as \bar c_w \bar \bs, \bs \tilde D \bar \bs)
	\ar[d]
\\
\CF(\as\bar c_w \bar \bs, \Ds_1)
	\ar[r]
	\ar[d]
	\ar[dr,dashed]
&
\CF(\as\bar c_w \bar \bs, \tilde{\Ds}_1)
\ar[d]
\\
\vdots\ar[d]\ar[dr,dashed]& \vdots \ar[d]\\
\CF(\as\bar c_w \bar \bs, \Ds_n)
	\ar[r]
	\ar[d]
	\ar[dr,dashed]
&
\CF(\as\bar c_w \bar \bs, \tilde{\Ds}_n)
	\ar[d]
\\
\CF(\as \bar c_w \bar \bs, \as D \bar \as)
	\ar[r]\ar[d] 
	\ar[dr,dashed]
&\CF(\as \bar c_w \bar \bs, \as \tilde D \bar \as)
\ar[d]
\\
\CF(\as \bar c_w \bar \bs, \as \bar c_w \bar \as)
\ar[r]
&\CF(\as \bar c_w \bar \bs, \as \bar c_w \bar \as)
\end{tikzcd}
\]  
to its trivially matched counterpart.
\end{lem}
\begin{proof}
We observe that all of the attaching curves which appear in the statement are disjoint from the connected sum tube which connects $\bar\Sigma$ and $\bT^2$. In fact, we may describe the entire 
hyperbox using a connected sum operation on hypercubes of attaching curves. This connected sum operation on hypercubes of attaching curves is a small modification of the one considered by Lipshitz, Ozsv\'{a}th and Thurston in \cite{LOTDoubleBranchedII}*{Section~3.5} (see \cite{ZemBordered}*{Section~9.5} for our present notation). If $\cL_{\b}=(\bs_\veps,\Theta_{\veps,\veps'})_{\veps\in \bE_n}$ and $\cL_{\b'}=(\bs'_{\nu}, \Theta_{\nu,\nu'})_{\nu\in \bE_m}$ are two hypercubes of attaching curves on Heegaard surfaces $\Sigma$ and $\Sigma'$, then $\cL_{\b}\# \cL_{\b'}$ is a hypercube of attaching curves of dimension $n+m$. The curves of $\cL_{\b}\# \cL_{\b'}$ are defined as
\[
\cB_{(\veps,\nu)}=\bs_{\veps}\cup \bs'_{\nu}
\]
with small translations of curves taken to achieve admissibility.  In our present setting, if the connected sum of $\Sigma$ and $\Sigma'$ is taken at either $w$ or $w'$ (deleting this basepoint), then the curves $\cB_{(\veps,\nu)}$ will also form an admissible diagram which is algebraically rigid, so we take the length 1 Floer chains of $\cL_{\b}\# \cL_{\b'}$ to be the top degree generators, and all higher length chains to vanish.

In the hyperbox in the statement of the lemma, each constituent hypercube of chain complexes is obtained by pairing the 0-dimensional hypercube $\as \bar c_w \bar \bs$ with a 2-dimensional hypercube of beta attaching curves. We view $\as \bar c_w \bar \bs$ as the connected sum of 0-dimensional hypercubes $\as \bar \bs$ with $\bar c_w$. In the top-most and bottom-most hypercubes, the beta hypercube of attaching curves are connected sums of a 0-dimensional hypercube on $\Sigma\# \bar \Sigma$ with a 2-dimensional hypercube on $\bT^2$. The remaining beta hypercubes are connected sums of a 1-dimensional hypercube on $\Sigma\# \bar \Sigma$ and a 1-dimensional hypercube on $\bT^2$.

 Work of the last author \cite{ZemBordered}*{Section~11.3} constructs, under suitable hypotheses, a homotopy equivalence between the hyperbox in the lemma statement (i.e. the hyperbox obtained by counting curve pairs with perfect matching) and the hyperbox obtained by counting curve pairs with trivial matching. The precise statement may be found in \cite{ZemBordered}*{Proposition~11.8} and involves the following somewhat technical definitions:
 \begin{enumerate}
 \item The alpha and beta hypercubes of attaching curves must be \emph{algebraically rigid} and consist of handleslide equivalent attaching curves.
 \item \label{item:technical-point} A technical condition about the placement of the connected sum points and the basepoint must be satisfied. See \cite{ZemBordered}*{Definition~11.4}. We say that a hypercube of algebraically rigid, handleslide equivalent attaching curves is \emph{graded} by the connected sum point $p$ if for all nonnegative classes $\phi\in \pi_2(\Theta_{\veps_1,\veps_2},\dots, \Theta_{\veps_{n-1},\veps_n})$ (where $\Theta_{\veps_i,\veps_{i+1}}$ are chains from the hypercube), we have
 \[
 \mu(\phi)\ge 2n_p(\phi).
 \]
 \end{enumerate}

We recall the statement of \cite{ZemBordered}*{Proposition~11.8}. Suppose that $\Sigma$ and $\Sigma'$ are Heegaard surfaces with distinguished points $p\in \Sigma$ and $p'\in \Sigma'$, and $\cL_{\a}$ and $\cL_{\b}$ are algebraically rigid hypercubes on $\Sigma$ and $\cL_{\a'}$ and $\cL_{\b'}$ are algebraically rigid hypercubes on $\Sigma'$. If $\cL_{\a}$ is graded by $p$ and $\cL_{\b'}$ is graded by $p'$, then the map
\[
\Psi^{\Sigma'<\Sigma}\colon \CF_{\wedge}(\cL_{\a}\# \cL_{\a'}, \cL_{\b}\# \cL_{\b'})\to  \CF_{\otimes}(\cL_{\a}\# \cL_{\a'}, \cL_{\b}\# \cL_{\b'})
\]
is a homotopy equivalence. Here $\CF_{\wedge}$ has differential counting perfectly matched curve pairs, while $\CF_{\otimes}$ counts trivially matched curve pairs.

We now return to the original setting of the present lemma. By construction, each hypercube is algebraically rigid. Furthermore, the alpha hypercubes $\as\bar \bs$ and $\bar c_w$ are vacuously graded by the connected sum point since they are 0-dimensional. All of the hypercubes on $\Sigma\# \bar \Sigma$ are graded by the connected sum point since the connected sum point is adjacent to the basepoint and the hypercubes are algebraically rigid. The beta hypercubes on $\bT^2$ are not all graded by the basepoint, but all of the alpha hypercubes which lie on $\bT^2$ are. The only difference in our present situation and \cite{ZemBordered} is that the hypercubes of attaching curves do not all consist of handleslide equivalent attaching curves. For example, $\bs\bar \bs$ is not handleslide equivalent to $\Ds_1^0$. Nonetheless, the proof of \cite{ZemBordered}*{Proposition~11.8} still goes through since Heegaard diagram formed from a pair of curves in the diagram represents an algebraically rigid diagram for a connected sum of $S^1\times S^2$'s and the grading assumption in \eqref{item:technical-point} is satisfied. Hence, the same argument as in \cite{ZemBordered}*{Proposition~11.8} provides the homotopy equivalence in the statement.
\end{proof}

Next, we view the above hyperbox as consisting of attaching curves on the disjoint union of $\Sigma\# \bar \Sigma$ and $\bT$, instead of the connected sum of these two diagrams. Note that compression commutes with pairing up to homotopy, so we may instead compress the hyperbox of attaching curves:
\begin{equation}
\begin{tikzcd}[row sep=.4cm, column sep=1.4cm]
\bs \bar c_w \bar \bs\ar[r]\ar[d] &\bs \bar c_w \bar \bs\ar[d]\\
\bs D \bar \bs \ar[r]\ar[d]& \bs \tilde D \bar \bs\ar[d]\\
\Ds_1\ar[r]\ar[d] &\tilde{\Ds}_1 \ar[d]\\
\vdots\ar[d]& \vdots\ar[d]\\
\Ds_n\ar[r]\ar[d] &\tilde{\Ds}_n \ar[d]\\
\as D \bar \as\ar[r]\ar[d]& \as \tilde D \bar \as\ar[d]\\
\as\bar c_w \bar \as\ar[r] &\as \bar c_w \bar \as
\end{tikzcd}\label{eq:expanded-hyperbox-tw-map}
\end{equation}
Note that we have not defined compression of hyperboxes of attaching curves. Nonetheless, the function-composition approach from Section~\ref{sec:hypercubes} generalizes to this setting with minimal care.
In our present case this amounts to viewing each of the $n+2$  levels in the above diagram
as a twisted complex in the Fukaya category. Each pair of adjacent levels in the diagram determines a morphism between 1-dimensional hypercubes of attaching curves. A pair of morphisms between hypercubes of attaching curves may be composed, so we compose iteratively all of these morphisms. Note that this is not the same as applying the ordinary $A_\infty$-composition map $\mu_{n+2}$. Rather, it is the result of applying $\mu_2$ repeatedly.

\begin{lem}
\label{lem:compressed-Fuk-hyperbox}
The compression of the diagram in Equation~\eqref{eq:expanded-hyperbox-tw-map} takes the following form:
\[
\begin{tikzcd}[row sep=.6cm, column sep=1.4cm]
\bs \bar c_w \bar \bs\ar[r, "1"]\ar[d, "\ve{X}|\theta^-"] \ar[dr, "\ve{X}'|\theta^+"] &\bs \bar c_w \bar \bs\ar[d,"\ve{X}|\theta^-"]\\
\as\bar c_w \bar \as\ar[r, "1"] &\as \bar c_w \bar \as
\end{tikzcd}
\]
Here,  $\ve{X}$ is the iterated composition of the vertical arrows on $\Sigma\# \bar \Sigma$, and $\theta^{\pm}$ are the intersection points of $\bar c_w\cap \bar c_w'$. Furthermore, the cycle $\ve{X}'$ is homologous to $\ve{X}$. 
\end{lem}
\begin{proof} We view the diagram as being obtained by composing $n+1$ morphisms between 1-dimensional hypercubes of Lagrangians. Viewing this diagram as a sequence of morphisms $\theta_{n+1},\dots, \theta_1$, we form the composition as
\[
\mu_2(\theta_{n+1},\dots, \mu_2(\theta_2,\theta_1)).
\]
We first claim that except for the final composition, there are no diagonal terms in any partial composition. To see this, consider the first composition; the same argument will apply to all later compositions except for the final. There are three holomorphic quadrilateral counts which contribute to the diagonal term, corresponding to the following quadrilaterals:
\begin{enumerate}
\item $(\bs \bar c_w \bar \bs, \bs \bar c_w \bar \bs, \bs \tilde D \bar \bs, \tilde \Ds_1)$.
\item $(\bs \bar c_w \bar \bs, \bs D \bar \bs, \bs \tilde D \bar \bs, \tilde \Ds_1)$.
\item $(\bs \bar c_w \bar \bs, \bs D \bar \bs, \Ds_1, \tilde \Ds_1)$.
\end{enumerate}
Each of these Heegaard quadruples decomposes as a connected sum where both factors have one set of attaching curves which is a small translation of an adjacent set. Since we are counting curves using trivial matchings, we use \cite{HHSZExact}*{Proposition~11.5} to see that no holomorphic quadrilaterals contribute. Similarly, using a model count on the genus 1 diagram, we observe that the two triangle compositions in the vertical direction coincide.

Therefore the final composition is of a diagram of the following form:
\begin{equation}
\begin{tikzcd}
\bs \bar c_w \bar \bs\ar[r]\ar[d, "\ve{X}|\theta"] &\bs \bar c_w \bar \bs\ar[d, "\ve{X}|\theta"]\\
\as D \bar \as\ar[r]\ar[d] &\as \tilde D \bar \as \ar[d]\\
\as\bar c_w \bar \as\ar[r] &\as \bar c_w \bar \as
\end{tikzcd}
\label{eq:hyperbox-lagrangians-to-be-compressed}
\end{equation}
where $\theta$ denotes the intersection point of $\bar c_w\cap D$ or $\bar c_w \cap \tilde{D}$.

As a first step, we consider the length 1 maps in the compression of Equation~\eqref{eq:hyperbox-lagrangians-to-be-compressed}. The length one chain on the left side of the compression is given by counting triangles on $(\Sigma\# \bar \Sigma, \bs \bar \as, \as \bar \as, \as \bar \as)$ and $(\bT^2, \bar c_w, D, \bar c_w)$. There is no matching condition, so we count triangles on the two diagrams separately and see that the triangle maps applied to $(\ve{X}|\theta, \Theta^+| \theta)$ have output $\ve{X}|\theta^-$. The computation on $\Sigma\# \bar \Sigma$ follows from nearest-point triangle counts, and the computation on $\bT^2$ is a direct and straightforward triangle count. 

We now consider the length 2 map of the compression. The small translate theorems \cite{HHSZExact}*{Propositions~11.1 and 11.5} prohibit holomorphic quadrilaterals except on 
\[
(\Sigma\# \bar \Sigma,\bs\bar \bs, \as \bar \as, \as \bar \as, \as \bar \as)\#(\bT^2,\bar c_w,D, \tilde D, \bar c_w).
\]
Since the diagram on $\Sigma\# \bar \Sigma$ has several adjacent sets of attaching curves which are small translates of one another, the only curves counted consist of pairs $(u,u')$ where $u$ is an index 0 quadrilateral on $\Sigma\# \bar \Sigma$ and $u'$ is an index $-1$ quadrilateral on $\bT^2$. Furthermore, the two quadrilaterals share the same almost complex structure parameter (viewing the moduli space of rectangles as $\R$). There is a unique such curve $u'$, shown in Figure~\ref{twist_fig:7}. Write $\ve{X}'|\theta^-$ for the output. 

To see that $\ve{X}'$ is homologous to $\ve{X}$ we deform the matching, by instead requiring the quadrilaterals $u$ and $u'$ to have almost $\ev_{\Box}(u)-\ev_{\Box}(u')=t$, where $\ev_{\Box}$ denotes the evaluation map to the moduli space of complex rectangles, identified with $\R$. We then send $t\to \infty$. The effect is to change $\ve{X}'$ by a boundary. This shows that $\ve{X}'$ is homologous to the composition 
\[
\mu_2(\mu_2(\ve{X}, \Theta_{\a\bar \a, \a \bar \a}), \Theta_{\a\bar \a, \a \bar \a}).
\]
 which we identify with $\ve{X}$ via the nearest point map. This completes the proof.
\end{proof}


\subsection{Proof of Theorem~\ref{thm:twist-map}}

We now complete our proof of the computation of the map $\tw_*$:
\begin{proof}[Proof of Theorem~\ref{thm:twist-map}] It remains to put
the pieces together to finish the proof. We build a 3-dimensional hyperbox to realize the homotopy. The back face is the map $\tw_*$, computed using perfect matching between $\Sigma\# \bar \Sigma$ and $\bT^2$.

 The front face is defined very similarly, using the trivial (i.e. tensor product) matchings, with a modification of one cube. In place of
\[
\begin{tikzcd}[column sep=2cm, row sep=1.5cm, labels=description]
\CF_\wedge( \bar c_w \bar \bs,\bar c_w \bar \as)
	\ar[r,"\id"]
	\ar[d, "F_3^{c_w,c_w}\Psi^{(\bar c_w \bar \a)\to (c_w \bar \a)}_{(\bar c_w \bar \b)\to (c_w \bar \b)}"]& 
\CF_\wedge(\bar c_w \bar \bs,\bar c_w \bar \as)
	\ar[d, "F_3^{c_w,c_w}\Psi^{(\bar c_w \bar \a)\to (c_w \bar \a)}_{(\bar c_w \bar \b)\to (c_w \bar \b)}"]
\\
\CF(\bar \bs, \bar \as)\ar[r, "\id"]
& \CF(\bar \bs, \bar \as)
\end{tikzcd}
\]
we have
\[
\begin{tikzcd}[column sep=1.5cm, row sep=1cm, labels=description]
\CF_\otimes(\bar c_w \bar \bs,\bar c_w \bar \as)
	\ar[r,"\id"]
	\ar[d, "\Psi'"]&
 \CF_\otimes(\bar c_w \bar \bs,\bar c_w \bar \as )
	\ar[d, "\Psi'"]
\\
\CF(\bar \bs, \bar \as   )\ar[r, "\id"]
& \CF(\bar \bs, \bar \as)
\end{tikzcd}
\]
Where 
\[
\Psi'=\left(\begin{array}{l} \xs\otimes \theta^-\mapsto \xs\\
\xs\otimes \theta^+\mapsto \Phi_w(\xs)
\end{array} \right)\circ \Psi^{\bar\Sigma<\bT}
\]
Lemmas~\ref{lem:1-handle-c_w-commutation}, \ref{lem:tensor-product-hyerpcube-middle} and Equation~\eqref{eq:hypercube-3-handle-preliminary} show that the hyperbox relations are satisfied for appropriate choices of length 3 maps on the interior of the cube.

Next, we compress the front face of the hypercube described above. We pair the hypercube from Lemma~\ref{lem:compressed-Fuk-hyperbox} with $\as \bar c_w \bar \bs$ and then stack with the 2-dimensional hypercube
\[
\begin{tikzcd}\CF(\bar\Sigma,\bar \bs, \bar \as)\ar[r, "\id"]\ar[d, "\eta"]& \CF(\bar \Sigma, \bar \bs, \bar \as) \ar[d, "\eta"]\\
\CF(\Sigma,\as,\bs)\ar[r, "\id"]& \CF(\Sigma,\as,\bs)
\end{tikzcd}
\] The so-constructed cube has the form
\[
\begin{tikzcd}\CF(\as, \bs)\ar[r, "\id"]\ar[dr, "\Phi_w \iota"] \ar[d, "\iota"] & \CF(\as, \bs)\ar[d, "\iota"]\\
\CF(\as, \bs)\ar[r, "\id"] & \CF(\as, \bs),
\end{tikzcd}
\]
completing the proof.
\end{proof}

\subsection{Knots and links}\label{subsec:twist-knots}

\label{sec:twisting-knots-links}

If $K$ is a knot, then \emph{a-priori} there is an action $\tw^\mu_*$ corresponding to rotating a neighborhood of $K$ in the meridional direction. (Rotation in the longitudinal direction is well known to induce an interesting automorphism by work of Sarkar \cite{SarkarMovingBasepoints}). Despite the interesting potential for a new map, we have the following:

\begin{prop} 
The map $\tw^{\mu}_*$ is chain homotopic to $\id$ on $\cCFKI(K)$. The same holds for links. 
\end{prop}
\begin{proof}We consider the stabilized model of the involution. For knots and links, the doubling arcs $\ds$ have endpoints at a point $p$ which is not one of the basepoints of the knot. Hence, applying $\tw^{\mu}$ to each Heegaard diagram has no effect on any of the attaching curves.
\end{proof}

\begin{figure}[h]
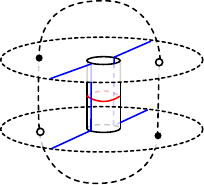
\caption{A doubled knot diagram (for the unknot). Meridianal twisting preserves the diagram. }
\end{figure}

\begin{rem} We can also recover the above result on the non-stabilized doubling models of the knot involution, considered in Section~\ref{sec:doubling-knots} and \cite{HHSZExact}*{Section~2.2}, though the argument is more subtle. Here, we double at one of the basepoints (say $w\in \Sigma$), with the $\ds$ curves having endpoints at $w$. In this model, after attaching 1-handles, we move $w$ to the position of $\bar z$. The map $\tw^\mu$ still changes the curves $\Ds$, but since there is no basepoint in the tube region (unlike in the case of closed 3-manifolds), the effect on the level of diagrams is to act by a surface diffeomorphism which is isotopic to the identity relative to the basepoints on $\Sigma\# \bar \Sigma$. 
\end{rem}

%% file: twist_fig1.pdf_tex
\begingroup%
  \makeatletter%
  \providecommand\color[2][]{%
    \errmessage{(Inkscape) Color is used for the text in Inkscape, but the package 'color.sty' is not loaded}%
    \renewcommand\color[2][]{}%
  }%
  \providecommand\transparent[1]{%
    \errmessage{(Inkscape) Transparency is used (non-zero) for the text in Inkscape, but the package 'transparent.sty' is not loaded}%
    \renewcommand\transparent[1]{}%
  }%
  \providecommand\rotatebox[2]{#2}%
  \newcommand*\fsize{\dimexpr\f@size pt\relax}%
  \newcommand*\lineheight[1]{\fontsize{\fsize}{#1\fsize}\selectfont}%
  \ifx\svgwidth\undefined%
    \setlength{\unitlength}{221.83820152bp}%
    \ifx\svgscale\undefined%
      \relax%
    \else%
      \setlength{\unitlength}{\unitlength * \real{\svgscale}}%
    \fi%
  \else%
    \setlength{\unitlength}{\svgwidth}%
  \fi%
  \global\let\svgwidth\undefined%
  \global\let\svgscale\undefined%
  \makeatother%
  \begin{picture}(1,0.69449752)%
    \lineheight{1}%
    \setlength\tabcolsep{0pt}%
    \put(0,0){\includegraphics[width=\unitlength,page=1]{twist_fig1.pdf}}%
    \put(0.3747111,0.56682597){\color[rgb]{0,0,1}\makebox(0,0)[lt]{\lineheight{1.25}\smash{\begin{tabular}[t]{l}$\ds$\end{tabular}}}}%
    \put(0.88805677,0.60241408){\color[rgb]{0,0,1}\makebox(0,0)[lt]{\lineheight{1.25}\smash{\begin{tabular}[t]{l}$\tilde\ds$\end{tabular}}}}%
  \end{picture}%
\endgroup%

%% file: twist_fig5.pdf_tex
\begingroup%
  \makeatletter%
  \providecommand\color[2][]{%
    \errmessage{(Inkscape) Color is used for the text in Inkscape, but the package 'color.sty' is not loaded}%
    \renewcommand\color[2][]{}%
  }%
  \providecommand\transparent[1]{%
    \errmessage{(Inkscape) Transparency is used (non-zero) for the text in Inkscape, but the package 'transparent.sty' is not loaded}%
    \renewcommand\transparent[1]{}%
  }%
  \providecommand\rotatebox[2]{#2}%
  \newcommand*\fsize{\dimexpr\f@size pt\relax}%
  \newcommand*\lineheight[1]{\fontsize{\fsize}{#1\fsize}\selectfont}%
  \ifx\svgwidth\undefined%
    \setlength{\unitlength}{405.57373808bp}%
    \ifx\svgscale\undefined%
      \relax%
    \else%
      \setlength{\unitlength}{\unitlength * \real{\svgscale}}%
    \fi%
  \else%
    \setlength{\unitlength}{\svgwidth}%
  \fi%
  \global\let\svgwidth\undefined%
  \global\let\svgscale\undefined%
  \makeatother%
  \begin{picture}(1,0.70382282)%
    \lineheight{1}%
    \setlength\tabcolsep{0pt}%
    \put(0,0){\includegraphics[width=\unitlength,page=1]{twist_fig5.pdf}}%
    \put(0.32925765,0.19077608){\makebox(0,0)[lt]{\lineheight{1.25}\smash{\begin{tabular}[t]{l}$w$\end{tabular}}}}%
    \put(0,0){\includegraphics[width=\unitlength,page=2]{twist_fig5.pdf}}%
    \put(0.25296278,0.53385793){\makebox(0,0)[lt]{\lineheight{1.25}\smash{\begin{tabular}[t]{l}$p$\end{tabular}}}}%
    \put(0.38047159,0.53492213){\makebox(0,0)[lt]{\lineheight{1.25}\smash{\begin{tabular}[t]{l}$w$\end{tabular}}}}%
    \put(0,0){\includegraphics[width=\unitlength,page=3]{twist_fig5.pdf}}%
    \put(0.31329792,0.62185282){\color[rgb]{0,0,1}\makebox(0,0)[lt]{\lineheight{1.25}\smash{\begin{tabular}[t]{l}$\ds$\end{tabular}}}}%
    \put(0,0){\includegraphics[width=\unitlength,page=4]{twist_fig5.pdf}}%
    \put(0.63229666,0.53213515){\makebox(0,0)[lt]{\lineheight{1.25}\smash{\begin{tabular}[t]{l}$p$\end{tabular}}}}%
    \put(0.7598056,0.53492213){\makebox(0,0)[lt]{\lineheight{1.25}\smash{\begin{tabular}[t]{l}$w$\end{tabular}}}}%
    \put(0,0){\includegraphics[width=\unitlength,page=5]{twist_fig5.pdf}}%
    \put(0.69263184,0.62185282){\color[rgb]{0,0,1}\makebox(0,0)[lt]{\lineheight{1.25}\smash{\begin{tabular}[t]{l}$\tilde{\ds}$\end{tabular}}}}%
    \put(0,0){\includegraphics[width=\unitlength,page=6]{twist_fig5.pdf}}%
    \put(0.20056136,0.19100584){\color[rgb]{1,0,0}\makebox(0,0)[lt]{\lineheight{1.25}\smash{\begin{tabular}[t]{l}$c_p$\end{tabular}}}}%
    \put(0,0){\includegraphics[width=\unitlength,page=7]{twist_fig5.pdf}}%
    \put(0.74710981,0.14771891){\makebox(0,0)[lt]{\lineheight{1.25}\smash{\begin{tabular}[t]{l}$w$\end{tabular}}}}%
    \put(0,0){\includegraphics[width=\unitlength,page=8]{twist_fig5.pdf}}%
    \put(0.61558465,0.26952007){\color[rgb]{0,0,1}\makebox(0,0)[lt]{\lineheight{1.25}\smash{\begin{tabular}[t]{l}$D$\end{tabular}}}}%
    \put(0,0){\includegraphics[width=\unitlength,page=9]{twist_fig5.pdf}}%
    \put(0.65362226,0.01027431){\color[rgb]{1,0,1}\makebox(0,0)[lt]{\lineheight{1.25}\smash{\begin{tabular}[t]{l}$\tilde{D}$\end{tabular}}}}%
    \put(0.80492452,0.19971219){\color[rgb]{1,0,0}\makebox(0,0)[lt]{\lineheight{1.25}\smash{\begin{tabular}[t]{l}$\bar c_w$\end{tabular}}}}%
    \put(0.80360737,0.12430635){\color[rgb]{1,0,0}\makebox(0,0)[lt]{\lineheight{1.25}\smash{\begin{tabular}[t]{l}$c_w$\end{tabular}}}}%
    \put(0.40304419,0.29859724){\makebox(0,0)[lt]{\lineheight{1.25}\smash{\begin{tabular}[t]{l}$\bar\Sigma$\end{tabular}}}}%
    \put(0.42888095,0.14241683){\makebox(0,0)[lt]{\lineheight{1.25}\smash{\begin{tabular}[t]{l}$\Sigma$\end{tabular}}}}%
    \put(0,0){\includegraphics[width=\unitlength,page=10]{twist_fig5.pdf}}%
    \put(0.23770366,0.28204555){\color[rgb]{0,0,1}\makebox(0,0)[lt]{\lineheight{1.25}\smash{\begin{tabular}[t]{l}$D$\end{tabular}}}}%
  \end{picture}%
\endgroup%

%% file: twist_fig7.pdf_tex
\begingroup%
  \makeatletter%
  \providecommand\color[2][]{%
    \errmessage{(Inkscape) Color is used for the text in Inkscape, but the package 'color.sty' is not loaded}%
    \renewcommand\color[2][]{}%
  }%
  \providecommand\transparent[1]{%
    \errmessage{(Inkscape) Transparency is used (non-zero) for the text in Inkscape, but the package 'transparent.sty' is not loaded}%
    \renewcommand\transparent[1]{}%
  }%
  \providecommand\rotatebox[2]{#2}%
  \newcommand*\fsize{\dimexpr\f@size pt\relax}%
  \newcommand*\lineheight[1]{\fontsize{\fsize}{#1\fsize}\selectfont}%
  \ifx\svgwidth\undefined%
    \setlength{\unitlength}{174.7030358bp}%
    \ifx\svgscale\undefined%
      \relax%
    \else%
      \setlength{\unitlength}{\unitlength * \real{\svgscale}}%
    \fi%
  \else%
    \setlength{\unitlength}{\svgwidth}%
  \fi%
  \global\let\svgwidth\undefined%
  \global\let\svgscale\undefined%
  \makeatother%
  \begin{picture}(1,0.95200474)%
    \lineheight{1}%
    \setlength\tabcolsep{0pt}%
    \put(0,0){\includegraphics[width=\unitlength,page=1]{twist_fig7.pdf}}%
    \put(0.61395214,0.26698994){\color[rgb]{0,0,1}\makebox(0,0)[lt]{\lineheight{1.25}\smash{\begin{tabular}[t]{l}$\bar c_w$\end{tabular}}}}%
    \put(0.12376047,0.27253206){\color[rgb]{0,0,1}\makebox(0,0)[lt]{\lineheight{1.25}\smash{\begin{tabular}[t]{l}$\bar c_w'$\end{tabular}}}}%
    \put(0.49661182,0.18366684){\color[rgb]{1,0,0}\makebox(0,0)[lt]{\lineheight{1.25}\smash{\begin{tabular}[t]{l}$D$\end{tabular}}}}%
    \put(0.83754842,0.5842289){\color[rgb]{1,0.5254902,0}\makebox(0,0)[lt]{\lineheight{1.25}\smash{\begin{tabular}[t]{l}$\tilde{D}$\end{tabular}}}}%
    \put(0.72441241,0.6606722){\makebox(0,0)[lt]{\lineheight{1.25}\smash{\begin{tabular}[t]{l}$w$\end{tabular}}}}%
    \put(0,0){\includegraphics[width=\unitlength,page=2]{twist_fig7.pdf}}%
  \end{picture}%
\endgroup%

%% file: twist_fig8.pdf_tex
\begingroup%
  \makeatletter%
  \providecommand\color[2][]{%
    \errmessage{(Inkscape) Color is used for the text in Inkscape, but the package 'color.sty' is not loaded}%
    \renewcommand\color[2][]{}%
  }%
  \providecommand\transparent[1]{%
    \errmessage{(Inkscape) Transparency is used (non-zero) for the text in Inkscape, but the package 'transparent.sty' is not loaded}%
    \renewcommand\transparent[1]{}%
  }%
  \providecommand\rotatebox[2]{#2}%
  \newcommand*\fsize{\dimexpr\f@size pt\relax}%
  \newcommand*\lineheight[1]{\fontsize{\fsize}{#1\fsize}\selectfont}%
  \ifx\svgwidth\undefined%
    \setlength{\unitlength}{173.57680535bp}%
    \ifx\svgscale\undefined%
      \relax%
    \else%
      \setlength{\unitlength}{\unitlength * \real{\svgscale}}%
    \fi%
  \else%
    \setlength{\unitlength}{\svgwidth}%
  \fi%
  \global\let\svgwidth\undefined%
  \global\let\svgscale\undefined%
  \makeatother%
  \begin{picture}(1,0.95491554)%
    \lineheight{1}%
    \setlength\tabcolsep{0pt}%
    \put(0,0){\includegraphics[width=\unitlength,page=1]{twist_fig8.pdf}}%
    \put(0.72425744,0.62875944){\makebox(0,0)[lt]{\lineheight{1.25}\smash{\begin{tabular}[t]{l}$w$\end{tabular}}}}%
    \put(0,0){\includegraphics[width=\unitlength,page=2]{twist_fig8.pdf}}%
    \put(0.28085552,0.52456521){\makebox(0,0)[lt]{\lineheight{1.25}\smash{\begin{tabular}[t]{l}$\bar c_w$\end{tabular}}}}%
    \put(0.66973236,0.28691808){\makebox(0,0)[lt]{\lineheight{1.25}\smash{\begin{tabular}[t]{l}$c_w$\end{tabular}}}}%
  \end{picture}%
\endgroup%

%% file: twist_fig10.pdf_tex
\begingroup%
  \makeatletter%
  \providecommand\color[2][]{%
    \errmessage{(Inkscape) Color is used for the text in Inkscape, but the package 'color.sty' is not loaded}%
    \renewcommand\color[2][]{}%
  }%
  \providecommand\transparent[1]{%
    \errmessage{(Inkscape) Transparency is used (non-zero) for the text in Inkscape, but the package 'transparent.sty' is not loaded}%
    \renewcommand\transparent[1]{}%
  }%
  \providecommand\rotatebox[2]{#2}%
  \newcommand*\fsize{\dimexpr\f@size pt\relax}%
  \newcommand*\lineheight[1]{\fontsize{\fsize}{#1\fsize}\selectfont}%
  \ifx\svgwidth\undefined%
    \setlength{\unitlength}{224.18988235bp}%
    \ifx\svgscale\undefined%
      \relax%
    \else%
      \setlength{\unitlength}{\unitlength * \real{\svgscale}}%
    \fi%
  \else%
    \setlength{\unitlength}{\svgwidth}%
  \fi%
  \global\let\svgwidth\undefined%
  \global\let\svgscale\undefined%
  \makeatother%
  \begin{picture}(1,0.83905751)%
    \lineheight{1}%
    \setlength\tabcolsep{0pt}%
    \put(0,0){\includegraphics[width=\unitlength,page=1]{twist_fig10.pdf}}%
    \put(0.17189963,0.55934039){\makebox(0,0)[lt]{\lineheight{1.25}\smash{\begin{tabular}[t]{l}$\d$\end{tabular}}}}%
    \put(0,0){\includegraphics[width=\unitlength,page=2]{twist_fig10.pdf}}%
    \put(0.75608346,0.56159601){\makebox(0,0)[lt]{\lineheight{1.25}\smash{\begin{tabular}[t]{l}$\d$\end{tabular}}}}%
  \end{picture}%
\endgroup%

%% file: twist_fig9.pdf_tex
\begingroup%
  \makeatletter%
  \providecommand\color[2][]{%
    \errmessage{(Inkscape) Color is used for the text in Inkscape, but the package 'color.sty' is not loaded}%
    \renewcommand\color[2][]{}%
  }%
  \providecommand\transparent[1]{%
    \errmessage{(Inkscape) Transparency is used (non-zero) for the text in Inkscape, but the package 'transparent.sty' is not loaded}%
    \renewcommand\transparent[1]{}%
  }%
  \providecommand\rotatebox[2]{#2}%
  \newcommand*\fsize{\dimexpr\f@size pt\relax}%
  \newcommand*\lineheight[1]{\fontsize{\fsize}{#1\fsize}\selectfont}%
  \ifx\svgwidth\undefined%
    \setlength{\unitlength}{374.27461989bp}%
    \ifx\svgscale\undefined%
      \relax%
    \else%
      \setlength{\unitlength}{\unitlength * \real{\svgscale}}%
    \fi%
  \else%
    \setlength{\unitlength}{\svgwidth}%
  \fi%
  \global\let\svgwidth\undefined%
  \global\let\svgscale\undefined%
  \makeatother%
  \begin{picture}(1,0.5438152)%
    \lineheight{1}%
    \setlength\tabcolsep{0pt}%
    \put(0,0){\includegraphics[width=\unitlength,page=1]{twist_fig9.pdf}}%
    \put(0.15880498,0.2639178){\makebox(0,0)[lt]{\lineheight{1.25}\smash{\begin{tabular}[t]{l}$\Ss{\bar \bs \bar c_w}$\end{tabular}}}}%
    \put(0.07553652,0.29824484){\makebox(0,0)[rt]{\lineheight{1.25}\smash{\begin{tabular}[t]{r}$\Ss{\bar \as \bar c_w}$\end{tabular}}}}%
    \put(0.07553652,0.40150841){\makebox(0,0)[rt]{\lineheight{1.25}\smash{\begin{tabular}[t]{r}$\Ss{\bar \as \bar c_w}$\end{tabular}}}}%
    \put(0.07570019,0.48134229){\makebox(0,0)[rt]{\lineheight{1.25}\smash{\begin{tabular}[t]{r}$\Ss{\bar \as c_w}$\end{tabular}}}}%
    \put(0.15880498,0.33642595){\makebox(0,0)[lt]{\lineheight{1.25}\smash{\begin{tabular}[t]{l}$\Ss{\bar \bs c_w}$\end{tabular}}}}%
    \put(0.15880498,0.4318798){\makebox(0,0)[lt]{\lineheight{1.25}\smash{\begin{tabular}[t]{l}$\Ss{\bar \bs c_w}$\end{tabular}}}}%
    \put(0,0){\includegraphics[width=\unitlength,page=2]{twist_fig9.pdf}}%
    \put(0.40133502,0.02540277){\makebox(0,0)[lt]{\lineheight{1.25}\smash{\begin{tabular}[t]{l}$\Ss{\bar \bs \bar c_w}$\end{tabular}}}}%
    \put(0.31765066,0.05972959){\makebox(0,0)[rt]{\lineheight{1.25}\smash{\begin{tabular}[t]{r}$\Ss{\bar \as \bar c_w}$\end{tabular}}}}%
    \put(0.31765066,0.1629932){\makebox(0,0)[rt]{\lineheight{1.25}\smash{\begin{tabular}[t]{r}$\Ss{\bar \as \bar c_w}$\end{tabular}}}}%
    \put(0.31782549,0.24543054){\makebox(0,0)[rt]{\lineheight{1.25}\smash{\begin{tabular}[t]{r}$\Ss{\bar \as c_w}$\end{tabular}}}}%
    \put(0.40133502,0.09791094){\makebox(0,0)[lt]{\lineheight{1.25}\smash{\begin{tabular}[t]{l}$\Ss{\bar \bs c_w}$\end{tabular}}}}%
    \put(0.40133502,0.19336476){\makebox(0,0)[lt]{\lineheight{1.25}\smash{\begin{tabular}[t]{l}$\Ss{\bar \bs c_w}$\end{tabular}}}}%
    \put(0,0){\includegraphics[width=\unitlength,page=3]{twist_fig9.pdf}}%
    \put(0.27799804,0.37761623){\makebox(0,0)[lt]{\lineheight{1.25}\smash{\begin{tabular}[t]{l}$\d$\end{tabular}}}}%
    \put(0,0){\includegraphics[width=\unitlength,page=4]{twist_fig9.pdf}}%
    \put(0.43478332,0.3762651){\makebox(0,0)[lt]{\lineheight{1.25}\smash{\begin{tabular}[t]{l}$\d$\end{tabular}}}}%
    \put(0,0){\includegraphics[width=\unitlength,page=5]{twist_fig9.pdf}}%
    \put(0.66316217,0.26346252){\makebox(0,0)[lt]{\lineheight{1.25}\smash{\begin{tabular}[t]{l}$\Ss{\bar \bs \bar c_w}$\end{tabular}}}}%
    \put(0.57947782,0.30580457){\makebox(0,0)[rt]{\lineheight{1.25}\smash{\begin{tabular}[t]{r}$\Ss{\bar \as \bar c_w}$\end{tabular}}}}%
    \put(0.57947782,0.40906824){\makebox(0,0)[rt]{\lineheight{1.25}\smash{\begin{tabular}[t]{r}$\Ss{\bar \as \bar c_w}$\end{tabular}}}}%
    \put(0.57965409,0.49150561){\makebox(0,0)[rt]{\lineheight{1.25}\smash{\begin{tabular}[t]{r}$\Ss{\bar \as c_w}$\end{tabular}}}}%
    \put(0.66316217,0.34799392){\makebox(0,0)[lt]{\lineheight{1.25}\smash{\begin{tabular}[t]{l}$\Ss{\bar \bs c_w}$\end{tabular}}}}%
    \put(0.66316217,0.43943995){\makebox(0,0)[lt]{\lineheight{1.25}\smash{\begin{tabular}[t]{l}$\Ss{\bar \bs c_w}$\end{tabular}}}}%
    \put(0.16481868,0.3017706){\makebox(0,0)[lt]{\lineheight{1.25}\smash{\begin{tabular}[t]{l}$\Ss{\Theta|\theta^+}$\end{tabular}}}}%
    \put(0,0){\includegraphics[width=\unitlength,page=6]{twist_fig9.pdf}}%
    \put(0.93667533,0.08000924){\makebox(0,0)[lt]{\lineheight{1.25}\smash{\begin{tabular}[t]{l}$\Ss{\bar \bs \bar c_w}$\end{tabular}}}}%
    \put(0.85299023,0.11834389){\makebox(0,0)[rt]{\lineheight{1.25}\smash{\begin{tabular}[t]{r}$\Ss{\bar \as \bar c_w}$\end{tabular}}}}%
    \put(0.85299023,0.22160756){\makebox(0,0)[rt]{\lineheight{1.25}\smash{\begin{tabular}[t]{r}$\Ss{\bar \as \bar c_w}$\end{tabular}}}}%
    \put(0.85314629,0.3040449){\makebox(0,0)[rt]{\lineheight{1.25}\smash{\begin{tabular}[t]{r}$\Ss{\bar \as c_w}$\end{tabular}}}}%
    \put(0.93667533,0.16454064){\makebox(0,0)[lt]{\lineheight{1.25}\smash{\begin{tabular}[t]{l}$\Ss{\bar \bs c_w}$\end{tabular}}}}%
    \put(0.93667533,0.25197897){\makebox(0,0)[lt]{\lineheight{1.25}\smash{\begin{tabular}[t]{l}$\Ss{\bar \bs c_w}$\end{tabular}}}}%
    \put(0,0){\includegraphics[width=\unitlength,page=7]{twist_fig9.pdf}}%
    \put(0.78470848,0.37761621){\makebox(0,0)[lt]{\lineheight{1.25}\smash{\begin{tabular}[t]{l}$\d$\end{tabular}}}}%
    \put(0,0){\includegraphics[width=\unitlength,page=8]{twist_fig9.pdf}}%
    \put(0.67116773,0.3105697){\makebox(0,0)[lt]{\lineheight{1.25}\smash{\begin{tabular}[t]{l}$\Ss{\Theta|\theta^-}$\end{tabular}}}}%
    \put(0.94199384,0.12252585){\makebox(0,0)[lt]{\lineheight{1.25}\smash{\begin{tabular}[t]{l}$\Ss{\Theta|\theta^+}$\end{tabular}}}}%
    \put(0,0){\includegraphics[width=\unitlength,page=9]{twist_fig9.pdf}}%
  \end{picture}%
\endgroup%

%% file: twist_fig11.pdf_tex
\begingroup%
  \makeatletter%
  \providecommand\color[2][]{%
    \errmessage{(Inkscape) Color is used for the text in Inkscape, but the package 'color.sty' is not loaded}%
    \renewcommand\color[2][]{}%
  }%
  \providecommand\transparent[1]{%
    \errmessage{(Inkscape) Transparency is used (non-zero) for the text in Inkscape, but the package 'transparent.sty' is not loaded}%
    \renewcommand\transparent[1]{}%
  }%
  \providecommand\rotatebox[2]{#2}%
  \newcommand*\fsize{\dimexpr\f@size pt\relax}%
  \newcommand*\lineheight[1]{\fontsize{\fsize}{#1\fsize}\selectfont}%
  \ifx\svgwidth\undefined%
    \setlength{\unitlength}{97.69787159bp}%
    \ifx\svgscale\undefined%
      \relax%
    \else%
      \setlength{\unitlength}{\unitlength * \real{\svgscale}}%
    \fi%
  \else%
    \setlength{\unitlength}{\svgwidth}%
  \fi%
  \global\let\svgwidth\undefined%
  \global\let\svgscale\undefined%
  \makeatother%
  \begin{picture}(1,0.90487954)%
    \lineheight{1}%
    \setlength\tabcolsep{0pt}%
    \put(0,0){\includegraphics[width=\unitlength,page=1]{twist_fig11.pdf}}%
    \put(0.74399439,0.85281655){\makebox(0,0)[lt]{\lineheight{1.25}\smash{\begin{tabular}[t]{l}$K$\end{tabular}}}}%
  \end{picture}%
\endgroup%